\documentclass{amsart}

\usepackage{amssymb,amsmath,amscd,graphicx,amsthm}

\newtheorem{theorem}{Theorem}
\newtheorem{lemma}[theorem]{Lemma}
\newtheorem{proposition}[theorem]{Proposition}

\newtheorem{conjecture}[theorem]{Conjecture}

\theoremstyle{definition}

\theoremstyle{remark}
\newtheorem{remark}[theorem]{Remark}

\numberwithin{equation}{section}
\numberwithin{theorem}{section}

\def\A{{\mathcal A}}
\def\AA{{\mathbb A}}

\def\C{{\mathbb C}}
\def\CC{{\mathcal C}}
\def\D{{\mathfrak d}}

\def\FF{{\mathcal F}}
\def\FFF{{\mathcal F}}
\def\G{{\mathcal G}}
\def\H{\mathcal H}

\def\M{{\mathcal M}}
\def\N{\mathbb N}
\def\O{{\mathcal O}}
\def\P{{\mathcal P}}
\def\PP{{\mathcal P}}
\def\Phhi{\Phi}  
\def\Pssi{\Psi}

\def\SS{{\mathcal S}}
\def\TE{{\mathcal T}}
\def\T{{\mathcal T}}

\def\UU{\overline{\A}}

\def\X{{\mathcal X}}
\def\Y{{\mathcal Y}}
\def\Z{{\mathbb Z}}

\def\d{\mathbf d}

\def\g{\mathfrak g}
\def\gl{\mathfrak g\mathfrak l}
\def\sl{\mathfrak s\mathfrak l}

\def\h{\mathfrak h}

\def\one{\mathbf 1}

\def\phhi{{\varphi}}
\def\pssi{{\psi}}

\def\thetta{{\theta}}

\def\wB{{\widetilde{B}}}

\def\hB{{\widehat{B}}}

\def\wx{{\widetilde{\bf x}}}
\def\x{{\bf x}}

\def\Ad{\operatorname{Ad}}

\def\End{\operatorname{End}}

\def\Id{{\operatorname {Id}}}

\def\Mat{\operatorname{Mat}}

\def\Poi{{\{\cdot,\cdot\}}}

\def\Tr{\operatorname{Tr}}

\def\bangle{\genfrac{\langle}{\rangle}{0pt}{}}

\def\deg{{\operatorname{deg}}}
\def\diag{\operatorname{diag}}
\def\dim{\operatorname{dim}}
\def\dnabla{{\raisebox{2pt}{$\bigtriangledown$}}\negthinspace}

\def\rank{\operatorname{rank}}
\def\romon{\mbox{\tiny\rm I}}
\def\romtw{\mbox{\tiny\rm II}}
\def\romth{\mbox{\tiny\rm III}}
\def\romfo{\mbox{\tiny\rm IV}}
\def\romfi{\mbox{\tiny\rm V}}
\def\romsi{\mbox{\tiny\rm VI}}

\def\unabla{\boldsymbol\nabla}

\def\:{{:\ }}

\begin{document}

\dedicatory
{Dedicated to the memory of Andrei Zelevinsky}

\title[Exotic cluster structure on $SL_n$]
{Exotic cluster structures on $SL_n$: the Cremmer--Gervais case}

\author{M. Gekhtman}

\address{Department of Mathematics, University of Notre Dame, Notre Dame,
IN 46556}
\email{mgekhtma@nd.edu}

\author{M. Shapiro}
\address{Department of Mathematics, Michigan State University, East Lansing,
MI 48823}
\email{mshapiro@math.msu.edu}

\author{A. Vainshtein}
\address{Department of Mathematics \& Department of Computer Science, University of Haifa, Haifa,
Mount Carmel 31905, Israel}
\email{alek@cs.haifa.ac.il}

\begin{abstract} This is the second paper in the series of papers dedicated to
the study of natural cluster structures in the rings of regular functions on simple complex Lie groups and Poisson--Lie
structures compatible with these cluster structures.
According to our main conjecture, each class in the Belavin--Drinfeld classification of Poisson--Lie structures on $\G$ corresponds to a cluster structure in $\O(\G)$. 
We have shown before that this conjecture holds for any $\G$ in the case of the standard Poisson--Lie structure   and for all Belavin-Drinfeld classes in $SL_n$, $n<5$.
In this paper we establish it for the Cremmer--Gervais Poisson--Lie structure on $SL_n$, which is the least similar to the standard one. 
\end{abstract}


\maketitle

\tableofcontents

\section{Introduction}
In \cite{GSVMMJ} we initiated a systematic study of
multiple cluster structures in the rings of regular functions
on simple Lie groups following an approach developed and implemented 
in \cite{GSV1, GSV5, GSVb} 
for constructing cluster structures 
on algebraic varieties. It is based on the notion of a Poisson bracket compatible with a cluster structure. The key point is that if an algebraic 
Poisson variety $\left ( \mathcal{M}, \Poi\right )$ possesses a  coordinate chart
that consists of regular functions whose logarithms have pairwise constant Poisson brackets, 
then one can try to use this chart to define a cluster structure $\CC_\M$ compatible with $\Poi$. Algebraic 
structures corresponding to $\CC_\M$ (the cluster algebra and the upper cluster algebra)
are closely related 
to the ring $\O(\M)$ of regular functions on $\mathcal{M}$. 
In fact, under certain rather mild conditions, $\O(\M)$ can be obtained by tensoring one of these
algebras with $\C$.

This construction was applied in \cite[Ch.~4.3]{GSVb} to double Bruhat cells in semisimple Lie groups
equipped with (the restriction of) the {\em standard\/} Poisson--Lie structure. It was shown that
the resulting cluster structure coincides with the one built in \cite{CAIII}. Recall that it was proved in
\cite{CAIII} that the corresponding upper cluster algebra is isomorphic to the ring of regular functions on 
the double Bruhat cell. Moreover, any cluster variable is a regular function on the whole Lie
group. Taking this into account, and since the open double Bruhat cell is dense in the corresponding Lie group, 
one can equip  the ring of regular functions on the Lie group with the same cluster structure.
The standard Poisson--Lie structure is a particular case of Poisson--Lie structures corresponding to quasi-triangular
Lie bialgebras. Such structures are associated with solutions to the classical Yang--Baxter equation (CYBE).
Their complete classification was obtained by Belavin and Drinfeld in \cite{BD}. In \cite{GSVMMJ} we conjectured that any such solution gives rise to a compatible cluster structure on the Lie group
  and provided several examples supporting this conjecture by showing that it holds true for the class of the standard Poisson--Lie  structure in any simple complex Lie group, and for 
 the whole Belavin--Drinfeld classification in $SL_n$ for $n=2,3,4$. We call the cluster structures
 associated with the non-trivial Belavin--Drinfeld data {\it exotic}.

In this paper, we prove the conjecture of \cite{GSVMMJ} in the case of the 
Cremmer--Gervais Poisson structure on $SL_n$. We chose to consider this case because, in a sense, the Poisson structure in question differs the most from the standard: the discrete data featured in the Belavin--Drinfeld classification is trivial in the standard case and has the ``maximal size'' in the Cremmer--Gervais case. Our result allows to equip $SL_n$ with a new cluster structure.
In fact, we construct a new cluster structure $\CC_{CG}$ for the affine space $\Mat_n$ of $n\times n$ matrices and then adapt it
to the $GL_n$ and $SL_n$ case using first localization and then specialization.
In Section 2 below, we collect the necessary information on cluster algebras and Poisson--Lie groups and formulate the main conjecture from \cite{GSVMMJ}. In Section 3, we present the definition of the Cremmer--Gervais Poisson bracket, formulate our main result, introduce the cluster structure $\CC_{CG}$ and outline the proof of the main theorem by breaking it into a series of intermediate results about $\CC_{CG}$ and showing how the final part of the proof ---  the fact that the ring of regular functions is isomorphic to the upper cluster algebra associated with $\CC_{CG}$ ---
 follows from them. 

The first result we need is the existence of a log-canonical coordinate system on $GL_n$ with respect to the natural extension of the Cremmer-Gervais Poisson structure. The proof is contained in Section 4 as a corollary of a more general claim: using an embedding of $GL_n$ into its double as a diagonal subgroup, we prove that a certain family of functions is log-canonical on the double with respect to the appropriate Poisson--Lie structure. We then obtain the coordinate system we are interested in
 as a restriction  of this family to the diagonal subgroup. The resulting collection of functions serves as an initial cluster 
for $\CC_{CG}$. The fact that this cluster and the quiver $Q_{CG}$ introduced in Section 3 form an initial seed of a 
cluster structure compatible with the Cremmer--Gervais Poisson structure is proved in Section 5. As a byproduct of our construction we obtain explicit expressions for the Casimir functions of the Cremmer--Gervais bracket.
 To the best of our knowledge, this is the first time these Casimirs  have been presented in the literature.
  In Section 6, we prove that cluster variables in the clusters adjacent to the initial one are regular as well,
and that $\CC_{CG}$ admits a global toric action mandated by the main conjecture. It is worth mentioning that, like in Section 4, the regularity statement follows from a more general one valid in the double. This hints at a possibility of endowing the double with a cluster structure associated with the Cremmer--Gervais bracket. 
 
 The most involved part of the paper is Section 7, where we construct two sequences of cluster transformations needed to implement the final part of the proof presented in Section 3. The first sequence is used to perform an induction step from $\Mat_n$ to $\Mat_{n-1}$, and the second one allows to realize a linear map on $\Mat_n$ that is anti-Poisson with respect to the Cremmer--Gervais bracket. To construct these sequences we use two
 crucial features of the initial seed.  First, the initial quiver can be embedded into a torus; second, the cluster variables forming the initial cluster, as well as those obtained in the process of applying cluster transformation forming our sequences, can be represented as minors of a certain
 matrix of dimensions quadratic in $n$. The corresponding cluster transformations can be modeled on classical determinantal identities and translation invariance properties of this matrix. The final section contains some technical results on cluster algebras that are used at various stages of the proof.

Main results of this paper were presented without proofs in \cite{PNAS}. 
There we also proved
that, for $SL_3$, the cluster algebra
and the upper cluster algebra  corresponding to the Cremmer--Gervais cluster structure do not coincide.
We conjecture that in contrast with the case of the standard cluster structure,
the gap between the cluster and the upper cluster algebra exists for any exotic cluster structure.

\section{Cluster structures and Poisson--Lie groups}
\label{SecPrel}

\subsection{Cluster structures}
We start with the basics on cluster algebras of geometric type. The definition that we present
below is not the most general one, see, e.g.,
\cite{FZ2, CAIII} for a detailed exposition. In what follows, we will write $[i,j]$ for an interval
$\{i, i+1, \ldots , j\}$ in $\mathbb{N}$, and we will denote $[1, n]$ by $[n]$. Besides, given
a sequence $\{i_1,\dots,i_n\}$ and a subset $J\subseteq[n]$, we will write $i_J$ for the subsequence
$\{i_j\: j\in J\}$.
 
The {\em coefficient group\/} $\PP$ is a free multiplicative abelian
group of finite rank $m$ with generators $g_1,\dots, g_m$.
An {\em ambient field\/}  is
the field $\FFF$ of rational functions in $n$ independent variables with
coefficients in the field of fractions of the integer group ring
$\Z\PP=\Z[g_1^{\pm1},\dots,g_m^{\pm1}]$ (here we write
$x^{\pm1}$ instead of $x,x^{-1}$).

A {\em seed\/} (of {\em geometric type\/}) in $\FFF$ is a pair
$\Sigma=(\x,\widetilde{B})$,
where $\x=(x_1,\dots,x_n)$ is a transcendence basis of $\FFF$ over the field of
fractions of $\Z\PP$ and $\widetilde{B}$ is an $n\times(n+m)$ integer matrix
whose principal part $B$ is skew-symmetrizable (recall that the principal part of a rectangular matrix  
is its maximal leading square submatrix). Matrices $B$ and $\wB$ are 
called the {\it exchange matrix\/} and the {\it extended exchange matrix}, respectively.

The $n$-tuple  $\x$ is called a {\em cluster\/}, and its elements
$x_1,\dots,x_n$ are called {\em cluster variables\/}. Denote
$x_{n+i}=g_i$ for $i\in [m]$. We say that
$\widetilde{\x}=(x_1,\dots,x_{n+m})$ is an {\em extended
cluster\/}, and $x_{n+1},\dots,x_{n+m}$ are {\em stable
variables\/}. It is convenient to think of $\FFF$ as
of the field of rational functions in $n+m$ independent variables
with rational coefficients. 

In what follows, we will only deal with the case when the exchange matrix is skew-symmetric. In this
situation the extended exchange  
matrix can be conveniently represented by a {\it quiver\/} $Q=Q(\wB)$. It is a directed graph on the 
vertices $1,\dots,n+m$ corresponding to all variables; the vertices corresponding to stable variables are called
{\it frozen}. Each entry $b_{ij}>0$ of the matrix $\wB$ gives rise to $b_{ij}$ edges going from
the vertex $i$ to the vertex $j$; each such edge is denoted $i\to j$. Clearly, $\wB$ can be restored uniquely from $Q$, and we will eventually write $\Sigma=(\x,Q)$ instead of $\Sigma=(\x,\wB)$.
Note that $B$ is irreducible if and only if the 
subquiver of $Q$ induced by non-frozen vertices is connected.

Given a seed as above, the {\em adjacent cluster\/} in direction $k\in [n]$
is defined by
$$
\x_k=(\x\setminus\{x_k\})\cup\{x'_k\},
$$
where the new cluster variable $x'_k$ is given by the {\em exchange relation}
\begin{equation}\label{exchange}
x_kx'_k=\prod_{\substack{1\le i\le n+m\\  b_{ki}>0}}x_i^{b_{ki}}+
       \prod_{\substack{1\le i\le n+m\\  b_{ki}<0}}x_i^{-b_{ki}};
\end{equation}
here, as usual, the product over the empty set is assumed to be
equal to~$1$. In terms of the quiver $Q$, the exchange relation can be rewritten as
\begin{equation*}
x_kx'_k=\prod_{k\to i}x_i+\prod_{i\to k}x_i.
\end{equation*}

We say that $\wB'$ is
obtained from $\wB$ by a {\em matrix mutation\/} in direction $k$
and
write $\wB'=\mu_k(\wB)$ 
 if
\[
b'_{ij}=\begin{cases}
         -b_{ij}, & \text{if $i=k$ or $j=k$;}\\
                 b_{ij}+\displaystyle\frac{|b_{ik}|b_{kj}+b_{ik}|b_{kj}|}2,
                                                  &\text{otherwise.}
        \end{cases}
\]
It can be easily verified that 
$B$ and $B'$ are skew-symmetric simultaneously, and hence the quiver $Q'=\mu_k(Q)$ is well defined.
The corresponding {\it quiver mutation\/} can be described as follows: (i) all edges $i\to k$ and all edges $k \to i$ are reversed; (ii) for any two-edge path $i\to k\to j$ in $Q$, $e(i,j)$ edges $i\to j$ are added, where $e(i,j)$ is the number of two-edge paths $i\to k\to j$; 
(iii) every edge $j\to i$ (if it exists) annihilates with an edge $i\to j$.

Given a seed $\Sigma=(\x,\widetilde{B})$, we say that a seed
$\Sigma'=(\x',\widetilde{B}')$ is {\em adjacent\/} to $\Sigma$ (in direction
$k$) if $\x'$ is adjacent to $\x$ in direction $k$ and $\widetilde{B}'=
\mu_k(\widetilde{B})$. Two seeds are {\em mutation equivalent\/} if they can
be connected by a sequence of pairwise adjacent seeds. 
The set of all seeds mutation equivalent to $\Sigma$ is called the {\it cluster structure\/} 
(of geometric type) in $\FFF$ associated with $\Sigma$ and denoted by $\CC(\Sigma)$; in what follows, 
we usually write $\CC(\wB)$, or even just $\CC$ instead. 

Following \cite{FZ2, CAIII}, we associate
with $\CC(\wB)$ two algebras of rank $n$ over the {\it ground ring\/} $\AA$, $\Z\subseteq\AA \subseteq\Z\P$:
the {\em cluster algebra\/} $\A=\A(\CC)=\A(\wB)$, which 
is the $\AA$-subalgebra of $\FF$ generated by all cluster
variables in all seeds in $\CC(\wB)$, and the {\it upper cluster algebra\/}
$\UU=\UU(\CC)=\UU(\wB)$, which is the intersection of the rings of Laurent polynomials over $\AA$ in cluster variables
taken over all seeds in $\CC(\wB)$. The famous {\it Laurent phenomenon\/} \cite{FZ3}
claims the inclusion $\A(\CC)\subseteq\UU(\CC)$. The natural choice of the ground ring for the geometric type
is the polynomial ring in stable variables $\AA=\Z\P_+=\Z[x_{n+1},\dots,x_{n+m}]$; this choice is assumed unless
explicitly stated otherwise. 

Let $V$ be a quasi-affine variety over $\C$, $\C(V)$ be the field of rational functions on $V$, and
$\O(V)$ be the ring of regular functions on $V$. Let $\CC$ be a cluster structure in $\FF$ as above.
Assume that $\{f_1,\dots,f_{n+m}\}$ is a transcendence basis of $\C(V)$. Then the map $\varphi: x_i\mapsto f_i$,
$i\in [n+m]$, can be extended to a field isomorphism $\varphi: \FF_\C\to \C(V)$,  
where $\FF_\C=\FF\otimes\C$ is obtained from $\FF$ by extension of scalars.
The pair $(\CC,\varphi)$ is called a cluster structure {\it in\/}
$\C(V)$ (or just a cluster structure {\it on\/} $V$), $\{f_1,\dots,f_{n+m}\}$ is called an extended cluster in
 $(\CC,\varphi)$.
Occasionally, we omit direct indication of $\varphi$ and say that $\CC$ is a cluster structure on $V$. 
A cluster structure $(\CC,\varphi)$ is called {\it regular\/}
if $\varphi(x)$ is a regular function for any cluster variable $x$. 
The two algebras defined above have their counterparts in $\FF_\C$ obtained by extension of scalars; they are
denoted $\A_\C$ and $\UU_\C$.
If, moreover, the field isomorphism $\varphi$ can be restricted to an isomorphism of 
$\A_\C$ (or $\UU_\C$) and $\O(V)$, we say that 
$\A_\C$ (or $\UU_\C$) is {\it naturally isomorphic\/} to $\O(V)$.

The following statement is an analog of Proposition~3.37 in \cite{GSVb}.

\begin{proposition}\label{regfun}
Let $V$ be a Zariski open subset in $\C^{n+m}$ and $(\CC=\CC(\wB),\varphi)$ be a cluster structure in $\C(V)$  
with $n$ cluster and $m$ stable variables such that


{\rm(i)} there exists an extended cluster $\wx=(x_1,\dots,x_{n+m})$ in $\CC$ such that $\varphi(x_i)$ is
regular on $V$ for $i\in [n+m]$;

{\rm(ii)} for any cluster variable $x_k'$, $k\in [n]$, obtained via the exchange relation~\eqref{exchange} 
applied to $\wx$, $\varphi(x_k')$ is regular on $V$,

{\rm(iii)} for any stable variable $x_{n+i}$, $i\in [m]$, $\varphi(x_{n+i})$ vanishes at some point of $V$;

{\rm(iv)} each regular function on $V$ belongs to $\varphi(\UU_\C(\CC))$.

\noindent Then $\CC$ is a regular cluster structure and $\UU_\C(\CC)$ is naturally isomorphic to $\O(V)$.
\end{proposition}

Note that 
Proposition~\ref{regfun} essentially coincides with ~\cite[Prop.~2.1]{GSVMMJ}. The latter one contains a redundant
condition $\rank\wB=n$ that may be omitted.

\subsection{Compatible Poisson brackets and toric actions}
Let $\Poi$ be a Poisson bracket on the ambient field $\FFF$, and $\CC$ be a cluster structure in $\FFF$. 
We say that the bracket and the cluster structure are {\em compatible\/} if, for any extended
cluster $\widetilde{\x}=(x_1,\dots,x_{n+m})$,  one has
\begin{equation}\label{cpt}
\{x_i,x_j\}=\omega_{ij} x_ix_j,
\end{equation}
where $\omega_{ij}\in\Z$ are
constants for all $i,j\in[n+m]$. The matrix
$\Omega^{\widetilde \x}=(\omega_{ij})$ is called the {\it coefficient matrix\/}
of $\Poi$ (in the basis $\widetilde \x$); clearly, $\Omega^{\widetilde \x}$ is
skew-symmetric.

A complete characterization of Poisson brackets compatible with a given cluster structure $\CC=\CC(\wB)$ in the case $\rank\wB=n$ is given in \cite{GSV1}, see also \cite[Ch.~4]{GSVb}. In particular, the following statement is an 
immediate corollary of Theorem~1.4  in \cite{GSV1}.

\begin{proposition}\label{Bomega}
The following two statements are equivalent:

 {\rm(i)} $\rank\wB=n$, and the bracket
$\Poi$ is compatible with $\CC(\wB)$; 

{\rm(ii)} $\wB\Omega=[\Delta\;\; 0]$ for a non-degenerate diagonal matrix $\Delta$. 
\end{proposition}

Clearly, the notion of compatibility and the result stated above extend to Poisson brackets on $\FF_\C$ without any changes.

A different description of compatible Poisson brackets on $\FF_\C$ is based on the notion of a toric action is suggested
in \cite[Proposition~2.2]{GSSV}. We will not use this description below, so let us concentrate on the properties of the toric 
action that are relevant to this paper.

 Fix an arbitrary extended cluster
$\wx=(x_1,\dots,x_{n+m})$ and define a {\it local toric action\/} of rank $r$ as the map 
$\TE^W_{\d}:\FF_\C\to
\FF_\C$ given on the generators of $\FF_\C=\C(x_1,\dots,x_{n+m})$ by the formula 
\begin{equation}
\TE^W_{\d}(x_i)=x_i \prod_{\alpha=1}^r d_\alpha^{w_{i\alpha}},\quad i\in [n+m],\qquad
\d=(d_1,\dots,d_r)\in (\C^*)^r,
\label{toricact}
\end{equation}
where $W=(w_{i\alpha})$ is an integer $(n+m)\times r$ {\it weight matrix\/} of full rank, and extended naturally to the whole $\FF_\C$. 

Let $\wx'$ be another extended cluster, then the corresponding local toric action defined by the weight matrix $W'$
is {\it compatible\/} with the local toric action \eqref{toricact} if the following diagram is commutative for
any fixed $\d\in (\C^*)^r$:
$$
\begin{CD}
\FF_\C=\C(\wx) @>>> \FF_\C=\C(\wx')\\
@V\TE^W_{\d} VV @VV \TE^{W'}_{\d}V\\
\FF_\C=\C(\wx) @>>> \FF_\C=\C(\wx')
\end{CD}
$$
(here the horizontal arrows are induced by $x_i\mapsto x'_i$ for $i\in [n+m]$). If local toric actions at all clusters are compatible, they define a {\it global toric action\/} $\TE_{\d}$ on $\FF_\C$ called the extension of the local toric action \eqref{toricact}. Lemma~2.3 in \cite{GSV1} claims that \eqref{toricact} extends 
to a unique global action of $(\C^*)^r$  if and only if $\wB W = 0$. Therefore, if $\rank\wB=n$, then the maximal possible rank of a global toric action equals $m$. Any global toric action can be obtained from a toric action of
the maximal rank by setting some of $d_i$'s equal to~$1$.

\subsection{Poisson--Lie groups}
Let $\G$ be a Lie group equipped with a Poisson bracket $\Poi$.
$\G$ is called a {\em Poisson--Lie group\/}
if the multiplication map
$$
\G\times \G \ni (x,y) \mapsto x y \in \G
$$
is Poisson. Perhaps, the most important class of Poisson--Lie groups
is the one associated with classical R-matrices. 

Let $\g$ be the Lie algebra of $\G$
equipped with a nondegenerate invariant bilinear form, 
$\mathfrak{t}\in \g\otimes\g$ be the corresponding Casimir element.
For an arbitrary element $r=\sum_i a_i\otimes b_i\in\g\otimes\g$ denote
\[
[[r,r]]=\sum_{i,j} [a_i,a_j]\otimes b_i\otimes b_j+\sum_{i,j} a_i\otimes [b_i,a_j]\otimes b_j+
\sum_{i,j} a_i\otimes a_j\otimes [ b_i,b_j]
\]
and $r^{21}=\sum_i b_i\otimes a_i$.

A {\em classical R-matrix} is an element $r\in \g\otimes\g$ that satisfies
{\em the classical Yang-Baxter equation (CYBE)} 
\begin{equation}
[[r, r]] =0
\label{CYBE}
\end{equation}
together with the condition $r + r^{21} = \mathfrak{t}$.

Given a solution $r$ to \eqref{CYBE}, one can construct explicitly the Poisson--Lie bracket on the Lie group $\G$.
Choose a basis $\{I_i\}$ in $\g$, and let $\partial^R_i$ and $\partial^L_i$ be the right
and the left invariant vector fields on $\G$ whose values at the unit element equal $I_i$. Write $r$ as
$r=\sum_{i,j} r_{ij}I_i\otimes I_j$, then the Poisson--Lie bracket on $\G$ is given by
\begin{equation}\label{sklya}
\{f_1,f_2\}_r=\sum_{i,j}r_{ij}\left(\partial^R_i f_1\partial^R_j f_2-
\partial^L_i f_1\partial^L_j f_2\right),
\end{equation}
see \cite[Proposition 4.1.4]{KoSo}. This bracket is called the {\it Sklyanin bracket\/} corresponding to $r$.

Denote by 
$\langle \ , \ \rangle$ the Killing form on $\g$, and by 
$\nabla^R$, $\nabla^L$ the right and
left gradients of functions on $\G$ with respect to the Killing form. Let $\pi_{>0}, \pi_{<0}$ be projections of  
$\g$ onto subalgebras spanned by positive and negative roots, $\pi_0$ be the projection onto the Cartan 
subalgebra $\h$. Define also $\pi_{\geq 0}=\pi_{>0} + \pi_0$ and $\pi_{\leq 0}=\pi_{<0} + \pi_0$.

It will be convenient to rewrite formula (\ref{sklya}) for $\Poi_r$ as
\begin{equation}
\{f_1,f_2\}_r = \langle R_\pm(\nabla^L f_1), \nabla^L f_2 \rangle - \langle R_\pm(\nabla^R f_1), \nabla^R f_2 \rangle,
\label{sklyabra}
\end{equation}
where $R_\pm\in \End\g$ are given by $\langle R_+ \eta, \zeta \rangle =
- \langle R_- \zeta, \eta \rangle =\langle r, \eta\otimes\zeta \rangle$.

The images of $\g$ under $R_\pm$ are Lie subalgebras of $\g$; we denote them by $\g_\pm$,
and the corresponding Lie subgroups of $\G$, by $\G_\pm$.

Following \cite{r-sts}, let us recall the construction of {\em the Drinfeld double}. The double of $\g$ is 
$D(\g)=\g  \oplus \g$ equipped with an invariant nondegenerate bilinear form
$$
\langle\langle (\xi,\eta), (\xi',\eta')\rangle\rangle = \langle \xi, \xi'\rangle - \langle \eta, \eta'\rangle. 
$$
Define subalgebras $\D_\pm$ of $D(\g)$ by
\begin{equation}\label{ddeco}
\D_+=\{( \xi,\xi)\: \xi \in\g\}, \quad \D_-=\{ (R_+(\xi),R_-(\xi))\: \xi \in\g\}.
\end{equation}
Then $\D_\pm$ are isotropic subalgebras of $D(\g)$ and $D(\g)= \D_+ \dot + \D_-$. In other words,
$(D(\g), \D_+, \D_-)$ is {\em a Manin triple}. Then the operator $R_D= \pi_{\D_+} - \pi_{\D_-}$ can be used to define 
a Poisson--Lie structure on $D(\G)=\G\times \G$, the double of the group $\G$, via
\begin{equation}
\{f_1,f_2\}_D = \frac{1}{2}\left (\langle\langle R_D(\dnabla^L f_1), \dnabla{^L} f_2 \rangle\rangle 
- \langle\langle R_D(\dnabla^R f_1), \dnabla^R f_2 \rangle\rangle \right),
\label{sklyadouble}
\end{equation}
where $\dnabla^R$ and $\dnabla^L$ are right and left gradients with respect to $\langle\langle \cdot ,\cdot \rangle\rangle$.
Restriction of this bracket to $\G$ identified with the diagonal subgroup of $D(\G)$ (whose Lie algebra is $\D_+$) 
coincides with the Poisson--Lie bracket $\Poi_r$ on $\G$.

The classification of classical R-matrices for simple complex Lie groups was given by Belavin and Drinfeld in \cite{BD}. Let $\G$ be a simple complex Lie group,
$\g$ be the corresponding Lie algebra, $\h$ be its Cartan subalgebra,
$\Phi$ be the root system associated with $\g$, $\Phi^+$ be the set of positive roots, and $\Delta\subset \Phi^+$ be the 
set of positive simple roots. 
A {\em Belavin--Drinfeld triple} $T=(\Gamma_1,\Gamma_2, \gamma)$
consists of two subsets $\Gamma_1,\Gamma_2$ of $\Delta$ and an isometry $\gamma:\Gamma_1\to\Gamma_2$ nilpotent in the 
following sense: for every $\alpha \in \Gamma_1$ there exists $m\in\mathbb{N}$ such that $\gamma^j(\alpha)\in \Gamma_1$ for $j\in [0,m-1]$, but $\gamma^m(\alpha)\notin \Gamma_1$. The isometry $\gamma$ extends in a natural way to a map between root systems $\Phi_1, \Phi_2$ generated by $\Gamma_1, \Gamma_2$. This allows one to define a partial ordering on $\Phi$: $\alpha \prec_T \beta$ if $\beta=\gamma^j(\alpha)$ for some $j\in \mathbb{N}$. 

Select root vectors $e_\alpha \in\g$ satisfying 
$(e_{-\alpha},e_\alpha)=1$. According to the Belavin--Drinfeld classification, the following is true (see, e.g., \cite[Chap.~3]{CP}).

\begin{proposition}\label{bdclass}
{\rm(i)} Every classical R-matrix is equivalent {\rm(}up to an action of $\sigma\otimes\sigma$, where $\sigma$ is an 
automorphism of $\g$\/{\rm)} to the one of the form
\begin{equation}
\label{rBD}
r= r_0 + \sum_{\alpha\in \Phi^+} e_{-\alpha}\otimes e_\alpha + \sum_{\stackrel{\alpha \prec_T \beta}{\alpha,\beta\in\Phi^+}} e_{-\alpha}\wedge e_\beta.
\end{equation}

{\rm(ii)} $r_0\in \h\otimes\h$ in \eqref{rBD} satisfies 
\begin{equation}
(\gamma(\alpha)\otimes \Id )r_0 + (\Id\otimes \alpha )r_0 = 0  
\label{r01}
\end{equation}
for any $\alpha\in\Gamma_1$ and
\begin{equation}
r_0 + r_0^{21} = \mathfrak{t}_0,
\label{r02}
\end{equation}
where $\mathfrak{t}_0$ is the $\h\otimes\h$-component of $\mathfrak{t}$. 

{\rm(iii)} Solutions $r_0$ to \eqref{r01}, \eqref{r02}
form a linear space of dimension $\frac{k_T(k_T-1)}{2}$, where 
$k_T= | \Delta \setminus \Gamma_1 |$; more precisely, define
\begin{equation*}
\h_T=\{ h\in\h \ : \ \alpha(h)=\beta(h)\ \mbox{if}\ \alpha\prec_T\beta\}, 
\end{equation*}
then $\dim\h_T=k_T$, and if $r_0'$ is a fixed solution of \eqref{r01}, \eqref{r02}, then
every other solution has a form $r_0=r_0' + s$, where $s$ is an arbitrary element of $\h_T\wedge\h_T$.
\end{proposition}

We say that two classical R-matrices that have a form \eqref{rBD} belong to the same {\em Belavin--Drinfeld class\/}
if they are associated with the same Belavin--Drinfeld triple.

\subsection{Main conjecture}
Let $\G$ 
be a simple complex Lie group. Given a Belavin--Drinfeld triple $T$ for $\G$,
define the torus $\H_T=\exp \h_T\subset\G$.

We conjecture (see \cite[Conjecture~3.2]{GSVMMJ}) that there exists a classification of regular cluster structures on $\G$ that is completely
parallel to the Belavin--Drinfeld classification.

\begin{conjecture}
\label{ulti}
Let $\G$ be a simple complex Lie group.
For any Belavin--Drinfeld triple $T=(\Gamma_1,\Gamma_2,\gamma)$ there exists a cluster structure
$\CC_T$ on $\G$ such that

{\rm (i)}
the number of stable variables is $2k_T$, and the corresponding extended exchange matrix has a full rank;

{\rm (ii)} $\CC_T$ is regular, and the corresponding upper cluster algebra $\UU_\C(\CC_T)$ 
is naturally isomorphic to $\O(\G)$;

{\rm (iii)} the global toric action of $(\mathbb{C}^*)^{2k_T}$ on $\C(\G)$ 
is generated by the action
of $\H_T\times \H_T$ on $\G$ given by $(H_1, H_2)(X) = H_1 X H_2$;

 {\rm (iv)} for any solution of CYBE that belongs to the Belavin--Drinfeld class specified  by $T$, the corresponding Sklyanin bracket is compatible with $\CC_T$;

{\rm (v)} a Poisson--Lie bracket on $\G$ is compatible with $\CC_T$ only if it is a scalar multiple
of the Sklyanin bracket associated with a solution of CYBE that belongs to the Belavin--Drinfeld class specified  by $T$.
\end{conjecture}

The different parts of the conjecture are related to each other in the following way 
(see \cite[Theorem~4.1]{GSVMMJ}).

\begin{theorem}
\label{partial}
Let $T=(\Gamma_1, \Gamma_2,\gamma)$ be a Belavin--Drinfeld triple and $\CC_T$ be a cluster structure
on $\G$.
Suppose that assertions {\rm(i)} and {\rm(iii)} of Conjecture {\rm\ref{ulti}} are valid and that 
assertion {\rm(iv)} is valid for one particular R-matrix in the Belavin--Drinfeld 
class specified  by $T$. Then {\rm(iv)} and {\rm(v)} are valid for the whole Belavin--Drinfeld class specified  
by $T$.
\end{theorem}

\section{Main result and the outline of the proof}
The Belavin--Drinfeld data (triple, class) is said to be {\it trivial\/} if $\Gamma_1=\Gamma_2=\varnothing$.
In this case,   $\H_T=\H$ is the Cartan
subgroup in $\G$, $r_0$ in \eqref{rBD} can be chosen to be $\frac{1}{2} \mathfrak{t}_0$ and the corresponding 
$R_\pm\in \End\g$ becomes
$$
R_+=\frac{1}{2} \pi_{0} +  \pi_{> 0},\quad  R_-=-\frac{1}{2} \pi_{0} -  \pi_{< 0}.
$$
The resulting Poisson bracket is called   the {\it standard Poisson--Lie structure\/} on $\g$.
Conjecture \ref{ulti} in this case was verified in \cite{GSVMMJ}.

In this paper we will consider $\G=SL_n$ and the Belavin--Drinfeld data that is "the farthest" from the 
trivial data, namely, $\Gamma_1=\{\alpha_2, \dots, \alpha_{n-1}\}$,  
 $\Gamma_2=\{\alpha_1, \dots, \alpha_{n-2}\}$ and $\gamma(\alpha_i) = \alpha_{i-1}$
for $i\in [2,n-1]$.  The resulting Poisson--Lie bracket on $SL_n$ is called the
{\em Cremmer--Gervais structure\/}. The main result of this paper is

\begin{theorem}\label{CGtrue}
Conjecture {\rm \ref{ulti}} is valid for the Cremmer--Gervais Poisson--Lie structure.
\end{theorem}

\begin{remark}
1. Usually, the name Cremmer--Gervais is associated with the Belavin--Drinfeld data 
opposite to the one described above, that is, with $\Gamma_1$ and $\Gamma_2$ switched and $\gamma$ 
replaced with $\gamma^{-1}$ (see, e.g., \cite{CrGe}, \cite{Ho}). The cluster structures that 
correspond to this version are isomorphic under the transposition $X\mapsto X^t$.

 2. The claim of the theorem was verified for $n=3,4$ in \cite{GSVMMJ} and for $n=5$ in \cite{Ei}.
\end{remark}

We prove Theorem~\ref{CGtrue} by producing a cluster structure $\CC_{CG}=\CC_{CG}(n)$ that possesses all the
needed properties. In fact, we will construct an exotic cluster structure in the space $\Mat_n$ of 
$n\times n$ matrices compatible with a natural extension of the Cremmer--Gervais Poisson bracket and derive the 
required properties of $\CC_{CG}$ from similar features of the latter cluster structure. 
 Note that in the ``intermediate'' case of $GL_n$ we also obtain a regular cluster structure compatible
with the extension of the Cremmer--Gervais Poisson bracket, however, in this case the ring of regular functions on $GL_n$ is isomorphic to the localization of the upper cluster algebra with respect to the function $\det X$.
In what follows we use the same notation $\CC_{CG}$ for all three cluster structures and indicate explicitely which
one is meant when needed.

To describe the initial cluster for  $\CC_{CG}$, we need to introduce some notation.
For a matrix $A$, we denote by $A_{i_1\ldots i_l}^{j_1\ldots j_m}$ its submatrix formed
by rows $i_1,\ldots, i_l$ and columns $j_1,\ldots, j_m$. If all rows (respectively, columns) of $A$ are 
involved, we will omit the lower (respectively, upper) list of indices.
If $X$, $Y$ are two $n\times n$ matrices, denote by  $\X$ and $\Y$  $(n-1)\times (n+1)$ matrices 
$$
\X = \left [ X_{[2,n]}\  0\right ], \quad \Y = \left [0\   Y_{[1,n-1]} \right ].
$$
Put $k=\lfloor \frac{n+1}{2}\rfloor$ and $N=k (n-1)$. Define a $k (n-1) \times (k+1) (n+1)$ matrix
\begin{equation}
U(X, Y) = \left [
\begin{array}{ccccc}
\Y & \X & 0 & \cdots & 0\\
0 & \Y & \X  & 0 & \cdots\\
0 & \ddots& \ddots &\ddots & 0\\
0 & \cdots & 0 & \Y & \X
\end{array}
\right ].
\label{uho}
\end{equation}

Define three families of functions in $X\in SL_n$ 
via
\begin{equation}\label{inclust}
 \begin{aligned}
\thetta_i(X)&=\det X_{[n-i+1,n]}^{[n-i+1,n]}, \; i\in [n-1];\\
\phhi_i(X)&=\det U(X,X)_{[N-i+1, N]}^{[k (n+1) - i +1, k (n+1)]}, \; i\in [N];\\
\pssi_i(X)&=\det U(X,X)_{[N-i+1, N]}^{[k (n+1) - i +2, k (n+1)+1]}, \; i\in [M].
\end{aligned}
\end{equation}
In the last family, $M=N$ if $n$ is even and $M=N-n+1$ if $n$ is odd.

\begin{theorem}
\label{logcan}
The functions~\eqref{inclust} form a log-canonical family with respect to the Cremmer--Gervais
bracket. 
\end{theorem}

Consequently, we choose family~\eqref{inclust} as an initial (extended) cluster for  $\CC_{CG}(n)$. 
Furthermore, functions $\phhi_N$ and $\pssi_M$ are the only stable variables for $\CC_{CG}(n)$;
for the motivation for this choice see Proposition~\ref{sigizmund} below. 

Next, we need to describe the quiver  that corresponds to the initial cluster above. 
In fact, it will be convenient to do this in the $\Mat_n$ rather than $SL_n$ situation, in
which case the initial cluster is augmented by the  addition of one more stable
variable, 
$\thetta_n(X)=\det X$. The Cremmer--Gervais Poisson structure is extended to $\Mat_n$
by requiring that $\thetta_n(X)$ is a Casimir function.
We denote the quiver corresponding to the augmented initial cluster by $Q_{CG}(n)$.
Its vertices are $n^2$ nodes of the $n\times n$ rectangular grid indexed by pairs
$(i,j)$, $i, j \in [n]$, with $i$ increasing top to bottom and $j$ increasing left
to right. Before
describing  edges of the quiver, let us explain the correspondence between the
cluster variables and the vertices of $Q_{CG}(n)$.

For any cluster or stable variable $f$ in the augmented initial cluster, consider the upper
left matrix
entry of the submatrix of $U(X,X)$ (or $X$) associated with this variable. This
matrix entry is $x_{ij}$ for some $i, j \in [n]$. 
Thus we define a correspondence  $\rho \: f \leftrightarrow (i,j)\in
[n]\times [n]$.

\begin{lemma}
\label{vertices}
$\rho$ is a one-to-one correspondence between the augmented initial cluster and
$[n]\times [n]$.
\end{lemma}

We will assign each cluster variable $f$ to the vertex indexed by $(i,j) = \rho(f)$.
In particular, the stable variable $\thetta_n$ is assigned to vertex $(1,1)$, 
 the stable variable $\phhi_N$ is assigned to vertex $(2,1)$ if $n$ is odd and $(1,n)$
 if $n$ is even, and  the stable variable $\pssi_M$ is assigned to vertex $(1,n)$ if
$n$ is odd and $(2,1)$
 if $n$ is even.

Now, we describe arrows of $Q_{CG}(n)$. There are horizontal arrows
$(i,j+1) \to (i,j)$ for  all $i \in [n], j\in [n-1]$ except $(i,j)=(1, n-1)$;
vertical arrows
$(i+1,j) \to (i,j)$ for  all $i \in [n-1], j\in [n]$ except $(i,j)=(1,1)$; 
diagonal arrows $(i,j) \to (i+1,j+1)$ for  all $i, j\in [n-1]$.
In addition, we have arrows between vertices of the first and the last rows:
$(n,j) \to (1,j)$ and $(1,j) \to (n,j+1)$ for $j\in [2, n-1]$;  and 
 arrows between vertices of the first and the last columns:
$(i,n) \to (i+2,1)$ and $(i+2,1) \to (i+1, n)$ for $i\in [1, n-2]$. This
concludes the description of $Q_{CG}(n)$. The quiver  $Q'_{CG}(n)$ that correspond to the
$SL_n$ case is obtained
from $Q_{CG}(n)$ by deleting the vertex $(1,1)$ and erasing all arrows incident to this
vertex. Quivers $Q_{CG}(4)$ and $Q_{CG}(5)$ are shown on Fig.~\ref{fig:Qcg45}.

\begin{figure}[ht]
\begin{center}
\includegraphics[height=6.4cm]{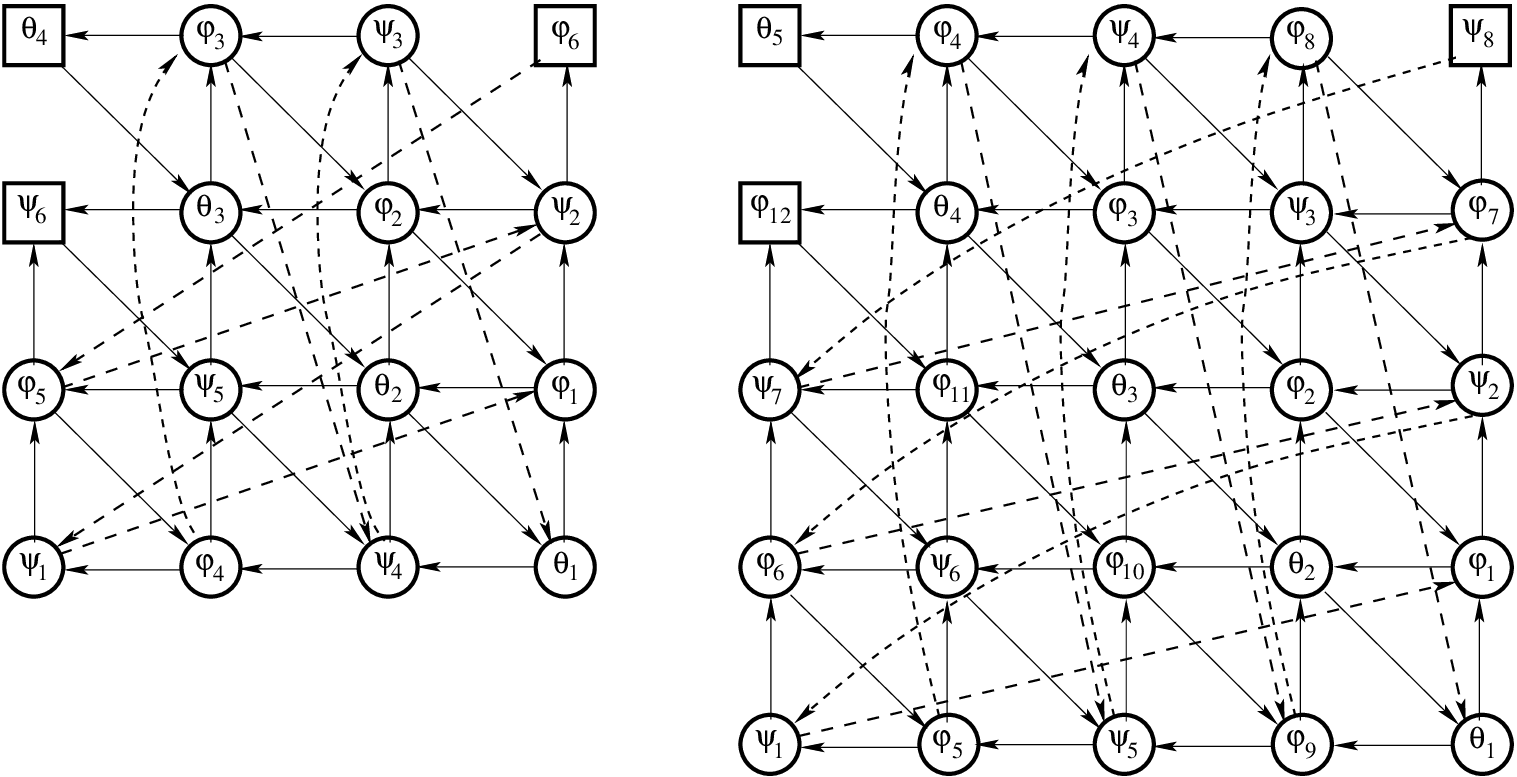}
\caption{Quivers $Q_{CG}(4)$ and $Q_{CG}(5)$}
\label{fig:Qcg45}
\end{center}
\end{figure}

\begin{theorem}
\label{quiver} {\rm (i)}
The quivers $Q_{CG}(n)$  and $Q'_{CG}(n)$ define cluster structures compatible with the
Cremmer-Gervais Poisson structure on $\Mat_n$ and $SL_n$ respectively.

{\rm (ii)} The corresponding extended exchange matrices are of full rank.
\end{theorem}

Note that $k_T=1$ in the Cremmer--Gervais case, and hence the corresponding
Belavin--Drinfeld class contains a unique R-matrix. Therefore, Theorem~\ref{quiver} establishes
parts (i) and (iv) of Conjecture~\ref{ulti}  in the Cremmer--Gervais case.

Another property of the cluster structure $\CC_{CG}$ is given by the following theorem.

\begin{theorem}
\label{regular}
Mutation of any function in the family~\eqref{inclust} except for $\phhi_N$ and 
$\pssi_M$ results in a regular function on $\Mat_n$.
\end{theorem}

As a corollary, we get part (iii) of  Conjecture~\ref{ulti}, see  Proposition~\ref{torus} below.
Consequently, by Theorem~\ref{partial}, part (v) of Conjecture~\ref{ulti} is established as well.

To establish part (ii) of the conjecture, we invoke Proposition~\ref{regfun}. Condition (i) follows 
from the definition of the initial cluster~\eqref{inclust},
condition (ii) is verified by Theorem~\ref{regular}, condition (iii) is verified by direct calculation.
Therefore, to settle part (ii) of Conjecture~\ref{ulti} and thus to complete the proof of Theorem~\ref{CGtrue},
it remains to check that every function in $\O(\Mat_n)$ belongs to 
the upper cluster algebra $\UU_{CG}=\UU_{\C}(\CC_{CG})$.

The proof relies on induction on $n$.  For $n=2$ the Cremmer--Gervais cluster structure coincides with the standard one,
which was treated already in~\cite{FZ2}.
The main ingredient of the proof is a construction of
two distinguished sequences of cluster
 transformations. The first sequence, $\mathcal S$, followed by freezing some of the
cluster variables and localization
 at a single cluster variable $\phhi_{n-1}(X)$, leads to a map $\zeta \: \Mat_n \setminus
 \{X\:\phhi_{n-1}(X)=0\}\to
\Mat_{n-1}$ that ``respects" the Cremmer--Gervais cluster
 structure. The map $\zeta$ is needed to perform an induction step. However, because
of the localization mentioned above, we also need a second sequence, $\mathcal T$, 
of transformations. This sequence can be viewed as a cluster-algebraic realization of the
anti-Poisson involution
 $X \mapsto W_0 X W_0$ on $\Mat_n$ equipped with the Cremmer--Gervais Poisson bracket, where
$W_0$ is the matrix corresponding to the longest permutation $w_0$
 (cf.~Lemma \ref{antipoiss} below). This allows one to apply $\zeta$ to
  $W_0 X W_0$ as well and then invoke certain general properties of cluster algebras
to fully utilize the induction assumption.

If $v=(v_i)_{i=1}^n$ is a vector in $\mathbb{C}^n$, we denote by $N(v)$ a
lower-triangular Toeplitz matrix
 $$
 N(v) = \left (
 \begin{matrix}
 v_1 & 0 & \cdots & 0\\
 v_2 & v_1 & \ddots  & \vdots\\
 \vdots & \ddots & \ddots  & 0\\
 v_n  & \cdots & v_2 & v_1
 \end{matrix}
 \right ).
 $$
Let 
\begin{equation}
\label{v_j}
v_j(X) = \frac{\left ( X^{-1}\right)_{jn}}{\left ( X^{-1}\right)_{1n}}
=\frac{(-1)^{j-1}\det X_{[n-1]}^{[j-1]\cup[j+1,n]}}{\phhi_{n-1}(X)}.
\end{equation}
Then
\begin{equation}
\label{zeta}
 X N(v(X)) = \left (
 \begin{array}{cc}
 0 & \zeta(X)\\
 \star & \star
 \end{array}
 \right ),
\end{equation}
 where $\zeta(X)$ is an $(n-1)\times (n-1)$ matrix and $\star$ stands for an expression whose explicit form is not essential for future
computations. Note that $v_1(X)=1$ for any $X$, and hence
$N(v(X))$ is unipotent. 

Denote by $\hat Q_{CG}(n)$ the quiver obtained by adding to $Q_{CG}(n)$ two
additional arrows: $(1,1) \to (n,2)$ and $(n,1)\to (1,1)$.

 \begin{theorem}
 \label{transform1}
  There exists a sequence $\mathcal S$ of cluster transformations in $\CC_{CG}(n)$
such that $\mathcal S(Q_{CG}(n))$ contains
  a subquiver isomorphic to $\hat Q_{CG}(n-1)$ and cluster variables indexed by
vertices of this subquiver satisfy
  $$\mathcal S(f)_{ij}(X)=
  \phhi_{n-1}(X)^{\varepsilon_{ij}}
  f_{ij}(\zeta(X)),\quad i,j\in [n-1],
  $$
  where $\varepsilon_{ij}=1$ if $(i,j)$ is a $\pssi$-vertex in $Q_{CG}(n-1)$
and $\varepsilon_{ij}=0$ otherwise.
 \end{theorem}
 
For any function $g$ on $\Mat_n$ define $g^{w_0}(X)=g(W_0XW_0)$.
 Besides, for any quiver $Q$ denote by $Q^{opp}$ the quiver obtained from $Q$ by reversing all
arrows. 

  \begin{theorem}
 \label{transform2}
 There exists a sequence $\mathcal T$ of cluster transformations in $\CC_{CG}$ such
that $\mathcal T(Q_{CG}(n))$ is isomorphic
 to $Q_{CG}^{opp}(n)$ and cluster variables indexed by vertices of
 $\mathcal T(Q_{CG}(n))$ satisfy
 $$\mathcal T(f)_{ij}(X) = f^{w_0}_{ij}(X).$$
 \end{theorem}

 To complete the proof of Theorem \ref{CGtrue}, it remains to show that functions $x_{ij} \in \O (\Mat_n)$ belong to the upper cluster algebra $\UU_{CG}(n)$. We will use induction on $n$.

Consider  the localization of $\UU_{CG}(n)$ at $\phhi_{n-1}(X)$, denoted by
$\UU_{CG}(n)\langle \phhi_{n-1}\rangle $.
Minors $\det X_{[n-1]}^{[j-1]\cup[j+1,n]}$ are cluster variables; for $j\in [2, n-1]$ this
can be seen from the third line in the formula presented in Lemma \ref{cycles} below with
$l=j-1, p= l+1, q =1$, and for $j=n$, from the fact that $\theta_{n-1}^{w_0}(X)=\det X_{[n-1]}^{[ n-1]}$. It follows from  \eqref{v_j} that
the entries of $N(v(X))^{-1}$, which are polynomial in $v_j(X)$, belong to 
$\UU_{CG}(n)\langle \phhi_{n-1}\rangle $. 

Next, we claim  that Theorem \ref{transform1}  allows one to apply Lemma \ref{MishaSha} to seeds
$(\hat Q_{CG}(n-1), \{ \mathcal S(f)_{ij}(X) \}_{i,j\in[n-1]})$ and  
$(  Q_{CG}(n-1), \{ f_{ij}(\zeta(X)) \}_{i,j\in[n-1]})$.  The role
of $x_{n+m}$ featured in  Lemma~\ref{MishaSha} is played by the cluster variable $\phhi_{n-1}(X)$ 
attached to the  vertex $(1,1)$  in both $\hat Q_{CG}(n-1)$ and $Q_{CG}(n-1)$. It follows immediately from
Theorem~\ref{quiver} that the vector $\tilde w=(\deg f_{ij}/(n-1))_{i,j=1}^{n-1}$ satisfies the assumptions
of Lemma~\ref{MishaSha}. Next, let us show that the vector $\hat w=(\varepsilon_{ij}+\deg f_{ij}/(n-1))_{i,j=1}^{n-1}$ satisfies these assumptions as well.
 By Theorem~\ref{transform1}, the condition we need to check
translates into the following one: for any non-frozen vertex $(i,j)$ the difference between the multiplicity 
$\hat \kappa_{ij}$ of the arrow between $(i,j)$ and $ (1,1)$ in $\hat Q_{CG}(n-1)$ 
and the multiplicity $\kappa_{ij}$ of the same arrow
in $Q_{CG}(n-1)$ equals the  difference 
between the number of $\pssi$-vertices pointing to $(i,j)$ and the number of $\pssi$-vertices pointed from 
$(i,j)$; note that $(1,1)$ is not a $\pssi$-vertex, and hence  the arrows
involved in this condition are the same in $\hat Q_{CG}(n-1)$ and $Q_{CG}(n-1)$. Observe now that for every vertex $(i,j)$ 
other than $(n-1,1)$ and $(n-1,2)$, the number of $\pssi$-vertices pointed from $(i,j)$ 
 is equal to the number of $\pssi$-vertices pointing to $(i,j)$, and hence 
$\hat \kappa_{ij}= \kappa_{ij}$. Next, 
$(n-1,1)$ is connected only with one $\pssi$-vertex: it is pointed from $(n-3,n-1)$, 
and $(n-1,2)$ is connected to three $\pssi$-vertices: it points to
$(n-1,1)$ and $(n-2,2)$  and is pointed from $(n-1,3)$. 
In both cases the condition on 
$\hat \kappa_{ij}-\kappa_{ij}$ is satisfied, and thus Lemma \ref{MishaSha} can be invoked
with  $\gamma=(\varepsilon_{ij})_{i,j=1}^{n-1}$.
It guarantees that if
we apply the same sequence of  cluster transformations to both seeds above and call the resulting cluster variables $\hat y (X)$ and $y(X)$, respectively, then $\hat y = \phhi_{n-1}(X)^{\alpha} y$ for some $\alpha$.  
Moreover, $\alpha$ is integer since both $\hat y$ and $y$ are rational functions of $X$. By the induction hypothesis, for any cluster obtained from the seed $(  Q_{CG}(n-1), \{ f_{ij}(\zeta(X)) \}_{i,j\in[n-1]})$,
 matrix entries of $\zeta(X)$ are Laurent polynomials in cluster variables forming this cluster. This, together with the previous observation, means
matrix entries of $\zeta(X)$
viewed as functions of $X$  belong to $\UU_{CG}(n)\langle \phhi_{n-1}\rangle $.
Therefore, the same is true for matrix entries of $X_{[n-1]}= \left [ 0\ \zeta(X) \right ] N(v(X))^{-1}$.

We will now show that $x_{nj}\in \UU_{CG}(n)\langle \phhi_{n-1}\rangle $ for $j\in [2,n-1]$ (recall that $x_{n1}$ and $x_{nn}$ belong to our initial cluster. Since $\theta_i(X) = \det X^{[n-i+1,n]}_{[n-i+1,n]}$, $i\in[n]$, 
and $\phhi_i(X)=  \det X^{[n-i+1,n]}_{[n-i,n-1]}$, $i\in[n-1]$, we see that $x_{n, n-i} = \frac {1}{\phhi_i(X)} \left ((-1)^{i} \theta_{i+1}(X) + P_i(X)\right )$, where
$P_i(X)$ is a polynomial in $x_{ij}$ with $i < n$ and $x_{nj}$ with $j > n -i$. Applying these relations recursively for $i\in [2, n-2]$, we conclude
that $x_{n, n-i}$ can be represented as an element of $ \UU_{CG}(n)\langle \phhi_{n-1}\rangle $ divided by the cluster variable $\phhi_i(X)$.
Next, recall from Remark \ref{more_cl_var} that functions $-\Lambda_{i+2}(X, X)$,  $i\in [n-2]$, where $\Lambda_{i}(X,X)$ is the $i\times i$ principal dense
trailing minor of the matrix $V(X,X)$, are cluster variables. We then can write $x_{n,n-i}=\frac {1}{\pssi_{i+1}(X)} \left ((-1)^{i-1} \Lambda_{i+2}(X,X) + \tilde P_i(X)\right )$, where
$\tilde P_i(X)$ is another polynomial in $x_{ij}$ with $i < n$ and $x_{nj}$ with $j > n -i$, and conclude that $x_{n, n-i}$ can be alternatively represented as an element of $ \UU_{CG}(n)\langle \phhi_{n-1}\rangle $ divided by the cluster variable $\pssi_{i+1}(X)$.

Note that $\phhi_i(X)$ and $\pssi_{i+1}(X)$ are irreducible for $i\in[n-2]$. Hence,
by Lemma \ref{twoloc}, the existence of two representations as described above means that $x_{n,n-i}$
multiplied by a suitable power of $\phhi_{n-1}(X)$ belongs to $\UU_{CG}(n)$, and hence $x_{n,n-i}$
itself belongs to $\UU_{CG}(n)\langle \phhi_{n-1}\rangle $. 

 We have shown that $x_{ij} \in \UU_{CG}(n)\langle \phhi_{n-1}\rangle $ for $i,j\in [n]$. 
 Similarly, it follows from Theorem \ref{transform2} that
  $x_{ij} \in \UU_{CG}(n)\langle \phhi^{w_0}_{n-1}\rangle $. Since $\phhi_{n-1}(X) \ne \phhi^{w_0}_{n-1}(X)$
  are both irreducible, we can invoke Lemma \ref{twoloc} again to conclude that  $x_{ij} \in \UU_{CG}(n) $.
The proof is complete.

\section{Initial cluster}

The goal of this section is to prove Theorem~\ref{logcan}.

A Poisson-Lie bracket on $SL_n$ can be extended to a Poisson bracket on $\Mat_n$ by requiring that
the determinant is a Casimir function. The restriction of this bracket to $GL_n$ is Poisson--Lie.
Formulae \eqref{sklya}, \eqref{sklyabra} remain valid in this case provided $\langle \cdot,\cdot\rangle$ is
understood as the Killing form for $\sl_n$ extended in the natural way to the whole $\gl_n$, that is,
$\langle X, Y\rangle = \Tr(XY)$.

Put $S = e_{12} + \cdots + e_{n-1 n}$ and $D_n=\diag (1,2,\ldots,n)$.  For any square matrix $A$ define
$$
\gamma_+ (A) = S A S^T, \quad \gamma_- (A) = S^T A  S.
$$
Besides, write $A_{>0}$, $A_{\geq 0}$, $A_{<0}$, $A_{\leq 0}$ and $A_0$
instead of $\pi_{>0}A$,  $\pi_{\geq 0}A$, $\pi_{<0}A$, $\pi_{\leq 0}A$ and $\pi_0A$, respectively.

\begin{lemma}
\label{CGR}
The Cremmer--Gervais Poisson--Lie bracket on $GL_n$ is given by 
\[
\{ f_1, f_2\}(X) = \left \langle  R_+ (\nabla f_1(X) X),  \nabla f_2(X) X\right \rangle -  
\left \langle  R_+ (  X \nabla f_1(X) ), X \nabla f_2(X)  \right \rangle,
\]
where
$\nabla$ is the gradient with respect to $\langle \cdot,\cdot\rangle$ and the operator $R_+\in \End \gl_n$ 
is defined by
\begin{equation}
R_+(\eta) =  \frac{1}{1-\gamma_+} \eta_{> 0} - \frac{\gamma_-}{1-\gamma_-} \eta_{\leq 0} 
+\frac {n-1} {2n}\Tr \eta \cdot\one   + \frac{1}{n} \big (\Tr \eta \cdot D_n - \Tr (D_n \eta)\cdot\one\big ).
\label{R+CG} 
\end{equation}
\end{lemma}

\begin{proof} Let $r$ be the r-matrix associated with the Cremmer--Gervais data.
We will compute $R_+\in \End \gl_n$ that corresponds to $r$ viewed as an element of 
$ \gl_n\otimes \gl_n$ via $\langle R_+ \eta, \zeta \rangle =
\langle r, \eta\otimes\zeta \rangle$, see~\eqref{sklyabra}.

First observe that in the Cremmer--Gervais case, the map $\gamma : \alpha_{i+1} \mapsto \alpha_i$,  
$i\in  [n-1]$, 
translates into the map $e_{i,i+1} \mapsto e_{i-1,i} = S e_{i,i+1} S^T = \gamma_+(e_{i,i+1} )$ on root vectors 
corresponding to simple
positive roots in $\Gamma_1$. Therefore, if $\alpha$ is any positive root in the root system generated by 
$\Gamma_1$, then
$e_{\gamma(\alpha)} = \gamma_+(e_\alpha)$. Moreover, for any $\alpha$ we have $\gamma^m(\alpha)\notin \Gamma_1$ 
if and only if 
$\gamma_+^m(e_\alpha)=0$. This means that for positive $\alpha$ \eqref{rBD} yields
\[
\langle R_+ e_\alpha, e_\beta \rangle = 
\left\langle e_{-\alpha} \otimes \sum_{m\geq 0} \gamma^m_+(e_\alpha),e_{\alpha} \otimes  e_\beta\right\rangle=
\left\langle \sum_{m\geq 0} \gamma^m_+(e_\alpha),e_\beta\right\rangle,
\]
and so
\[
R_+e_\alpha=\sum_{m\geq 0} \gamma^m_+(e_\alpha)=\frac{1}{1-\gamma_+} e_\alpha. 
\]

Similarly, for any positive $\beta$
\[
\langle R_+ e_{-\beta}, e_\alpha \rangle = 
-\left\langle \sum_{m\geq 1} \gamma^m_+(e_\alpha) \otimes e_{-\alpha}, 
e_{-\beta}\otimes e_\alpha\right\rangle=-\left\langle e_\alpha,  \sum_{m\geq 1} \gamma^m_-(e_{-\beta})\right\rangle,
\]
and hence
 \[
R_+e_{-\beta}=-\sum_{m\geq 1} \gamma^m_-(e_{-\beta})=-\frac{\gamma_-}{1-\gamma_-} e_{-\beta}.
\]

Next, let us turn to the action of $R_+$ on the Cartan subalgebra; it is defined by the term $r_0$ in \eqref{rBD}. 
Let $h_i=e_{ii} - e_{i+1,i+1}$, $i\in [n-1]$, be the standard basis of the Cartan subalgebra in $\sl_n$ and 
let $\hat h_i\in \h$ form the dual basis defined by $\alpha_j(\hat h_i)=\langle h_j, \hat h_i \rangle =\delta_{ji}$.
Since $\mathfrak t_0$ in \eqref{r02} is given by 
$\mathfrak t_0 = \sum_{i,j=1}^{n-1} c_{ij} \hat h_i\otimes \hat h_j$ where
$C=\left[2 \delta_{ij} - (\delta_{i,j-1}+\delta_{i,j+1})\right]_{i,j=1}^{n-1}$ is the Cartan matrix for 
$\sl_n$,
it is easy to check that for the Cremmer--Gervais data, the unique choice of $r_0$ that satisfies \eqref{r01}, 
\eqref{r02} is given by the formula 
$r_0 = \sum_{i=1}^{n-1} \hat h_i \otimes \hat h_i - \sum_{i=1}^{n-2} \hat h_i \otimes \hat h_{i+1}$. 
This implies  $R_+h_i = \hat h_i - \hat h_{i+1}$ where $\hat h_n=0$.   
Since $\sum_{j=1}^{n-1} c_{ij}\hat h_j  = h_i$, we get
$R_+ h_i =\sum_{i=1}^{n-1} [A  C^{-1}]_{ij} h_j$ where $A$ is the $(n-1)\times( n-1)$ matrix with $1$'s on the diagonal and $-1$'s on the 
first superdiagonal.  A straightforward computation shows that
$$
A C^{-1}  = -\left [ \begin{array}{ccccc} 0 & 1 & 1 & \cdots & 1\\ \cdots & \cdots & \cdots & \cdots &\cdots \\
\\ \cdots & 0& 0 & 1 & 1 \\  0 & \cdots & 0 &0 &  1\\
0 & 0 & \cdots & 0 & 0 \end{array}
\right ] + [1, \cdots, 1]^T\left [ \frac 1 n, \frac 2n,\cdots, \frac{n-1} n \right ].
$$
Thus, 
\begin{equation*}
\begin{aligned}
R_+ h_i &=-( h_{i+1} + \cdots + h_{n-1}) + \frac 1 n \sum_{j=1}^{n-1} j h_{j} \\
&=- \frac{\gamma_-}{1-\gamma_-} h_i + e_{nn} + \frac 1 n \sum_{j=1}^{n-1} j h_{j}=
-\frac{\gamma_-}{1-\gamma_-} h_i + \frac 1 n \one
\end{aligned}
\end{equation*}
as prescribed by \eqref{R+CG}.

Finally, since $\langle r, \one\otimes \one \rangle = 0$, we must have $R_+\one=0$.
Plugging $\eta=\one$ into \eqref{R+CG} one gets
\[
R_+\one=-\frac{\gamma_-}{1-\gamma_-} \one -\frac{n-1} 2 \one + D_n - \frac{n+1} 2 \one=
\one - D_n + D_n -\one=0
\] 
as required. Thus we checked that the action of $R_+$ on all basis vectors in $\mathfrak{gl}_n$ is consistent 
with equation \eqref{R+CG}, and the proof is completed.
\end{proof}
       
In what follows, we will need to compute Poisson brackets of certain functions
on $GL_n$ with respect to the Cremmer--Gervais structure. It turns out, however, that 
computations become more 
transparent if they are performed in the double $D(GL_n)$  of $GL_n$ equipped with 
the Poisson-Lie structure \eqref{sklyadouble} for functions that reduce to the ones we are interested in
after restriction to the diagonal subgroup of $D(GL_n)$.

For a function $f$ on $D(GL_n)$, we have 
\begin{eqnarray}\nonumber
\dnabla^R f(X,Y) &=\left (X\nabla_X f(X,Y), -Y\nabla_Y f(X,Y)\right ), \\ \nonumber
 \dnabla^L f(X,Y) &=\left (\nabla_X f(X,Y)X, -\nabla_Y f(X,Y)Y\right ),
\end{eqnarray}
where $\nabla_X, \nabla_Y$ denote gradients of $f$ as a function of $X$ and $Y$ with respect to the trace-form. 
Then a simple computation shows that \eqref{sklyadouble} can be rewritten
as
\begin{equation}\label{sklyadoubleGL}
\begin{split}
\{f_1,f_2\}_D = &\langle R_+(E_L f_1), E_L f_2\rangle -  \langle R_+(E_R f_1), E_R f_2\rangle\\
&+  \langle X\nabla_X  f_1, Y\nabla_Y f_2\rangle - \langle\nabla_X  f_1\cdot X, \nabla_Y f_2 \cdot Y\rangle,
\end{split}
\end{equation}
where
\[
E_R = X \nabla_X + Y\nabla_Y, \quad E_L =  \nabla_X X+ \nabla_Y Y.
\]

We will need to compute Poisson brackets for functions of the form
$f(U(X,Y))$ with respect to \eqref{sklyadoubleGL}. Since $U(X,Y)$ is a $k\times (k+1)$ block
matrix with $(n-1)\times (n+1)$ blocks, we can represent the gradient of $f$ viewed as a function on 
such matrices and taken with respect to the appropriate trace form as a $(k+1)\times k$ block
matrix  $\unabla f$ with $(n+1)\times (n-1)$ blocks $\unabla_{ij} f$. 
We will view $U(X,Y) \unabla f(U(X,Y))$ as a $k\times k$ block matrix with $(n-1)\times (n-1)$ blocks
and $\unabla f(U(X,Y)) U(X,Y)$ as a $(k+1)\times (k+1)$ block matrix with $(n+1)\times (n+1)$ blocks. 
We use notation $\Phhi(f)$ and $\Pssi(f)$ for "block-traces" of these two matrices:
\begin{equation}\label{blocktrace}
\begin{aligned} 
&\Phhi(f)=\Phhi(f)(U(X,Y))=\sum_{i=1}^k \left ( U(X,Y) \unabla f(U(X,Y)) \right )_{ii}, \\ 
&\Pssi(f)=\Pssi(f)(U(X,Y))=\sum_{i=1}^{k+1} \left ( \unabla f(U(X,Y))  U(X,Y)\right )_{ii}.
\end{aligned}
\end{equation}

Define
\[
 I=I(f)=\sum_{i=1}^k \unabla_{ii} f,\qquad J=J(f)=\sum_{i=1}^{k} \unabla_{i+1,i} f
\]
and write down the products $J\X$ and $I\Y$ as
\[
 J\X=[JX_{[2,n]}\  0]=\left [\begin{array}{cc}
J_X & 0 \\ \star & 0\end{array}\right ],\qquad
I\Y=[0\  IY_{[1,n-1]}]=\left [\begin{array}{cc}
0 & \star \\ 0 & I_Y\end{array}
\right ].
\]

\begin{lemma}\label{uhogradlemma} 
The gradients $\nabla_X f(U(X,Y))$ and $\nabla_Y f(U(X,Y))$ satisfy relations
\begin{equation}\label{uhograd}
\begin{aligned}
\nabla_X f(U(X,Y)) X = J_X, &\quad 
X \nabla_X f(U(X,Y))  = \left [\begin{array}{cc}
0 & \star \\ 0 & \X J\end{array}\right ]=J_{\X};\\
\nabla_Y f(U(X,Y)) Y = I_Y,
 &\quad Y \nabla_Y f(U(X,Y))  = \left [\begin{array}{cc}
\Y I & 0 \\ \star & 0\end{array}\right ]=I_{\Y}.
\end{aligned}
\end{equation}
\end{lemma}

\begin{proof} By the definition of the gradient,
\[
  \langle \delta X, \nabla_X f(U(X,Y))\rangle=\left.\frac{d}{dt}\right\vert_{t=0} f(U(X+t \delta X, Y))
\]
for any $n\times n$ matrix $\delta X$. Further,
\[
\left.\frac{d}{dt}\right\vert_{t=0} f(U(X+t \delta X, Y))=
\left \langle\unabla f,\left.\frac{d}{dt}\right\vert_{t=0} U(X+t \delta X, Y)  \right \rangle=
 \Tr\left([\delta X_{[2,n]}\  0]J\right).
  \]
Therefore,
\begin{equation*}
 \begin{aligned}
  \Tr \left(\delta X \nabla_X f(U(X,Y))X\right)&=\Tr \left(X\delta X\nabla_X f(U(X,Y))\right)=
  \Tr \left(\X\left[\begin{array}{cc}
\delta X & 0 \\ 0 & 0\end{array}\right ]J\right)\\
&=\Tr\left(\left[\begin{array}{cc}
\delta X & 0 \\ 0 & 0\end{array}\right ]J\X\right)=\Tr(\delta XJ_X),
 \end{aligned}
\end{equation*}
and the first relation in \eqref{uhograd} follows.
 
Similarly,
\[
 \langle\nabla_X f(U(X,Y)), \delta X\rangle=\Tr \left(J[\delta X_{[2,n]}\  0]\right).
\]
Therefore,
\begin{equation*}
 \begin{aligned}
  \Tr \left( X \nabla_X f(U(X,Y))\delta X\right)&=\Tr \left(\nabla_X f(U(X,Y))\delta X X\right)=
  \Tr \left(J\delta X_{[2,n]}\left[\begin{array}{c}
\star  \\ \X\end{array}\right ]\right)\\
&=\Tr\left(\left[\begin{array}{c}
\star \\ \X\end{array}\right ][0\  J]\left[\begin{array}{c}
\delta X_1 \\ \delta X_{[2,n]}\end{array}\right ]\right)=\Tr(J_{\X}\delta X),
 \end{aligned}
\end{equation*}
and the second relation in \eqref{uhograd} follows.
 
The remaining two relations for  $\nabla_Y f(U(X,Y))$ are obtained in the same way via
$\langle \delta Y, \nabla_Y f(U(X,Y))\rangle =\Tr\left ( [0\  \delta Y_{[1,n-1]}]I\right)$.
\end{proof}

\begin{proposition}\label{brack_uho}
The Poisson bracket between  $f_1(U(X,Y))$, $f_2(U(X,Y))$
viewed as functions on the double of $GL_n$ is given by the formula
\begin{equation}\label{brack_uho_form}
 \begin{split}
\{ f_1, f_2\}_D&= \frac{1}{n} \left (\Tr(D_{n-1} \Phhi^1) \Tr\Phhi^2 - \Tr(D_{n-1} \Phhi^2) \Tr\Phhi^1 \right)\\
&   - \frac{1}{n} \left (\Tr(D_{n+1} \Pssi^1) \Tr\Pssi^2 - \Tr(D_{n+1} \Pssi^2) \Tr\Pssi^1 \right) \\
&  - \left \langle  \frac{1}{1-\gamma_+} 
\Phhi^1_{>0} - \frac{\gamma_-}{1-\gamma_-}\Phhi^1_{\leq 0}, \Phhi^2 \right \rangle
+ \left \langle  \frac{1}{1-\gamma_+} 
\Pssi^1_{>0} - \frac{\gamma_-}{1-\gamma_-}\Pssi^1_{\leq 0}, \Pssi^2 \right \rangle\\ 
& + \langle \Y I^1, \X J^2\rangle 
-\langle  I^1\Y,J^2\X\rangle,
\end{split}
\end{equation}
where superscripts $1$ and $2$ refer to objects associated with $f_1$ and $f_2$, respectively.
\end{proposition}

\begin{proof} First, observe that by \eqref{uhograd}, $E_L=J_X + I_Y$
and $E_R= J_{\X}+I_{\Y}$. 
Besides, $\Phhi=\X J+\Y I$ and $\Pssi= J\X+I\Y$, hence
\begin{equation}\label{trcom}
\Tr J_X =\Tr J\X=\Tr\X J=\Tr J_{\X},\quad  \Tr I_Y=\Tr I\Y=\Tr\Y I=\Tr I_{\Y}
\end{equation}
implies
\begin{equation}\label{tre}
\begin{aligned}
 \Tr E_R=\Tr\Phhi,&\quad \Tr(D_nE_R)=\Tr (D_{n-1}\Phhi)+\Tr J_X,\\
 \Tr E_L=\Tr\Pssi,&\quad \Tr(D_nE_L)=\Tr (D_{n+1}\Pssi)-\Tr I_Y.
\end{aligned}
 \end{equation}

Consider the first term in \eqref{sklyadoubleGL}: using Lemma \ref{CGR} and \eqref{tre}, we obtain
\begin{equation}\label{1}
\begin{aligned}
\langle R_+(E^1_L), E_L^2\rangle&= \left \langle  \frac{1}{1-\gamma_+} 
(E^1_L)_{>0} - \frac{\gamma_-}{1-\gamma_-}(E^1_L)_{\leq 0}, E_L^2 \right \rangle 
+  \frac {n-1} {2n} \Tr\Pssi^1 \Tr\Pssi^2  \\ 
&
+ \frac{1}{n} \big (\Tr\Pssi^1 ( \Tr(D_{n+1} \Pssi^2)-\Tr I_Y^2) - (\Tr (D_{n+1} \Pssi^1)-\Tr I_Y^1)
\Tr \Pssi^2 \big ).
\end{aligned}
\end{equation}

Since $J_X$ is the upper-left $n\times n$ block of $J \X$, while
$I_Y$ is the lower-right $n\times n$ block of $I \Y$,
we conclude that
$$
\frac {1}{1-\gamma_+} (E_L)_{\geq 0} =   (J_X )_{\geq 0} +
\left [\one_n\ 0\right ] \frac {\gamma_+}{1-\gamma_+} \Pssi_{\geq 0} 
\left [\begin{array}{c} \one_n\\ 0\end{array} \right ]
$$
and
$$
\frac {\gamma_-}{1-\gamma_-} (E_L)_{\leq 0} =   -(J_X )_{\leq 0} +
\left [0\ \one_n\right ] \frac {\gamma_-}{1-\gamma_-} \Pssi_{\leq 0} \left [\begin{array}{c} 0\\ \one_n\end{array} \right ].
$$
Clearly, the projection $\pi_{\geq 0}$ in the first relation above can be replaced by $\pi_{>0}$.
Therefore, the first term in \eqref{1} is equal to
\[
 \left \langle  
J^1_X, E_L^2 \right \rangle +
 \left \langle  \frac{\gamma_+}{1-\gamma_+} 
\Pssi^1_{>0}, \left [\begin{array}{cc} E_L^2 & 0\\ 0 & 0 \end{array}\right ]\right\rangle - 
\left\langle\frac{\gamma_-}{1-\gamma_-}\Pssi^1_{\leq 0}, \left [\begin{array}{cc} 0 & 0\\ 0 &E_L^2  
\end{array}\right ]\right\rangle.
\]
Note that $\langle A_{>0},B\rangle=\langle A_{>0}, B_{\leq 0}\rangle$ 
and $\langle A_{\leq 0},B\rangle=\langle A_{\leq 0}, B_{\geq 0}\rangle$ 
for any matrices $A$ and $B$, hence
\begin{equation*}
\begin{aligned} 
 \left\langle  \frac{\gamma_+}{1-\gamma_+} \Pssi^1_{>0}, 
  \left[
 \begin{array}{cc} E_L^2 & 0\\ 
                       0 & 0 \end{array}
\right ]\right\rangle & =
\left\langle  \Pssi^1_{>0}, 
\left[
\begin{array}{cc} 0 & 0\\ 
                  0 & \frac{1}{1-\gamma_-}(E_L^2)_{\leq 0} \end{array}
\right]\right\rangle\\
 & =\left \langle  \Pssi^1_{>0}, 
 \left[
 \begin{array}{cc} 0 & 0\\ 
                   0 & ( I_Y^2 )_{\leq 0} \end{array}
\right ]\right\rangle + \left \langle  \Pssi^1_{>0},
 \frac{\gamma_-}{1-\gamma_-}\Pssi^2_{\leq 0} 
\right\rangle\\
 & =\left \langle  \Pssi^1_{>0}, 
 \left [
 \begin{array}{cc} 0 & 0\\ 
                   0 & I_Y^2 \end{array}\right ]
 \right\rangle + \left \langle  
\frac{\gamma_+}{1-\gamma_+}\Pssi^1_{>0}, \Pssi^2 \right\rangle\\
 & =\left \langle  \Pssi^1_{>0},  \Pssi^2 \right\rangle - 
\left \langle  \Pssi^1_{>0}, J^2\X\right\rangle + \left \langle  
\frac{\gamma_+}{1-\gamma_+}\Pssi^1_{>0}, \Pssi^2 \right\rangle\\
 & = - \left \langle  \Pssi^1_{>0}, J^2\X\right\rangle + \left \langle  
\frac{1}{1-\gamma_+}\Pssi^1_{>0}, \Pssi^2 \right\rangle
\end{aligned}
\end{equation*}
and
\begin{align*}
 \left \langle  \frac{\gamma_-}{1-\gamma_-} 
\Pssi^1_{\leq 0}, \left [\begin{array}{cc} 0 & 0\\ 0 & E_L^2  \end{array}\right ]\right\rangle 
&=\left \langle  \Pssi^1_{\leq 0}, 
\left [\begin{array}{cc}  \frac{1}{1-\gamma_+}( E_L^2 )_{\geq 0} & 0 \\ 0 & 0 \end{array}\right ]
\right\rangle\\
 &=\left \langle  
\Pssi^1_{\leq 0}, \left [\begin{array}{cc} ( J_X^2 )_{\geq 0} & 0 \\ 0 & 0 \end{array}\right ]\right\rangle 
+ \left \langle  \Pssi^1_{\leq 0}, 
 \frac{\gamma_+}{1-\gamma_+}\Pssi^2_{\geq 0} 
\right\rangle\\
 &=\left \langle  
\Pssi^1_{\leq 0}, J^2\X\right\rangle + \left \langle  
 \frac{\gamma_-}{1-\gamma_-}\Pssi^1_{\leq 0},  \Pssi^2 \right\rangle.
\end{align*}
Therefore,
\begin{equation}\label{2}
 \begin{aligned}
\left \langle  \frac{1}{1-\gamma_+} 
{E^1_L}_{>0} - \frac{\gamma_-}{1-\gamma_-}{E^1_L}_{\leq 0}, E_L^2 \right \rangle &=
\left \langle  J^1_X, E_L^2 \right \rangle - \left \langle  \Pssi^1, J^2\X\right\rangle\\
& +\left \langle  
 \frac{1}{1-\gamma_+}\Pssi^1_{> 0}-
 \frac{\gamma_-}{1-\gamma_-}\Pssi^1_{\leq 0},  \Pssi^2 
\right\rangle.  
\end{aligned}
\end{equation}

Now, let us turn to the second term in \eqref{sklyadoubleGL}. 
Using Lemma \ref{CGR} and \eqref{tre}, we obtain
\begin{equation}\label{3}
\begin{aligned}
\langle R_+(E^1_R), E_R^2\rangle&= \left \langle  \frac{1}{1-\gamma_+} 
(E^1_R)_{>0} - \frac{\gamma_-}{1-\gamma_-}(E^1_R)_{\leq 0}, E_R^2 \right \rangle 
+  \frac {n-1} {2n} \Tr\Phhi^1 \Tr\Phhi^2  \\ 
&
+ \frac{1}{n} \big (\Tr\Phhi^1 ( \Tr(D_{n-1} \Phhi^2)+\Tr J_X^2) - (\Tr (D_{n-1} \Phhi^1)+\Tr J_X^1)
\Tr \Phi^2 \big ).
\end{aligned}
\end{equation}

Since $\X J$ is the lower-right $(n-1)\times (n-1)$ block of $J_\X$, while
$\Y I$ is the upper-left $(n-1)\times (n-1)$ block of $I_\Y$,
we conclude that
$$
\frac {1}{1-\gamma_+}(E_R )_{\geq 0} =  
(J_{\X})_{\geq 0}+ \frac {1}{1-\gamma_+}
\left [\begin{array}{cc}
\Phhi_{\geq 0} & 0 \\  0 & 0\end{array}\right ]
$$
and
$$
\frac {\gamma_-}{1-\gamma_-} \left (E_R\right )_{\leq 0} = 
-(J_{\X})_{\leq 0} + \frac {1}{1-\gamma_-}
\left [\begin{array}{cc} 0 & 0 \\
0 & \Phhi_{\leq 0} \end{array}\right].
$$

Therefore, the first term in \eqref{3}  is equal to 
$$
\left \langle  J_{\X} + \frac {1}{1-\gamma_+}
\left [\begin{array}{cc}
\Phhi^1_{>0} & 0 \\  0 & 0\end{array}\right ] - \frac {1}{1-\gamma_-}
\left [\begin{array}{cc} 0 & 0 \\
0 & \Phhi^1_{\leq 0} \end{array}\right ] , E_R^2
 \right \rangle.
$$
Similarly to the previous case, we have
\begin{align*}
\left \langle  \frac {1}{1-\gamma_+}
\left [\begin{array}{cc}
\Phhi^1_{>0} & 0 \\  0 & 0\end{array}\right ], E_R^2
 \right \rangle &=
\left \langle 
\left [\begin{array}{cc}
\Phhi^1_{>0} & 0 \\  0 & 0\end{array}\right ],  \frac {1}{1-\gamma_-}( E_R^2)_{\leq 0}
 \right \rangle\\
 &\hskip -1cm=
\left \langle 
\left [\begin{array}{cc}
\Phhi^1_{>0} & 0 \\  0 & 0\end{array}\right ], (I_{\Y})_{\leq 0}+ 
\frac {1}{1-\gamma_-}\left [\begin{array}{cc} 0 & 0 \\
0 & \Phhi^2_{\leq 0} \end{array}\right ]
 \right \rangle\\
& \hskip -1cm = \left \langle \Phhi^1_{>0}, (\Y I^2)_{\leq 0}\right \rangle
+ \left \langle
\left [\begin{array}{cc}
\Phhi^1_{>0} & 0 \\  0 & 0\end{array}\right],
\frac {1}{1-\gamma_-}
\left [\begin{array}{cc} 0 & 0 \\
0 & \Phhi^2_{\leq 0} \end{array}\right]
 \right \rangle
\\
&\hskip -1cm =  \left \langle \Phhi^1_{>0}, (\Y I^2)_{\leq 0}\right \rangle
+ \left \langle
\frac {\gamma_+}{1-\gamma_+}\Phhi^1_{>0},
\Phhi^2
 \right \rangle
\\
&\hskip -1cm = - \left \langle \Phhi^1_{>0}, \X J^2\right \rangle
+ \left \langle
\frac {1}{1-\gamma_+}\Phhi^1_{>0},
\Phhi^2
 \right \rangle
\end{align*}
and
\begin{align*}
\left \langle  \frac {1}{1-\gamma_-}
\left [\begin{array}{cc} 0 & 0 \\
0 & \Phhi^1_{\leq 0}\end{array}\right], E_R^2
 \right \rangle &=
\left \langle 
\left [\begin{array}{cc} 0 & 0 \\
0 & \Phhi^1_{\leq 0}\end{array}\right],  \frac {1}{1-\gamma_+}( E_R^2)_{\geq 0}
 \right \rangle\\
& = \left \langle \left [\begin{array}{cc} 0 & 0 \\
0 & \Phhi^1_{\leq 0}\end{array}\right], (J^2_{\X})_{\geq 0}+
\frac {1}{1-\gamma_+}
\left [\begin{array}{cc}  \Phhi^2_{\geq 0} & 0 \\ 0 & 0\end{array}\right]
 \right \rangle
\\
&=  \left \langle \Phhi^1_{\leq 0}, \X J^2\right \rangle
+ \left \langle
\frac {\gamma_-}{1-\gamma_-}\Phhi^1_{\leq 0},
\Phhi^2 \right \rangle,
\end{align*}
and so
\begin{equation}\label{4}
 \begin{aligned}
\left \langle  \frac{1}{1-\gamma_+} 
{E^1_R}_{>0} - \frac{\gamma_-}{1-\gamma_-}{E^1_R}_{\leq 0}, E_R^2 \right \rangle &=
\left \langle  J^1_{\X}, E_R^2 \right \rangle - \left \langle  \Phhi^1, \X J^2\right\rangle\\
& +\left \langle  
 \frac{1}{1-\gamma_+}\Phhi^1_{> 0}
- \frac{\gamma_-}{1-\gamma_-}\Phhi^1_{\leq 0},  \Phhi^2 
\right\rangle.  
\end{aligned}
\end{equation}

Using \eqref{sklyadoubleGL}, \eqref{uhograd} and (\ref{1})--(\ref{4}), 
we obtain
\begin{align*}
\{f_1,f_2\}_D &=
 \frac{1}{n} \left (\Tr(D_{n-1} \Phhi^1) \Tr\Phhi^2 - \Tr(D_{n-1} \Phhi^2) \Tr\Phhi^1 \right)\\
&  - \frac{1}{n} \left (\Tr(D_{n+1} \Pssi^1) \Tr\Pssi^2 - \Tr(D_{n+1} \Pssi^2) \Tr\Pssi^1 \right) \\
&  -\frac{1}{n} \left (\Tr\Pssi^1 \Tr I_Y^2+\Tr\Phhi^1 \Tr J_X^2 - \Tr I_Y^1 \Tr\Pssi^2 - 
\Tr J_X^1  \Tr\Phhi^2 \right) \\
&  - \left \langle  \frac{1}{1-\gamma_+} 
\Phhi^1_{>0} - \frac{\gamma_-}{1-\gamma_-}\Phhi^1_{\leq 0}, \Phhi^2 \right \rangle
+ \left \langle  \frac{1}{1-\gamma_+} 
\Pssi^1_{>0} - \frac{\gamma_-}{1-\gamma_-}\Pssi^1_{\leq 0}, \Pssi^2 \right \rangle\\ 
& 
+\left \langle  
J^1_X, J_X^2 \right \rangle -
\left \langle  
\Pssi^1, J^2\X\right\rangle 
-\left \langle  
J^1_{\X}, J_{\X}^2 \right \rangle
+\left \langle \Phhi^1, \X J^2\right \rangle.
\end{align*}
The expression on the third line is equal to zero by \eqref{trcom}. Finally,
$\langle J^1_X,J^2_X\rangle=\langle J^1\X,J^2\X\rangle$ and $\langle J^1_{\X}, J^2_{\X}\rangle
=\langle \X J^1,\X J^2\rangle$. Therefore,
\[
\left \langle  J^1_X, J_X^2 \right \rangle -\left \langle  \Pssi^1, J^2\X\right\rangle 
-\left \langle  J^1_{\X}, J_{\X}^2 \right \rangle+\left \langle \Phhi^1, \X J^2\right \rangle=
-\langle I^1\Y,J^2\X\rangle + \langle \Y I^1, \X J^2\rangle,
\]
and \eqref{brack_uho_form} follows.
\end{proof}

Consider the family of functions on the double of $GL_n$ given by
\begin{equation}\nonumber
 \begin{aligned}
\thetta_i=\thetta_i(X,Y)&=\det X_{[n-i+1,n]}^{[n-i+1,n]}, \quad i\in [n-1];\\
\phhi_i=\phhi_i(X,Y)&=\det U(X,Y)_{[N-i+1, N]}^{[k (n+1) - i +1, k (n+1)]}, \quad i\in [N];\\
\pssi_i=\pssi_i(X,Y)&=\det U(X,Y)_{[N-i+1, N]}^{[k (n+1) - i +2, k (n+1)+1]}, \quad i\in [M];
\end{aligned}
\end{equation}
recall that $k=\lfloor \frac{n+1}{2}\rfloor$, $N=k (n-1)$,  $M=N$ if $n$ is even and $M=N-n+1$ if $n$ is odd.

\begin{theorem}
\label{dlogcan}
The functions $\thetta_i$, $\phhi_j$, $\pssi_l$ form a log-canonical family with respect to $\Poi_D$. 
\end{theorem}

Theorem~\ref{logcan} is obtained from Theorem~\ref{dlogcan} by restriction to the diagonal subgroup
in the double.

The rest of this section is devoted the proof
of Theorem \ref{dlogcan}, which will only be presented for $n$ odd. The case of $n$ even can be 
treated in a similar way. Observe that if we extend the definition
of $\pssi_i$ above to $i\in [M+1, N]$, we will obtain $\pssi_{M+j}=\pssi_M \thetta_j$ for $j\in [1, n-1]$.
Thus, to show that the entire family of functions above is log-canonical, it is sufficient
to deal with $2N$ functions $\phhi_i$, $\pssi_i$ for $i\in [N]$. Moreover, since \eqref{cpt} is
equivalent to $\{\log x_i,\log x_j\}=\omega_{ij}$, it will be convenient to deal with $\bar\phhi_i=\log\phhi_i$ and 
$\bar\pssi_i=\log\pssi_i$ instead of $\phhi_i$ and $\pssi_i$. 

Let us start with computing  right and left gradients of $\bar\phhi_i$ and $\bar\pssi_i$, as well as
``block-traces'' $\Phhi(\bar\phhi_i)$, $\Phhi(\bar\pssi_i)$, $\Pssi(\bar\phhi_i)$ and $\Pssi(\bar\pssi_i)$.

Denote by $\unabla^i$ the gradient of $\bar \phhi_i$ viewed as a function of $U$. 

\begin{lemma}\label{gradfi}
{\rm (i)} For any $q\in [N]$,
 \begin{equation}\label{aux0}
\unabla^q=
\left [\begin{array}{ccc} 0& \cdots & 0\\ \unabla^q_{11} & \cdots &\unabla^q_{1k} \\ 
\cdots &\cdots &\cdots\\
\unabla^q_{k-1,1} & \cdots &\unabla^q_{k-1,k} \\  0 & \cdots & 0
\end{array}\right ],
\end{equation}
where the entries  are blocks of size $(n+1)\times (n-1)$ satisfying equations
\begin{equation}\label{aux1}
 \begin{aligned}
\Y \unabla^q_{i-1,j} + \X \unabla^q_{i j} &= 0, \qquad 1\leq j<i\leq k-1, \\ 
\Y \unabla^q_{k-1,j} &= 0, \qquad 1\leq j\leq k-1,\\ 
\unabla^q_{i j}\X + \unabla^q_{i,j+1}\Y&=0, \qquad 1\leq i<j\leq k-1.
\end{aligned}
\end{equation}

{\rm (ii)}  For any $q\in [N]$, $\Phhi(\bar\phhi_q)$ is an upper triangular matrix with the diagonal
entries not depending on $X, Y$, and $\Pssi(\bar\phhi_q)$ is a lower triangular matrix with the diagonal
entries not depending on $X, Y$.
\end{lemma}

\begin{proof}
(i) Let us re-write $U(X,Y)$ as 
$$
U(X,Y)= \left [ A\, B\, C \right ],
$$
where $A$ and $C$ are block-columns with blocks of size $(n-1)\times(n+1)$
and $B$ is the $k\times (k-1)$ block matrix 
\[
B = \left [
\begin{array}{ccc}
 \X & 0 & \cdots \\
\Y & \X & 0  \\
 \ddots& \ddots &\ddots \\
 \cdots & 0 & \Y
\end{array}
\right ];
\]
note that $B$ is a square matrix, since $k(n-1)=(k-1)(n+1)$.
The family $\phhi_q$, $q\in [N]$, is given by 
principal contiguous minors of $B$ extending from the lower right 
corner, and hence \eqref{aux0} follows.

Denote by $B_q$ the principal submatrix of $B$ that corresponds to $\phhi_q$. Then
$\unabla^q$ can be written as
\[
\unabla^q=\left [\begin{array}{c} 0\\ \unabla^q_*\\ 0 \end{array}\right ],
\]
where 
$$
\unabla^q_* = \left [\begin{array}{cc} 0 & 0 \\ 0  & B_q^{-1} \end{array}\right ].
$$
Therefore, we have
\begin{equation}\label{Unabla}
U \unabla^q = B \unabla^q_* = \left [\begin{array}{cc} 0 & \star \\ 0  & \one_q \end{array}\right ]  
\end{equation}
and
\begin{equation}\label{nablaU}
\unabla^q U=
\left [\begin{array}{ccc} 0 & 0 & 0\\ \star & \unabla^q_* B & \star \\ 0 & 0 & 0 \end{array}\right ] =
\left [\begin{array}{cccc} 0& 0 & 0  & 0\\ 
\star & 0 & 0 & \star \\ \star & \star & \one_q & \star\\
0& 0 & 0  & 0
\end{array}\right ],
\end{equation}
where the first and the last block columns have width $n+1$.

On the other hand,
\[
B \unabla^q_*= \left [\Y \unabla^q_{i-1,j} + \X \unabla^q_{i j} \right ]_{i,j=1}^k, \quad 
\unabla^q_* B= \left [\unabla^q_{i j}\X + \unabla^q_{i,j+1}\Y \right ]_{i,j=1}^{k-1},
\]
which gives \eqref{aux1} (here and in what follows $\unabla^q_{kk}=0$).

(ii) Follows immediately from \eqref{Unabla} and \eqref{nablaU}.
\end{proof}

Similarly, denote by $\bar\unabla^q$ the gradient of $\bar \pssi_q$ viewed as a function
of $U$. 

\begin{lemma}\label{gradpsi}
{\rm (i)} For any $q\in [N]$,
 \begin{equation}\label{aux11}
\bar\unabla^q=
\left [\begin{array}{ccc} 0& \cdots & 0\\ 
\bar\unabla^q_{11} & \cdots
&\bar\unabla^q_{1k} \\ 
\cdots &\cdots &\cdots\\
\bar\unabla^q_{k,1} & \cdots &\bar\unabla^q_{k,k} 
\end{array}\right ],
\end{equation}
where the entries 
are blocks of size $(n+1)\times (n-1)$, the first row of each block $\bar\unabla^q_{1i}$ is zero,
and all rows of the blocks $\bar\unabla^q_{ki}$ except for the first one are zero.
Moreover, nonzero
entries satisfy  
equations
\begin{equation}\label{aux2}
 \begin{aligned}
\Y \bar\unabla^q_{i-1,j} + \X \bar\unabla^q_{i j} &= 0, \qquad
1\leq j<i\leq k, \\
\bar\unabla^q_{i j}\X + \bar\unabla^q_{i,j+1}\Y&=0, \qquad 1\leq i<j\leq k-1, \\
\bar\unabla^q_{ik}\X&=\delta_{ik}e_{11}, \quad 1\leq i \leq k.
\end{aligned}
\end{equation}

{\rm (ii)} For any $q\in [N]$, $\Phhi(\bar\pssi_q)$ is an upper triangular matrix
with the diagonal
entries not depending on $X, Y$, while $\Pssi(\bar\pssi_q)$ is a sum of 
$(n+1)\times(n+1)$
lower triangular matrices with the diagonal entries not depending on $X, Y$ and the
matrix
$$
\bar\unabla^q_{kk} \X=\left [\begin{array}{cc} 1 &  \star\\ 0 & 0
\end{array} \right ].
$$
\end{lemma}

\begin{proof}
(i) The family $\pssi_q,  q\in [N]$, is formed by principal contiguous minors
extending from the lower right 
corner of the matrix 
\[
\bar B = \left [
\begin{array}{ccccc}
\X^{[2,n+1]} & 0 & 0 & \cdots & 0 \\
\Y^{[2,n+1]} & \X & 0 &\cdots & 0  \\
0 &\Y &\X &\cdots & 0\\
\vdots& \ddots& \ddots &\ddots &\vdots \\
0 & \cdots & 0 &  \Y & \X^1
\end{array}
\right ]; 
\]
here the entries in the first column are $(n-1)\times n$ blocks, entries in the last
column are $(n-1)\times 1$
blocks, and all other entries are $(n-1)\times (n+1)$ blocks. This 
suggests to re-write $U(X,Y)$ as
$$
U(X,Y)= \left [ \bar A \;\bar B \;\bar C \right ],
$$
where $\bar A$ 
and $\bar C$ 
are block-columns of width $n+2$ and $n$, respectively, and immediately implies
\eqref{aux11}.

Denote by $\bar B_q$ the principal submatrices of $\bar B$  that corresponds to
$\pssi_q$. Then
$\bar\unabla^q$ can be written as
\[
\bar\unabla^q=\left [\begin{array}{c} 0\\\bar \unabla^q_*\\ 0 \end{array}\right ],
\]
 and
$$
\bar\unabla^q_* = \left [\begin{array}{cc} 0 & 0 \\ 0  & \bar B_q^{-1}
\end{array}\right ].
$$
Therefore, 
\begin{equation}\label{Ubnabla}
U \bar\unabla^q = \bar B \bar\unabla^q_* = 
\left [\begin{array}{cc} 0 & \star \\ 0  & \one_q \end{array}\right ] 
\end{equation}
and
\begin{equation}\label{bnablaU}
\bar\unabla^q U=\left [\begin{array}{ccc} 0 & 0 & 0\\ 
\star & \bar\unabla^q_* B & \star \\ 0 & 0 & 0 \end{array}\right ] =
\left [\begin{array}{cccc} 0& 0 & 0  & 0\\ 
\star & 0 & 0 & \star \\ \star & \star & \one_q & \star\\
0& 0 & 0  & 0
\end{array}\right ].
\end{equation}
Similarly to the previous case, this implies \eqref{aux2}.

(ii) Follows immediately from \eqref{Ubnabla} and \eqref{bnablaU}.
\end{proof}

Consider two functions $f_1, f_2$ from the set $\{\bar\phhi_i, \bar\pssi_i\: i\in[N]\}$.
Lemmas \ref{gradfi}(ii), \ref{gradpsi}(ii) 
imply that the 
first two summands in the right hand side of \eqref{brack_uho_form} do not depend on $X$ and $Y$.
Furthermore, $\Phhi^2=\Phhi^2_{\geq 0}$ and $\Phhi^1_{\leq 0}=\Phhi^1_{0}$, and therefore
\begin{equation}
\label{aux3}
\begin{aligned}
\left \langle  \frac{1}{1-\gamma_+} 
\Phhi^1_{>0} - \frac{\gamma_-}{1-\gamma_-}\Phhi^1_{\leq 0}, \Phhi^2 \right \rangle&=
\left \langle  \frac{1}{1-\gamma_+} 
\Phhi^1_{>0} - \frac{\gamma_-}{1-\gamma_-}\Phhi^1_{0}, \Phhi^2_{\geq 0} \right \rangle\\
&= - \left \langle \frac{\gamma_-}{1-\gamma_-}\Phhi^1_{0}, \Phhi^2_{0} \right \rangle
\end{aligned}
\end{equation}
does not depend on $X$ and $Y$ as well. On the other hand,
\begin{equation}
\label{aux4}
\begin{aligned}
\left \langle  \frac{1}{1-\gamma_+} 
\Pssi^1_{>0} - \frac{\gamma_-}{1-\gamma_-}\Pssi^1_{\leq 0}, \Pssi^2 \right \rangle&=
\left \langle 
\Pssi^1_{>0} - \frac{\gamma_-}{1-\gamma_-}\Pssi^1_{0}, \Pssi^2_{\leq 0} \right \rangle\\
& = \left \langle 
\Pssi^1_{>0}, \Pssi^2_{\leq 0} \right \rangle
 - \left \langle \frac{\gamma_-}{1-\gamma_-}\Pssi^1_{0}, \Pssi^2_{0} \right \rangle.
\end{aligned}
\end{equation}
Here we use the fact that  matrices with non-zero entries only in the first row are  
annihilated by $\gamma_+$ and are orthogonal  to the image of $\gamma_-$.
The second term in the resulting expression does not depend on $X$ and $Y$, while 
\begin{equation}\label{psipsi}
\left \langle 
\Psi^1_{>0}, \Psi^2_{\leq 0} \right \rangle=
\begin{cases}
\left \langle 
\bar\unabla^i_{kk}\X-e_{11},\Pssi^2 \right \rangle,\qquad\text{if   $f_1=\log\pssi_i$},\\
0,\qquad\qquad\qquad\qquad\qquad \text{if $f_1=\log\phhi_i$}. 
 \end{cases}
\end{equation}

Thus, to complete the proof of Theorem \ref{dlogcan}, we have to show that the expression
\begin{equation}\label{bad}
 \left \langle \Psi^1_{>0}, \Psi^2_{\leq 0} \right \rangle+ 
 \langle \Y I^1,\X J^2\rangle  -  \langle I^1\Y, J^2\X\rangle
\end{equation}
does not depend on $X, Y$. 
Consider the following  cases:
\begin{equation}\label{4cases}
\begin{aligned}
&1.\ f_1=\bar\phhi_{q_1}, f_2=\bar\phhi_{q_2},\quad 1\leq q_2 \leq q_1 \leq N;\\
&2.\ f_1=\bar\pssi_{q_1}, f_2=\bar\pssi_{q_2},\quad 1\leq q_2 \leq q_1 \leq N;\\
&3.\ f_1=\bar\phhi_{q_1}, f_2=\bar\pssi_{q_2},\quad 1\leq q_2 \leq q_1 \leq N;\\
&4.\ f_1=\bar\pssi_{q_1}, f_2=\bar\phhi_{q_2},\quad 1\leq q_2 \leq q_1 \leq N.
\end{aligned}
\end{equation}

{\em Case 1}. By \eqref{aux0}, \eqref{aux3} and \eqref{psipsi}, we need to evaluate
$$
 \langle \Y I^1,\X J^2\rangle  -  \langle I^1\Y, J^2\X\rangle=\sum_{i,j=1}^{k-1}\left 
(\Tr( \Y\unabla_{i,i+1}^{q_1} \X \unabla_{jj}^{q_2}) - \Tr(
\unabla_{i,i+1}^{q_1} \Y \unabla_{jj}^{q_2} \X)\right ).
$$

Let $j=i+s$ with $1\leq i\leq k-1$ and  $0\leq s \leq k-i-1$. Applying the first and the third equalities in \eqref{aux1} 
alternately for $k-j-1$ times each and, additionally, the second equality, we obtain
\begin{equation}\label{indu1}
\begin{aligned}
 \unabla_{i,i+1}^{q_1} \Y \unabla_{i+s,i+s}^{q_2} \X&=-\unabla_{i,i+1}^{q_1} \X \unabla_{i+s+1,i+s}^{q_2} 
 \X\\
 &=\unabla_{i,i+2}^{q_1} \Y \unabla_{i+s+1,i+s}^{q_2} \X=\ldots=
\unabla_{i,k-s}^{q_1} \Y \unabla_{k-1,i+s}^{q_2} \X=0.
\end{aligned}
\end{equation}
Similarly, for $1\leq i\leq k-2$ and $1\leq s\leq k-i-1$, we apply the third and the first equalities in \eqref{aux1} alternately
for $k-j-1$ times each and, additionally, the third and the second equalities  to get
\begin{equation*}
\Y\unabla_{i,i+1}^{q_1} \X \unabla_{i+s,i+s}^{q_2} =-\Y\unabla_{i,i+2}^{q_1} \Y \unabla_{i+s,i+s}^{q_2}
=\ldots=-\Y\unabla_{i,k-s+1}^{q_1} \Y \unabla_{k-1,i+s}^{q_2}=0. 
\end{equation*}

Assume now that $j=i-s$ with $2\leq i\leq k-1$ and $1\leq s\leq i-1$. Applying the first and the third equalities in \eqref{aux1} 
alternately for $k-i-1$ times each and, additionally, the first equality once more, we obtain
\begin{equation*}
\unabla_{i,i+1}^{q_1} \Y \unabla_{i-s,i-s}^{q_2} \X=-\unabla_{i,i+1}^{q_1} \X 
\unabla_{i-s+1,i-s}^{q_2} \X=
\ldots=-\unabla_{ik}^{q_1} \X \unabla_{k-s,i-s}^{q_2} \X. 
\end{equation*}
Similarly, for $1\leq i\leq k-1$ and $0\leq s\leq i-1$, we apply the third and the first equalities in \eqref{aux1} alternately
for $k-i-1$ times each  to get
\begin{equation*}
\Y\unabla_{i,i+1}^{q_1} \X \unabla_{i-s,i-s}^{q_2} =
-\Y\unabla_{i,i+2}^{q_1} \Y \unabla_{i-s,i-s}^{q_2}
=\ldots=
\Y\unabla_{ik}^{q_1} \X \unabla_{k-s-1,i-s}^{q_2} \X.
\end{equation*}

Therefore,
\begin{equation*}
\begin{aligned}
& \langle \Y I^1,\X J^2\rangle  -  \langle I^1\Y, J^2\X\rangle =
\\
&
= \sum_{i=2}^{k-1} \sum_{s=1}^{i-1}\Tr\left( \unabla_{ik}^{q_1} \X \unabla_{k-s,i-s}^{q_2} \X\right) +
\sum_{j=1}^{k-1} \sum_{t=0}^{j-1} \Tr\left(\Y\unabla_{jk}^{q_1} \X \unabla_{k-t-1,j-t}^{q_2}\right) 
\\
& 
= \sum_{j=1}^{k-2} \sum_{t=0}^{j-1}\Tr\left(\X \unabla_{j+1,k}^{q_1} \X \unabla_{k-t-1,j-t}^{q_2}\right) 
+\sum_{j=1}^{k-1} \sum_{t=0}^{j-1} \Tr\left(\Y\unabla_{jk}^{q_1} \X \unabla_{k-t-1,j-t}^{q_2}\right) 
\\
&
= \Tr\left(\sum_{j=1}^{k-2} \sum_{t=0}^{j-1}\left[\X\unabla_{j+1,k}^{q_1} + \Y\unabla_{jk}^{q_1}\right]
\X \unabla_{k-t-1,j-t}^{q_2}  +
\sum_{s=1}^{k-1} \Y\unabla_{k-1,k}^{q_1} \X \unabla_{ss}^{q_2}\right ). 
\end{aligned}
\end{equation*}

As it was noted in the proof of Lemma~\ref{gradfi}, 
$\left[\X\unabla_{j+1,k}^{q_1} + \Y\unabla_{jk}^{q_1}\right]=[B\unabla^{q_1}_*]_{j+1,k}$. By \eqref{Unabla},
this block vanishes provided it lies off the diagonal within the unit submatrix $\one_{q_1}$. These two 
conditions are equivalent to $(k-j)(n-1)\leq q_1$ and $j+1<k$. Besides, the block $\unabla_{k-t-1,j-t}^{q_2}$
of $\unabla^{q_2}_*$ vanishes provided it lies below the diagonal outside the submatrix $B_{q_2}^{-1}$. These 
two conditions are equivalent to $(k-j+t)(n-1)\geq q_2$ and $j<k-1$. Therefore, the first term above vanishes,
since otherwise $q_1<(k-j)(n-1)\leq (k-j+t)(n-1)<q_2$, a contradiction.

By \eqref{Unabla}, the second term can be re-written as 
\[ 
\Tr\left(
 [B \unabla^{q_1}_*]_{kk} \sum_{s=1}^{k-1} \X \unabla_{ss}^{q_2}\right )
= \Tr(\X J^2),
\]
since  $\sum_{s=1}^{k-1} \X \unabla_{ss}^{q_2}=0$ for $q_2\leq q_1 < n-1$,
and $ [B \unabla^{q_1}_*]_{kk}=\one_{n-1}$
for $q_1\geq n-1$.
Thus it remains to show that $\Tr(\X J^2)$ does not depend on $X, Y$. To do this, 
we set $q_1=q_2$ and use \eqref{brack_uho_form}, \eqref{aux3}, \eqref{aux4} :
$
0 =  \left \langle \frac{\gamma_-}{1-\gamma_-}\Phhi^2_{0}, \Phhi^2_{0} \right \rangle
-  \left \langle \frac{\gamma_-}{1-\gamma_-}\Pssi^2_{0}, \Pssi^2_{0} \right \rangle + \Tr(\X J^2), 
$
and so
\begin{equation}
\label{bad1}
\Tr(\X J^2) = - \left \langle \frac{\gamma_-}{1-\gamma_-}\Phhi^2_{0}, \Phhi^2_{0} \right \rangle
+ \left \langle \frac{\gamma_-}{1-\gamma_-}\Pssi^2_{0}, \Pssi^2_{0} \right \rangle
\end{equation}
does not depend on $X, Y$.

{\em Case 2}. As in the the previous case, apply alternately equalities in \eqref{aux2} to get
\begin{equation*}
\begin{aligned}
 \bar\unabla_{i,i+1}^{q_1} \Y \bar\unabla_{i+s,i+s}^{q_2}
\X&=-\bar\unabla_{i,i+1}^{q_1} \X \bar\unabla_{i+s+1,i+s}^{q_2} 
 \X\\
 &=\bar\unabla_{i,i+2}^{q_1} \Y \bar\unabla_{i+s+1,i+s}^{q_2} \X=\ldots=
-\bar\unabla_{i,k-s}^{q_1} \X \bar\unabla_{k,i+s}^{q_2} \X\\
&=(\delta_{s,0}-1)\bar\unabla_{i,k-s}^{q_1} \X  \bar\unabla_{k,i+s}^{q_2} \X\\ &
=(1-\delta_{s,0})
\bar\unabla_{i,k-s+1}^{q_1} \Y \bar\unabla_{k,i+s}^{q_2} \X
\end{aligned}
\end{equation*}
for $1\leq i\leq k-1$ and $0\leq s \leq k-i-1$. Note that the resulting relation trivially holds for $s=k-i$ as well.  

Recall that 
 $ \bar\unabla_{kj}^{q_2}$ has only the first row nonzero, therefore
 \begin{equation}\label{zerotr}
  \Tr (A_1B\bar\unabla_{kj}^{q_2}A_2)=0
 \end{equation}
for any matrices $A_1$ and $A_2$ and any $B$ that has a zero first column. Using this relation for 
$A_1=\bar\unabla_{i,k-s+1}^{q_1}$, $A_2= \X$, $B=\Y$ and $j=i+s$, we conclude
 that $\Tr (\bar\unabla_{i,i+1}^{q_1} \Y \bar\unabla_{i+s,i+s}^{q_2} \X)=0$.

Similarly, 
for   $1\leq i\leq k-1$ and $1\leq s\leq k-i$  we 
get
\begin{equation*}
\begin{aligned}
\Y\bar\unabla_{i,i+1}^{q_1} \X \bar\unabla_{i+s,i+s}^{q_2}
&=-\Y\bar\unabla_{i,i+2}^{q_1} \Y \bar\unabla_{i+s,i+s}^{q_2}
=\ldots\\
&=
\begin{cases} -\Y\bar\unabla_{i,k-s+2}^{q_1} \Y  \bar\unabla_{k,i+s}^{q_2} &
\mbox{if}\ s >1,\\
\Y \bar\unabla_{ik}^{q_1}\X   \bar\unabla_{k,i+s}^{q_2} & \mbox{if}\ s=1.
\end{cases}
\end{aligned}
\end{equation*}
Note that by the last relation in \eqref{aux2}, $\bar\unabla_{ik}^{q_1}\X$ has a zero
first column. Therefore, using \eqref{zerotr}
for $A_1=\Y$, $A_2=\one_{n-1}$, $B=\bar\unabla_{ik}^{q_1}\X$ and $j=i+s$, we 
conclude that $\Tr (\Y\bar\unabla_{i,i+1}^{q_1} \X \bar\unabla_{i+s,i+s}^{q_2})=0$.

Next, arguing as in Case 1, but using \eqref{aux2} rather than \eqref{aux1}, we obtain
\begin{equation*}
\bar\unabla_{i,i+1}^{q_1} \Y \bar\unabla_{i-s,i-s}^{q_2} \X=-\bar\unabla_{ik}^{q_1}
\X \bar\unabla_{k-s,i-s}^{q_2} \X\ 
\quad \text{for $2\leq i\leq k-1$, $1\leq s \leq i-1$}.
\end{equation*}
and
\begin{equation*}
\Y\bar\unabla_{i,i+1}^{q_1} \X \bar\unabla_{i-s,i-s}^{q_2} =
\Y\bar\unabla_{ik}^{q_1} \X \bar\unabla_{k-s-1,i-s}^{q_2} 
\quad\text{for $1\leq i\leq k-1$, $0\leq s \leq i-1$}.
\end{equation*}

Using relations above, we compute
\begin{equation*}
\begin{aligned}
&\langle \Y I^1,\X J^2\rangle  - \langle I^1\Y, J^2\X\rangle    =\\
&
= \sum_{i=2}^{k-1} \sum_{s=1}^{i-1}\Tr\left( \bar\unabla_{ik}^{q_1} \X
\bar\unabla_{k-s,i-s}^{q_2} \X\right) +
\sum_{j=1}^{k-1} \sum_{t=0}^{j-1} \Tr\left(\Y\bar\unabla_{jk}^{q_1} \X
\bar\unabla_{k-t-1,j-t}^{q_2}\right) \\
&= \Tr\left(\sum_{j=1}^{k-2} \sum_{t=0}^{j-1}
[\bar B\bar \unabla^{q_1}_*]_{j+1,k}
\X \bar\unabla_{k-t-1,j-t}^{q_2}  +
\sum_{s=1}^{k-1} \Y\bar\unabla_{k-1,k}^{q_1} \X \bar\unabla_{ss}^{q_2}\right ). 
\end{aligned}
\end{equation*}
The argument identical to the one used in Case 1 shows that the inner double sum  above vanishes. 

Consider now the term $\left \langle 
\Psi^1_{>0}, \Psi^2_{\leq 0} \right \rangle$ in \eqref{bad}. Using \eqref{psipsi}
and \eqref{zerotr} for $A_1=I^2$, $A_2=\X$, $B=\Y$, $j=k$, we can rewrite it as 
\begin{equation*}
\begin{aligned}
\left \langle 
\Psi^1_{>0}, \Psi^2_{\leq 0} \right \rangle & = \left \langle 
\bar\unabla^{q_1}_{kk}\X - e_{11}, \Psi^2 \right \rangle= \left \langle
\bar\unabla^{q_1}_{kk}\X , J^2\X \right \rangle -\Pssi^2_{11} \\
&= \left \langle \bar\unabla^{q_1}_{kk}\X, \sum_{s=1}^{k} 
\bar\unabla_{ss}^{q_2}\X 
\right \rangle -\Pssi^2_{11}
 \\
&=\Tr\left ( \X\bar\unabla^{q_1}_{kk}
\sum_{s=1}^{k} \X \bar\unabla_{ss}^{q_2} \right )-\Pssi^2_{11}.
\end{aligned}
\end{equation*}
Therefore, taking into account that $\Tr (\Y\bar\unabla_{k-1,k}^{q_1} \X \bar\unabla_{kk}^{q_2})=0$,
we get
\begin{equation*}
\begin{aligned}
&\left \langle 
\Psi^1_{>0}, \Psi^2_{\leq 0} \right \rangle + \langle \Y I^1,\X J^2\rangle  -
\langle I^1\Y, J^2\X\rangle    =
\\
&
= \Tr\left(
\sum_{s=1}^{k}\left [ \Y\bar\unabla_{k-1,k}^{q_1} + \X\bar\unabla^{q_1}_{kk}
\right ] \X \bar\unabla_{ss}^{q_2}\right )-\Pssi^2_{11}
\\
&
=\Tr \left ([\bar B\bar \unabla^{q_1}_*]_{kk} 
 \sum_{s=1}^{k} \X \bar\unabla_{ss}^{q_2}\right )-\Pssi^2_{11}
=\Tr(\X J^2) - \Pssi^2_{11}.
\end{aligned}
\end{equation*}
As in the previous case, we set $q_1=q_2$ and use 
 \eqref{brack_uho_form}, \eqref{aux3}, \eqref{aux4} to see that
\begin{equation}
\label{bad2}
\Tr(\X J^2) - \Pssi^2_{11} = - \left \langle \frac{\gamma_-}{1-\gamma_-}\Phhi^2_{0}, \Phhi^2_{0} \right \rangle
+ \left \langle \frac{\gamma_-}{1-\gamma_-}\Pssi^2_{0}, \Pssi^2_{0} \right \rangle 
\end{equation}
does not depend on $X, Y$.

{\em Case 3}. 
Using alternately \eqref{aux1} and \eqref{aux2}, we obtain relations similar to those utilized in Cases 1 and 2:
\begin{equation*}
\begin{aligned}
&\Tr (\unabla_{i,i+1}^{q_1} \Y \bar\unabla_{i+s,i+s}^{q_2} \X)=0  
\quad \text{for $1\leq i\leq k-1$, $1\leq s \leq k- i$},\\
& \unabla_{i,i+1}^{q_1} \Y \bar\unabla_{i-s,i-s}^{q_2} \X=
-\unabla_{ik}^{q_1} \X \bar\unabla_{k-s,i-s}^{q_2} \X \quad \text{for $1\leq i\leq k-1$, $0\leq s \leq i-1$},\\
&\Tr(\Y\unabla_{i,i+1}^{q_1} \X \bar\unabla_{i+s,i+s}^{q_2}) = 0 \quad 
\text{for $1\leq i \leq k-2$,  $1\leq s \leq k-i$},\\
& 
\Y \unabla_{i,i+1}^{q_1} \X \bar\unabla_{i-s+1,i-s+1}^{q_2} =
\Y\unabla_{ik}^{q_1} \X \bar\unabla_{k-s,i-s+1}^{q_2} 
\quad\text{for $1\leq i\leq k-1$, $0\leq s \leq  i$}.
\end{aligned}
\end{equation*}
These relations allow us to compute
\begin{equation*}
\begin{aligned}
& \langle \Y I^1,\X J^2\rangle  -  \langle I^1\Y, J^2\X\rangle = \Tr(\Y\unabla_{k-1,k}^{q_1} \X \bar\unabla_{kk}^{q_2})
\\
&
+ \sum_{i=1}^{k-1} \sum_{s=0}^{i-1}\Tr\left( \unabla_{ik}^{q_1} \X \bar\unabla_{k-s,i-s}^{q_2} \X\right) +
\sum_{j=1}^{k-1} \sum_{t=0}^{j} \Tr\left(\Y\unabla_{jk}^{q_1} \X \bar\unabla_{k-t,j-t+1}^{q_2}\right) 
\\
&
= \Tr\left(
\sum_{j=1}^{k-1} \sum_{t=0}^{j-1}[B\unabla^{q_1}_*]_{jk}
\X \bar\unabla_{k-t,j-t}^{q_2}  +
[B\unabla^{q_1}_*]_{kk}\sum_{s=1}^{k}  \X \bar\unabla_{ss}^{q_2}  \right ).
\end{aligned}
\end{equation*}
Arguing as in the previous cases, we conclude that $[B\unabla^{q_1}_*]_{jk}
\X \bar\unabla_{k-t,j-t}^{q_2} =0$ for $1\leq t \leq j-1$. If $t=0$, $[B\unabla^{q_1}_*]_{jk}
\X \bar\unabla_{kj}^{q_2} $ is nonzero provided $(k-j)(n-1) < q_2 \leq q_1 < (k-j+1)(n-1)$.
However, in this case the last $q_1 - (k-j)(n-1)$ rows of $[B\unabla^{q_1}_*]_{jk}$ are zero
and the first $(k-j+1)(n-1) - q_2$ columns of $\bar\unabla_{kj}^{q_2} $ are zero, which means
that $\Tr([B\unabla^{q_1}_*]_{jk} \X \bar\unabla_{kj}^{q_2})=0 $.
Furthermore, if $q_2\leq q_1 < n-1$, then the first $n-1 -q_1$ columns of $\bar\unabla_{kk}^{q_2}$ are zero and the last $q_1$ rows of $[B\unabla^{q_1}_*]_{kk}$
coincide with those of the identity matrix $\one_{n-1}$, therefore
$$
\Tr \left ([B \unabla^{q_1}_*]_{kk}  \X \bar\unabla_{kk}^{q_2}\right )= 
\Tr \left (\X \bar\unabla_{kk}^{q_2}\right )=1.
$$
 Thus,
\begin{equation}\label{bad3}
\begin{aligned}
 \langle \Y I^1,\X J^2\rangle  -  \langle I^1\Y, J^2\X\rangle &=
\Tr \left ([B \unabla^{q_1}_*]_{kk} 
 \sum_{s=1}^{k} \X \bar\unabla_{ss}^{q_2}\right )
=\Tr (\X J^2)\\
&=\Pssi^2_{11}-\left\langle \frac{\gamma_-}{1-\gamma_-}\Phhi_0^2,\Phhi_0^2\right\rangle+
\left\langle \frac{\gamma_-}{1-\gamma_-}\Pssi_0^2,\Pssi_0^2\right\rangle
\end{aligned}
\end{equation}
does not depend on $X, Y$.

{\em Case 4}. Similarly to Case 1, we have
\begin{equation*}
\begin{aligned}
&\bar\unabla_{i,i+1}^{q_1} \Y \unabla_{i+s,i+s}^{q_2} \X=0 \ 
\quad \text{for $1\leq i\leq k-1$, $0\leq s \leq k-i-1$},\\
& 
\bar\unabla_{i,i+1}^{q_1} \Y \unabla_{i-s,i-s}^{q_2} \X=-\bar\unabla_{ik}^{q_1} \X \unabla_{k-s,i-s}^{q_2} \X\ 
\quad \text{for $2\leq i\leq k-1$, $1\leq s \leq i-1$},\\
&\Y\bar\unabla_{i,i+1}^{q_1} \X \unabla_{i+s,i+s}^{q_2} = 0 \quad \text{for $1\leq i\leq k-2$,  $1\leq s \leq k-i-1$},\\
& 
\Y\bar \unabla_{i,i+1}^{q_1} \X \unabla_{i-s+1,i-s+1}^{q_2} =
\Y\bar\unabla_{ik}^{q_1} \X \unabla_{k-s,i-s+1}^{q_2} 
\quad\text{for $1\leq i\leq k-1$, $1\leq s \leq  i$}\ 
\end{aligned}
\end{equation*}
and
\begin{equation}\label{bad4}
\begin{aligned}
&\left \langle 
\Psi^1_{>0}, \Psi^2_{\leq 0} \right \rangle +  \langle \Y I^1,\X J^2\rangle  -  \langle I^1\Y, J^2\X\rangle = 
\Tr\left (\X\bar\unabla^{q_1}_{kk}
\sum_{s=1}^{k-1} \X \unabla_{ss}^{q_2} \right ) -\Pssi_{11}^2
\\
&
+ \sum_{i=2}^{k-1} \sum_{s=1}^{i-1}\Tr\left(\bar \unabla_{ik}^{q_1} \X \unabla_{k-s,i-s}^{q_2} \X\right) +
\sum_{j=1}^{k-1} \sum_{t=1}^{j} \Tr\left(\Y\bar\unabla_{jk}^{q_1} \X \unabla_{k-t,j-t+1}^{q_2}\right) 
\\
&
= \Tr\left(
\sum_{j=1}^{k-2} \sum_{t=1}^{j-1}[\bar B\bar \unabla^{q_1}_*]_{jk}
\X \unabla_{k-t,j-t+1}^{q_2}  +
[\bar B \unabla^{q_1}_*]_{kk}\sum_{s=1}^{k-1}  \X \unabla_{ss}^{q_2}  \right ) - \Pssi^2_{11}
\\
&
=\Tr \left ([\bar B\bar \unabla^{q_1}_*]_{kk} 
 \sum_{s=1}^{k-1} \X \unabla_{ss}^{q_2}\right )  - \Pssi^2_{11}
=\Tr(\X J^2) - \Pssi^2_{11}\\
&=-\left\langle \frac{\gamma_-}{1-\gamma_-}\Phhi_0^2,\Phhi_0^2\right\rangle+
\left\langle \frac{\gamma_-}{1-\gamma_-}\Pssi_0^2,\Pssi_0^2\right\rangle-\Pssi_{11}^2
\end{aligned}
\end{equation}
does not depend on $X, Y$.

The proof of Theorem \ref{dlogcan} is complete.
\qed

We finish this section with a proposition which, though not needed for the proof of our main theorem,
 serves as a justification for our choice of functions $\phhi_N(X)$, $\pssi_M(X)$ as stable 
variables in the cluster structure $\CC_{CG}(GL_n)$ in addition to $\thetta_n(X)$. Indeed, in all known examples of cluster structures on Poisson varieties stable variables have two key properties: they behave well under certain
natural group actions, and the vanishing locus of each of these functions foliates into a union of non-generic symplectic
leaves. The latter property usually is a consequence of each stable variable being log-canonical with certain globally defined
 coordinate functions. The proposition below shows that both these properties are valid in our current situation.

\begin{proposition}
\label{sigizmund}
{\rm(i)}  
Functions $\phhi_N(X,Y), \pssi_M(X,Y)$ are semi-invariants of the right and left action of $D_-$ in $D(GL_n)$.

{\rm(ii)}  
For any $i,j\in [n]$, $\phhi_N(X)$ and $\pssi_M(X)$ are log-canonical with $x_{ij}$ with respect to  
the Cremmer--Gervais Poisson--Lie bracket.
\end{proposition}

\begin{proof} (i) In the Cremmer--Gervais case, the subgroup $D_-$ of $D(GL_n)$ that corresponds to the subalgebra $\D_-$ of $\D(\gl_n)$ defined in \eqref{ddeco} is
given by
$$
D_-=\left \{\left (  
\left [
\begin{array}{cc} a & \star\\ 0 & A\end{array}
\right ], \quad
\left [
\begin{array}{cc} A & 0\\ \star & a\end{array}\ 
\right ) 
\right ]\right \},
$$
where $A\in GL_{n-1}$ (see, e.g. \cite{Y}).
Thus, the two-sided action of $D_-$ on $( X, Y)\in D(GL_n)$,
$$
(X, Y) \mapsto \left ( \left [
\begin{array}{cc} a & \star\\ 0 & A\end{array}
\right ] X \left [
\begin{array}{cc} a' & \star\\ 0 & A'\end{array}
\right ], \left [
\begin{array}{cc} A & 0\\ \star & a\end{array}\right ] \ Y \left [
\begin{array}{cc} A' & 0\\ \star & a'\end{array}\ 
\right ]
\right ),
$$
translates into a transformation of $U(X,Y)$ given by the multiplication on the left by the block diagonal
matrix diag$(\underbrace{A,\ldots, A}_{k \ \mbox{times}})$ and on the right by the block diagonal
matrix diag$(\underbrace{B,\ldots, B}_{k+1 \ \mbox{times}})$, where $B$ is an $(n+1)\times (n+1)$ matrix of the form
$ B=\left [
\begin{array}{ccc} a' & \star& 0 \\ 0 & A' & 0\\ 0 & \star & a'\end{array}
\right ] $. If $n$ is odd, $\phhi_N$ is the determinant of the $N\times N$ submatrix of $U(X,Y)$ obtained by 
deleting the first and the last $n+1$ columns. Under the action above, it is multiplied by 
$\det A^{k} \det B^{k-1}$.  If $n$ is even,  $\phhi_N$ is the determinant of the $N\times N$ submatrix of 
$U(X,Y)$ obtained by deleting the first $n+2$ and the last $n$ columns. 
This determinant is transformed via multiplication 
by $\det A^{k} \det B^{k-1}a'$.
This proves the semi-invariance of $\phhi_N(X,Y)$.

Similarly, if $n$ is even, then $\pssi_M(X,Y)= \pssi_N(X,Y)$ is the determinant of the 
submatrix of $U(X,Y)$ obtained by deleting the first $n+2$ and the last $n$ columns. Under the action 
described above, this determinant is multiplied by
$\det A^{k} \det B^{k-1} a'$. If $n$ is odd then   $\pssi_M(X,Y)= \pssi_{N-n+1}(X,Y)$ 
is the determinant of the  submatrix of $U(X,Y)$ obtained by deleting the first 
$n-1$ rows and the first $2n+1$ and the last $n$ columns.
Under the action, it gets multiplied by $\det A^{k-1} \det B^{k-2}a'^2$. 
This completes the proof of the first claim.

(ii) For $i\in [2,n]$, $j\in [n]$, let us consider $\{\log x_{ij}, \log \phhi_N\}_D$. 
If we view $x_{ij}$ as the $(i-1,j)$-th matrix entry of the lower right block
of $U(X,Y)$, we can use Proposition \ref{brack_uho} with $f_1(X,Y)=\log  x_{ij}$ and 
$f_2(X,Y)=\log \phhi_N(X,Y)$. By \eqref{Unabla}, \eqref{nablaU}, $\Phhi^2 = k \one_{n-1}$ and 
$\Pssi^2 = (k-1) \one_{n+1}$. Also, $I^1=0$ and $J^1=\frac {1} {x_{ij}} e_{j,i-1}$,
 so $\Phhi^1 = \frac {1} {x_{ij} }\X e_{j,i-1}$, $\Pssi^1 = \frac {1} {x_{ij} }e_{j,i-1} \X$. 
 It follows that the last two terms in \eqref{brack_uho_form} vanish and the rest of the terms 
 depend only on diagonal parts $\Phhi^1_0=e_{i-1,i-1}$ and $\Pssi^1_0 =  e_{jj}$. 
 Then \eqref{brack_uho_form} implies that $\{\log x_{ij}, \log \phhi_N\}_D$ is constant. 

Next, for $i\in [n-1]$, $j\in [n]$, let us consider $\{ \log \phhi_N, \log y_{ij}\}_D$, 
where $y_{ij}$ is viewed as the $(i,j+1)$-th matrix entry of the upper left block
of $U(X,Y)$. Now, $\Phhi^1 = k \one_{n-1}$ and $\Pssi^1 = (k-1) \one_{n+1}$, $J^2 = 0$, 
$I^2 = \frac {1} {y_{ij}} e_{j+1,i}$ and
\eqref{brack_uho_form} involves only diagonal parts $\Phhi^2_0= e_{ii}$ and $\Pssi^2_0= e_{j+1, j+1}$ 
so $\{ \log \phhi_N, \log y_{ij}\}_D$ is constant. Assertion (ii) follows after restricting the bracket 
$\Poi_D$ to the diagonal subgroup of the double. The case of $\pssi_M$ can be treated similarly.
\end{proof}

\section{Initial quiver}
The goal of this section is to prove Theorem~\ref{quiver}.

As in the previous section, we concentrate on the case $n$ odd.
It follows from \eqref{brack_uho_form}, \eqref{aux3},
\eqref{aux4}, \eqref{bad1}, \eqref{bad2}, \eqref{bad3} and \eqref{bad4} that the bracket 
$\Poi_D$ for $f_{1}$, $f_2$ as in~\eqref{4cases} is given by
\begin{equation}\label{omegan}
\begin{aligned}
\{f_1,f_2\}_D&=\frac 1n\left(\Tr(D_{n-1}\Phhi_0^1)\Tr\Phhi_0^2-\Tr(D_{n-1}\Phhi_0^2)\Tr\Phhi_0^1\right)\\
&-\frac 1n\left(\Tr(D_{n+1}\Pssi_0^1)\Tr\Pssi_0^2-\Tr(D_{n+1}\Pssi_0^2)\Tr\Pssi_0^1\right)\\
&+\left\langle\sigma_-\Phhi_0^1,\Phhi_0^2\right\rangle -\left\langle\sigma_-\Phhi_0^2,\Phhi_0^2\right\rangle\\
&-\left\langle\sigma_-\Pssi_0^1,\Pssi_0^2\right\rangle+\left\langle\sigma_-\Pssi_0^2,\Pssi_0^2\right\rangle
+\varepsilon\Pssi_{11}^2,
\end{aligned}
\end{equation}
where $\varepsilon=0$ in cases 1 and 2, $\varepsilon=1$ in case 3, $\varepsilon=-1$ in case 4, 
and $\sigma_-=\gamma_-/(1-\gamma_-)$ is the sum of all shifts in the south-eastern direction.

For any nonnegative integer $q$ and any  $m \in \mathbb{N}$ define
$$
T_{m} (q) =\left  \lfloor  \frac q m\right \rfloor \one_m + \left [
\begin{array}{cc}  0 & 0 \\ 0 & \one_{\hat q}\end{array} \right ],
$$
where $\hat q=q\bmod m$ and $0\le \hat q\le m-1$.

Lemmas \ref{gradfi}, \ref{gradpsi} imply that 
\begin{gather*}
\Phhi_0(\bar\phhi_q) =\Pssi_0(\bar\pssi_q)= T_{n-1}(q),\\ 
\quad \Pssi_0(\bar\phhi_q) = T_{n+1}(q),
\quad \Pssi_0(\bar\pssi_q) = e_{11} + T_{n+1}(q-1),
\end{gather*}
and hence
$$
\Tr(\Phhi_0(\bar\phhi_q))=\Tr(\Pssi_0(\bar\phhi_q))=
\Tr(\Phhi_0(\bar\pssi_q))=\Tr(\Pssi_0(\bar\pssi_q))=q.
$$

In what follows we will need several identities involving $T_m(q)$ and $\sigma_-$.
Denote $\Delta_m(q)=T_m(q)-T_m(q-1)$ for $q\in \N$.

\begin{proposition}\label{identities}
{\rm (i)} For any $m, q_1\in\N$ and any integer $r$ such that $q_1+rm\ge 0$,
\begin{equation}\label{Tshift}
 T_m(q_1+rm)=T_m(q_1)+r\one_m.
\end{equation}

{\rm (ii)} For any $m\in \N$ and $q_1\in [m]$ 
\begin{equation}\label{idty0}
 \Tr\left(D_m(\Delta_m(q_1+1)-\Delta_m(q_1))\right)=-1.
\end{equation}
 
 {\rm (iii)} For any $m, q_1\in \N$
\begin{equation}\label{idty1}
 \langle\sigma_-\Delta_m(q_1),\Delta_m(q_1)\rangle=0.
 \end{equation}

 {\rm (iv)} For any $m\in \N$ and any diagonal $\xi$
 \begin{equation}\label{idty2}
  \sigma_-\xi+\sigma_+\xi=\Tr\xi\cdot\one_m-\xi,
 \end{equation}
where $\sigma_+=\gamma_+/(1-\gamma_+)$ is conjugate to $\sigma_-$ with respect to the trace-form.
 
 {\rm (v)}  For any $m, q_1\in \N$ such that $m\nmid q_1$ and any diagonal $\xi$
 \begin{equation}\label{idty7}
  \langle\sigma_-\xi, \Delta_m(q_1)\rangle+
  \langle\sigma_+\xi, \Delta_m(q_1+1)\rangle
  =\Tr\xi.
 \end{equation}
 
 {\rm (vi)} For any $m, q_1\in \N$ such that $m\nmid q_1$ and any diagonal $\xi$
 \begin{equation}\label{idty6}
  \left\langle\sigma_-\xi,\Delta_m(q_1+1)-\Delta_m(q_1)\right\rangle=-\langle \xi, \Delta_m(q_1+1)\rangle.
 \end{equation}
 
  {\rm (vii)} For any $m, q_1\in \N$
 \begin{equation}\label{idty3}
  \langle\sigma_-\one_m,T_m(q_1)\rangle+\langle\sigma_-T_m(q_1),\one_m\rangle=(m-1)q_1.
 \end{equation}

 {\rm (viii)}
 For any $m, q_1, q_2\in \N$
 \begin{equation}\label{idty4}
  \langle\sigma_-T_m(q_1),\Delta_m(q_2)\rangle+\langle\sigma_-\Delta_m(q_2), T_m(q_1)\rangle
  =q_1-\langle T_m(q_1), \Delta_m(q_2)\rangle.
 \end{equation}

 {\rm (ix)} For any $m, q_1\in \N$
 \begin{equation}\label{idty5}
  \langle\sigma_-T_m(q_1),T_m(q_1)\rangle-\langle\sigma_-T_m(q_1-1),T_m(q_1-1)\rangle=
  q_1-\langle T_m(q_1), \Delta_m(q_1)\rangle.
 \end{equation}
 \end{proposition}

\begin{proof}
 (i) Evident.
 
 (ii) Clearly,  $\Delta_m(q_1)=e_{m+1- q_1, m+1- q_1}$ for $q_1\in [m]$, 
 hence the trace in question is the difference of the two 
 consecutive diagonal entries of $D_m$.
 
 (iii) Evident.
 
 (iv) Consider the $(i,i)$-th entry of $\sigma_-\xi+\sigma_+\xi$. Clearly, the contribution of $\sigma_-\xi$ equals $\xi_{11}+\dots+\xi_{i-1,i-1}$, whereas the contribution of $\sigma_+\xi$
 equals $\xi_{i+1,i+1}+\dots+\xi_{mm}$, since $\sigma_+$ is the sum of all shifts in the north-west direction. Now \eqref{idty2} follows immediately.

(v) Indeed, 
$$
\left\langle\sigma_-\xi,\Delta_m(q_1)\right\rangle+\left\langle\sigma_+\xi,\Delta_m(q_1+1)\right\rangle
=\left\langle\xi,\sigma_+\Delta_m(q_1)\right\rangle+
\left\langle\xi,\sigma_-\Delta_m(q_1+1)\right\rangle.
$$
If $m\nmid q_1$ then the first summand above equals $\xi_{11}+\dots+\xi_{m-\hat q_1, m-\hat q_1}$, whereas
the second summand equals $\xi_{m+1-\hat q_1, m+1-\hat q_1}+\dots+\xi_{mm}$, and \eqref{idty7} follows. 

 (vi) Indeed, 
$$
\left\langle\sigma_-\xi,\Delta_m(q_1+1)-\Delta_m(q_1)\right\rangle=
\left\langle\xi,\sigma_+(\Delta_m(q_1+1)-\Delta_m(q_1))\right\rangle.
$$
  Therefore, if $m\nmid q_1$, then $\sigma_+(\Delta_m(q_1+1)-\Delta_m(q_1))=-\Delta_m(q_1+1)$, and \eqref{idty6} follows.
  
 (vii) By \eqref{idty2} for $\xi=\one_m$,
 \begin{gather*}
 \langle\sigma_-\one_m,T_m(q_1)\rangle+\langle\sigma_-T_m(q_1),\one_m\rangle=
 \langle\sigma_-\one_m+\sigma_+\one_m, T_m(q_1)\rangle\\
 =\langle m\one_m,T_m(q_1)\rangle-\langle \one_m,T_m(q-1)\rangle=(m-1)q_1.
 \end{gather*}
 
 (viii) By \eqref{idty2} for $\xi=T_m(q_1)$,
 \begin{gather*}
 \langle\sigma_-T_m(q_1),\Delta_m(q_2)\rangle+\langle\sigma_-\Delta_m(q_2), T_m(q_1)\rangle=  
  \langle\sigma_-T_m(q_1)+\sigma_+T_m(q_1),\Delta_m(q_2)\rangle \\
 = q_1\langle \one_m, \Delta_m(q_1)\rangle-\langle T_m(q_1),\Delta_m(q_1)\rangle=
 q_1 -   \langle T_m(q_1),\Delta_m(q_1)\rangle.
 \end{gather*}
 
 (ix) Indeed,
 \begin{equation*}
 \begin{aligned}
  &\langle\sigma_-T_m(q_1),T_m(q_1)\rangle-\langle\sigma_-T_m(q_1-1),T_m(q_1-1)\rangle\\
  & =\langle\sigma_-T_m(q_1),\Delta_m(q_1)\rangle+\langle\sigma_-\Delta_m(q_1),T_m(q_1-1)\rangle\\
&= \langle\sigma_-T_m(q_1),\Delta_m(q_1\rangle+\langle\sigma_-\Delta_m(q_1),T_m(q_1)\rangle\\ 
&\qquad-\langle\sigma_-\Delta_m(q_1),\Delta_m(q_1)\rangle.
\end{aligned}
\end{equation*} 
The third term vanishes by \eqref{idty1}, and the sum of the first two terms equals $q_1-\langle T_m(q_1), \Delta_m(q_1)\rangle$ by \eqref{idty4} for $q_2=q_1$.
\end{proof}

Let $\widetilde{B}(n)$ denote the (extended) exchange matrix corresponding to the quiver $Q_{CG}(n)$, and $\Omega(n)$ denote the 
matrix of brackets $\Poi_D$ for functions in~\eqref{4cases} given by \eqref{omegan}. By Proposition~\ref{Bomega}, 
we have to check that 
\begin{equation}\label{criterion}
\widetilde{B}(n)\Omega(n)=[\lambda\one_{n^2-3}\ 0]
\end{equation}
for some constant $\lambda$. We will prove that this equality holds
with $\lambda=1$. Clearly, this implies that $\wB(n)$ is of full rank. Similar statements for
the quiver $Q'_{CG}(n)$ and the Cremmer--Gervais bracket on $SL_n$ follow immediately.

There are three types of vertices in the quiver $Q_{CG}(n)$: $\phhi$-vertices, $\pssi$-vertices and 
$\thetta$-vertices. In turn, $\phhi$-vertices can be broken into two types: standard (correspond to 
$\phhi_q$ with $q\in [n+1,N-1]$, $q\ne M$) and superdiagonal (correspond to $\phhi_q$ with $q\in [2,n-1]$), and 
three special vertices that correspond to $\phhi_1$, $\phhi_n$ and $\phhi_M$. 
Standard and superdiagonal
$\phhi$-vertices have 6 incident edges, as shown on Fig.~\ref{fig:typ1}. 

\begin{figure}[ht]
\begin{center}
\includegraphics[height=4cm]{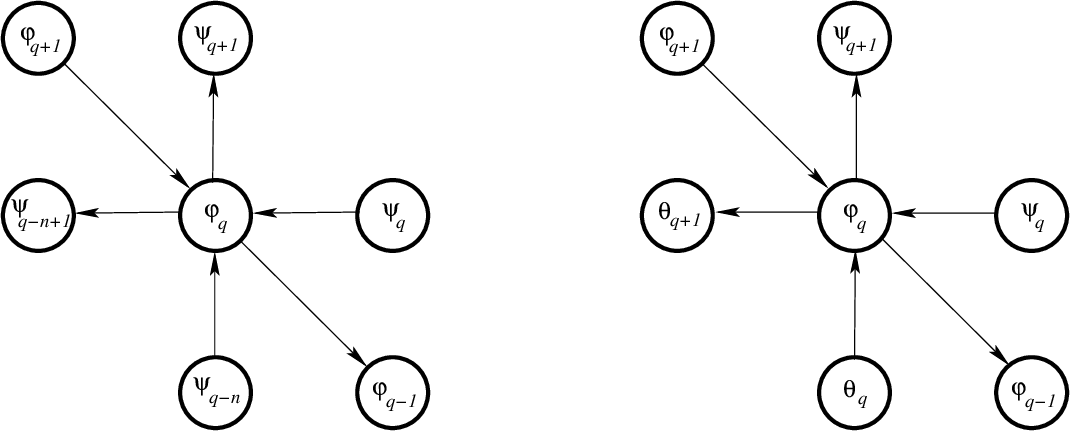}
\caption{Standard and superdiagonal $\phhi$-vertices}
\label{fig:typ1}
\end{center}
\end{figure}

Therefore, for a standard $\phhi$-vertex, checking \eqref{criterion} with $\lambda=1$ amounts to checking
\begin{equation}\label{check11}
\begin{gathered}
 \{\bar\pssi_{q+1}-\bar\phhi_{q+1}+\bar\phhi_{q-1}-\bar\pssi_q+\bar\pssi_{q-n+1}-\bar\pssi_{q-n},
 \bar\phhi_p\}_D=\delta_{pq},\\
 \{\bar\pssi_{q+1}-\bar\phhi_{q+1}+\bar\phhi_{q-1}-\bar\pssi_q+\bar\pssi_{q-n+1}-\bar\pssi_{q-n},
  \bar\pssi_p\}_D=0
\end{gathered}
\end{equation}
for all $p\in [N]$.
Subdiagonal $\phhi$-vertices
that correspond to $\phhi_q$ with $q\in [M+1,N-1]$ are included into standard ones
for the following reason. For subdiagonal vertices, $\bar\pssi_q$ and $\bar\pssi_{q+1}$ should be replaced
by $\bar\thetta_{q-M}$ and $\bar\thetta_{q-M+1}$, respectively. However, as we have already noted,   
$\bar\thetta_r=\bar\pssi_{M+r}-\bar\pssi_M$;
therefore, relations \eqref{check11} remain relevant for subdiagonal $\phhi$-vertices as well.
For the same reason, a similar relation for $\bar\thetta_p$ is not included, since it follows from
the second relation in~\eqref{check11}.

Consider the contribution of the first line in ~\eqref{omegan} to $\Poi_D$. Clearly, this contribution for a
pair of functions $f_{q_1}$ and $f_{q_2}$ equals
\begin{equation*}\label{line1}
\frac 1n \left(q_2\Tr(D_{n-1}T_{n-1}(q_1))-q_1\Tr(D_{n-1}T_{n-1}(q_2))\right)
\end{equation*}
for $f=\bar\phhi$ and $f=\bar\pssi$ alike and arbitrary $q_1, q_2$. Taking into account ~\eqref{Tshift} for 
$m=n-1$, $r=-1$ and $q_1=q$ or $q_1=q-1$,
we see that the contribution of the first line in~\eqref{omegan} to the left hand sides of both relations
in~\eqref{check11} vanishes.

The contribution of the second line in ~\eqref{omegan} to $\Poi_D$ equals
$$
\frac 1n \left(q_1\Tr(D_{n+1}T_{n+1}(q_2))-q_2\Tr(D_{n+1}T_{n+1}(q_1))\right)
$$
for a pair $\bar\phhi_{q_1}, \bar\phhi_{q_2}$,
$$
\frac 1n \left(q_1+q_1\Tr(D_{n+1}T_{n+1}(q_2-1))-q_2\Tr(D_{n+1}T_{n+1}(q_1))\right)
$$
for a pair $\bar\phhi_{q_1}, \bar\pssi_{q_2}$, and
$$
\frac 1n \left(q_1-q_2+q_1\Tr(D_{n+1}T_{n+1}(q_2-1))-q_2\Tr(D_{n+1}T_{n+1}(q_1-1))\right)
$$
for a pair $\bar\pssi_{q_1}, \bar\pssi_{q_2}$. Taking into account ~\eqref{Tshift} for $m=n+1$, $r=-1$,
and $q_1=q+2$ or $q_1=q+1$,
we see that the contribution of the second line in~\eqref{omegan} to the left hand sides of both relations
in~\eqref{check11} vanishes.

The contribution of the third line in ~\eqref{omegan} to $\Poi_D$ equals
$$
\langle\sigma_-T_{n-1}(q_1),T_{n-1}(q_*)\rangle -\langle\sigma_-T_{n-1}(q_2),T_{n-1}(q_*)\rangle
$$
with $q_*=\min\{q_1,q_2\}$, for $f=\bar\phhi$ and $f=\bar\pssi$ alike and arbitrary $q_1, q_2$.
Therefore, for $p\ge q+1$ the  contribution of the third line in ~\eqref{omegan} to the left hand 
sides of  both relations in~\eqref{check11} equals
\begin{equation*}
\begin{split} 
\langle \sigma_-T_{n-1}(q-1),\one_{n-1}\rangle +\langle \sigma_-\one_{n-1},T_{n-1}(q-1)\rangle\\
-\langle \sigma_-T_{n-1}(q),\one_{n-1}\rangle -\langle \sigma_-\one_{n-1},T_{n-1}(q)\rangle=2-n
\end{split}
\end{equation*}
by~\eqref{idty3} for $m=n-1$ and $q_1=q$ or $q_1=q-1$. 

For $q-n+1\le p\le q$ this contribution equals
\begin{equation*}
\begin{split} 
\langle \sigma_-T_{n-1}(q-1),\one_{n-1}\rangle +\langle \sigma_-\one_{n-1},T_{n-1}(q-1)\rangle
-\langle \sigma_-T_{n-1}(q),\one_{n-1}\rangle \\-\langle \sigma_-\one_{n-1},T_{n-1}(q)\rangle
+\langle\sigma_-T_{n-1}(q), T_{n-1}(q)\rangle -\langle\sigma_- T_{n-1}(q-1), T_{n-1}(q-1)\rangle\\
-\langle\sigma_-\Delta_{n-1}(q), T_{n-1}(p)\rangle-\langle\sigma_-T_{n-1}(p),
\Delta_{n-1}(q)\rangle\\
=2-n+q-p+\langle  T_{n-1}(p)-T_{n-1}(q), \Delta_{n-1}(q)\rangle
\end{split}
\end{equation*}
by \eqref{idty3}, \eqref{idty4} for $q_1=p$ and $q_2=q$ and \eqref{idty5} for $q_1=q$. 
Clearly,
$$
\langle T_{n-1}(p)-T_{n-1}(q), \Delta_{n-1}(q)\rangle=\delta_{pq}-1
$$ 
when $q-n+1\le p\le q$, hence the contribution equals $1-n+q-p+\delta_{pq}$.

Finally, for $p\le q-n$ the  contribution of the third line in ~\eqref{omegan} to the left hand 
sides of  both relations in~\eqref{check11} vanishes.

The contribution of the fourth line in ~\eqref{omegan} to $\Poi_D$ equals
$$
\langle\sigma_-T_{n+1}(q_2),T_{n+1}(q_*)\rangle -\langle\sigma_-T_{n+1}(q_1),T_{n+1}(q_*)\rangle
$$
for a pair $\bar\phhi_{q_1}, \bar\phhi_{q_2}$,
$$
\langle\sigma_-T_{n+1}(q_2-1),T_{n+1}(q_*-1)\rangle -\langle\sigma_-T_{n+1}(q_1-1),T_{n+1}(q_*-1)\rangle
$$
for a pair $\bar\pssi_{q_1}, \bar\pssi_{q_2}$, and
$$\texttt{}
\langle\sigma_-T_{n+1}(q_2-1),T_{n+1}(q_\circ)\rangle -\langle\sigma_-T_{n+1}(q_1),T_{n+1}(q_\circ)\rangle
+q_*,
$$
where $q_\circ$ equals $q_1$ if $q_1\le q_2$ and $q_2-1$ if $q_2\le q_1$, 
for a pair $\bar\phhi_{q_1}, \bar\pssi_{q_2}$; the coincidence of the two results for 
$q_1=q_2$ follows immediately from~\eqref{idty1} with $m=n+1$ and $q_1=q$. 
Therefore, for $p\ge q+1$ the  contribution of the third line 
in ~\eqref{omegan} to the left hand 
sides of  both relations in~\eqref{check11} equals
\begin{equation*}
\begin{split} 
\langle \sigma_-T_{n+1}(q+1),\one_{n+1}\rangle +\langle \sigma_-\one_{n+1},T_{n+1}(q+1)\rangle\\
-\langle \sigma_-T_{n+1}(q),\one_{n+1}\rangle -\langle \sigma_-\one_{n+1},T_{n+1}(q)\rangle-2=n-2
\end{split}
\end{equation*}
by~\eqref{idty3} for $m=n+1$ and $q_1=q+1$ or $q_1=q$. 

For $q-n+1\le p\le q$ this contribution equals
\begin{equation*}
\begin{split} 
\langle \sigma_-T_{n+1}(q+1),\one_{n+1}\rangle +\langle \sigma_-\one_{n+1},T_{n+1}(q+1)\rangle
-\langle \sigma_-T_{n+1}(q),\one_{n+1}\rangle \\-\langle \sigma_-\one_{n+1},T_{n+1}(q)\rangle
-\langle\sigma_-T_{n+1}(q+1), T_{n+1}(q+1)\rangle +\langle\sigma_- T_{n+1}(q), T_{n+1}(q)\rangle
-\delta_{f\phhi}\\+\langle\sigma_-\Delta_{n+1}(q+1), T_{n+1}(p-\delta_{f\pssi})\rangle+\langle\sigma_-T_{n+1}(p-\delta_{f\pssi}),
\Delta_{n+1}(q+1)\rangle-\delta_{f\pssi}\delta_{pq}\\
=n-2-q+p-\delta_{f\psi}\delta_{pq}-\langle T_{n+1}(p-\delta_{f\pssi})-T_{n+1}(q+1), \Delta_{n+1}(q+1)\rangle
\end{split}
\end{equation*}
by \eqref{idty3}, \eqref{idty4} for $q_1=p-\delta_{f\pssi}$ and $q_2=q+1$ and \eqref{idty5} for $q_1=q+1$. Expression
$\delta_{f\pssi}\delta_{pq}$ is a consequence of the following fact: the contribution of the fourth line 
in~\eqref{omegan} to the bracket
$\{\bar\phhi_{q-1}-\bar\pssi_q, \bar f_q\}_D$
equals $q$ for  $f=\phhi$ and $q-1$ for $f=\pssi$.
Clearly,
$$
\langle T_{n+1}(p-\delta_{f\pssi})-T_{n+1}(q+1), \Delta_{n+1}(q+1)\rangle=-1
$$ 
when $q-n+1\le p\le q$, hence the contribution equals $n-1-q+p-\delta_{f\pssi}\delta_{pq}$.

Finally, for $p\le q-n$ the  contribution of the fourth line in ~\eqref{omegan} to the left hand 
sides of  both relations in~\eqref{check11} vanishes.

Comparing the obtained results we see that relations~\eqref{check11} are valid for all standard $\phhi$-vertices.

For superdiagonal $\phhi$-vertices corresponding to $q\in [2,n-2]$, checking \eqref{criterion} amounts to checking
\begin{equation}\label{check12}
\begin{gathered}
 \{\bar\pssi_{q+1}-\bar\phhi_{q+1}+\bar\phhi_{q-1}-\bar\pssi_q+\bar\pssi_{M+q+1}-\bar\pssi_{M+q},
 \bar\phhi_p\}_D=\delta_{pq},\\
 \{\bar\pssi_{q+1}-\bar\phhi_{q+1}+\bar\phhi_{q-1}-\bar\pssi_q+\bar\pssi_{M+q+1}-\bar\pssi_{M+q},
  \bar\pssi_p\}_D=0
\end{gathered}
\end{equation}
for all $p\in [N]$. Here, as in the case of subdiagonal $\phhi$-vertices, we have used the identity
$\bar\thetta_{M+q+1}-\bar\thetta_{M+q}= \bar\pssi_{M+q+1}-\bar\pssi_{M+q}$. For $q=n-1$ the term
$\bar\pssi_{M+q+1}$ in both relations in \eqref{check12} should be replaced by $\bar\pssi_M$, since $\thetta_n=\det X$ is a Casimir function.

The contributions of each line in ~\eqref{omegan} to $\Poi_D$ are given by the same formulas as in the case 
of standard $\phhi$-vertices. Recall that $M=(k-1)(n-1)$ with
$k=\frac{n+1}2\in\N$. Therefore,
taking into account~\eqref{Tshift} for 
$m=n-1$, $r=k-1$ and $q_1=q+1$ or $q_1=q$, and~\eqref{idty0} for $m=n-1$, $q_1=q$,
we see that for $q\in [2,n-2]$ the contribution of the first line in~\eqref{omegan} to the left hand sides of both relations in~\eqref{check12} equals $-p/n$. 

Note that $M-2=(k-2)(n+1)$, so,  
taking into account ~\eqref{Tshift} for $m=n+1$, $r=k-2$,
and $q_1=q+2$ or $q_1=q+1$,
we see that for $q\in [2,n-2]$ the contribution of the second line in~\eqref{omegan} to the left hand sides of both relations
in~\eqref{check12} equals $p/n$.

Let $p\ge M+q+1$ and $q\in [2,n-2]$, 
then the  contribution of the third line in ~\eqref{omegan} to the left hand 
sides of  both relations in~\eqref{check12} equals
\begin{equation*}
\begin{split} 
(k-1)\left(\langle \sigma_-T_{n-1}(q+1),\one_{n-1}\rangle +\langle \sigma_-\one_{n-1},T_{n-1}(q+1)\rangle\right)\\
-(k-1)\left(\langle \sigma_-T_{n-1}(q),\one_{n-1}\rangle +\langle \sigma_-\one_{n-1},T_{n-1}(q)\rangle\right)\\
-\langle \sigma_-T_{n-1}(p),\Delta_{n-1}(q+1)-\Delta_{n-1}(q)\rangle
+\langle \sigma_-T_{n-1}(q+1),T_{n-1}(q+1)\rangle\\
-2\langle \sigma_-T_{n-1}(q),T_{n-1}(q)\rangle+\langle \sigma_-T_{n-1}(q-1),T_{n-1}(q-1)\rangle\\
=1+(k-1)(n-2)+\langle T_{n-1}(p), \Delta_{n-1}(q+1)\rangle\\
-\langle T_{n-1}(q+1),\Delta_{n-1}(q+1)\rangle+\langle T_{n-1}(q),\Delta_{n-1}(q)\rangle
\end{split}
\end{equation*}
by~\eqref{idty6}  for $m=n-1$, $\xi=T(p)$, $q_1=q$ (\eqref{idty6} is applicable since $q\in [2,n-2]$), and~\eqref{idty3}, \eqref{idty5} for $m=n-1$ and $q_1=q+1$ or $q_1=q$.
Clearly,
\begin{equation}\label{adidty}
 \langle T_{n-1}(q),\Delta_{n-1}(q)\rangle=
 \langle T_{n-1}(q+1),\Delta_{n-1}(q+1)\rangle=1
\end{equation}
since $q\in [2,n-2]$, and 
$$
\langle T_{n-1}(p), \Delta_{n-1}(q+1)\rangle=\left\lfloor\frac p{n-1}\right\rfloor=k-1
$$
since $p\ge M+q+1$, hence the contribution equals $1+(k-1)(n-1)=M+1$.

For $q+1\le p\le M+q$ and $q\in [2,n-2]$ this contribution equals
\begin{equation*}
 \begin{split}
 \langle \sigma_-T_{n-1}(q-1), T_{n-1}(q-1)\rangle-
 \langle \sigma_-T_{n-1}(p), T_{n-1}(q-1)\rangle\\-
 \langle \sigma_-T_{n-1}(q), T_{n-1}(q)\rangle+
 \langle \sigma_-T_{n-1}(p), T_{n-1}(q)\rangle\\+
  \langle \sigma_-T_{n-1}(q+1), T_{n-1}(p)\rangle-
  \langle \sigma_-T_{n-1}(q), T_{n-1}(p)\rangle\\
 = -q+\langle T_{n-1}(q), \Delta_{n-1}(q)\rangle
  \langle \sigma_-T_{n-1}(p), \Delta_{n-1}(q)\rangle\\+
   \langle\sigma_+T_{n-1}(p), \Delta_{n-1}(q+1)\rangle
   =1-q+p
 \end{split}
 \end{equation*}
by~\eqref{idty5} for $m=n-1$ and $q_1=q$, \eqref{idty7} for $\xi=T_{n-1}(p)$ and $q_1=q$, and~\eqref{adidty}.

Finally, for $p\le q$ and $q\in [2,n-2]$ this contribution equals
\begin{equation*}
  -\langle \sigma_-\Delta_{n-1}(q),T_{n-1}(p)\rangle+
   \langle\sigma_-\Delta_{n-1}(q+1)), T_{n-1}(p)\rangle=\delta_{pq}
 \end{equation*}
since 
$$
\langle\sigma_-\Delta_{n-1}(q_1), T_{n-1}(p)\rangle=\begin{cases}
p\quad\text{if $p<q_1<n$,}\\
p-1\quad\text{if $p=q_1<n$.}
\end{cases}
$$

Let $p\ge M+q+1$ and $q\in [2,n-2]$, then the  contribution of the fourth line in ~\eqref{omegan} to the left hand 
sides of  both relations in~\eqref{check12} equals
\begin{equation*}
\begin{split} 
-2-(k-2)\left(\langle \sigma_-T_{n+1}(q+2),\one_{n+1}\rangle +\langle \sigma_-\one_{n+1},T_{n+1}(q+2)\rangle\right)\\
+(k-2)\left(\langle \sigma_-T_{n+1}(q+1),\one_{n+1}\rangle +\langle \sigma_-\one_{n+1},T_{n+1}(q+1)\rangle\right)\\
+\langle \sigma_-T_{n+1}(p-\delta_{f\pssi}),\Delta_{n+1}(q+2)-\Delta_{n+1}(q+1)\rangle
-\langle\sigma_-T_{n+1}(q+2), T_{n+1}(q+2)\rangle\\ +2\langle\sigma_-T_{n+1}(q+1), T_{n+1}(q+1)\rangle
-\langle\sigma_-T_{n+1}(q), T_{n+1}(q)\rangle\\
=-3-(k-2)n + \langle T_{n+1}(q+2),\Delta_{n+1}(q+2)\rangle\\
-\langle T_{n+1}(q+1),\Delta_{n+1}(q+1)\rangle
-\langle T_{n+1}(p-\delta_{f\pssi}), \Delta_{n+1}(q+2)\rangle
\end{split}
\end{equation*}
by~\eqref{idty5} for $m=n+1$ and $q_1=q+2$ or $q_1=q+1$,~\eqref{idty6}  for $m=n+1$, $\xi=T(p-\delta_{f\pssi})$, $q_1=q+1$, 
and~\eqref{idty3} for $m=n+1$ and $q_1=q+2$ or $q_1=q+1$.
Clearly,
\begin{equation}\label{adidty2}
 \langle T_{n+1}(q+1),\Delta_{n+1}(q+1)\rangle=
 \langle T_{n+1}(q+2),\Delta_{n+1}(q+2)\rangle=1
\end{equation}
since $q\in [2,n-2]$, and 
$$
\langle T_{n+1}(p-\delta_{f\pssi}), \Delta_{n+1}(q+2)\rangle=\left\lfloor\frac {p-\delta_{f\pssi}}{n+1}\right\rfloor=k-2
$$
since $p\ge M+q+1$, hence the contribution equals $-3-(k-2)(n+1)=-M-1$.

For $q+1\le p\le M+q$ and $q\in [2,n-2]$ this contribution equals
\begin{equation*}
 \begin{split}
  -1-\delta_{f\pssi}+q-\langle \sigma_-T_{n+1}(p-\delta_{f\pssi}), \Delta_{n+1}(q+1)\rangle-
   \langle\sigma_+T_{n+1}(p-\delta_{f\pssi}), \Delta_{n+1}(q+2)\rangle\\
   +\langle T_{n+1}(q+1), \Delta_{n+1}(q+1)\rangle=-1+q-p
 \end{split}
 \end{equation*}
by~\eqref{idty5} for $m=n+1$ and $q_1=q+1$,
\eqref{idty7} for $m=n+1$, $\xi=T_{n+1}(p-\delta_{f\pssi})$ and $q_1=q+1$ and~\eqref{adidty2}.

Finally, for $p\le q$ and $q\in [2,n-2]$ this contribution equals
\begin{equation*}
 \begin{split}
  -\langle \sigma_-\Delta_{n+1}(q+2),T_{n+1}(p-\delta_{f\pssi})\rangle+
   \langle\sigma_-\Delta_{n+1}(q+1)), T_{n+1}(p-\delta_{f\pssi})\rangle-\delta_{f\pssi}\delta_{pq}\\
   =-\langle\sigma_+T_{n+1}(p-\delta_{f\pssi}), \Delta_{n+1}(q+2)\rangle
   +\langle\sigma_+T_{n+1}(p-\delta_{f\pssi}), \Delta_{n+1}(q+1)\rangle-\delta_{f\pssi}\delta_{pq}\\=
   -\delta_{f\pssi}\delta_{pq}
 \end{split}
 \end{equation*}
since 
$$
\langle\sigma_- \Delta_{n+1}(q_1),T(p)\rangle=p
$$
for $p<q_1<n$.

Comparing the obtained results we see that relations~\eqref{check12} are valid for all superdiagonal 
$\phhi$-vertices with $q\in [2,n-2]$.

Special $\phhi$-vertices have 5 incident edges, as shown on Fig.~\ref{fig:spec1}. It is easy to check that relations similar 
to~\eqref{check11} and~\eqref{check12} are valid for special $\phhi$-vertices. In fact, this check does not involve
any additional computations. For example, since $\bar\thetta_1=\bar\pssi_{M+1}-\bar\pssi_M$, relations for
$\phhi_M$ coincide with relations for a standard $\phhi$-vertex and are proved in the same way. To treat vertices $\phhi_1$
and $\phhi_n$ it suffices to introduce an artificial vertex corresponding to $q=0$ (this vertex should be a $\phhi$-vertex
in the former case and a $\pssi$-vertex in the latter).

\begin{figure}[ht]
\begin{center}
\includegraphics[height=3.5cm]{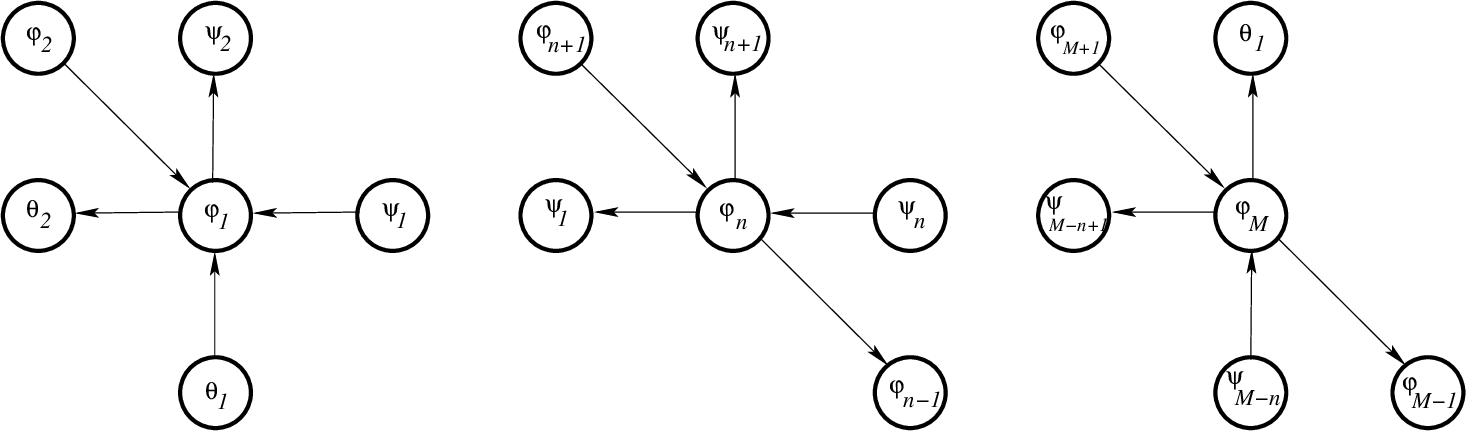}
\caption{Special $\phhi$-vertices}
\label{fig:spec1}
\end{center}
\end{figure}

Most of the $\pssi$-vertices and $\thetta$-vertices are standard, except for the 
two special vertices corresponding
to $\pssi_1$ and $\thetta_1$.
Standard 
$\pssi$-vertices and $\thetta$-vertices have 6 incident edges, as shown on Fig.~\ref{fig:typ2}. 

\begin{figure}[ht]
\begin{center}
\includegraphics[height=3.5cm]{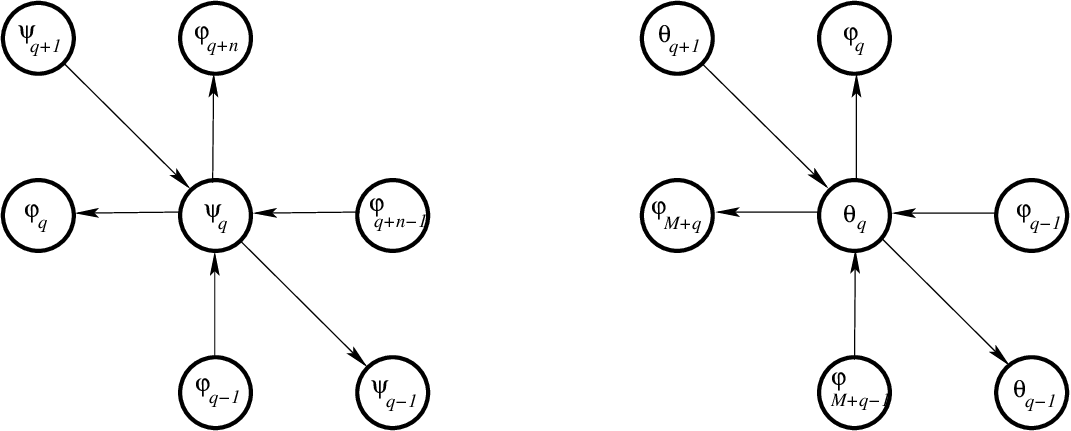}
\caption{Standard $\pssi$-vertices and $\thetta$-vertices}
\label{fig:typ2}
\end{center}
\end{figure}

The special $\pssi$-vertex and $\thetta$-vertex have 4 incident edges, as shown on Fig.~\ref{fig:spec2}.

\begin{figure}[ht]
\begin{center}
\includegraphics[height=3.5cm]{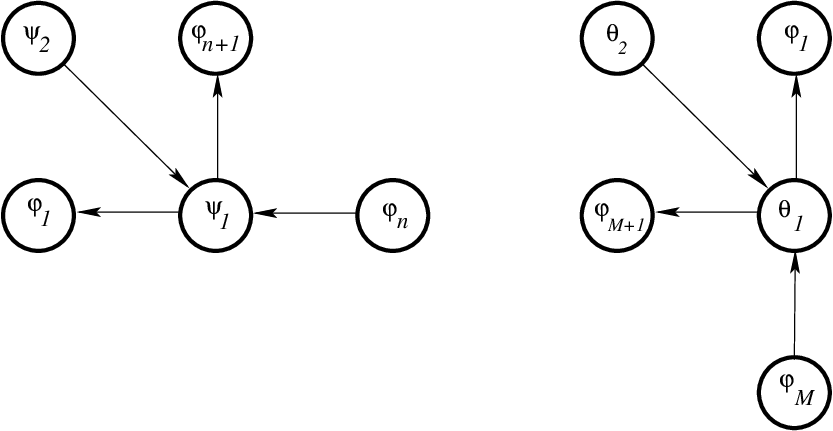}
\caption{Special $\pssi$-vertex and special $\thetta$-vertex}
\label{fig:spec2}
\end{center}
\end{figure}

For standard $\pssi$-vertices and $\thetta$-vertices, checking \eqref{criterion} amounts to checking
\begin{equation*}
\begin{gathered}
\{\bar\phhi_{q}-\bar\phhi_{q-1}-\bar\pssi_{q+1}+\bar\pssi_{q-1}+\bar\phhi_{q+n}-\bar\phhi_{q+n-1},
 \bar\phhi_p\}_D=0,\\
 \{\bar\phhi_{q}-\bar\phhi_{q-1}-\bar\pssi_{q+1}+\bar\pssi_{q-1}+\bar\phhi_{q+n}-\bar\phhi_{q+n-1},
 \bar\pssi_p\}_D=\delta_{pq},
\end{gathered}
\end{equation*}
for all $p\in [N]$. These relations are verified similarly to~\eqref{check11} and~\eqref{check12}. 
Special $\pssi$-vertices are treated similarly to special $\phhi$-vertices.

\begin{remark}
\label{casimirs}
{\rm A log-canonical coordinate system on a Poisson manifold is a useful tool for computing
the rank of the Poisson structure at a generic point and finding Casimir functions, as both tasks reduce to
computing the rank and the kernel of a constant matrix. In particular,~\eqref{criterion} guarantees that the maximal
rank of the Cremmer--Gervais structure on $SL_n$ is no less that $n^2-3$.
With a little more work along the lines of arguments used in this section, one can show that the maximal rank
is, in fact, $n^2-1$ for $n$ odd and $n^2-2$ for $n$ even, with the Casimir function in the latter case given
by $\frac {\pssi_N} {\phhi_N}$. 
}
\end{remark}

\section{Regularity}\label{sec:regular}
The goal of this section is to prove Theorem~\ref{regular} and to establish, as a corollary,
part (iii) of Conjecture~\ref{ulti} in the Cremmer--Gervais case. 

We need to show that for any mutable cluster variable $f$ in  the initial cluster given by 
~\eqref{inclust}, the adjacent variable $f'$ is a regular function
on $\Mat_n$. We will rely on two classical determinantal identities. 
The first is the Desnanot--Jacobi identity for minors of a square matrix $A$ (see \cite[Th.~3.12]{Br}):
\begin{equation}
\label{jacobi}
\det A \det A_{\hat \alpha \hat \beta}^{\hat \gamma \hat \delta} = 
\det A_{\hat \alpha }^{ \hat \gamma} \det A_{\hat \beta}^{ \hat \delta}-
\det A_{\hat \alpha }^{ \hat \delta} \det A_{\hat \beta}^{ \hat \gamma}
\end{equation}
where "hatted" subscripts and superscripts indicate deleted rows and columns, respectively. 
The second identity is for minors
of a matrix $B$ with the number of rows one less than the number of columns:
\begin{equation}
\label{notjacobi}
\det B^{\hat \alpha} \det B^{\hat \beta \hat \gamma}_{\hat \delta} + 
\det B^{\hat \gamma} \det B^{\hat \alpha\hat \beta}_{\hat \delta} = 
\det B^{\hat \beta} \det B^{\hat \alpha \hat \gamma}_{\hat \delta}.
\end{equation}
It can be obtained from~\eqref{jacobi} by adjoining an appropriate row.

 We will apply these identities to minors of a modification $V(X,Y)$ of the matrix
 $U(X,Y)$ defined in \eqref{uho}. This modification is defined as follows: 
 first, $U(X,Y)$ is {\it augmented\/} via
 attaching two extra rows, $[\underbrace{0\cdots 0}_{n+1}\ x_{11} \cdots x_{1n} \ 0 \cdots 0]$ on top 
 and $[ 0  \cdots 0 \ y_{n1} \cdots y_{nn} \ \underbrace{0\cdots 0}_{n+1}]$ at the bottom;
 second, the first and the last $n$ columns of the augmented matrix are deleted:
 $$
 V(X,Y)=\begin{bmatrix}
 Y_* & X & 0 & 0 & 0 & \hdots & 0\\
 0 & \Y & \X & 0 & 0 & \hdots & 0\\
 0 & 0 & \Y & \X & 0 & \hdots & 0\\
   &  & \ddots & \ddots & \ddots & \ddots &\\
 0 & 0 & \hdots & 0 & \Y & \X & 0\\
 0 & 0 & \hdots & 0 & 0 & Y & X^*  
 \end{bmatrix}
 $$
 with $X^*=[x_{21}, \hdots, x_{n1}, 0]^T$ and $Y_*=[0, y_{1n}, \hdots, y_{n-1,n}]^T$. 
 If $n$ is odd, $V(X, Y)$ is a square $(N+2)\times (N+2)$ matrix, and if $n$ is even, it
 is an $(N+2)\times (N+1)$ matrix.
 
  As before, we will only present the proof for odd $n$. Although in the proof of Theorem \ref{regular} 
  we only need to consider $V(X,X)$, we will work with identities satisfied by the minors of $V(X,Y)$.
Recall that functions $\phhi_q(X,Y)$, $q\in [N]$, represent dense minors whose lower right corner 
is $V(X,Y)_{N+1,N+1}=y_{n-1, n}$
  and $\pssi_q(X,Y)$, $q\in [N]$, represent dense minors whose lower right corner  is $V(X,Y)_{N+1,N+2}=x_{n1}$. 
  Functions $\thetta_q(X,Y)=\det X_{[n-q+1,n]}^{[n-q+1,n]}$, $q\in [n]$, can be represented as 
  $\thetta_q(X)=\det V(X,Y)_{[n-q+1,n]}^{[n-q+2,n+1]}$. We also have 
  $\theta_q(Y,X)=\det V(X,Y)_{[N-q+3,N+2]}^{[N-q+2,N+1]}$. Observe that 
$$
\det V(X,Y)_{[N-q+3,N+2]}^{[N-q+2,N+1]} = \thetta_n(Y,X) \pssi_{q-n} (X, Y)\qquad\text{for $q> n$}
$$
and 
$$
\det V(X,Y)^{[N-q+3,N+2]}_{[N-q+2,N+1]} =\pssi_q(X,Y)= \thetta_{q-M}(X,Y) \pssi_{M} (X, Y)\qquad
 \text{for $q> M$}.
$$

Denote by $\Lambda_q$ the $q\times q$ dense lower right minor of $V(X,Y)$. 
Applying the Desnanot--Jacobi identity to the $(q+2) \times (q+2)$ dense lower right submatrix of $V(X,Y)$
with $\alpha=\gamma=1$ and $\beta=\delta=q+2$, we obtain
\begin{equation}\label{supernei}
\Lambda_{q+2}(X, Y)\phhi_q (X,Y) = \Lambda_{q+1}(X, Y)\phhi_{q+1} (X,Y) - \thetta_{q+1}(Y,X)\pssi_{q+1} (X,Y) 
\end{equation}
for $q\in [n-1]$. Note that
$\Lambda_1=0$ and $\Lambda_2=- x_{n1} y_{n n} = - \phhi_1(X,Y) \thetta_1(Y,X)$.
Comparing~\eqref{supernei} for $q=1$ with the leftmost part of Fig.~\ref{fig:spec1}, one sees that the 
cluster mutation applied to 
$\phhi_1(X)=\phhi_1(X,X)$ transforms it into $\phhi_1'(X)=-\Lambda_3(X,X)$. 

For $q > 1$, we can eliminate $\Lambda_{q+1}$ from the $(q-1)$st and the $q$th equations in~\eqref{supernei} 
by multiplying the former by $\phhi_{q+1}(X,Y)$ and the latter by $\phhi_{q-1}(X,Y)$ and adding the two. 
This results in
\begin{equation*}
\begin{aligned}
(\phhi_{q+1}(X,Y)&\Lambda_{q}(X, Y)- \phhi_{q-1}(X,Y)\Lambda_{q+2}(X, Y) )\phhi_q (X,Y) \\
&   = \phhi_{q+1}(X, Y)\thetta_q(Y,X)\pssi_{q} (X,Y) + \phhi_{q-1}(X, Y)\thetta_{q+1}(Y,X)\pssi_{q+1} (X,Y)
\end{aligned}
\end{equation*}
for all superdiagonal $\phhi$-vertices.
Since the right-hand side above is consistent with the rightmost part of Fig.~\ref{fig:typ1},
 we see that for all such vertices 
 $\phhi'_q (X) = \phhi_{q+1}(X)\Lambda_{q}(X, X)- \phhi_{q-1}(X)\Lambda_{q+2}(X, X)$ is regular
on $\Mat_n$. 

Denote by $\Upsilon_q$ the determinant of the $q\times q$ submatrix of $V(X,Y)$ obtained by deleting the 
second column from the $q\times (q+1)$ dense submatrix of $V(X,Y)$ whose lower right corner is $V(X,Y)_{N+1,N+1}$,
and by $\bar\Upsilon_q$ the determinant of  the $q\times q$ submatrix of $V(X,Y)$ obtained by deleting the second 
column from the $q\times (q+1)$ dense submatrix of $V(X,Y)$ whose lower right corner is $V(X,Y)_{N+1,N+2}$.

Assume that $B$ in~\eqref{notjacobi} is the submatrix of $V(X,Y)$ involved in the definition of $\Upsilon_q$
and $\alpha=\delta=1$, $\beta=2$, $\gamma=q+1$.
Then
\begin{gather*}
B^{\hat\alpha}=\phhi_q(X,Y), \quad B^{\hat\beta}=\Upsilon_q(X,Y), \quad B^{\hat\gamma}=\varkappa(Y)\pssi_{q+1-n}(X,Y),\\
B^{\hat\beta\hat\gamma}_{\hat\delta}=\varkappa(Y)\bar\Upsilon_{q-n}(X,Y),\quad 
B^{\hat\alpha\hat\gamma}_{\hat\delta}=\varkappa(Y)\pssi_{q-n}(X,Y),\quad
B^{\hat\alpha\hat\beta}_{\hat\delta}=\phhi_{q-1}(X,Y),
\end{gather*}
where $\varkappa(Y) = \det Y_{[1,n-1]}^{[1,n-1]}$. 
Applying \eqref{notjacobi}  we obtain
\begin{equation}
\label{aux54}
\phhi_{q}(X,Y) \bar\Upsilon_{q-n}(X,Y)+ \phhi_{q-1}(X,Y)\pssi_{q+1-n}(X,Y)  
 = \Upsilon_{q}(X,Y) \pssi_{q-n}(X,Y),
\end{equation}
for $q\in [n+1,N]$.

Similarly, assume that $B$ in~\eqref{notjacobi} is the submatrix of $V(X,Y)$ involved in the definition 
of $\bar\Upsilon_q$
and $\alpha=\delta=1$, $\beta=2$, $\gamma=q+1$.
Then
\begin{gather*}
B^{\hat\alpha}=\pssi_q(X,Y), \quad B^{\hat\beta}=\bar\Upsilon_q(X,Y), \quad B^{\hat\gamma}=\phhi_{q}(X,Y),\\
B^{\hat\beta\hat\gamma}_{\hat\delta}=\Upsilon_{q-1}(X,Y),\quad 
B^{\hat\alpha\hat\gamma}_{\hat\delta}=\phhi_{q-1}(X,Y),\quad
B^{\hat\alpha\hat\beta}_{\hat\delta}=\pssi_{q-1}(X,Y).
\end{gather*}
Applying \eqref{notjacobi}  we obtain
\begin{equation}
\label{aux55}
\pssi_{q}(X,Y) \Upsilon_{q-1}(X,Y) + \phhi_{q}(X,Y) \pssi_{q-1}(X,Y) = \bar\Upsilon_q(X,Y) \phhi_{q-1}(X,Y)
\end{equation}
for $q\in [2,N]$.

Note that $\bar\Upsilon_1(X,Y)=y_{n-1,n}=\phhi_1(X,Y)$. Therefore,~\eqref{aux54} for $q=n+1$ reads
\begin{equation*}
\phhi_{n+1}(X,Y) \phhi_{1}(X,Y) + \pssi_{2}(X,Y) \phhi_{n}(X,Y)
= \Upsilon_{n+1}(X,Y) \pssi_{1}(X,Y). 
\end{equation*}
Since the left hand side above is consistent with the leftmost part of Fiq.~\ref{fig:spec2}, we infer
that $\pssi'_{1}(X)=\Upsilon_{n+1}(X,X)$ is regular on $\Mat_n$.

Next, $\Upsilon_n(X,Y)=\pssi_1(X,Y)\phhi_{n-1}(X,Y)$ 
hence~\eqref{aux55} for $q=n+1$ becomes
\begin{equation*}
\pssi_{n+1}(X,Y) \pssi_{1}(X,Y) \phhi_{n-1}(X,Y) + \phhi_{n+1}(X,Y) \pssi_{n}(X,Y) = 
\bar\Upsilon_{n+1}(X,Y) \phhi_{n}(X,Y).
\end{equation*}
The left hand side above is consistent with the central part of Fiq.~\ref{fig:spec1}, 
thus $\phhi'_{n}(X)=\bar\Upsilon_{n+1}(X,X)$ is regular on $\Mat_n$.

For $q\in [2,M-1]$, we eliminate $\bar\Upsilon_q$ from the $q$th equation in~\eqref{aux55} and 
the $(q+n)$th equation in~\eqref{aux54} by adding $\phhi_{q+n}$ times the former to $\phhi_{q-1}$ 
times the latter. Then we get
\begin{equation*}
\begin{aligned}
&\pssi_{q}(X,Y) \left (\Upsilon_{q+n}(X,Y)\phhi_{q-1}(X,Y) - \Upsilon_{q-1}(X,Y)\phhi_{q+n}(X,Y)\right ) \\
& \qquad =  \phhi_{q-1}(X,Y) \phhi_{q-1+n}(X,Y)\pssi_{q+1}(X,Y) + \phhi_{q}(X,Y) \phhi_{q+n}(X,Y)\pssi_{q-1}(X,Y).
\end{aligned}
\end{equation*}
After restricting to $Y=X$, the right hand side becomes consistent with the leftmost part of Fig.~\ref{fig:typ2}, 
hence $\pssi_{q}'(X)= \Upsilon_{q+n}(X,X)\phhi_{q-1}(X) - \Upsilon_{q-1}(X,X)\phhi_{q+n}(X)$
is a regular function on $\Mat_n$  for all standard $\pssi$-vertices.

On the other hand, we can eliminate $\Upsilon_q$ from the $q$th equation in~\eqref{aux54} and 
the$(q+1)$st equation in~\eqref{aux55} to obtain
\begin{equation}
\begin{aligned}
\label{aux56}
&\phhi_{q}(X,Y) \left (\bar\Upsilon_{q+1}(X,Y)\pssi_{q-n}(X,Y) - \bar\Upsilon_{q-n}(X,Y)\pssi_{q+1}(X,Y)\right ) \\
& \qquad =  \phhi_{q+1}(X,Y) \pssi_{q -n }(X,Y)\pssi_{q}(X,Y) + \phhi_{q-1}(X,Y) \pssi_{q+1}(X,Y)\pssi_{q-n+1}(X,Y)
\end{aligned}
\end{equation}
for $q\in [n+1,N-1]$. The right hand side of~\eqref{aux56} for $q\in [n+1,M-1]$ is consistent with the leftmost 
part of Fig.~\ref{fig:typ1}, hence 
$\phhi'_{q}(X,X)=\bar\Upsilon_{q+1}(X,X)\pssi_{q-n}(X,X) - \bar\Upsilon_{q-n}(X,X)\pssi_{q+1}(X,X)$
is a regular function on $\Mat_n$ for the corresponding standard $\phhi$-vertices.

Next, for $q\geq M$ we have $\pssi_q(X,Y)=\thetta_{q-M}(X,Y)\pssi_M(X,Y)$ and $\bar\Upsilon_{q+1}(X,Y)=
\det X_{[q+1-M,n]}^{q-M,n-1]}\pssi_M(X,Y)$ with $\thetta_0(X,Y)=1$.
Therefore,~\eqref{aux56} transforms into
\begin{equation}
\begin{aligned}
\label{aux57}
&\phhi_{q}(X,Y) \left (\det X^{[q-M,n-1]}_{[q+1-M,n]}\pssi_{q-n}(X,Y) - \bar\Upsilon_{q-n}(X,Y)\pssi_{q+1-M}(X,Y)\right ) \\
& \qquad =  \phhi_{q+1}(X,Y) \pssi_{q -n }(X,Y)\pssi_{q-M}(X,Y) + \phhi_{q-1}(X,Y) \pssi_{q+1-M}(X,Y)\pssi_{q-n+1}(X,Y).
\end{aligned}
\end{equation}
Since the right hand side of~\eqref{aux57} is consistent with the rightmost part of Fig.~\ref{fig:spec1} for $q=M$
and with the leftmost part of Fig.~\ref{fig:typ1} for $q>M$ (with $\pssi_q$ and $\pssi_{q+1}$ replaced by
$\thetta_{q-M}$ and $\thetta_{q-M+1}$ as explained immediately after Fig.~\ref{fig:typ1}), we infer that 
$\phhi'_{q}(X,X) =\det X^{[q-M,n-1]}_{[q+1-M,n]}\pssi_{q-n}(X,X) - \bar\Upsilon_{q-n}(X,X)\pssi_{q+1-M}(X,X)$
is regular on $\Mat_n$ for all the remaining $\phhi$-vertices. 

So far, we established that condition (iii) of Proposition~\ref{regfun} holds for all $\phhi$-vertices and
all $\pssi$-vertices. It remains to check it for $\thetta$-vertices.
To this end, denote by
$\hat \Upsilon_q$, $q\in [M+1,N]$,  the determinant of the submatrix obtained by deleting the $(q - M+1)$st 
row from the $(q+1)\times q$ dense submatrix 
 of $V(X,Y)$ whose lower right corner is $V(X,Y)_{N+1,N+1}$. 
  Assume that $B$ in~\eqref{notjacobi} is the transpose of the submatrix of $V(X,Y)$ involved in the definition 
of $\hat\Upsilon_q$,
and $\alpha=\delta=1$, $\beta=2$, $\gamma=q+1$.
Then
\begin{gather*}
B^{\hat\alpha}=\phhi_q(X,Y), \quad B^{\hat\beta}=\hat\Upsilon_q(X,Y), \quad 
B^{\hat\gamma}=\varkappa(X,Y)\thetta_{q-M+1}(X,Y),\\
B^{\hat\beta\hat\gamma}_{\hat\delta}=\varkappa(X,Y)\phhi_{q-M}(Y,X),\quad
B^{\hat\alpha\hat\gamma}_{\hat\delta}=\varkappa(X,Y)\thetta_{q-M}(X,Y),\quad
B^{\hat\alpha\hat\beta}_{\hat\delta}=\hat\Upsilon_{q-1}(X,Y),
\end{gather*}
 where now $\varkappa(X,Y) = \det  V(X,Y)_{[n+1,N]}^{[n+2,N+1]}$.
Applying \eqref{notjacobi}  we obtain
 \begin{equation}
\label{aux58}
\phhi_{M+\bar q}(X,Y)\phhi_{\bar q}(Y,X) + \thetta_{\bar q+1}(X,Y)\hat\Upsilon_{M+\bar q-1}(X,Y) 
 = \hat\Upsilon_{M+\bar q}(X,Y) \thetta_{\bar q}(X,Y)
\end{equation}
 for $\bar q=q-M\in [n-1]$.
 
 Note that $\hat\Upsilon_M(X,Y)=\phhi_M(X,Y)$. Therefore, 
 for $q=M+1$ \eqref{aux58} becomes
 \begin{equation*}
 \phhi_{M+1}(X,Y)\phhi_{1}(Y,X)  + \thetta_{2}(X,Y) \phhi_{M}(X,Y) = \hat\Upsilon_{M+1}(X,Y) \thetta_{1}(X,Y).
\end{equation*}
The left hand side of this equation restricted to the diagonal $Y=X$ is consistent with the rightmost 
part of Fig.~\ref{fig:spec2},
hence $\thetta'_1(X)=\hat\Upsilon_{M+1}(X,X)$ is a regular function on $\Mat_n$.

For $\bar q\in [2, n-1]$, we can eliminate $\hat\Upsilon_{\bar q+M-1}$ from the $(\bar q-1)$st and the $\bar q$th 
equation  in~\eqref{aux58} and obtain 
\begin{equation*}
\begin{aligned}
&  \thetta_{\bar q-1}(X,Y) \phhi_{M+\bar q}(X,Y) \phhi_{\bar q}(Y,X)+ \thetta_{\bar q+1}(X,Y)
\phhi_{M+\bar q-1}(X,Y) \phhi_{\bar q-1}(Y,X) \\
&\qquad\qquad=\left (\thetta_{\bar q-1}(X,Y) \hat\Upsilon_{M+\bar q}(X,Y)- \thetta_{\bar q+1}(X,Y) 
\hat\Upsilon_{M+\bar q-2}(X,Y)\right ) 
\thetta_{\bar q}(X,Y). 
\end{aligned}
\end{equation*}
The left-hand side of this equation 
is consistent with the rightmost part of Fig.~\ref{fig:typ2},  
and hence 
$\thetta'_{\bar q}(X)=\thetta_{\bar q-1}(X) \hat\Upsilon_{M+\bar q}(X,X)- 
\thetta_{\bar q+1}(X) \hat\Upsilon_{M+\bar q-2}(X,X)$ is regular on $\Mat_n$ for all standard $\thetta$-vertices.
 This concludes the proof.

\begin{remark}
 \label{more_cl_var}
 {\rm Formula \eqref{aux58} with $\bar q = 1$ and $X=Y$ shows that the cluster transformation applied to
 $\thetta_1(X)$ in the initial cluster results in the cluster variable $\hat \Upsilon_{M+1}(X, X)$.
 It is easy to deduce from the structure of the initial quiver that in the resulting cluster, the cluster transformation
 applied to $\thetta_2(X)$ is again given by~\eqref{aux58} with  $X=Y$ and, this time, $\bar q = 2$. Applying cluster
 transformations to $\thetta_i(X)$, $i\in [n-1]$, consecutively and using~\eqref{aux58} at every step, one concludes
 that all the functions $\hat \Upsilon_{M+i}(X, X)$ are cluster variables in our cluster algebra. Similarly,
 applying cluster
 transformations to $\phhi_i(X)$, $i\in [n-2]$, consecutively and using~\eqref{supernei} and the structure of the initial quiver,
 we see that all the functions $-\Lambda_{i+2}(X,X)$ belong to the set of cluster variables as well.
 }
 \end{remark}

 We are now in a position to prove part (iii) of Conjecture \ref{ulti} for the 
Cremmer--Gervais
case. In this case, $\H_T=\exp \h_T$ is one-dimensional and $\h_T=\{ h\in\h \: 
\alpha_i(h)=\alpha_{i+1}(h), i\in [n-2]\}=
\left \{ t \left ( D_n - \frac{n+1}{2}\one_n\right )\:  t\in\mathbb{C}\right \}$.
We will extend  the claim to the case of the  Cremmer--Gervais structure on
$GL_n$ as follows.

\begin{proposition}
\label{torus}
The global toric action of $(\mathbb{C}^*)^{3}$ on $\A_{CG}(GL_n)$ is generated by
multiplication by scalars and the action
of $(\mathbb{C}^*)^{2}$
given by $(t_1, t_2)(X) = t_1^{D_n} X t_2^{D_n}$, where 
$t^D=\diag(t^{d_1},\dots,t^{d_m})$ for $D=\diag(d_1,\dots,d_m)$. The rank of this action equals $3$.
\end{proposition}

\begin{proof} The claim that multiplication of element of $GL_n$ by scalars
gives rise to a global toric action is equivalent to claiming that all cluster
variables in 
$\CC_{CG}$ are homogeneous functions of $X$. We only need to check this for cluster
variables in and adjacent to the initial cluster. This can be easily seen by direct
inspection
of expressions for the adjacent cluster variables obtained in the previous proof.
Indeed, all equations used in that proof are homogeneous in matrix entries.
Consequently, the left and right actions of $\H_T$ can be replaced by the actions
of $\exp(\{tD_n\: t\in\C\})$.

Next, observe that 
$$
U(t^{D_n} X, t^{D_n} X)=
\mbox{diag}(t^{D_{n-1} - (i-1)\one_{n-1}})_{i=1}^k U(X, X)
\mbox{diag}(t^{ (j-1)\one_{n+1}})_{j=1}^{k+1}
$$ 
and 
$$
U( X t^{D_n} , X t^{D_n})=
\mbox{diag}(t^{ - (i-1)\one_{n-1}})_{i=1}^k U(X, X)\mbox{diag}(t^{D_{n+1} +
j\one_{n+1}})_{j=1}^{k+1}.
$$
Since every initial cluster variable $f$ can be realized as a minor of $U(X,X)$,
it follows
that $f(t_1^{D_n} X t_2^{D_n})=t_1^\alpha t_2^\beta f(X)$ for some integers $\alpha,
\beta$. The fact that the same is true for cluster variables adjacent to the initial
cluster
follows from an easily verified fact that expressions we obtained for these variables
are either monomial in minors of $U(X,X)$ or binomials in minors such that every column
and row of $U(X,X)$ appears in both terms the same number of times.

To prove that the rank of the action equals $3$, assume to the contrary that it is less than $3$.
Then there exist $t_0, t_1, t_2\in\C^*$ such that $t_1^{D_n}t_0Xt_2^{D_n}=X$ and 
$(t_0,t_1,t_2)\ne (1,1,1)$. In particular, $t_1t_0x_{11}t_2=x_{11}$, $t_1t_0x_{12}t_2^2=x_{12}$
and $t_1^2t_0x_{22}t_2^2=x_{22}$, which implies $t_1t_0t_2=t_1t_0t_2^2=t_1^2t_0t_2^2=1$. Therefore,
$(t_0,t_1,t_2)=(1,1,1)$, a contradiction.
\end{proof}

 \section{Quiver transformations}
The goal of this section is to prove Theorems~\ref{transform1} and~\ref{transform2}.

Recall that the augmentation of $U(X,X)$ was defined at the beginning of Section~\ref{sec:regular}
via
 attaching two extra rows: $[\underbrace{0\cdots 0}_{n+1}\ x_{11} \cdots x_{1n} \ 0 \cdots 0]$ on top 
of $U(X,X)$ and  $[ 0  \cdots 0 \ x_{n1} \cdots x_{nn} \ \underbrace{0\cdots 0}_{n+1}]$ at the bottom.
Let $\bar U$ be the submatrix of this augmentation obtained by deleting the first and the last columns
(note that all the deleted elements are zeros). It has certain translation invariance 
properties inherited from
the matrix $U(X,X)$. In particular, any submatrix of $\bar U_{[2,n]}^{[n]}$ is invariant under 
the shift of row and column indices by $-1$ and $n$, respectively (see translation $\tau_1$ in Fig.~\ref{fig:baru}). Next, any submatrix of $\bar U_{[n-1]}^{[n+1,2n]}$ is invariant under 
the shift of row and column indices by $n$ and $1$, respectively (see translation $\tau_2$ in Fig.~\ref{fig:baru}). Finally, any submatrix of $\bar U_{\hat 1}$ is invariant under 
the shift of row and column indices by $n-1$ and $n+1$, respectively, provided the result of the translation fits into $\bar U$ (see translation $\tau_3$ in Fig.~\ref{fig:baru}).

 \begin{figure}[ht]
\begin{center}
\includegraphics[height=8cm]{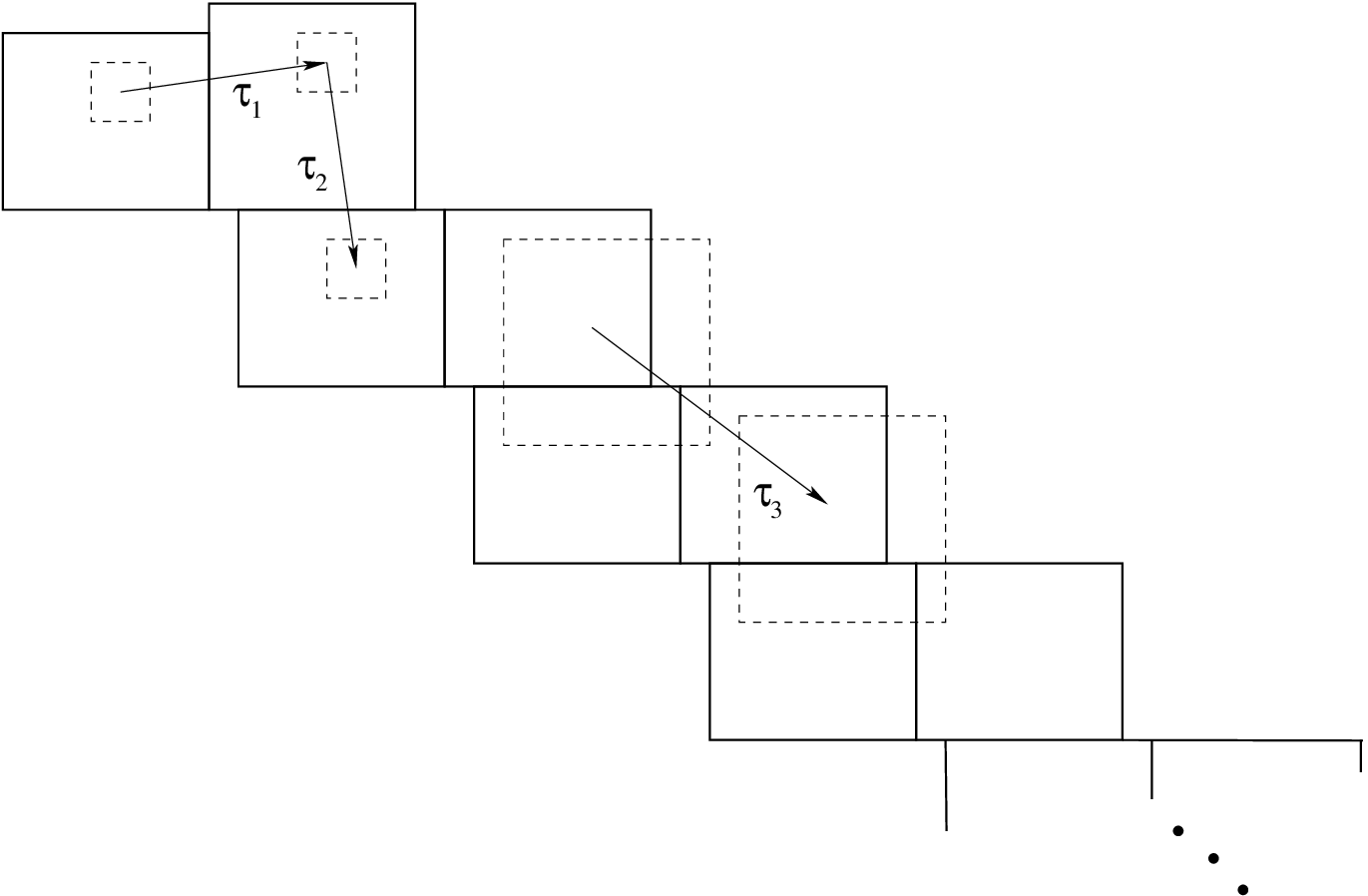}
\caption{Translation invariance properties of $\bar U$}
\label{fig:baru}
\end{center}
\end{figure}

 Define the {\em core\/} of a matrix to be its maximal square irreducible leading principal submatrix. 
 In what follows, 
  ${\bangle J I}_A$ will denote
 the determinant of the core of the submatrix  $A^{\hat J}_{\hat I}$ obtained from $A$ by deleting rows indexed by $I$ and columns indexed by $J$. 
 Occasionally, we will also use notation ${\bangle J I}^{(q)}_A$ to describe the product of determinants of the first $q$ irreducible components of the maximal
 square principal submatrix of $A^{\hat J}_{\hat I}$, so ${\bangle J I}^{(1)}_A$ is just ${\bangle J I}_A$.

A matrix $A=(a_{ij})$ is said to have a {\em staircase shape\/} if there exist two sequences of row indices $i_{t 1} < i_{t 2} < \cdots $ and of column indices 
$j_{t 1} < j_{t 2} < \cdots$, $t=1,2$, such that  for any $\alpha, \beta$ either both expressions
$i_{1\alpha}-i_{2\beta}-1$ and $j_{2\beta}-j_{1\alpha}-1$ vanish simultaneously, or at least one of them
is strictly positive, and $a_{ij} = 0 $ for $i \ge i_{1 \alpha}$, $j \le j_{1 \alpha}$ and  for 
$i\le i_{2\alpha}$, $j\ge j_{2 \alpha}$. 
For example, matrices $U(X,Y)$, $V(X,Y)$ and $\bar U$ have a staircase shape. 
Viewed as polynomial functions of its nonzero matrix
entries, minors of a matrix having a fixed staircase shape may not be irreducible. 
Thus, when one applies identities \eqref{jacobi}, \eqref{notjacobi} to such a matrix, a common factor 
may occur in all three terms. We have already encountered this situation in Section~\ref{sec:regular}. 
In what follows we will  utilize translation invariance to realize each cluster variable as a suitable 
minor of the augmentation $\bar U$. Consequently,
 all  cluster transformations that we use below will become equivalent to application of  
\eqref{jacobi}, \eqref{notjacobi} followed by the cancellation of common factors. To illustrate this 
strategy, we present a lemma that will be applied during the first stage of the construction of $\T$. Afterwards, we will simply point out submatrices of $\bar U$ to which we apply \eqref{jacobi}, \eqref{notjacobi}, and leave a straightforward but space-consuming verification to an interested reader.

\begin{lemma}
 \label{coredodgson}
 Let $A$ be a non-degenerate  $m\times m$ matrix  
 having a staircase shape. 
 Suppose that  $a_{\beta 1} a_{1\beta}\ne 0$
 for a fixed $\beta > 1$.
 Then the following identity is satisfied:
 \begin{equation}
{\bangle {1} {1}}_A\ {\bangle {\ } \beta}_A =  {\bangle {1} {\beta}}_A\ {\bangle {\ } 1}_A +{\bangle {\ } {\ }}_A\ {\bangle {1} { 1 \beta}}_A.
 \label{coredodgsonformula}
 \end{equation}
 \end{lemma}
 
 \begin{proof} We will derive \eqref{coredodgsonformula} from \eqref{jacobi} with $\alpha=\gamma=1$, $\delta=m$. Due to the conditions on $A$, all the minors
 involved in \eqref{jacobi} are non-zero, provided  $a_{i+1,i}\ne 0 $
for $i\in [\beta, m-1]$. We may assume this condition to be true; otherwise, if $i$ is the smallest index such that $a_{i+1,i}= 0 $, we can replace $A$ in 
\eqref{coredodgsonformula}
with $A^{[1,i]}_{[1,i]}$ and apply the argument below.

The shape of $A$ and the condition  $a_{\beta 1} a_{1\beta}\ne 0$ guarantee that if $s$ is the  smallest index  such that $a_{s,s+1}=0$ 
and $t$ is the smallest index such that $a_{t, t-2}=0$,  then $s\geq \beta$ and $ t\geq \beta+1$. This means that if $\Delta= \det A_{[t, m]}^{[t-1,m-1]}$
and $\Gamma=\det A_{[s+1, m]}^{[s+1,m]}$, then $\det A_{\widehat \beta}^{\widehat m} = {\bangle {\ } \beta}_A \Delta$, 
$\det A_{\hat 1}^{\widehat m} = {\bangle {\ } 1}_A \Delta$, 
 $\det A_{\hat 1 \widehat \beta}^{\hat 1 \widehat m} = {\bangle {1} { 1 \beta}}_A \Delta$ and $\det A = {\bangle {\ } \ }_A \Gamma$, $\det A_{\hat 1}^{\hat 1} = {\bangle {1} 1}_A \Gamma$, 
 $\det A_{ \widehat \beta}^{\hat 1 } = {\bangle {1} {  \beta}}_A \Gamma$, hence~\eqref{coredodgsonformula} follows from~\eqref{jacobi}.
\end{proof}
 
 \subsection{Sequence $\SS$}\label{sec:Strans}
The goal of this subsection is to prove Theorem~\ref{transform1}.
 
We  will construct $\mathcal S$ as a composition of two sequences
of transformations: $\mathcal S=\mathcal S_v\circ \mathcal S_h$, which are, in turn,
decomposed into subsequences to be described below.
To explain the construction of $\mathcal S_h$, we need to define, for $t\in [0,n]$,
a quiver $Q_t(n)$, whose vertices are indexed by pairs  $(i,j)$, $i\in [n-1]$, 
$j\in [n]$, and $(n,j)$, $j\in [n-t]$, for $t<n$. We will use a convention
$Q_0(n) = Q_{CG}(n)$. The vertex set 
of  $Q_t(n)$ is regarded as a subset of vertices of $Q_{t-1}(n)$ and, consequently,
of $Q_{CG}(n)$.
The quiver $Q_t(n)$ inherits all the edges in $Q_{CG}(n)$ whose both ends
belong to this subset except the edge $(n,n-t) \to (1,n-t)$ which is erased. In
addition, $Q_t(n)$  contains edges $(n-1,n-s) \to (1,n-s)$,
$(1,n-s) \to (n-1,n-s+1)$ for $s\in [t]$ (see Fig.~\ref{fig:q2(5)} for the quiver $Q_2(5)$). Just like in $Q_{CG}(n)$, vertices
$(1,1), (2,1)$ and $(1,n)$ are frozen vertices in $Q_t(n)$. 

\begin{figure}[ht]
\begin{center}
\includegraphics[height=5cm]{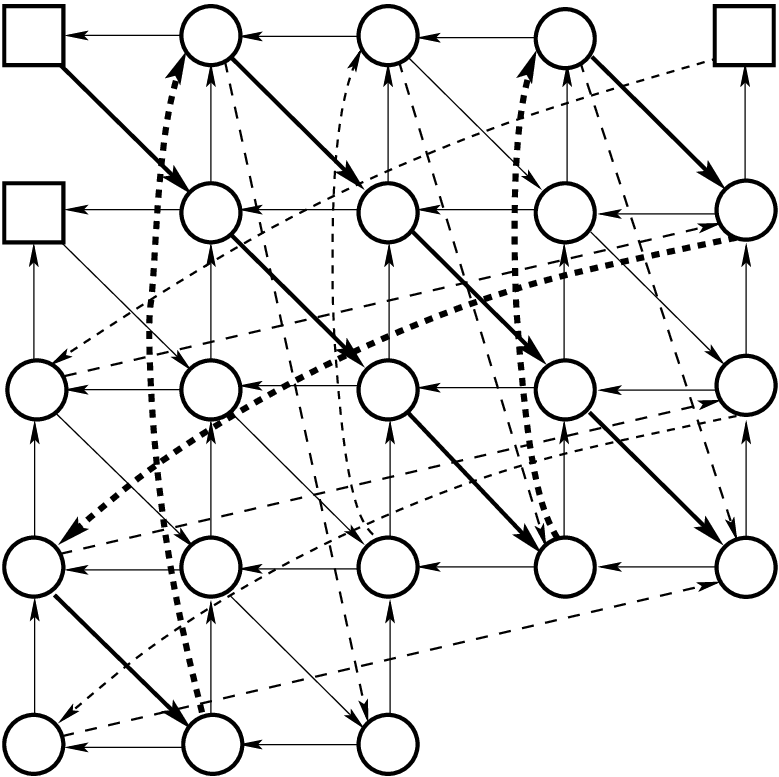}
\caption{Quiver $Q_2(5)$}
\label{fig:q2(5)}
\end{center}
\end{figure}

We now assign a function $f_{ij}^{(t)}(X)$ to each vertex $(i,j)$ of $Q_t(n)$. First of all,
 define $f_{ij}^{(0)}=f_{ij}$; 
recall that $f_{ij}=\rho^{-1}(i,j)$ is a variable of the augmented initial cluster (see Lemma~\ref{vertices}).
As in
the case of $Q_0(n)=Q_{CG}(n)$, $f_{ij}^{(t)}$ for $t>0$ is the determinant of a certain
submatrix $V_{ij}^{(t)}$
of $V(X,X)$ whose upper left entry is equal to $x_{ij}$. We start with submatrices that
correspond to frozen vertices in $Q_t(n)$. The $(1,2)$ entry of $V(X,X)$ is
$x_{11}$. If
$3l < t \leq 3 (l+1)$,  then $V_{11}^{(t)}$ is the core of $V(X,X)^{\hat 1}_{\hat \imath_{[l+1]}}$ 
with $i_\alpha=1+\alpha(n-1)$. Similarly, if $3l + 1 < t \leq 3 (l+1)+1$, then
 $V_{21}^{(t)}$ is the core of
$V(X,X)^{\hat 1}_{\hat \imath_{[0,l+1]}}$, and  if $3l + 2 < t \leq 3 (l+1)+2$, then
$V_{1n}^{(t)}$ is the core of
$V(X,X)_{\hat \imath_{[0,l+1]}}$.

Consider the collection of all distinct trailing principal minors  of submatrices
$V_{11}^{(t)},  V_{21}^{(t)},  V_{1n}^{(t)}$. The upper left entry of each such minor is 
 occupied by $x_{ij}$ for some $(i,j)$. This defines
a correspondence $\rho_t$ between distinct trailing principal minors  of 
$V_{11}^{(t)},  V_{21}^{(t)},  V_{1n}^{(t)}$ and pairs of indices $(i,j)$.

\begin{lemma}
\label{rho_t}
For any $t\in[0,n]$, 
$\rho_t$ is a one-to-one correspondence between distinct trailing principal minors 
of  $V_{11}^{(t)},  V_{21}^{(t)},  V_{1n}^{(t)}$ and vertices of $Q_t(n)$. 
\end{lemma}

\begin{proof} 
Define a {\em diagonal path\/} in $Q_t(n)$ to be the maximal directed path that starts at a
frozen vertex and consists of edges of the type $(i,j)\to (i+1,j+1)$, $(n, i) \to (1,
i)$ or $(n-1,i)\to (1,i)$, and $(i, n) \to (i+2, 1)$. Since exactly one such edge enters each non-frozen
vertex, the three diagonal paths are simple and do not intersect. We denote them 
$\P_1^{(t)}=\P_{11}^{(t)}$, $\P_2^{(t)}=\P_{21}^{(t)}$ and $\P_3^{(t)}=\P_{1n}^{(t)}$,  
according to the starting vertex. Besides, exactly one such edge
leaves each vertex, except for the following three: $(n,1)$, $(n,n-t)$ and $(n-1,n)$ for $t\in [0,n-2]$,
$(n,1)$, $(n-1,1)$ and $(n-1,n)$ for $t=n-1$, and $(n-2,n)$, $(n-1,1)$ and $(n-1,n)$ for $t=n$. 
These three vertices
are, therefore, the last vertices of the three diagonal paths. The diagonal path $\P_1^{(2)}$ in
$Q_2(5)$ is shown in bold in Fig.~\ref{fig:q2(5)}; it ends at $(4,5)$.

Let us prove that the three diagonal paths cover all vertices in $Q_t(n)$. Indeed, the path $\P_i^{(t)}$, 
$i\in [1,3]$, hits the first column of $Q_t(n)$ at vertices $(3s+i,1)$ for $0\le s\le (t+3-i)/3$.
After that the path that passed through $(t+1,1)$ hits $(n,n-t)$ and terminates (for $t=n$ the
corresponding path terminates at $(n-2,n)$). Consequently for, $t\in [0,n-2]$, the
two remaining paths hit the first column at vertices $(t+2s+2,1)$ and $(t+2s+3,1)$ for $s\ge 0$, until
one of them terminates at $(n,1)$ and the other passes through $(n-1,1)$ and terminates at $(n-1,n)$.
For $t=n-1, n$ the remaining two paths do not hit the first column any more; the one that passed through
$(n-1,1)$ terminates there, and the one that passed through $(n-2,1)$ terminates at $(n-1,n)$. 
Therefore, all vertices in the first column of $Q_t(n)$ are covered for any $t\in [0,n]$. Assume there are uncovered 
vertices in the first row of $Q_t(n)$. Let $(1,p)$ be the leftmost among these vertices, then go backwards
along the edge $(n,p)\to (1,p)$ (if $p<n-t$) or $(n-1,p)\to (1,p)$ (if $p\ge n-t$), and proceed backwards 
along the edges $(i,j)\to (i+1,j+1)$. Clearly, all vertices that are reached in this way are not covered by
 the diagonal paths. Eventually, we get to a non-covered vertex in the first column, a contradiction.

It remains to prove that the vertices in the paths  $\P_{11}^{(t)}$, 
$\P_{21}^{(t)}$ and $\P_{1n}^{(t)}$ correspond to diagonal elements of the matrices $V^{(t)}_{11}$, 
$V^{(t)}_{21}$ and $V^{(t)}_{1n}$ in the natural order,
and that the diagonal elements that ere not involved in this correspondence do not define new trailing
minors. First of all, note that one step in the south-east direction
along a diagonal in each one of these matrices corresponds to an edge $(i,j)\to (i+1,j+1)$ provided we remain in the same block, to an edge $(n,i)\to(1,i)$ provided we move from an $\X$-block to a $\Y$-block, to an edge
$(n-1,i)\to (1,i)$ provided we move from an $\X$-block with the last row deleted to a $\Y$-block, and to
an edge $(i,n)\to (i+2,1)$ provided we move from a $\Y$-block to an $\X$-block.

Denote by $V_1^{(t)}$, $V_2^{(t)}$ and $V_3^{(t)}$ the sequences of pairs of indices corresponding to the
diagonal elements of the matrices $V^{(t)}_{11}$, 
$V^{(t)}_{21}$ and $V^{(t)}_{1n}$, respectively, taken in the natural order. 
Denote by $F^{(t)}_i$ the first occurrence of 
a pair $(n,*)$ in $V^{(t)}_i$. 
Each sequence
$V_i^{(t)}$ can be decomposed as follows: $V^{(t)}_i=U_i^{(t)}F^{(t)}_iW^{(t)}_i$ (the second and the third 
subsequences are empty if there is no occurrence of $(n,*)$ in $V^{(t)}_i$. It is easy to prove by
induction the following characterization of the sequences $V^{(t)}_i$ and $\P^{(t)}_i$:
let $r$ take values in $\{1,2,3\}$ and let $r\equiv t\bmod 3$,  then
\begin{equation}\label{frecvp}
\begin{aligned}
&V^{(t)}_r= U^{(t-1)}_rW^{(t-1)}_{r+1},\qquad V^{(t)}_i= V^{(t-1)}_i\quad \text{for $i\ne r$, $t\in [n]$},\\
&\P^{(t)}_{r+1}= U^{(t)}_{r+1}F^{(t)}_{r+1},\qquad\quad \P^{(t)}_i= V^{(t)}_i\quad \text{for $i\ne r+1$,
$t\in [0,n]$}
\end{aligned}
\end{equation}
with the initial condition $W^{(0)}_{1}=\varnothing$.

For example, 
consider the matrix $V^{(t)}_{11}$. If $t=3s$ then it is obtained from $V(X,X)$ by deleting the last
row from the first $s$ blocks, and hence $F_{1}^{(t)}=(n,n-t)$. 
On the other hand, the number of edges $(n-1,i)\to (1,i)$ in the path
$\P_{1}^{(t)}$ equals $\lfloor \frac{3s+2}3\rfloor=s$, it ends at $(n,n-t)$ and does not
contain other vertices of type $(n,*)$. We thus see that in this case $\P_{r+1}^{(t)}=\P_1^{(t)}=
U_1^{(t)}F_1^{(t)}$. If $t=3s+1$ then the last row should be deleted from the $(s+1)$th block as well.
Therefore, the pair that follows $(n-1,n-t-1)$ in $V_1^{(t)}$ is $(1,n-t-1)$. In $V_2^{(t-1)}$
it followed $(n,n-t-1)$, which was the first occurrence of the last row entry, and hence
we get $V_r^{(t)}=V_1^{(t)}=U_1^{(t-1)}W_2^{(t-1)}$. Other relations in~\eqref{frecvp} are proved in a similar way.

 The one-to-one correspondence between the paths  $\P_{11}^{(t)}$, 
$\P_{21}^{(t)}$ and $\P_{1n}^{(t)}$ and the set of 
distinct trailing principal minors 
of  $V_{11}^{(t)},  V_{21}^{(t)},  V_{1n}^{(t)}$ follows immediately from~\eqref{frecvp}.
\end{proof}

\begin{remark}
In particular, since $V_{11}^{(0)}=X$, for $t=0$ we recover Lemma \ref{vertices}.
\end{remark}

Lemma \ref{rho_t} gives rise to a well-defined assignment $(i,j) \mapsto
f^{(t)}_{ij}(X)=\rho_t^{-1}(i,j)$ of minors of $V(X,X)$ to vertices of $Q_t(n)$.

Consider a seed $((f^{(t)}_{ij}), Q_t(n) )$, $t\geq 0$. Define $r\in\{1,2,3\}$ via $r\equiv t\bmod 3$
as above and 
consider the composition
of cluster transformations applied to consecutive mutable vertices along the diagonal path
$\P_{r+1}^{(t)}$ in the opposite direction 
(starting  at $(n, n-t)$).  
After the mutation at a vertex of $\P_{r+1}^{(t)}$ it is shifted to the position of
the previous vertex of the same path; so, at this moment there are two vertices occupying
the same position (see Fig.~\ref{fig:q2(5)31}).
After the second vertex of the path is mutated and  shifted, it is frozen, and the first vertex is erased.
We denote the resulting transformation $\mathcal S_h^{(t)}$.

\begin{lemma}
\label{Sh} 
For any $t\in [0,n-1]$, transformation $\mathcal S_h^{(t)}$ takes the seed $((f_{ij}^{(t)}), Q_t(n) )$ into
the seed $((f_{ij}^{(t+1)}), Q_{t+1}(n) )$.
\end{lemma}

\begin{proof}   
Let us first describe the evolution of the quiver $Q_t(n)$. One can check by induction that prior to 
applying a mutation in the sequence $\mathcal S_h^{(t)}$ to a vertex $(p,q)\in\P_{r+1}^{(t)}$ the current state of
the quiver can be described as follows.

(i) Vertex $(p,q)$ is trivalent if $p=n$ and $q=1,2$, five-valent if $t>0$, $r=3$ and $(p,q)=(i,i)$ for
$i\in[2,n-1]$ or $(p,q)=(1,n-1)$, and four-valent otherwise.

(ii) The edges pointing from $(p,q)$ are $(p,q)\to (p,q-1)$ for $q\ne 1$ and $(p,1)\to (p-1,n)$ for $q=1$;
$(p,q)\to (p-1,q)$ for $p\ne 1$ and $(1,q)\to (n-1,q+1)$ for $p=1$; $(p,q)\to (1,n)$ whenever $(p,q)$ is
five-valent.
 
(iii) The edges pointing to $(p,q)$ are $(p-1,q-1)\to(p,q)$ for $p\ne 1$ and  $q\ne 1$, 
$(n-1,q)\to(1,q)$ for $p=1$, and $(p-2,n)\to (p,1)$ 
for $q=1$; the special edge from the vertex of $\P_{r+1}^{(t)}$ that has been just mutated and moved to $(p,q)$ for $p\ne n$. 

(iv) There are no edges between the vertex that has been just mutated and moved to $(p,q)$  
and vertices $(p,q-1)$ and $(p-1,q)$ (or vertices $(p-1,n)$ or $(n-1,
q+1)$ in the situations described in (ii) above).

Consequently, when $(p,q)$ is moved one position back along $\P_{r+1}^{(t)}$, the local structure of the quiver is restored completely. 

Quiver $Q_2(5)$ prior to the mutation at the vertex $(3,1)$ is shown in Fig.~\ref{fig:q2(5)31}. 
Currently there
are two vertices at position $(3,1)$: the ``new'' one (moved from $(4,2)$) and the ``old'' one, that 
will be moved to $(1,5)$ after the mutation.  

\begin{figure}[ht]
\begin{center}
\includegraphics[height=5cm]{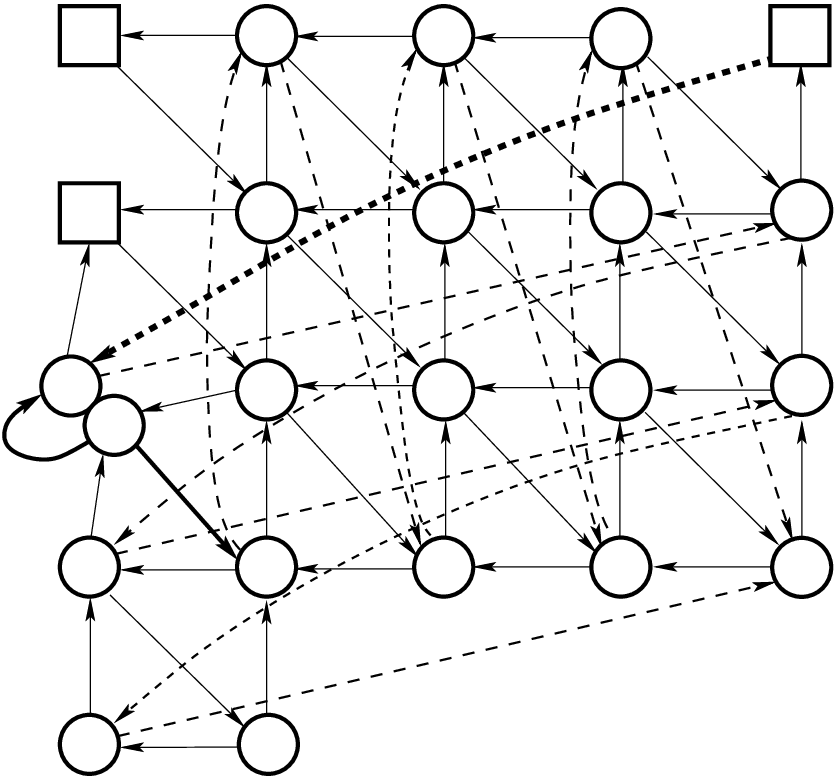}
\caption{Quiver $Q_2(5)$ prior to the mutation at the vertex $(3,1)$}
\label{fig:q2(5)31}
\end{center}
\end{figure}

Now let us turn to cluster transformations along the diagonal path $\mathcal
P^{(t)}_{r+1}$. First, we notice that for $t=0$ these transformations are
described
by identities \eqref{aux58} applied to $V(X,X)$ rather than $V(X,Y)$; namely, the transformation
corresponding to the mutation at $(p,p)$ is given by~\eqref{aux58} for $\bar q=n-p+1$, 
see Remark~\ref{more_cl_var}. 
To describe cluster transformations for $t> 0$, we apply Lemma~\ref{coredodgson} to the submatrix $A$ of $V(X,X)$ associated with the diagonal
path $\P^{(t)}_{r+1}$. More exactly, we define $A=V_{p_1q_1}^{(t)}$, where $(p_1,q_1)$ is the first
vertex of $\P_{r+1}^{(t)}$. Let $L$ be the length of $\mathcal
P^{(t)}_{r+1}$.
The $L$-th row of $A$ corresponds to 
the row of $V(X,X)$ that contains $F_{r+1}^{(t)}$, see the proof of Lemma~\ref{rho_t}.
Let us temporarily use notation $w^{(t)}_j(i)$ for  the cluster variable
attached to the $i$th vertex of the diagonal path $\mathcal P^{(t)}_j$.
Then, by Lemma~\ref{rho_t}, 
$w^{(t)}_{r+1}(i)=\det A^{\widehat{[i-1]}}_{\widehat{[i-1]}}= { \bangle
{[i-1]}  {[i-1] }}_A$. In order to verify the claim of the lemma, we need
to show that this variable transforms into $w^{(t+1)}_{r+1}(i-1)={ \bangle
{[i-2]}  {[i-2] \cup L}}_A$.

It follows from (i)--(iii) above that for $r=1$ the  relations we need to verify are
$$
 w^{(t+1)}_2(L-1) w^{(t)}_2(L) = w^{(t)}_2(L-1) w^{(t)}_3(L+1) +
w^{(t)}_1(L) w^{(t)}_3(L)
$$ 
for $t \in [n-3]$, 
$$ 
w^{(t+1)}_2(L-1) w^{(t)}_2(L) = w^{(t)}_2(L-1)  + w^{(t)}_1(L)
w^{(t)}_3(L)
$$  
for $t = n-1, n-2$, and
$$
w^{(t+1)}_2(i-1) w^{(t)}_2(i) = w^{(t)}_2(i-1) w^{(t+1)}_2(i) +
w^{(t)}_1(i) w^{(t)}_3(i)
$$
for $i\in [2,L-1]$, $t\in [n-1]$.
Note that
\begin{equation*}
w^{(t)}_1(i) ={ \bangle {[i-1]}  {[i-2] \cup L}}_A,\quad i \in [2,L], \qquad
w^{(t)}_3(i) ={ \bangle {[i-2]}  {[i-1] }}_A, \quad   i \in [2,L +1],
\end{equation*}
by Lemma~\ref{rho_t}.

Similarly, for $r=2$ we need to verify
$$ 
w^{(t+1)}_3(L-1) w^{(t)}_3(L) = w^{(t)}_3(L-1) w^{(t)}_1(L+n) +
w^{(t)}_1(L+n-1) w^{(t)}_2(L-1)
$$  
for $t \in [n-3]$,
$$ 
w^{(t+1)}_3(L-1)  w^{(t)}_3(L) = w^{(t)}_3(L-1)  + w^{(t)}_1(L+n-1)
w^{(t)}_2(L-1)
$$  
for  $t = n-1, n-2$, and
$$
w^{(t+1)}_3(i-1)  w^{(t)}_3(i) = w^{(t)}_3(i-1) w^{(t+1)}_3(i)  +
w^{(t)}_1(i+n-1) w^{(t)}_2(i-1)
$$
for  $i\in [2,L-1]$, $t\in [n-1]$.
In this case
\begin{equation*}
w^{(t)}_1(i+n) ={ \bangle {[i-2}  {[i-1] }}_A,\quad  i \in [2, L], \qquad
w^{(t)}_2(i-1) ={ \bangle {[i-1]}  {[i-2]\cup L }}_A, \quad   i \in [2, L],
\end{equation*}
by Lemma~\ref{rho_t};
note that the first relation involves the shift $\tau_3^{-1}$.

Finally, for $r=3$ we need to verify
$$
w^{(t+1)}_1(L-1)  w^{(t)}_1(L) = w^{(t)}_1(L-1) w^{(t)}_2(L) +
w^{(t)}_2(L-1 ) w^{(t)}_3(L-n)
$$  
for $t \in [n-3]$,
$$
w^{(t+1)}_1(L-1) w^{(t)}_1(L) = w^{(t)}_1(L-1)  + w^{(t)}_2(L-1)
w^{(t)}_3(L-n)
$$  
for $t =n-1, n-2$, and
$$
w^{(t+1)}_1(i-1) w^{(t)}_1(i) = w^{(t)}_1(i-1) w^{(t+1)}_1(i) +
w^{(t)}_2(i-1 ) w^{(t)}_3(i-n)
$$
for $i\in [n+1,L-1]$, $t\in [n-1]$. 
To treat the remaining case $i\in [2,n]$, note that the vertices $(p,p+1)$, $p\in [n-1]$, are not involved
in the sequence $\SS^{(t)}_h$, and hence the attached functions are $\phhi_{n-p}$, as it was in the initial
quiver. Therefore, we need to verify
$$
w^{(t+1)}_1(i-1) w^{(t)}_1(i) = w^{(t)}_1(i-1) w^{(t+1)}_1(i) +
w^{(t)}_2(i-1 ) \phhi_{n-i+1} w^{(t)}_3(1)
$$
for  $i\in [2,n]$, $t\in [n-1]$.
In the last four formulas 
\begin{align*}
w^{(t)}_2(i-1)& ={ \bangle {[i-2]}  {[i-1] }}_A,\quad  i \in [2, L+1],\\
w^{(t)}_3(i-n)&={ \bangle {[i-1]}  {[i-2]\cup L }}_A, \quad i \in [n+1,L],\\
\phhi_{n-i+1} w^{(t)}_3(1) &={ \bangle {[i-1]}  {[i-2]\cup L }}_A^{(2)}, \quad i \in [2, n].
\end{align*}

We  conclude that all relations  we need to establish are of the form
\begin{align*}
{\bangle {[i-1]} {[i-1]}}_A\ &{\bangle {[i-2] } {[i-2] \cup L}}_A =\\ 
&{\bangle
{[i-2]} {[i-2]}}_A \ {\bangle {[i-1]} { [i-1]\cup  L}}_A + { \bangle {[i-1]}
{[i-2] \cup L}}_A\ { \bangle {[i-2]}  {[i-1] }}_A^{(q)}
 \end{align*}
 with $q=2$ if $r=3$ and $i\in [2,n]$ and $q=1$ otherwise.
and thus their proof reduces to an application of Lemma \ref{coredodgson} to
the matrix $A^{\widehat{[i-2]}}_{\widehat{[i-2]}}$.
\end{proof}

\begin{remark} It follows immediately from (i)--(iv) in the proof of the lemma that at the end of the sequence of mutations involved in $\SS_h^{(t)}$ the first vertex of $\P_{r+1}^{(t)}$ becomes disconnected from the rest of the quiver, and hence could be erased.
\end{remark}

Define
$$
\mathcal S_h = \mathcal S_h^{(n-1)}\circ \mathcal S_h^{(n-2)}\circ \cdots \circ
\mathcal S_h^{(0)}.
$$

\begin{figure}[ht]
\begin{center}
\includegraphics[height=4.5cm]{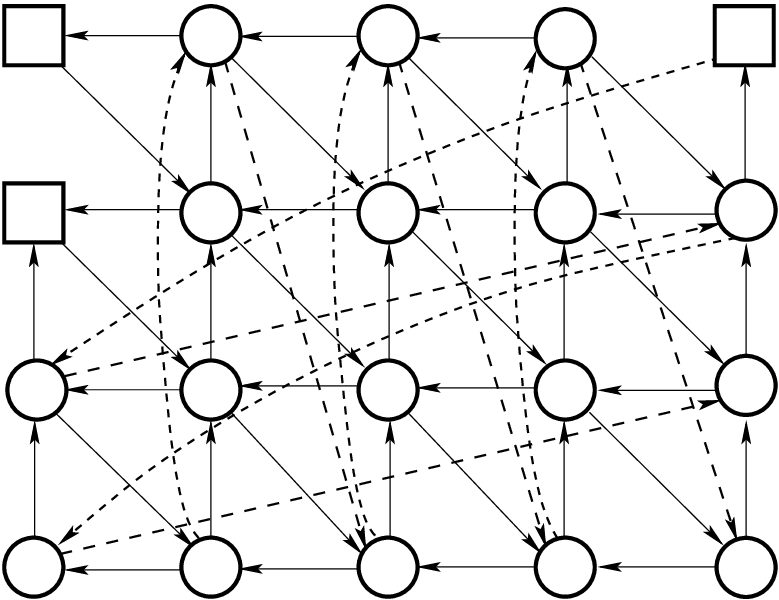}
\caption{Quiver $Q_5(5)$}
\label{fig:q5(5)}
\end{center}
\end{figure}

Lemmas above show that an application of $\mathcal S_h$ 
transforms our initial seed into $((f_{ij}^{(n)}), Q_n(n))$, where
vertices of $Q_n(n)$ form an $(n-1)\times n$ rectangular grid (see Fig.~\ref{fig:q5(5)} for the quiver
$Q_5(5)$). Furthermore, none of
the submatrices of $V(X,X)$ needed to define $f_{ij}^{(n)}(X)$  
contain the elements of the $n$th row of $X$. 
In fact, all of these defining submatrices ``fit"
(not necessarily uniquely) into the matrix
$\bar V$ obtained from $V(X,X)$ by deleting the first and the last column and 
rows numbered $n, 2n-1, \ldots, n + \nu (n-1)$ with
$\nu= \lceil \frac n 3 \rceil$. We embed matrices $V^{(n)}_{11}$, $V^{(n)}_{21}$, $V^{(n)}_{1n}$ into  $\bar V$ as dense submatrices
whose upper left corners are located at entries $(1,1)$, $(2,1)$ and $(n, n+1)$ of  $\bar V$, respectively.

It is convenient to subdivide  
 $\bar V$  into $k-1$ block columns of width $(n+1)$ and into $k$ block rows, where block rows numbered $2$ through
 $\nu$ have height $n-2$ and the remaining block rows have height $n-1$.
 Recall that the index of the last row of the $\alpha$th block row of $V(X,X)$ is $i_\alpha=1+\alpha(n-1)$. 
For $\alpha>\nu$, the index of the row in $\bar V$ that contains elements of the $i_\alpha$th row of $V(X,X)$ 
(in other words, the last row of $\alpha$th block row of $\bar V$) is $i^*_\alpha = i_\alpha - \nu$.
Further, the index of the first column in the $\beta$th block column of $\bar V$ is 
$j_\beta = 1 + (\beta-1)(n+1)$,  $\beta\in [k-1]$.

For example, for $n=5$,
 $\bar V$ is equal to
 $${\small
\left (
 \begin{array}{ccccccccccccc}
 x_{11} & x_{12} & x_{13} &  x_{14} & x_{15}& 0 & 0 & 0 & 0 & 0 & 0 & 0  \\
  x_{21} & x_{22} & x_{23} & x_{24} & x_{25}& 0 & 0 & 0 & 0 & 0 & 0 & 0 \\
 x_{31} & x_{32} & x_{33} &x_{34} & x_{35}& 0 & 0 & 0 & 0 & 0 & 0 & 0 \\
 x_{41} & x_{42} & x_{43} &x_{44} & x_{45}& 0 & 0 & 0 & 0 & 0 & 0 & 0 \\
   0 & x_{11} & x_{12} & x_{13} & x_{14} & x_{15} &x_{21}& x_{22} & x_{23} &
x_{24} & x_{25}& 0  \\
  0 & x_{21} & x_{22} & x_{23} & x_{24} & x_{25} &x_{31}& x_{32} & x_{33} &
x_{34} & x_{35}& 0  \\
 0 & x_{31} & x_{32} & x_{33} & x_{34} & x_{35} &x_{41}& x_{42} & x_{43} &
x_{44} & x_{45}& 0  \\
 0 & 0 & 0 & 0 & 0 &  0 & 0 & x_{11} & x_{12} & x_{13} & x_{14} & x_{15}  \\
  0 & 0 & 0 & 0 & 0 & 0 & 0 & x_{21} & x_{22} & x_{23} & x_{24} & x_{25} \\
  0 & 0 & 0 & 0 & 0 & 0 & 0 & x_{31} & x_{32} & x_{33} & x_{34} & x_{35}   \\
  0 & 0 & 0 & 0 & 0 & 0 & 0 & x_{41} & x_{42} & x_{43} & x_{44} & x_{45} 
 \end{array}
 \right ).
}
 $$
Here $\nu=2$, and the index of the last row in the third block row is $i^*_3=i_3-\nu=1+3(5-1)-2=11$.

To construct $\mathcal S_v$ we need yet another family of quivers denoted by
$Q_{n s}(n)$, $s\in [0, n-3]$, defined as follows. First, $Q_{n 0}(n)= Q_{n }(n)$. For
$s\geq 1$, the vertex set of $Q_{n s}(n)$ is obtained from that of $Q_{n}(n)$ by
removing vertices $(n-s,1), \ldots, (n-1,1)$. $Q_{n s}(n)$
 inherits all the edges in $Q_{n}(n)$ whose both ends
belong to this subset, except for the edge $(n-s-1,1)\to(n-s,2)$. In addition, $Q_{ns}(n)$  contains edges $(n-i-2,n) \to
(n-i,2)$,
$(n-i,2) \to (n-i-1,n)$ for $i\in [s]$. Vertices $(1,1), (2,1)$ and $(1,n)$ are
frozen in $Q_{ns}(n)$. The quiver $Q_{51}(5)$ is shown in Fig.~\ref{fig:q51(5)}.

\begin{figure}[ht]
\begin{center}
\includegraphics[height=4.5cm]{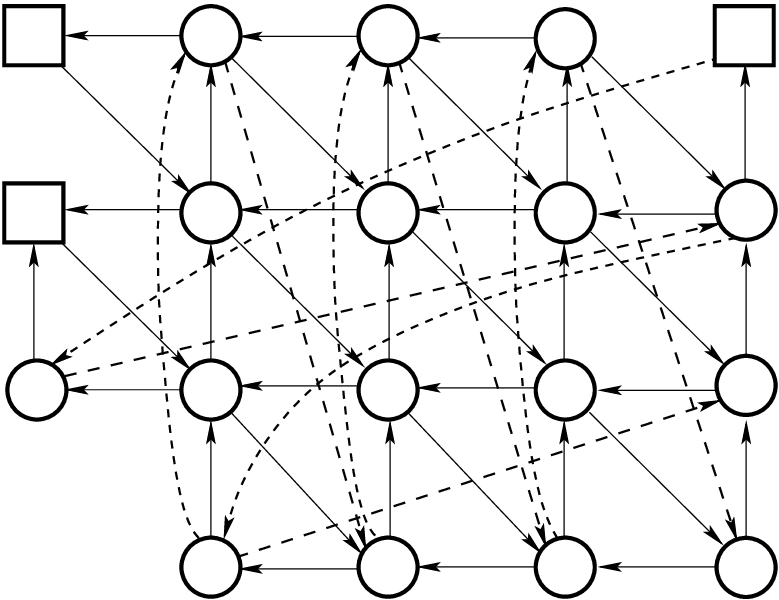}
\caption{Quiver $Q_{51}(5)$}
\label{fig:q51(5)}
\end{center}
\end{figure}

Let us assign a function $f_{ij}^{(ns)}(X)$ to each vertex $(i,j)$ of $Q_{ns}(n)$.
Since the functions will be written in terms of cores of submatrices of $\bar V$, from now on and till the end
of Section~\ref{sec:Strans} we omit explicit mentioning
of $\bar V$ in the subscript.

We start by defining submatrices  $\bar V_{11}^{(s)}$,  $\bar V_{21}^{(s)}$,  $\bar V_{1n}^{(s)}$ of $\bar V$. 
For $s=0$, we use 
$\bar V_{pq}^{(0)}= V^{(n)}_{pq}$. For $s>0$, the
formulas for $\bar V_{pq}^{(s)}$ and certain details of the proof (but not the statement) of Lemma \ref{Sv} below depend on the value of 
$n \bmod 3$, but not in a significant way. The differences do not warrant a tedious case by case consideration and thus, we will only discuss one of them, $n=6m - 1$.

In this case $\nu=2m$, and we define $\bar V_{11}^{(s)}$ as the core of
$\bar V^{\hat \jmath_{[2m-2l,2m+l]}}_{ \hat \imath^*_{[2m+1,2m+l]}}$ for $s\in[6l+1,6l+3]$, and as
the core of $\bar V^{\hat \jmath_{[2m-2l+1,2m+l-1]}}_{ \hat \imath^*_{[2m+1,2m+l]}}$ for $s\in[6l-2,6l]$.
Note that the lower right corner of $\bar V_{11}^{(s)}$ in the first case is $x_{n-1,n}$ 
located in the last row of the $(2m+l+1)$st block row of $\bar V$ (and in 
the last column of the $(2m+l)$th block column of $\bar V$). Indeed, the indices
of the row and the column of $\bar V^{\hat \jmath_{[2m-2l,2m+l]}}_{ \hat \imath^*_{[2m+1,2m+l]}}$ in
  which this entry is located are $(2l+r+1)(n-2) + 2$ and $(2l+r)(n+1) - 3r - 1$, respectively. 
  It is easy to check that these indices are equal and so the entry in question is on the diagonal of $\bar V^{\hat \jmath_{[2m-2l,2m+l]}}_{ \hat \imath^*_{[2m+1,2m+l]}}$. The entries below it are zero and thus it follows from examination of the staircase shape of $\bar V$ that $\bar V_{11}^{(s)}$ in this case can be as well
defined as the core of $\bar V^{\hat \jmath_{[2m-2l,2m+l+1]}}_{ \hat \imath^*_{[2m+1,2m+l]}}$.
Similarly, the lower right corner of $\bar V_{11}^{(s)}$ in the second case is $x_{n-2,n}$ located in the $(n-2)$nd  row of the $(2m+l)$th block row of $\bar V$, and hence $\bar V_{11}^{(s)}$ in this case can be as well defined as the core of $\bar V^{\hat \jmath_{[2m-2l+1,2m+l]}}_{\hat\imath^*_{[2m+1,2m+l]}}$.

Next, $\bar V_{21}^{(s)}$ is defined as the core of $\bar V^{\hat \jmath_{[2m-2l,2m+l-1]}}_{\hat 1\cup \hat \imath^*_{[2m+1,2m+l]}}$, or, equivalently, as the core of  $\bar V^{\hat \jmath_{[2m-2l,2m+l]}}_{\hat 1\cup \hat \imath^*_{[2m+1,2m+l]}}$
for $s\in[6l,6l+2]$, and as the core of
$\bar V^{\hat \jmath_{[2m-2l-1,2m+l]}}_{\hat 1\cup \hat \imath^*_{[2m+1,2m+l]}}$, or, equivalently, as the core of
 $\bar V^{\hat \jmath_{[2m-2l-1,2m+l+1]}}_{\hat 1\cup \hat \imath^*_{[2m+1,2m+l]}}$ for $s\in[6l+3,6l+5]$.
The justification for the two alternative representations follows from the fact that 
the lower right corner of $\bar V_{21}^{(s)}$ in the first case is $x_{n-1,n}$ located in the last row of the $(2m+l+1)$st block row of $\bar V$, while the lower right corner of $\bar V_{21}^{(s)}$ in
the second case is $x_{n-2,n}$ located in the $(n-2)$nd  row of the $(2m+l)$th block row of $\bar V$. 

Finally, $\bar V_{1n}^{(s)}$ is defined as the core of 
$\bar V^{\widehat{[n]}\cup \hat\jmath_{[2m-2l+1,2m+l]}}_{\widehat{[n-1]}\cup \hat \imath_{[2m+1,2m+l]}}$, or, equivalently,
as the core of  
$\bar V^{\widehat{[n]}\cup \hat\jmath_{[2m-2l+1,2m+l+1]}}_{\widehat{[n-1]}\cup \hat \imath_{[2m+1,2m+l]}}$ 
 for $s\in[6l-1,6l+1]$, and as the core of
 $\bar V^{\widehat{[n]}\cup \hat\jmath_{[2m-2l,2m+l]}}_{\widehat{[n-1]}\cup \hat \imath_{[2m+1,2m+l+1]}}$, or, equivalently, as the core of 
$\bar V^{\widehat{[n]}\cup \hat\jmath_{[2m-2l,2m+l+1]}}_{\widehat{[n-1]}\cup \hat \imath_{[2m+1,2m+l+1]}}$
 for $s\in[6l+2,6l+4]$. This time the justification is that the lower right corner of $\bar V_{1n}^{(s)}$ 
in the first case is $x_{n-1,n}$ located in the last row of the $(2m+l+1)$st block row of $\bar V$, while the lower right corner of $\bar V_{1n}^{(s)}$ in the second case  is $x_{n-2,n}$ located in the $(n-2)$nd  row of the $(2m+l+1)$st block row of $\bar V$.

Similar to the definition of $\rho_t$, define
a correspondence $\rho_{ns}$ between distinct trailing principal minors  of  $\bar
V_{11}^{(s)}$,  $\bar V_{21}^{(s)}$, $\bar  V_{1n}^{(s)}$ and pairs of indices $(i,j)$.

\begin{lemma}
\label{rho_ns}
For any $s\in [0,n-3]$, 
$\rho_{ns}$ is a one-to-one correspondence between distinct trailing principal
minors  of  $\bar V_{11}^{(s)}$, $\bar V_{21}^{(s)}$, $\bar  V_{1n}^{(s)}$ and vertices
of $Q_{ns}(n)$.
\end{lemma}

\begin{proof} 
The proof of the lemma is similar to the proof of Lemma~\ref{rho_t}. We start with defining  {\em diagonal 
paths\/} in $Q_{ns}(n)$ as maximal directed paths that start at a
frozen vertex and consist of edges of the type $(i,j)\to (i+1,j+1)$, $(n-1,i)\to (1,i)$, and $(i, n) \to (i+2, 1)$ or $(i,n) \to (i+2,2)$. Since exactly one such edge enters each non-frozen
vertex, the three diagonal paths are simple and do not intersect. We denote them 
$\bar\P_1^{(s)}=\bar\P_{11}^{(s)}$, $\bar\P_2^{(s)}=\bar\P_{21}^{(s)}$ and 
$\bar\P_3^{(s)}=\bar\P_{1n}^{(s)}$,  according to the starting vertex. 
Besides, exactly one such edge
leaves each vertex, except for the following three: $(n-s-1,1)$, $(n-2,n)$ and $(n-1,n)$. 
These three vertices
are, therefore, the last vertices of the three diagonal paths. 

The fact that the three diagonal paths cover all vertices in $Q_{ns}(n)$ is proved in the same way as the
similar fact for diagonal paths in $Q_t(n)$. 

It remains to prove that the vertices in the paths  $\bar\P_{11}^{(s)}$, 
$\bar\P_{21}^{(s)}$ and $\bar\P_{1n}^{(s)}$ correspond to diagonal elements of the matrices $\bar V^{(s)}_{11}$, 
$\bar V^{(s)}_{21}$ and $\bar V^{(s)}_{1n}$ in the natural order,
and that the diagonal elements that are not involved in this correspondence do not define new trailing
minors. To do that we define sequences 
$\bar V_i^{(s)}$, $i\in [1,3]$,  similarly to $V_i^{(t)}$ in the proof of Lemma~\ref{rho_t}.
Denote by $\bar F^{(s)}_i$ the last occurrence of 
a pair $(*,1)$ in $\bar V^{(s)}_i$.  
Each sequence
$\bar V_i^{(s)}$ can be decomposed as follows: $\bar V^{(s)}_i=\bar U_i^{(s)}\bar F^{(s)}_i\bar W^{(s)}_i$ (the first and the second  
subsequences are empty if there is no occurrence of $(*,1)$ in $\bar V^{(s)}_i$). It is easy to prove by
induction the following characterization of the sequences $\bar V^{(s)}_i$ and $\bar\P^{(s)}_i$:
let $\bar r$ take values in $\{1,2,3\}$ and let $\bar r\equiv n-s-1\bmod 3$,  then
\begin{equation}\label{frecbvp}
\begin{aligned}
&\bar V^{(s)}_{\bar r+1}= \bar U^{(s-1)}_{\bar r+1}\bar W^{(s-1)}_{\bar r},\qquad \bar V^{(s)}_i= \bar V^{(s-1)}_i\quad \text{for $i\ne \bar r+1$, $s\in [n-3]$},\\
&\bar\P^{(s)}_{\bar r}= \bar U^{(s)}_{\bar r}\bar F^{(s)}_{\bar r},\qquad\qquad \bar\P^{(s)}_i= \bar V^{(s)}_i\quad \text{for $i\ne \bar r$,
$s\in [0,n-3]$}
\end{aligned}
\end{equation}
with the initial condition $W^{(0)}_{\bar r_0}=\varnothing$ for $\bar r_0\equiv n-1\bmod3$.

Relations \eqref{frecbvp} are proved similarly to relations~\eqref{frecvp}; they immediately
imply the sought for one-to one correspondence.
\end{proof}

Assign a function $f^{(ns)}_{ij}(X)=\rho_{ns}^{-1}(i,j)$
to each vertex $(i,j)$ of $Q_{ns}(n)$.

Consider a seed $((f^{(ns)}_{ij}), Q_{ns}(n) )$. Denote by $\mathcal S_v^{(s)}$ a
composition
of cluster transformations applied to consecutive mutable vertices of the diagonal
of $Q_{ns}(n)$ starting with $(n-s-1, 1)$ followed by the shift along the diagonal,
freezing of the last mutated vertex and erasure of the first vertex of the diagonal.
As before, we assume that whenever a vertex is shifted, the function assigned to it is not changed.

\begin{lemma}
\label{Sv}
For any $s\in [0,n-4]$, the transformation $\mathcal S_v^{(s)}$ transforms the seed $((f_{ij}^{(ns)}), Q_{ns}(n) )$ into
 $((f_{ij}^{(n, s+1)}), Q_{n, s+1}(n) )$.
\end{lemma}

\begin{proof} The proof of the lemma is similar to the proof of Lemma~\ref{Sh}. 
Let us first describe the evolution of the quiver $Q_{ns}(n)$. One can check by induction that prior to 
applying a mutation in the sequence $\mathcal S_v^{(s)}$ to a vertex $(p,q)\in\bar\P_{\bar r}^{(s)}$ the current state of
the quiver can be described as follows.

(i) Vertex $(p,q)$ is trivalent if $p=n$ and $q=1,2$, five-valent if $t>0$, $r=3$ and $(p,q)=(i,i)$ for
$i\in[2,n-1]$ or $(p,q)=(1,n-1)$, and four-valent otherwise.

(ii) The edges pointing from $(p,q)$ are $(p,q)\to (p,q-1)$ for $q\ne 1$ and $(p,1)\to (p-1,n)$ for $q=1$;
$(p,q)\to (p-1,q)$ for $p\ne 1$ and $(1,q)\to (n-1,q+1)$ for $p=1$; $(p,q)\to (1,n)$ whenever $(p,q)$ is
five-valent.
 
(iii) The edges pointing to $(p,q)$ are $(p-1,q-1)\to(p,q)$ for $p\ne 1$ and  $q\ne 1$, 
$(n-1,q)\to(1,q)$ for $p=1$, and $(p-2,n)\to (p,1)$ 
for $q=1$; the special edge from the vertex of $\bar\P_{\bar r}^{(s)}$ that has been just mutated and moved to $(p,q)$ for $p\ne n$. 

(iv) There are no edges between the vertex that has been just mutated and moved to $(p,q)$ 
and vertices $(p,q-1)$ and $(p-1,q)$ (or vertices $(p-1,n)$ or $(n-1,
q+1)$ in the situations described in (ii) above).

Consequently, when $(p,q)$ is moved one position back along $\bar\P_{\bar r}^{(s)}$, the local structure of the quiver is restored completely. 
   
To describe cluster transformations along the diagonal path $\bar\P^{(s)}_{\bar r}$ we consider the  submatrix $A=\bar V_{p_1q_1}^{(s)}$ of $\bar V$, where $(p_1,q_1)$ is the first
vertex of $\bar\P_{\bar r}^{(s)}$.  Let $L=L_s$ be the length of $\bar\P^{(s)}_{\bar r}$.
 The $L$-th column of $A$ corresponds to 
the column of $\bar V$ that contains $\bar F_{\bar r}^{(s)}$, see the proof of Lemma~\ref{rho_ns}.
Similar to the proof of Lemma \ref{Sh}, we use notation $w^{(s)}_j(i)$ for  the cluster variable
attached to the $i$th vertex of the diagonal path $\bar\P^{(s)}_j$.
Then, by Lemma~\ref{rho_ns}, 
$w^{(s)}_{\bar r}(i)=\det A^{\widehat{[i-1]}}_{\widehat{[i-1]}}$, and we need
to show that this variable transforms into $w^{(s+1)}_{\bar r}(i-1)$.

It follows from (i)--(iii) above that for $\bar r=1$ the  relations we need to verify are
$$
w^{(s+1)}_1(L-1)  w^{(s)}_1(L) = w^{(s)}_1(L-1) w^{(s)}_3(L-n+1) +
w^{(s)}_2(L-1 ) w^{(s)}_3(L-n)
$$  
for $s\in [n-4]$,
\begin{equation}\label{probe}
w^{(s+1)}_1(i-1) w^{(s)}_1(i) = w^{(s)}_1(i-1) w^{(s+1)}_1(i) +
w^{(s)}_2(i-1 ) w^{(s)}_3(i-n)
\end{equation}
for $i\in [n+1,L-1]$, $s\in [n-4]$, and
$$
w^{(s+1)}_1(i-1) w^{(s)}_1(i) = w^{(s)}_1(i-1) w^{(s+1)}_1(i) +
w^{(s)}_2(i-1 ) \phhi_{n-i+1} w^{(s)}_3(1)
$$
for  $i\in [2,n]$, $s\in [n-4]$.

Similarly, for $\bar r=2$ we need to verify
$$
 w^{(s+1)}_2(L-1) w^{(s)}_2(L) = w^{(s)}_2(L-1) w^{(s)}_1(L+1) +
w^{(s)}_1(L) w^{(s)}_3(L)
$$ 
for $s\in [n-4]$ and
$$
w^{(s+1)}_2(i-1) w^{(s)}_2(i) = w^{(s)}_2(i-1) w^{(s+1)}_2(i) +
w^{(s)}_1(i) w^{(s)}_3(i)
$$
for  $s\in [n-4]$, $i\in [2,L-1]$.

Finally, for $\bar r=3$ we need to verify
$$ 
w^{(s+1)}_3(L-1) w^{(s)}_3(L) = w^{(s)}_3(L-1) w^{(s)}_2(L) +
w^{(s)}_1(L+n-1) w^{(s)}_2(L-1)
$$  
for $s\in [n-4]$ and
$$
w^{(s+1)}_3(i-1)  w^{(s)}_3(i) = w^{(s)}_3(i-1) w^{(s+1)}_3(i)  +
w^{(s)}_1(i+n-1) w^{(s)}_2(i-1)
$$
for $s\in [n-4]$, $i\in [2,L-1]$.

As it was mentioned above, we will only verify these relations for the case $n= 6 m -1$. In this situation
$L_{3 \gamma} = j_{\nu - \gamma}$ for $\gamma \in[0,2m]$, $L_{3 \gamma+1} = j_{\nu - \gamma} - n$ for $\gamma \in[0,2m]$,
$L_{3 \gamma+2} = j_{\nu - \gamma} - 1$ for $\gamma \in [0,2m-1]$. 
This, together with the definition of submatrices $\bar V_{p_1q_1}^{(s)}$
implies that in ranges of index $i$ involved in relations we want to establish, 
functions $w^{(s)}_j(i) $ can be written as
\begin{equation}\label{wbar}
\begin{aligned}
&w_{1}^{(s)}(i)={\bangle {[i-1]\cup j_{[2m-2l,2m+l]}}{[i-1]\cup  i^*_{[2m+1,2m+l]}} },  
 \quad s\in[6l+1, 6l +3], \\
&w_{1}^{(s)}(i)={\bangle {[i-1]\cup j_{[2m-2l+1,2m+l-1]}}{[i-1]\cup i^*_{[2m+1,2m+l]}} },
\quad  s\in[6l-2, 6l],\\ 
&w_{2}^{(s)}(i)={\bangle {[i-1]\cup j_{[2m-2l,2m+l-1]}}{[i]\cup i^*_{[2m+1,2m+l]}} }, 
 \quad s\in[6l, 6l + 2], \\
&w_{2}^{(s)}(i)={\bangle {[i-1]\cup j_{[2m-2l - 1,2m+l]}}{[i]\cup i^*_{[2m+1,2m+l]}} },
\quad  s\in[6l+3, 6l +5],\\
&w_{3}^{(s)}(i)={\bangle {[n+i-1]\cup j_{[2m-2l+1,2m+l]}}{[n+i-2]\cup i^*_{[2m+1,2m+l]}} },
\quad s\in[6l-1, 6l + 1], \\
&w_{3}^{(s)}(i)={\bangle {[n+i-1]\cup j_{[2m-2l ,2m+l]}}{[n+i-2]\cup i^*_{[2m+1,2m+l+1]}} },
\quad  s\in[6l+2, 6l +4].
\end{aligned}
\end{equation}
For $i\in [2,n]$, the last two formulas have to be modified: the left hand side is replaced by $\phhi_{n-i+1} w^{(s)}_3(1)$, and in the right hand side $n+i$ is replaced by $i-2$ and $\bangle \ \ $ by ${\bangle \ \ }^{(2)}$. 
Note that for the reasons explained in the definition of matrices $\bar V_{pq}^{(s)}$ before
Lemma~\ref{rho_ns}, each expression of the form $\bangle{J\cup j_{[\alpha,\beta]}}{I}$ 
in~\eqref{wbar} can be replaced by $\bangle{J\cup j_{[\alpha,\beta+1]}}{I}$.
In addition, when treating the case $\bar r = 2$, we use the shift $\tau_3^{-1}$ (see Fig.~\ref{fig:baru}) to re-write
$$
w_{3}^{(6l+2)}(i)={\bangle {[i-2]\cup j_{[2m-2l-1 ,2m+l]}}{[i]\cup i^*_{[2m+1,2m+l]}} },\quad w_{3}^{(6l-1)}(i)={\bangle {[i-2]\cup j_{[2m-2l ,2m+l]}}{[i]\cup i^*_{[2m+1,2m+l-1]}} },
$$
while for $\bar r = 3$ we use $\tau_3^{-1}$ to re-write
\begin{align*}
&w_{2}^{(6l+1)}(n+i-1)={\bangle {[n+i-1]\cup j_{[2m-2l+1 ,2m+l]}}{[i]\cup  i^*_{[2m+1,2m+l+1]}} },\\
&w_{2}^{(6l-2)}(n+i-1)={\bangle {[n+i-1]\cup j_{[2m-2l+2 ,2m+l]}}{[n+i-3]\cup  i^*_{[2m+1,2m+l]}} }.
\end{align*}

Using representation \eqref{wbar} and the alternative representation together with the shifts, one 
can check  that all relations  we need to establish are of the form
\begin{equation}\label{gen012}
{\bangle {J \cup \delta} {I}}\ {\bangle {J \cup \gamma} {I\cup \alpha }} = {\bangle {J } {I}}\ {\bangle {J \cup \gamma \cup\delta} {I\cup \alpha }} + {\bangle {J \cup \gamma} {I}}^{(q)} {\bangle {J \cup \delta} {I\cup \alpha }}
 \end{equation}
(if $s\in[6l, 6 l +2]$),
or of the form
 \begin{equation}\label{gen345}
{\bangle {J \cup \delta} {I\cup \beta}}\ {\bangle {J \cup \gamma} {I\cup \alpha }} = {\bangle {J } {I}}\ {\bangle {J \cup \gamma \cup\delta} {I\cup \alpha \cup \beta }} + {\bangle {J \cup \gamma} {I\cup \beta}}^{(q)} {\bangle {J \cup \delta} {I\cup \alpha }}
 \end{equation}
 (in all the remaining cases) for suitable indices $\alpha$, $\beta$, $\gamma$, $\delta$ and index sets $I$, $J$
 with $q=2$ if $\bar r=1$ and $i\in[2,n]$ and $q=1$ otherwise. For instance, let us verify relation~\eqref{probe}. Here 
$s=6 l$, and we have to compute $w_1^{(6l+1)}(i-1)$, $w_1^{(6l+1)}(i)$, $w_3^{(6l)}(i-n)$,
$w_1^{(6l)}(i-1)$, $w_1^{(6l)}(i)$, and $w_2^{(6l)}(i-1)$. We use representation~\eqref{wbar}
for the former three functions, and the alternative reprentation for the latter three. 
Then~\eqref{probe} becomes~\eqref{gen012} with the choice 
 $\alpha = \gamma = i-1$, 
 $\delta= j_{2m-2l}$, $I =[i-2]\cup i^*_{[2m+1,2m+l]}$ and 
 $J = [i-2]\cup j_{[2m- 2l +1,2m +l]}$. 
 
 Now, either of the identities \eqref{gen012} and \eqref{gen345}
 can be established as follows. Consider the smallest dense submatrix  $C$ of $\bar V_{\hat I}^{\hat J}$ that contains al the minors  featured in the identity. For $s=0$, $C=\bar V_{[i,i_{2m+1}]}^{[i, j_{2m+1}-1]}$, which has one less rows than columns. We apply to $C$ identity \eqref{notjacobi} using columns in $C$ that correspond to columns $i, j_{2m}$ and $j_{2m+1}-1$ in $\bar V$ and the row in $C$ that correspond
 to row $i$ in $\bar V$. The needed relation is obtained after canceling the common factor  equal to $\det X_{[n-1]}^{[n-1]}$.
  If $s>0$, $C$ is a square matrix. We apply to it identity \eqref{jacobi} and cancel the common factors
 to obtain the needed relation. In doing so, we use columns in $C$ that correspond to $\gamma$ and $\delta$ and rows that correspond to $\alpha$ and $\beta$ for~\eqref{gen345}, or the row that 
 correspond to $\alpha$ and the last row of $C$ for~\eqref{gen012}.
\end{proof}

Define
$$
\mathcal S_v = \mathcal S_v^{(n-4)}\circ \mathcal S_v^{(n-5)}\circ \cdots \circ
\mathcal S_v^{(0)}.
$$
We have constructed a transformation $\mathcal S=\mathcal S_v\circ \mathcal S_h$.
To check that is has the properties stipulated in Theorem~\ref{transform1}, first consider the quiver  $Q_{n,n-3}(n)$. In this
quiver, frozen vertices $(1,1)$ and $(2,1)$ are only
connected to vertices $(1,2)$ and $(2,2)$. Freezing the latter two vertices and
erasing the former two, we obtain a quiver
isomorphic to $\hat Q_{CG}(n-1)$.

Next, consider  $\bar V_{11}^{(n-3)}$,  $\bar V_{21}^{(n-3)}$, $\bar  V_{1n}^{(n-3)}$.
We claim that
$\bar V_{11}^{(n-3)}$ is the $(n (k-1)+1)\times (n (k-1)+1)$ submatrix of $\bar V$  
obtained by deleting columns $j_2,\ldots, j_{k-1}$ and rows $i^*_{\nu +1}, 
\ldots, i^*_{k-1}$. As before, we will only
explain this for $n=6m-1$, since other cases can be dealt with in a similar way. We have $k-1= 3m -1$ and
$\nu+1=2m+1$, so 
$\bar V_{11}^{(n-3)}= \bar V_{11}^{(3(2m - 1) - 1)} $ is the core of $\bar V^{\hat \jmath_{[2,3m-1]}}_{ \hat \imath^*_{[2m+1,3m-1]}}$, as claimed.

 Similarly, $\bar V_{21}^{(n-3)}$ is the
 $(n-2)(k-1)\times (n-2)(k-1)$ submatrix of $\bar V$ obtained by deleting  
 the last $n-1$ rows and $n+1$ columns, columns $j_2,\ldots, j_{k-2}$ and rows $i^*_{\nu+1}, \ldots,i^*_{k-1}$.
 
Finally,  $\bar V_{1n}^{(n-3)}$ is the
 $(n-2)(k-1)\times (n-2)(k-1)$ submatrix of $\bar V$ obtained by deleting  
 the first $(n-1)$ rows and $n$ columns, columns $j_2,\ldots, j_{k-1}$ and rows $i^*_{\nu+1}, \ldots,i^*_{k}$.

It follows from the description above that among the functions attached to vertices
of $Q_{n,n-3}(n)$
only functions  $f_{11}^{(n, n-3)}= \det \bar V_{11}^{(n-3)}$ and $f_{21}^{(n,
n-3)}= \det \bar V_{21}^{(n-3)}$ are associated
with submatrices of $\bar V$ containing elements of the first column of the
first block column. After we removed
vertices $(1,1)$ and $(2,1)$, the remaining functions $f_{ij}^{(n, n-3)}$ are minors of the $(n
(k-1)+1)\times (n (k-1))$ matrix $A$ obtained from $\bar V$ by deleting
 the first column from  every block column and the last row from every block row of height $n-1$, 
 except for the first and the last one.   It will be convenient to make $A$ into a square 
  $(n(k-1)+1)\times (n (k-1)+1)$ matrix $V'$ by attaching the 
  vector $\mbox{col} (0, X_n^{[n-2]}, 0, \ldots, 0)$ as the first column.
For example, for  $n=5$,
 $V'$ is equal to
 $${\small
\left (
 \begin{array}{ccccccccccc}
 0  & x_{12} & x_{13} &  x_{14} & x_{15}& 0 & 0 & 0 & 0 & 0 & 0  \\
 x_{15} &  x_{22} & x_{23} & x_{24} & x_{25}& 0 & 0 & 0 & 0 & 0 & 0 \\
  x_{25} & x_{32} & x_{33} &x_{34} & x_{35}& 0 & 0 & 0 & 0 & 0 & 0 \\
x_{35} &  x_{42} & x_{43} &x_{44} & x_{45}& 0 & 0 & 0 & 0 & 0 & 0 \\
  0 &  x_{11} & x_{12} & x_{13} & x_{14} & x_{15} & x_{22} & x_{23} & x_{24} &
x_{25}& 0  \\
 0 &  x_{21} & x_{22} & x_{23} & x_{24} & x_{25} & x_{32} & x_{33} & x_{34} &
x_{35}& 0  \\
 0 &  x_{31} & x_{32} & x_{33} & x_{34} & x_{35} & x_{42} & x_{43} & x_{44} &
x_{45}& 0  \\
 0 & 0 & 0 & 0 & 0 & 0 &    x_{11} & x_{12} & x_{13} & x_{14} & x_{15}  \\
 0 & 0 & 0 & 0 & 0 & 0 &   x_{21} & x_{22} & x_{23} & x_{24} & x_{25}  \\
 0 & 0 & 0 & 0 & 0 & 0 &   x_{31} & x_{32} & x_{33} & x_{34} & x_{35}   \\
 0 & 0 & 0 & 0 & 0 & 0 &   x_{41} & x_{42} & x_{43} & x_{44} & x_{45}
 \end{array}
 \right ).
}
 $$
 
 In $V'$ we can use the shift
 $\tau_3^{-1}$ to embed $\bar V_{1n}^{(n-3)}$ alternatively as a dense submatrix whose upper left corner is located in the entry $(2,1)$.
 
Now, the set $\{ f_{ij}^{(n, n-3)}\: i\in [n-1]$, $j\in [2, n]\}$
coincides with the set of trailing principal minors of two dense submatrices of
$V'$: the first one, $V'_{22}$, is obtained by deleting the first row
and column; the second, $V'_{1n}$ is obtained by deleting the last $n$ columns and
the first one and the last $(n-1)$ rows.
Multiply $V'$ on the right by a block-diagonal unipotent lower triangular
matrix 
$$
Z=\mbox{diag}(1,\underbrace{N(v(X)),\ldots, N(v(X))}_{k-1 \ \mbox{times}}).
$$ 
As a result, both $V'_{22}$ and  $V'_{1n}$ will also be multiplied by a
unipotent lower triangular matrix, and therefore the values  of their trailing principal
minors will not change. Furthermore, the last $(k-1)n$ columns of   $V'$ are
subdivided into $k-1$ block columns of width $n$  that have a  form
$$
\left (
\begin{array}{c}
{\bf 0}_{n'_i\times n} \\
X_{[2-\delta_{i1},n-1]}^{[2,n]} \ \ 0\\
 X_{[1,n-2+\delta_{i,k-1}]} \\
{\bf 0}_{n_i''\times n} 
\end{array}
\right ) 
$$
with $n'_i=1-\delta_{i1}+(i-1)(n-2)$, $n_i''=1-\delta_{i,k-1}+(k-i-1)(n-2)$, $i\in [k-1]$.
Right multiplication of $V'$ by $Z$ results in right multiplication of each
block column
by $N(v(X))$, which, due to \eqref{zeta}, transforms the non-zero part of the block
column into
$$
\left (
\begin{array}{c}
\zeta(X)_{[2-\delta_{i1},n-1]} \ \  0\\
\ \ 0\ \ \zeta(X)_{[1,n-2+\delta_{i,k-1}]}
\end{array}
\right ).
$$
The first column of $V'$ does not change after multiplication by $Z$, however,
its nonzero entries can be re-written as $x_{in}=\zeta(X)_{in}$,  $i\in[n-2]$.
Denote  $\X'=\left [\zeta(X)_{[2,n-1]}\ 0\right ]$, $\Y'=\left [ 0\
\zeta(X)_{[1,n-2]}\right ]$.
We have transformed $V'_{22}$ into
\[
 \left [
\begin{array}{cccc}
 \X' & 0 & \cdots & \cdots\\
\Y' & \X' & 0 \cdots  \\
 \ddots& \ddots &\ddots & \ddots \\
 \cdots & 0 & \Y' & \X' \\
\cdots & 0 &0 &\  0\  \zeta(X)
\end{array}
\right ]
\]
and  $V'_{1n}$ into
\[
 \left [
\begin{array}{ccccc}
(\Y')^n_{[1,n-2]} & \X' & 0 & \cdots & 0 \\
0 & \Y' & \X' &\cdots & 0  \\
0 & 0 &\Y' &\X' &\cdots \\
\vdots& \ddots& \ddots &\ddots &\vdots \\
0 & \cdots & 0 & 0 &   \Y'
\end{array}
\right ].
\]
A comparison with the definition of the initial cluster for $\CC_{CG}(n-1)$ given by 
~\eqref{inclust} then shows that trailing minors of $V'_{22}$ are
$\thetta_i(\zeta(X))=\phhi_i(X)$, $i\in [n-1]$, and $\thetta_{n-1}(\zeta(X))
\pssi_i(\zeta(X))=\phhi_{n-1}(X) \pssi_i(\zeta(X))$,
$i\in [(k-1)(n-2)]$.
Trailing minors of $V'_{1n}$ are  $\phhi_i(\zeta(X))$,
$i\in [(k-1)(n-2)]$.

Finally, we re-label the vertices of the subquiver of  $Q_{n,n-3}(n)$ obtained by
deleting vertices $(1,1), (2,1)$ : a vertex $(i,j)$ is re-named $(i,j-1)$. The
resulting quiver coincides with $Q_{CG}(n-1)$ and, by the previous paragraph, the
functions attached to its vertices
are as prescribed by Theorem~\ref{transform1}.

 \subsection{Sequence $\T$}\label{Ttrans} 
 The goal of this subsection is to prove Theorem~\ref{transform2}.
 
We construct $\mathcal T$ in several stages, while maintaining two copies of the quiver. The same
mutations are applied to both copies. 
At the beginning of each stage, if the function attached to a vertex becomes equal to
$f^{w_0}_{ij}$ for some $i,j\in [n]$, the vertex in the first copy of the quiver is frozen. 
When this frozen vertex becomes isolated in the first copy, its counterpart in the second copy is moved 
to the $(i,j)$th position of the resulting quiver. By the end of the process, the function attached to 
each vertex $(i,j)$ of the resulting quiver $Q^*_{CG}(n)$ (the second copy) is $f^{w_0}_{ij}$.  
By Lemma~\ref{antipoiss}, $\{f^{w_0}_{ij}\}$ is a log-canonical basis. 

Let us identify the $\{f_{ij}^{w_0}\}=\{\thetta^{w_0}_i, i\in[n];  \phhi^{w_0}_\alpha, \alpha\in [N]; \pssi_\beta^{w_0}, \beta\in[M]\}$ that we seek to obtain via cluster transformations described below.
 Clearly, $\thetta^{w_0}_i = \det X^{[i]}_{[i]}$, while the other two families can be conveniently presented as 
 minors of $\bar U$. Indeed, it is not hard to observe that $\phhi_\alpha^{w_0}$ is the $\alpha\times \alpha$ leading principal minor of the dense $N\times N$ submatrix
 of $\bar U$ whose left upper corner is $\bar u_{2,n+1} = x_{21}$. (Note that the same $N\times N$ submatrix of $\bar U$ was used to define $\phi_\alpha$
 as its {\em trailing\/} principal minors.) Similarly, $\pssi_\beta^{w_0}$ is the $\beta\times \beta$ leading principal minor of the dense $M\times M$ submatrix
 of $\bar U$ whose left upper corner is $\bar u_{2,n} = x_{1n}$. (The $\tau_3$-shift of this $M\times M$ submatrix of $\bar U$ was used to define $\pssi_\beta$ as its {\em trailing\/} principal minors.)

\begin{lemma}
\label{antitorus}
The global toric action of $(\mathbb{C^*})^3$ on $\A_{CG}(GL_n)$ described in Proposition \ref{torus} induces a local toric
action of rank $3$ on the family $\{f_{ij}^{w_0}\}_{i,j=1}^n$.  Weight matrices of this action with respect to families $\{f_{ij}\}_{i,j=1}^n$ and $\{f_{ij}^{w_0}\}_{i,j=1}^n$ have the same column space.
\end{lemma}

\begin{proof} As we observed in the proof of Proposition \ref{torus}, the weight matrix of the global toric action of $(\mathbb{C^*})^3$ on $\A_{CG}(GL_n)$ with respect to the initial cluster $\{f_{ij}\}$ can be described via relations
$f_{ij} \left ( t_1^{D_n} (t_0 X) t_2^{D_n}\right ) = t_0^{\deg f_{ij}} t_1^{\alpha_{ij}} t_2^{\beta_{ij}} f_{ij}(X)$ for some integers $\alpha_{ij}$ and $\beta_{ij}$.
Note that $W_0 D_n W_0 = (n+1) \one_n - D_n$, and hence
\begin{align*}
&f^{w_0}_{ij} \left ( t_1^{D_n} (t_0 X) t_2^{D_n}\right ) = f_{ij} \left ( t_1^{W_0 D_n W_0} (t_0 W_0 X W_0) t_2^{W_0 D_n W_0}\right )\\
& = \left (t_0 (t_1 t_2)^{n+1}\right )^{\deg f_{ij}} f_{ij} \left ( t_1^{-D_n} W_0 X W_0 t_2^{-D_n}\right )\\
&=
t_0^{\deg f_{ij}} t_1^{-\alpha_{ij} + (n+1) \deg f_{ij}} t_2^{-\beta_{ij}+ (n+1) \deg f_{ij}} f^{w_0}_{ij}(X).
\end{align*}
Therefore, both the claim about the rank of the action and the claim about the column spaces follow
from the proof of Proposition~\ref{torus}.
\end{proof}

It follows from
Lemmas~\ref{Q-restore} and~\ref{antitorus} that the quiver corresponding to the basis 
$\{f^{w_0}_{ij}\}$ is $\lambda Q_{CG}(n)$ for 
some integer $\lambda\ne0$. It will be enough to check one edge of $Q^*_{CG}(n)$ to find that
 $\lambda=-1$, and hence  $Q^*_{CG}(n)$ is isomorphic to $Q_{CG}^{opp}(n)$.

Thus, we will focus our attention on the evolution of the first copy of the quiver. At the 
beginning of each stage the vertices of the first copy are uniquely labeled by pairs of integers. 
The quiver with such a labeling is denoted $Q_S(n)$, where $S$ is the
stage number (written in roman numerals). For a 
vertex labeled $(p,q)$, we denote by $f_S(p,q)$ the function attached to this vertex.  We then apply to
$Q_S(n)$ a sequence of cluster transformations $\T_S$, which consists of consecutive subsequences
$\T_S^l$; we denote $\T_S^{(l)}=\T_S^l\circ\dots\circ\T_S^1$. The function attached to $(p,q)$ after
the end of $\T_S^{(l)}$ is denoted $f_S^{(l)}(p,q)$.

 To describe the first sequence of mutations more conveniently, we relabel $n^2$ vertices of the initial quiver $Q_{CG}(n)$ in the following way: $(i,j)\mapsto (p(i,j),q(i,j))$ with
\begin{equation}\label{pqviaij}
 p(i,j)=\begin{cases}
 i, \qquad \hskip 2.5mm \text{if $i<j$,}\\
 i-1, \quad \text{if $i\geq j$,}
 \end{cases}\qquad
 q(i,j)=\begin{cases}
 j-i, \qquad\qquad\hskip 2mm\text{if $i<j$,}\\
 j-i+n+1, \quad \text{if $i\geq j$.}
 \end{cases}
\end{equation}
The relabeled nodes can be naturally placed on  an $(n-1)\times (n+1)$ grid with an additional vertex $(0,n+1)$ attached.
This is equivalent to changing the fundamental domain of the universal cover of $Q_{CG}(n)$ (see Fig.~\ref{fig:tilt} for the case $n=5$). Our initial labeling corresponds to the square fundamental
domain shown by dashed lines, while the new labeling corresponds to the truncated parallelogram
fundamental domain shown by solid lines. Therefore, the function $f_{\romon}(p,q)$ attached to a vertex $(p,q)$
is $f_{i(p,q),j(p,q)}$, where $i(p,q)$ and $j(p,q)$ are defined via the inverse of  \eqref{pqviaij}:
 \begin{equation*}
 i(p,q)=\begin{cases}
 p, \qquad \hskip 2.5mm \text{if $p+q\leq n$,}\\
 p+1, \quad \text{if $p+q >n$,}
 \end{cases}\qquad
 j(p,q)=\begin{cases}
 p+q, \qquad\hskip 2.5mm\text{if $p+q\leq n$,}\\
 p+q-n, \quad \text{if $p+q>n$.}
 \end{cases}
\end{equation*}
In other words, we have
\begin{equation}\label{fviaphi}
f_{\romon}(p,q)=\begin{cases}
\phhi_{1+r(n-1)-p}, \quad&\text{if $q=2r-1<n+1$,}\\
\pssi_{1+r(n-1)-p},  \quad&\text{if $q=2r<n+1$,}\\
\theta_{n-p}, \qquad&\text{if $q=n+1$.}
\end{cases}
\end{equation}

Let us rewrite these expressions in terms of cores of submatrices of $\bar U$. From now on we omit explicit mentioning
of $\bar U$ in the subscript.

 \begin{lemma}
 \label{lem_initcore}
 For any $p,q \in [n]$, 
 \begin{equation*}
 f_{\romon}(p,q) = \bangle {[p+q-1]}{[p]}.
  \label{initcore*}
 \end{equation*}
\end{lemma}
\begin{proof}
Indeed, for $q=n+1$ \eqref{fviaphi} implies $f_{\romon}(p, n+1)=\thetta_{n-p} = \det X_{[p+1]}^{[p+1]} =
\bangle {[p+n]}{[p]}$.
 Let $q=2r-1<n+1$, then by~\eqref{fviaphi} $f_{\romon}(p,q) = \phhi_{1 +r(n-1) - p}$ is the determinant 
 of the irreducible submatrix $A$ of $\bar U$ whose upper left and lower right entries are 
 $$
 \left ( (k-r+1)(n-1) + p+1, (k-r+1)(n+1) + p+q \right )
 $$ 
 and
 $$
 \left ( k(n-1) + n, k (n+1) + n \right ).
 $$ 
 Let $A'=\tau_3^{-k+r-1}(A)$; its
upper left and lower right entries are $\left ( p+1,  p+q \right )$ and
 $\left ( (r-1)(n-1) + n, (r-1) (n+1) + n \right )$. Since rows below $(r-1)(n-1) + n$ in $\bar U$ contain only zero entries in the first 
 $(r-1) (n+1) + n $ columns, and since $A'$ is irreducible, we obtain $\phhi_{1 +r(n-1) - p} = \bangle  {[p+q-1]} {[p]}$ as required.
 The case of even $q$ can be treated in the same way.
 \end{proof}

Thus, frozen vertices are now located at $(1,n-1)$, $(1,n)$ and $(0,n+1)$. The corresponding cluster variables are $\pssi^{w_0}_M$, $\phhi^{w_0}_N$ and $\thetta^{w_0}_n$.
The arrows in the rearranged quiver are described as follows: unless the vertex $(p,q+1)$ is frozen, there is a short horizontal edge directed to the left, $(p,q+1)\to (p,q)$; there are all possible short vertical edges directed down, $(p,q) \to (p+1,q)$ and all possible sort ``northeast" edges  $(p,q) \to (p-1,q+1)$, except for $(p,q)=(1,n)$. In addition, there is a directed path of long edges between the right and left boundaries of the grid: $(n-1,n+1)\to (n-1,1)\to (n-2,n-1) \to (n-2,1) \to \cdots \to (1,1) \to (0,n+1)$, 
 and a directed path of long edges between the lower and upper boundaries: $(n-1,3)\to (1,1)\to (n-1,4) \to (2,1) \to \cdots \to (1,n-2) \to (n-1,n+1)$.  
We will denote the rearrangement of $Q_{CG}$ by $Q_{\romon}(n)$. On Fig.~\ref{fig:tilt} the two paths are shown with thick dashed and dotted lines, and thick dashed lines, respectively. The vertices with the same filling pattern are identified.

\begin{figure}[ht]
\begin{center}
\includegraphics[height=6.95cm]{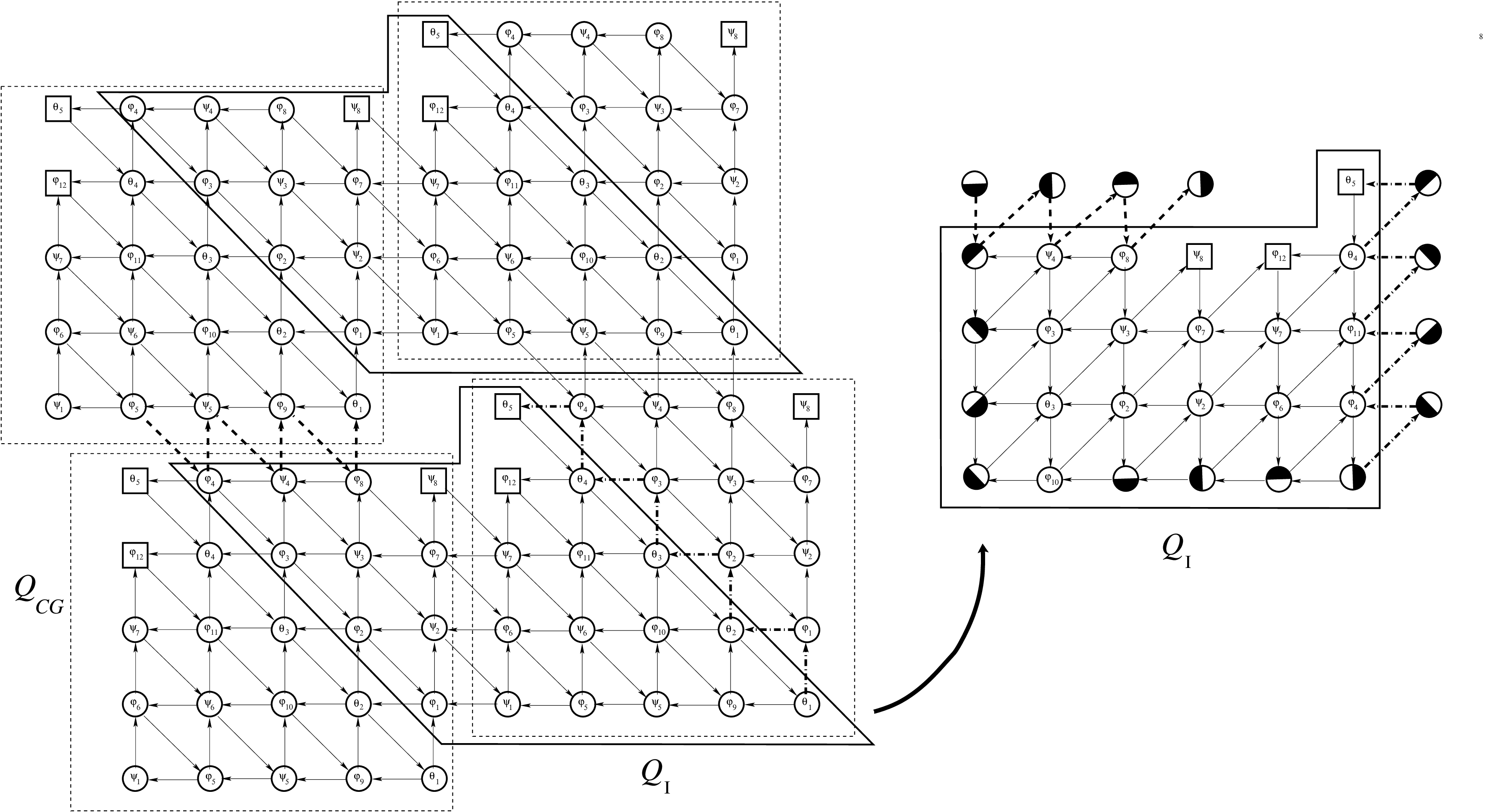}
\caption{Relabeling the vertices of $Q_{CG}(5)$}
\label{fig:tilt}
\end{center}
\end{figure}

It will be convenient to decompose $Q_{\romon}(n)$ into the {\it web\/} and the {\it weave\/}. 
The weave consists of the edges $(p,q)\to (p',q')$ such that $p=1$, $p'=n-1$, or $p=n-1$, $p'=1$, and 
hence is the second of the paths defined above (more exactly, the edges of this path), and the web is 
the rest of the quiver. Clearly, the web is a planar graph. Its vertices can be placed on $n-1$ 
concentric circles so that
all vertices $(p,q)$ with the same value of $p$ lie on the $p$th circle counting from inside. 
Consequently, the weave consists of the edges between the inner and the outer circles. Besides, the vertices with the same value of $q$ are arranged into $n+1$ rays going counterclockwise. Finally, 
we place the vertex $(0, n+1)$ in the center. The web and the weave for $Q_{\romon}(5)$ are shown on
 Fig.~\ref{fig:ann}. The weave is shown in dashed lines, and the dashed vertices should be identified 
 with the corresponding solid vertices.

 \begin{figure}[ht]
\begin{center}
\includegraphics[height=8cm]{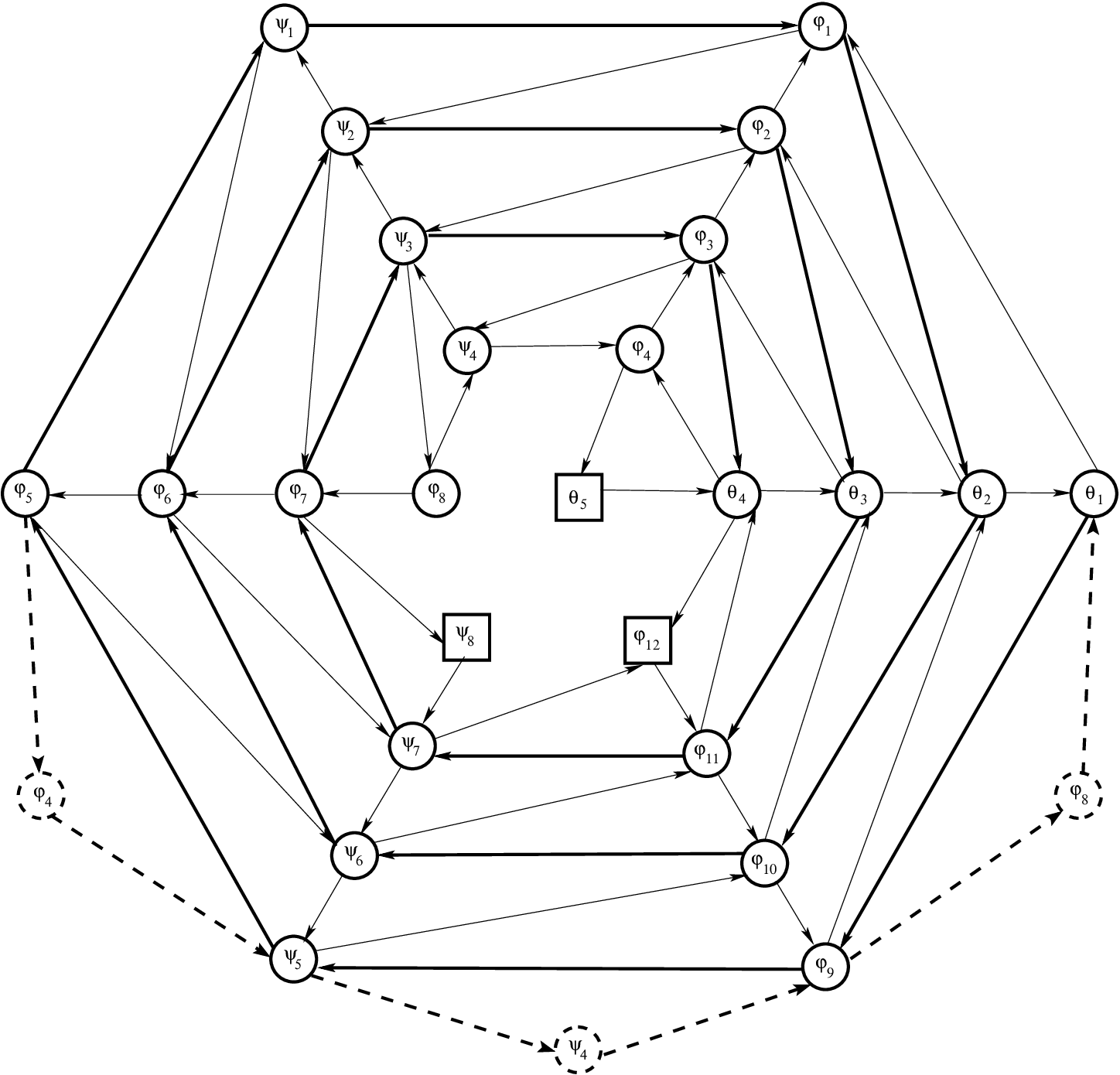}
\caption{Quiver $Q_{\romon}(5)$: the web and the weave}
\label{fig:ann}
\end{center}
\end{figure}

  The web is subdivided into consistently oriented triangles. More precisely, consider the trapezoid formed
by the vertices $(p,q)$, $(p,q-1)$, $(p-1,q-1)$ and $(p-1,q)$; here and in what follows 
$q$ is understood $\bmod\ n+1$ with values in $[n+1]$. We will 
say that $(p,q)$ is its {\it southeast\/} ({\it SE\/}) corner; {\it SW\/}, {\it NE\/} and {\it NW\/} corners are defined similarly. In what follows we will refer to a trapezoid by the position of its SE corner.
   
Each trapezoid contains one diagonal; we will distinguish NE, SE, NW and SW trapezoids, depending on the direction of the diagonal. The $(p,1)$ trapezoids  are NW, all the other trapezoids
are NE. Besides, the trapezoid $(2,n-1)$ is {\it incomplete\/}: its upper base is removed. Note that 
the trapezoid $(2,n)$ is not assumed incomplete, since we can freely add or delete edges between 
frozen vertices. Finally, there is a special triangle with vertices at $(1,1)$, $(0,n+1)$ and $(1,n+1)$.
The trapezoids are linearly ordered as follows: $(p,q)\prec (p',q')$ if $p>p'$ or $p=p'$ and $n+1\geq q>q'\geq 1$.

 We will apply a sequence $\T^1_{\romon}$ of $(n-2)(n+1)+1$ mutations along the directed path that starts at $(n-1,n+1)$, ends at $(1,n+1)$ and consists of edges $(p,q)\to (p,q-1)$, $q\in [2,n+1]$, and $(p,1)\to (p-1,n+1)$, 
 $p\in [2,n-1]$; this path is shown by thick solid lines on Fig.~\ref{fig:ann}.
Then, in the resulting quiver, we apply a sequence $\T^2_{\romon}$ of $(n-3)(n+1)+1$ mutations along a similar path
 that starts at $(n-1,n+1)$ and  ends at $(2,n+1)$, etc., the last sequence, $\T^{n-1}_{\romon}$, 
 being a single mutation at $(n-1,n+1)$.
 Denote $\T^{(l)}_{\romon} = \T^l_{\romon}\circ\cdots\circ \T^1_{\romon}$,
$Q^{(l)}_{\romon}(n)=\T^{(l)}_{\romon}(Q_{\romon}(n))$ for $l\in [n-1]$, $Q^{(0)}_{\romon}(n)=Q_{\romon}(n)$, $f^{(0)}_{\romon}(p,q)=f_{\romon}(p,q)$.

\begin{lemma}
 \label{4path} For any $l\in [n-1]$, 
    \begin{equation*}
 f^{(l)}_{\romon}(p,q)= \begin{cases}
 f^{(l-1)}_{\romon}(p,q), & \quad \text{if  $p\leq l-1$ or $p=l$, $q \leq n$},\\
 \bangle {[p+q-l-1]} {[p-l]\cup [n-l+1,n]},   &\quad  \text{if  $p\geq l+1$, $q > l$ or $p=l$, $q = n+1$},\\
  \bangle {[p+q+n-l-1]}{[p-l-1]\cup [n-l,n]},   &\quad  \text{if $p\geq l+1$, $q \leq  l$}. 
  \end{cases}
 \end{equation*}
\end{lemma}

\begin{proof}
Let us first describe the evolution of the quiver $Q_{\romon}(n)$. 
After the first mutation in $\T^1_{\romon}$, the trapezoids $(n-1,n+1)$ and $(1,n+1)$
turn to {\it oriented\/} trapezoids without a diagonal. With the subsequent mutations one of the oriented trapezoids propagates clockwise. 
One can check by induction that prior to applying a mutation in the sequence $\T^l_{\romon}$ 
to a vertex $(p,q)$, the current state of the quiver $Q^{(l-1)}_{\romon}(n)$ is described as follows.

(i) Vertex $(p,q)$  is trivalent if $p=n-1$ and $q=l+1, l+2$, and four-valent otherwise.
There are two edges pointing from $(p,q)$: $(p,q) \to (p,q\pm 1)$. 
Edges pointing towards $(p,q)$ are

i.1) $(p\pm 1, q) \to (p,q)$ if $ l \leq p \leq n-2$;

i.2) $(n- 2, q) \to (n-1,q)$ and $(1, q-l-2) \to (n-1,q)$ if $p=n-1$ and $l+3\leq q \leq n+1$;

i.3) $(n- 2, q) \to (n-1,q)$  if $p=n-1$ and $ q=l+1, l+ 2$;

i.4) $(n- 2, q) \to (n-1,q)$ and $(1, n-2+q-l) \to (n-1,q)$ if $p=n-1$ and $1\leq q \leq l$.

(ii) The web is subdivided into  oriented triangles and trapezoids. 

If $q<n+1$, there are at most six oriented trapezoids: 
$$
(p+1,q+1)\prec(p+1,1)\prec(p,q+1)\prec(p,1)\prec(l,n+1)\prec(l,1).
$$
The first two do not exist if $p=n-1$, and the last two do not exist if $l=1$.
Note that $l<p$, and hence indeed  $(p,1)\prec (l,n+1)$. 
All the trapezoids  $(p',q')\prec(p+1,q+1)$  are NE if $q'\ne 1$ and NW if $q'=1$.
All the trapezoids $(p+1,q+1)\prec(p',q')\prec (p,q+1)$ except for $(p+1,1)$ are SE.
All the trapezoids $(p,q+1)\prec(p',q')\prec(l,n+1)$ are NE if $q'\ne1$ and NW if $q'=1$, $p'\ne p$.
All the trapezoids $(l,n+1)\prec(p',q')$ except for $(l,1)$ are SE if $q'\ne n+1$ and SW if $q'=n+1$.
The trapezoid   $(2,n-1)$ is incomplete. Finally, the special triangle
degenerates after the unique mutation at $(1,1)$, when the edge $(1,1)\to (0,n+1)$ disappears.

If $q=n+1$, there are at most two oriented trapezoids: $(l,n+1)\prec(l,1)$; they do not exist if $l=1$.
Trapezoid $(p+2,1)$ is NW (it does not exist if $p=n-1, n=2$). Trapezoid $(p+1,1)$ is SW (it does not exist
if $p=n-1$). All the trapezoids  $(p',q')\prec(p+2,1)$  are NE if $q'\ne 1$ and NW if $q'=1$.
All the trapezoids $(p+2,1)\prec(p+1,q')\prec (p+1,1)$ are SE. All the trapezoids $(p+1,1)\prec(p',q')\prec(l,n+1)$ are NE if $q'\ne1$ and NW if $q'=1$.
The rest of the pattern, starting from the trapezoid $(l,n+1)$, is the same as above.

(iii) The weave is only affected by mutations along the outer circle. Prior to mutation at
$(n-1,q)$ it consists of two paths having together $2(n-1)$ edges. For $q\ne l, l+1$ these paths are
\begin{equation*}
\begin{split}
(n-1, i)\to (1, 1)\to &(n-1,i+1)\to (1,2)\to \cdots\\
& \cdots\to (n-1,q-1)\to (1,q-i)\to (n-1, q)
\end{split}
\end{equation*}
and 
\begin{equation*}
\begin{split}
(n-1, q+1)\to (1, q-i+1)\to &(n-1,q+2)\to (1,q-i+2)\to \cdots\\
& \cdots\to (n-1,i-3)\to (1,n-2)\to (n-1, i-2),
\end{split}
\end{equation*}
where
$$
i=\begin{cases} l+3,\quad \text{if $1\le q<l$},\\
                l+2,\quad \text{if $l+2\le q<n+1$}.
  \end{cases}
$$
The first path is empty for $q=l+2$, and the second path is empty for $q=n+1$, $l=1$ (that is, before the start
of mutations). 
For $q=l, l+1$ the two paths remain the same as for $q=l+2$.               

After the mutation at $(n-1,1)$ in $\T^{l}_{\romon}$  and prior to the start of $\T^{l+1}_{\romon}$, 
the two paths are given by the formulas above with $q=n+1$ and $i=l+2$.

The quiver $Q^{(2)}_{\romon}(5)$ before the mutation at $(4,2)$ in the sequence $\T^3_{\romon}$ is shown on Fig.~\ref{fig:ann2}.
The vertex $(4,2)$ is four-valent, the edges pointing towards $(4,2)$ come from $(3,2)$ and $(1,2)$, according
to i.4) above. 
Trapezoids $(5,3)$ and $(5,1)$ do not exist, since $p=n-1=4$. Trapezoids $(4,3)$, $(4,1)$, $(3,6)$ and $(3,1)$
are oriented. The types of all other trapezoids are exactly as described in (ii). Finally, $q=2<3=l$, and hence 
$i=6$ and the weave consists of the paths $(4,6)\to(1,1)\to(4,1)\to(1,2)\to(4,2)$ and
$(4,3)\to(1,3)\to(4,4)$, as stipulated by (iii). 

 \begin{figure}[ht]
\begin{center}
\includegraphics[height=8cm]{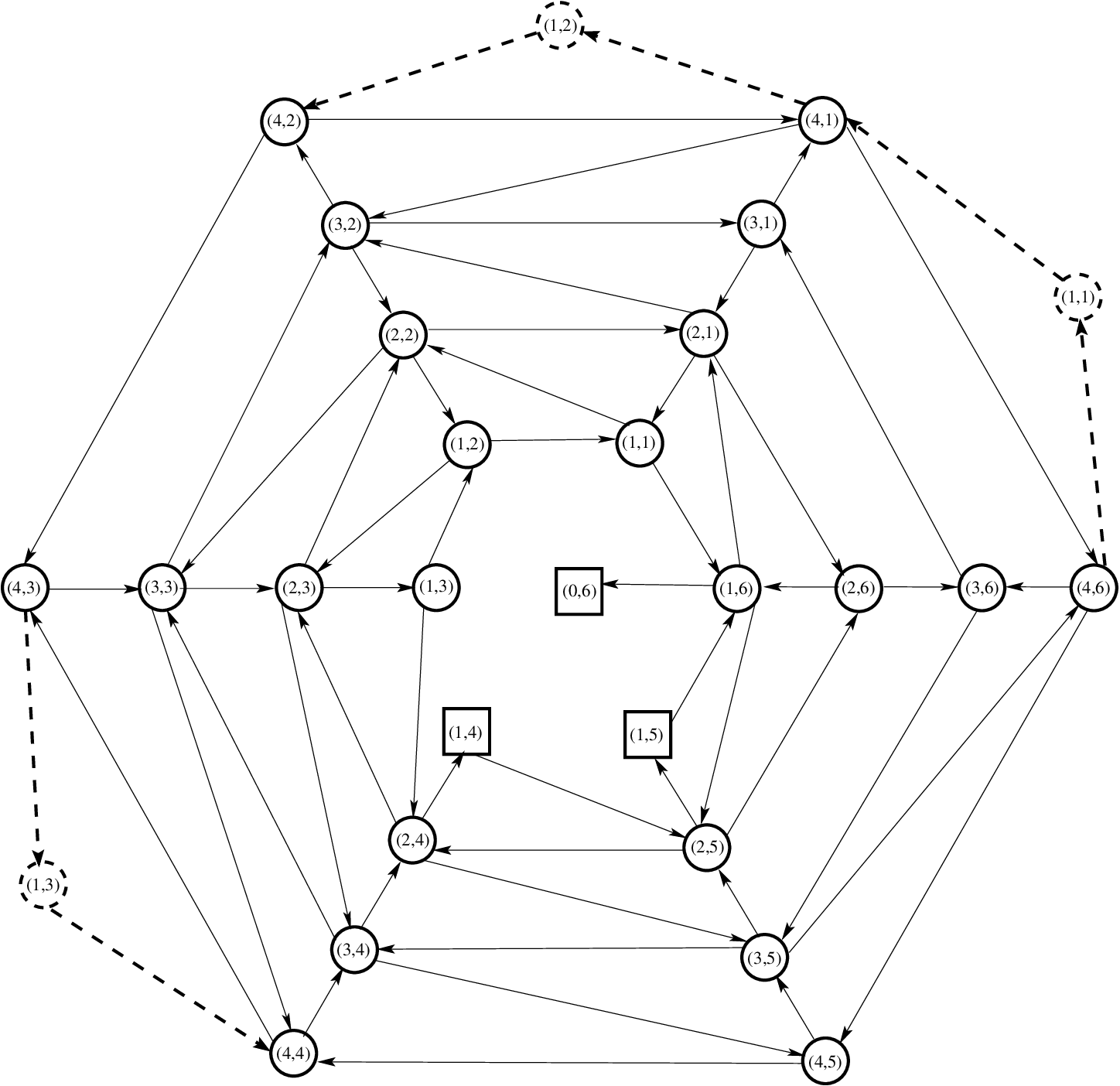}
\caption{Quiver $Q^{(2)}_{\romon}(5)$ before the mutation at $(4,2)$ in $\T^3_{\romon}$}
\label{fig:ann2}
\end{center}
\end{figure}

We need to establish that functions $f^{(l)}_{\romon} (p,q)$ satisfy identities dictated by the order in which transformations are applied and by property (i) in the above description of the current state of the quiver. Thus,  we must have
\begin{equation*}
f^{(l-1)}_{\romon} (p,q) f^{(l)}_{\romon} (p,q) = f^{(l-1)}_{\romon} (p,q-1) f^{(l)}_{\romon} (p,q+1) + f^{(l-1)}_{\romon} (p-1,q) f^{(l)}_{\romon} (p+1,q)
\end{equation*}
 for $l+1\leq p \leq n-2$, $q\leq n$,
 \begin{equation*}
 \begin{split}
f^{(l-1)}_{\romon} (p,n+1) f^{(l)}_{\romon} (p,n+1) = f^{(l-1)}_{\romon} (p,n)& f^{(l-1)}_{\romon} (p,1)\\ 
+ &f^{(l-1)}_{\romon} (p-1,n+1) f^{(l)}_{\romon} (p+1,n+1)
\end{split} 
\end{equation*}
 for $q=n+1$, $l\leq p \leq n-2$,
\begin{equation*}
\begin{split}
f^{(l-1)}_{\romon} (n-1,q) f^{(l)}_{\romon} (n-1,q) = f^{(l-1)}_{\romon} (n-1,q-1) &f^{(l)}_{\romon} (n-1,q+1)\\ 
+ &f^{(l-1)}_{\romon} (n-2,q) \Delta (q) 
\end{split}
\end{equation*}
for $p=n-1$, $q < n+1$, where
$$
\Delta (q)=\begin{cases}  
f_{\romon}^{(0)}(1, n+ q-l-2) &\quad  \mbox{if}\ q\leq l,\\
    1, &\quad  \mbox{if}\ q= l+1, l+2,\\
  f_{\romon}^{(0)}(1, q-l-2), &\quad  \mbox{if}\ q\geq l+3,
\end{cases}\ 
$$
and
\begin{equation*}
\begin{split}
f^{(l-1)}_{\romon} (n-1,n+1) f^{(l)}_{\romon} (n-1,n+1) = &f^{(l-1)}_{\romon} (n-1,n) f^{(l-1)}_{\romon} (n-1,1)\\ + &f^{(l-1)}_{\romon} (n-2,n+1) f_{\romon}^{(0)}(1, n-l-1).
\end{split}
\end{equation*}

We observe that all these relations can be written in the form
\begin{align}
\nonumber
&{\small \bangle {[\kappa] } {[\mu] \cup [\nu+1, n]}\ \bangle {[\kappa-1] } {[\mu-1] \cup [\nu, n]}= }\\ 
\label{magicformula}
& {\small \qquad
\bangle {[\kappa-1] } {[\mu] \cup [\nu+1, n]}\ \bangle {[\kappa] } {[\mu-1] \cup [\nu, n]} 
+ \bangle {[\kappa-1] } {[\mu-1] \cup [\nu+1, n]}\ \bangle {[\kappa] } {[\mu] \cup [\nu, n]},}
\end{align}
where
\begin{equation}\label{albega}
\begin{aligned}
&\mu=p-l+1,\quad \nu=n-l+1,\quad \kappa=p+q-l \qquad\text{for $q\geq l+1$},\\
&\mu=p-l,\quad \nu=n-l, \quad \kappa=p+q+n-l \qquad\text{for $q\leq l$}.
\end{aligned}
\end{equation}
Here one has to use an alternative representation for $f_{\romon}^{(l-1)} (p,l)$ obtained via $\tau_1$:
$$
f_{\romon}^{(l-1)} (p,l) =  \bangle {[p] } {[p-l+1]\cup [n-l+2,n]} =   \bangle {[n+p] } {[p-l]\cup [n-l+1,n]}, 
$$
and for $f_{\romon}^{(0)}(1,q)$ obtained via $\tau_3$:
$$
f_{\romon}^{(0)}(1,q) =  \bangle {[q] } {1} = \bangle {[n+1 + q] } {[n]},
$$
together with a convention
$$
\bangle {[n]} {[n]} = \bangle {[n+1]} {[n]} = 1;
$$
the latter is used when $(p,q)$ is trivalent.

Consider the matrix $C$ obtained from the augmentation $\bar U$ of $U(X,X)$ by deleting columns indexed by $[\kappa-1]$ and rows indexed by $[\mu-1] \cup [\nu+1, n]$. In other words, $C$ is the matrix used to define
 $\bangle {[\kappa-1] } {[\mu-1] \cup [\nu+1, n]}$. 
Then $C$ is in the staircase form since, for $p>1$, $C$ does not involve 
the first row of $\bar U$  and, for $(p,q)=(1,n+1)$,  $C$ does not involve 
the first $n$ columns of $\bar U$. Let $t$ be the smallest index such that $c_{ii}=0$ for $i>t$. Then we define $A$ to be a leading principal $t\times t$ submatrix of $C$ and observe that \eqref{magicformula} reduces to \eqref{coredodgsonformula} with $\beta=\nu - \mu +1$. The condition
$a_{\beta 1} a_{1 \beta} = \bar u_{\nu \kappa}\bar u_{\mu ,\kappa +\nu - \mu}\ne 0$ is satisfied since $1 < \nu \leq n$ and $\kappa +\nu - \mu$ lies in the range $ [n+1, 2n ]$, by~\eqref{albega}. 
\end{proof}

\begin{remark}
\label{first_last}
{\rm (i) Note that
$f^{(l)}_{\romon}(n-1,l+2) = \bangle {[n]}{[n-l-1]\cup [n-l+1,n]} = x_{n-l,1}$ and 
$f^{(l)}_{\romon}(n-1,l+1) = 
\bangle {[n-1]}{[n-l-1]\cup [n-l+1,n]} = x_{n-l-1,n}$, 
and thus Lemma~\ref{4path} implies that all matrix entries of the first and the last columns of $X$ are cluster variables.

(ii) Consider the evolution of the edge between the vertices $(n-1,n-1)$ and $(n-2,n-1)$ in the quivers $Q_{\romon}^{(l)}$. In the initial quiver (for $l=0$) this edge points from $(n-2,n-1)$ to $(n-1,n-1)$. It follows from condition (ii) in the proof of Lemma~\ref{4path} that each sequence $\T_{\romon}^l$ for $l\in [n-3]$ affects this edge twice by mutations at its endpoints. Therefore, the direction of the edge in $Q_{\romon}^{(n-3)}$ is the same as in the initial quiver. The sequence
$\T_{\romon}^{n-2}$ affects this edge only once, by the mutation at $(n-1,n-1)$, and the sequence $\T_{\romon}^{(n-1)}$ does not affect it at all. Hence, in the resulting quiver $Q_{\romon}^{(n-1)}$ the edge points from $(n-1,n-1)$ to
$(n-2,n-1)$.}
\end{remark}
 
Denote by $Q_{\romtw}(n)$ the
quiver obtained from $\T^{(n-1)}_{\romon}(Q_{\romon}(n))$ by shifting the vertices of the ray $n+1$ 
outwards: $(p,n+1)\mapsto (p+1,n+1)$, $p\in [0,n-1]$.

\begin{proposition}
\label{end4path} 
{\rm (i)} $Q_{\romtw}(n)$ is decomposed into the web and the weave. The web is subdivided into 
consistently oriented triangles,
forming trapezoids. All trapezoids $(p,q)$ are SE for $q\ne1$ and SW for $q=1$. Trapezoids $(2,1)$ 
and $(2,n-1)$ are incomplete: they lack the upper base.  Trapezoid $(n-1,n)$ is missing.
Besides, there is an additional triangle with vertices $(n,n+1)$, $(n-1,n+1)$, $(n-1,1)$.
The weave consists of a directed path
$(n-1,1)\to (1,1)\to (n-1,2) \to (1,2)\to\cdots \to(1,n-2) \to (n-1,n-1)$. 

{\rm (ii)} The cluster variables  $f_{\romtw}(p,q)$ attached to the vertices of $Q_{\romtw}(n)$ are
\begin{equation*}
 f_{\romtw}(p,q)= \begin{cases}
 \bangle {[n]}{[n-p+2,n]},  & \quad \text{if $q=n+1$}, \\
  \bangle {[n]}{1\cup [n-p+2,n]},   & \quad \text{if  $q=n$}, \\
  \bangle{[n-1]} {1\cup [n-p+2,n]},   &\quad  \text{if  $q=n-1$}, \\
   \bangle {[n+q]} {[n-p+1,n]},   &\quad  \text{if $q \leq  n-2$, $p \geq q$}, \\
 \bangle {[n+q+1]}{[n] \cup[2n-p+1,2n-1]},
 &\quad  \text{if  $q \leq  n-2$, $p < q$}.
    \end{cases}
 \end{equation*}
\end{proposition}

\begin{proof}
Follows immediately from Lemma~\ref{4path} and the description of the evolution of $Q_{\romon}(n)$.

Note that, apart from the stable variables, the cluster $\{f_{\romtw}(p,q)\}$ contains four  more variables from the cluster $\{f^{w_0}_{ij}\}$:
$f_{\romtw}(n,n+1)=x_{11} = \thetta^{w_0}_1=f_{nn}^{w_0}$,  $f_{\romtw}(n-1,n)=x_{21} =\phhi^{w_0}_1= 
f_{n-1,n}^{w_0}$,  
$f_{\romtw}(n-1,n-1)=x_{1n} = \pssi^{w_0}_1=f_{n1}^{w_0}$
and $f_{\romtw}(n-2,n-1)=\det \left (\begin{array}{cc} x_{1n}  &  x_{21}\\ x_{2n}  &  x_{31}
 \end{array}\right ) = \pssi^{w_0}_2=f_{n-2,n}^{w_0}$. The corresponding vertices are frozen in $Q_{\romtw}(n)$.
\end{proof}

The quiver $Q_{\romtw}(5)$ is shown on Fig.~\ref{fig:ann3}. 
The additional triangle is shown by dotted lines. The rest is as on Fig.~\ref{fig:ann}.

 \begin{figure}[ht]
\begin{center}
\includegraphics[height=8cm]{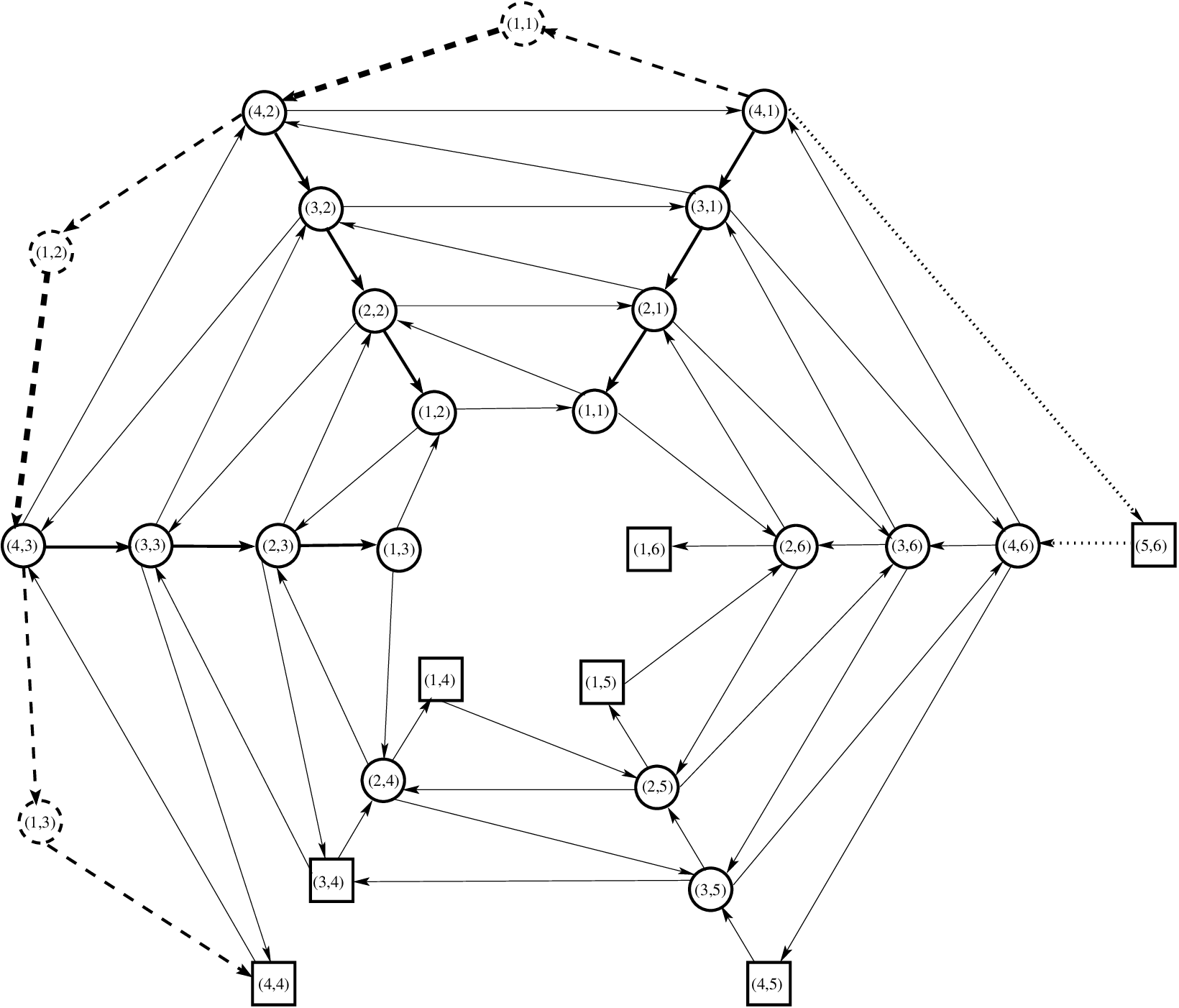}
\caption{Quiver $Q_{\romtw}(5)$}
\label{fig:ann3}
\end{center}
\end{figure}

The second stage consists in applying a  sequence $\T^1_{\romtw}$ of $(n-1)(n-2)$  mutations 
along the directed path that starts at $(n-1,1)$, ends at $(1,n-2)$ and consists of edges $(p,q)\to (p-1,q)$, $q\in 
[n-2]$, and $(1,q)\to (n-1,q+1)$, 
 $q\in [n-3]$; this path is shown by thick solid and dashed lines on Fig.~\ref{fig:ann3}.
Then, in the resulting quiver, we apply a sequence $\T^2_{\romtw}$ of $(n-1)(n-3)$ mutations  
along a similar path
 that starts at $(n-1,1)$ and  ends at $(1,n-3)$, etc., ending with a sequence $\T^{n-2}_{\romtw}$ of $n-1$ mutations  along the path  from $(n-1,1)$ to $(1,1)$.
 As before, denote $\T^{(l)}_{\romtw} = \T^l_{\romtw}\circ\cdots\circ \T^1$,
 $Q^{(l)}_{\romtw}(n)=\T^{(l)}_{\romtw}(Q_{\romtw}(n))$ for $l\in [n-2]$, $Q^{(0)}_{\romtw}(n)=Q_{\romtw}(n)$, $f^{(0)}_{\romtw}(p,q)=f_{\romtw}(p,q)$.

\begin{lemma}
 \label{cycles}
 For any $l\in [n-2]$,
   \begin{equation*}
f^{(l)}_{\romtw}(p,q)= \begin{cases}
 f^{(l-1)}_{\romtw}(p,q), & \quad \text{if   $q > n-l-1$},\\
 \bangle {[n]\cup[n+l+1,n+q+l]}{[n-p+l+1,n]},   &\quad  \text{if  $p > q+ l$, $q \leq n-l-1$},\\ 
 \bangle {[n+1]\cup[n+l+2,n+q+l+1]}{[n]\cup[2n-p+l+1,2n-1]},   &\quad  \text{if $l<p \leq q+ l$,  
 $q \leq n-l-1$},\\
 \bangle {[n]\cup[n+l+2,n+q+l+1]} {[l-p+3,n]},  &\quad  \text{if $p \leq l$, $q \leq n-l-1$}. \end{cases}\ 
 \end{equation*}
\end{lemma}

\begin{proof}
To describe the evolution of the quiver $Q_{\romtw}(n)$ under transformations $\T^l_{\romtw}$, 
it will be convenient to refine the description
of the web and to decompose it into three parts: 
the {\it truncated web\/}, the {\it lacing\/} and the {\it knob\/}.
The knob is the vertex $(n,n+1)$ and the edges incident to it, the lacing is formed by the edges between the vertices of the first and the $(n+1)$th rays (including the edge $(n-1,n+1)\to (1,1)$, if it exists), 
and the rest is the truncated web. The weave, as before, consists of the edges
 between the inner and the outer circles, except for the edge $(n-1,n+1)\to (1,1)$. 
Therefore, before the beginning of $\T^1_{\romtw}$ the knob
contains edges $(n,n+1)\to (n-1,n+1)$ and $(n-1,1)\to (n,n+1)$, while the lacing is the path
$(1,1)\to (2,n+1)\to (2,1)\to (3,n+1)\to\cdots\to (n-1,1)$.
 
One can check by induction that prior to applying a mutation in the sequence $\T^l_{\romtw}$ to a vertex $(p,q)$, the current state of the quiver $Q^{(l-1)}_{\romtw}(n)$ can be described as follows. 

(i) Vertex $(n-1,q)$ is five-valent if $l=1$, $(1,n-2)$ is trivalent if $l=1$, $(p,q)$ is four-valent otherwise.
There are two edges pointing from $(p,q)$: $(p,q) \to (p\pm 1,q)$, where $p$ is understood $\bmod\ 
n-1$ with values in $[n-1]$. Besides, if $l=1$ and $p=n-1$ there is an additional edge $(n-1,q) \to (n,n+1)$.
The edges pointing towards $(p,q)$ are

i.1) $(p, q\pm 1) \to (p,q)$ if $ q\ne1$  
($q$ is understood $\bmod\ n+1$ with values in $[n+1]$, as before);

i.2) $(1, n- 3) \to (1,n-2)$  if $l=1$;

i.3) $(p, 2) \to (p,1)$ and $(\{p-l+1\}, n+1) \to (p,1)$, where $\{p\}$ denotes $p\bmod n-2$ with values in $[2,n-1]$.

(ii) The truncated web is subdivided into  oriented triangles and trapezoids. 
Since no mutations are performed at the vertices $(p,n-1)$, $(p,n)$, and $(p,n+1)$ for $p\in [n-1]$,
the corresponding trapezoids $(p,n)$ and $(p,n+1)$ do not change. In what follows we only
describe the remaining trapezoids.
There are at most two oriented trapezoids: $(p+1,q)$ and $(p+1,q+1)$. The first of them does not exist
if $q=1$, and none of them exist if $p=n-1$. 
The trapezoids $(p',q')$ are SE if $q'\in [2,q-1]$, or $q'\in [q+2,n-l]$, 
or $q'=q$ and $p'\in [p+2,n-1]$, or $q'=q+1$ and $p'\in [2,p]$.
The rest of trapezoids are SW.  A special case is the trapezoid $(2,n-1)$: it is SE prior to the unique mutation at
$(1,n-2)$; after this mutation it has edges $(2,n-2)\to (1,n-2),\ (2,n-1)\to (1,n-1)$ and $(2,n-2)\to (2,n-1)$ and is not affected by any subsequent transformations in $\T^{(l)}_{\romtw}$.

(iii) The lacing is only affected by mutations at the vertices of the first ray. Prior to the mutation 
at $(p,1)$ in $\T^l_{\romtw}$ it is the path
$(p+1,1)\to (\{p-l+2\},n+1)\to (p+2,1)\to (\{p-l+3\},n+1)\to\cdots
\to (\{p-l+1\},n+1)\to (p,1)$.
The length of the path is $2(n-2)$.

(iv) The knob is only affected by mutations at the vertices of the outer circle in $\T^1_{\romtw}$. 
Prior to the mutation at $(n-1,q)$ the knob contains
 edges $(n,n+1)\to (n-1,q-1)$ and $(n-1,q)\to (n,n+1)$.
After the mutation at $(n-1,n-2)$ the only remaining edge of the knob is  $(n,n+1)\to (n-1,n-2)$.
It will be convenient to move the vertex of the knob
to the position $(n,n-2)$ (this does not affect any computations
since there will be no further no mutations at $(n-1,n-2)$ on the second stage).

(v) The weave is only affected by mutations along the inner and the outer circles. Prior to the mutation
at $(n-1,q)$ in $\T^l_{\romtw}$ it consists of the paths
$(n-1,n-2)\to (1,n-2) \to (n-1,n-3) \to (1,n-3)\to \cdots \to (n-1,n-l) \to (1,n-l)$,
$(n-1,1)\to(1,1)\to (n-1,2) \to (1,2)\to \cdots \to (n-1,q-1) \to (1,q-1)$,
$(n-1,q)\to(1,q)\to (n-1,q+1) \to (1,q+1)\to\cdots \to (1,n-l-1) \to (n-1,n-l)$
and the edge $(1,q)\to (n-1,q-1)$. The first of the paths does not exist if $l=1$.  
Prior to the mutation at $(1,q)$ in $\T^l_{\romtw}$ the first two paths are the same as above, 
the third path is
$(n-1,q+1)\to(1,q+1)\to (n-1,q+2) \to (1,q+2)\to\cdots \to (1,n-l-1) \to (n-1,n-l)$
(it is empty if $q=n-l-1$), and 
the edge is $(1,q)\to (n-1,q)$. 

The quiver $Q_{\romtw}^{(1)}(5)$ before the mutation at $(3,2)$ in the sequence $\T^2_{\romtw}$ is shown on Fig.~\ref{fig:ann4}.
The vertex $(3,2)$ is four-valent, the edges pointing towards $(3,2)$ come from $(4,2)$ and $(2,2)$, according to i.i) above. 
Trapezoids $(4,2)$ and $(4,3)$
are oriented. The types of all other trapezoids are exactly as described in (ii). 
The lacing consists of the path $(1,1)\to(3,6)\to(2,1)\to(4,6)\to(3,1)\to(2,6)\to(4,1)$; it is shown by thick solid lines. The vertex of the knob is already shifted to position $(5,4)$, the edges are as stipulated by
(iv). The first path of the weave is the edge $(4,3)\to(1,3)$, the second path is the edge $(4,1)\to(1,2)$, the 
third path does not exist since $q=n-l-1=2$. The edge of the weave is $(1,2)\to(4,2)$.

 \begin{figure}[ht]
\begin{center}
\includegraphics[height=8cm]{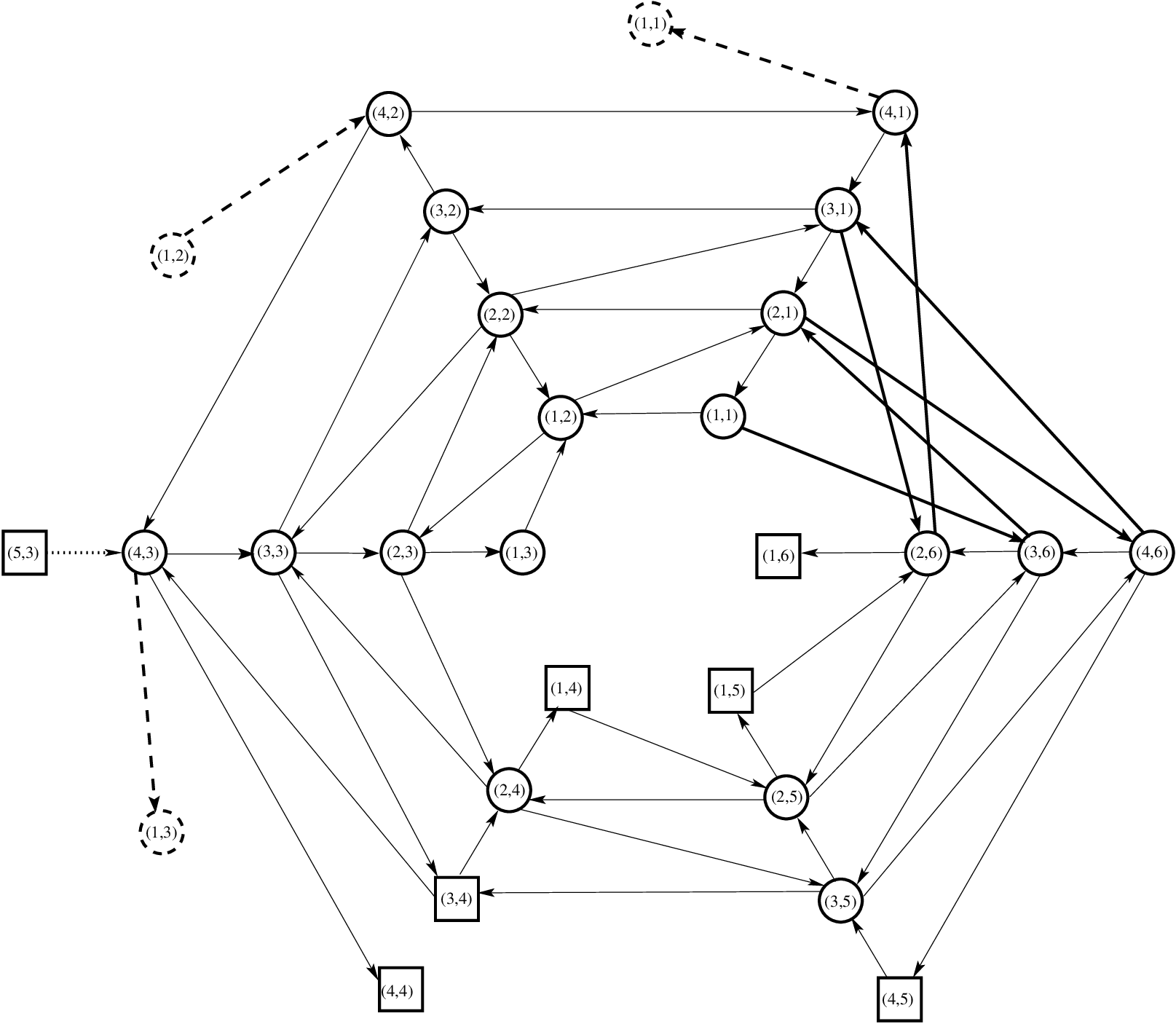}
\caption{Quiver $Q^{(1)}_{\romtw}(5)$ before the mutation at $(3,2)$ in $\T^2_{\romtw}$}
\label{fig:ann4}
\end{center}
\end{figure}

Proceeding as in the proof of Lemma~\ref{4path}, we can write identities for $f^{(l)}_{\romtw} (p,q)$
dictated by the order in which transformations are applied and by property (i) in the above description of the current state of the quiver. 
In the case $l<p\leq q+l$, we observe that all these identities are of the form
\begin{equation}\label{magicformula_12}
\begin{aligned}
  \bangle  {[n+1]\cup [\mu, \nu-1]}{[n]\cup[\kappa,2n-1] }&\bangle  {[n+1]\cup [\mu+1, \nu]}{[n]\cup[\kappa+1,2n-1] }\\ &= 
\ \bangle  {[n+1]\cup [\mu+1, \nu]}{[n]\cup[\kappa,2n-1] }\ \bangle  {[n+1]\cup [\mu, \nu-1]}{[n]\cup[\kappa+1,2n-1] }\\
&+ \bangle  {[n+1]\cup [\mu, \nu]}{[n]\cup[\kappa,2n-1] } \bangle  {[n+1]\cup [\mu+1, \nu-1]}{[n]\cup[\kappa+1,2n-1] }
\end{aligned}
\end{equation}
with
\begin{equation*}
\mu=n+l+1,\quad \nu=n+q+l+1,\quad \kappa=l-p+2n.
\end{equation*}
For $p< q+l$, this follows directly from the definition of $f^{(l)}_{\romtw} (p,q)$. 
The convention $\bangle {[2n]} {[n]}=1$ is used when $l=1$, $(p,q)=(1,n-2)$, in which case $(p,q)$
is trivalent.
If $p=q+l$, we use the translation $\tau_2$ to write the relation
\begin{align*}
f^{(l)}_{\romtw} (q+l,q)&= \bangle  {[n]\cup [n+l+1, n+q+l]}{[n-q+1,n] }\\
&  = \det X_{[n-q]}^{[1,l]\cup[q+l+1,n]} = 
\bangle  {[n+1]\cup [n+l+2, n+q+l+1]}{[n]\cup[2n-q+1,2n-1] },
\end{align*}
which is valid due to the inequality $n-q\geq l+1$.

For $p < l$ or $p>q+ l$  identities to be established are
 of  the form
 {\small
\begin{align}
\nonumber
& { \bangle  {[n]\cup [\mu, \nu-1]}{[\kappa,n] }\ \bangle  {[n]\cup [\mu+1, \nu]}{[\kappa+1,n] }}\\ 
\label{magicformula_11}
& {\quad=\ \bangle  {[n]\cup [\mu+1, \nu]}{[\kappa,n] }\ 
\bangle  {[n]\cup [\mu, \nu-1]}{[\kappa+1,n] }\ 
+ \bangle  {[n]\cup [\mu, \nu]}{[\kappa,n] }\ \bangle  {[n]\cup [\mu+1, \nu-1]}{[\kappa+1,n] }},
\end{align}}
and for $p=l$ they are
{\small
\begin{align}
\nonumber
& { \bangle  {[n+1]\cup [\mu, \nu-1]}{[n] }\ \bangle  {[n]\cup [\mu+1, \nu]}{[3,n] } }\\ 
\label{magicformula_13}
& { =\ \bangle  {[n+1]\cup [\mu+1, \nu]}{[n] }\ 
\bangle  {[n]\cup [\mu, \nu-1]}{[3,n] }\ 
+ \bangle  {[n+1]\cup [\mu, \nu]}{[n] }\ \bangle  {[n]\cup [\mu+1, \nu-1]}{[3,n] }\ }
\end{align}}
with
\begin{equation*}
\begin{aligned}
&\mu=n+l+1,\quad \nu=n+q+l+1,\quad \kappa=l-p+2 \qquad\text{for $p\leq l$},\\
&\mu=n+l,\quad \nu=n+q+l, \quad \kappa=l-p+n \qquad\text{for $p>q+ l$}.
\end{aligned}
\end{equation*}

Excluded from \eqref{magicformula_11} is the case $l=1$, $p=n-1$, when a cluster transformation is applied at a five-valent vertex of the quiver.
In this case, he identities that need to be established have a form
\begin{align*}
f^{(0)}_{\romtw} (n-1,1) f^{(1)}_{\romtw} (n-1,1) = f^{(0)}_{\romtw} (n-2,1) &f^{(0)}_{\romtw}(n,n+1) f^{(0)}_{\romtw} (1,1)\\ &+ f^{(0)}_{\romtw} (n-1,2) f_{\romtw}^{(1)}(n-2, n+1) 
\end{align*}
and
\begin{align*}
f^{(0)}_{\romtw} (n-1,q) f^{(1)}_{\romtw} (n-1,q) = f^{(0)}_{\romtw} (n-2,q) &f^{(0)}_{\romtw}(n,n+1) f^{(0)}_{\romtw} (1,q)\\ &+ f^{(0)}_{\romtw} (n-1,q+1) f_{\romtw}^{(1)}(n-1, q-1) 
\end{align*}
for $q \in [2, n-2]$.
Both can be re-written as 
{\small
\begin{align}
\nonumber
&  \bangle  {[n+q]}{[2,n] }\ \bangle  {[n]\cup [n+2,n+q+1]}{[3,n] }\ = \\ 
\label{magicformula_14}
&  \qquad {\bangle {[n] \cup [n+2, n+q+1]} {[2,n]}}^{(2)}
\bangle  {[n+q]}{[3,n] }\ 
+ \bangle  {[n+q+1]}{[2,n] }\ \bangle  {[n]\cup [n+2, n+q]}{[3,n] },
\end{align}}
where we used the fact that for any $q\in [n-2]$, 
$$
f^{(0)}_{\romtw}(n,n+1) f^{(0)}_{\romtw} (1,q)= x_{11} \bangle {[n+q+1]} {[n]} = {\bangle {[n] \cup [n+2, n+q+1]} {[2,n]}}^{(2)}.
$$ 

To establish \eqref{magicformula_12} - \eqref{magicformula_14}, consider the matrix $C$ obtained from the augmentation $\bar U$ of $U(X,X)$ by deleting columns and rows indicated in the second factor of the second term on the right hand side of the corresponding relation. For rows and columns of $C$ we retain the same indices they had as rows and columns of $\bar U$.
 $C$ has a staircase shape. 
 Let $t$ be the smallest index such that diagonal entries of $C$ are zero in rows with the index larger than $t$. Then we define $B$ to be the leading principal $t\times (t+1)$ submatrix of $C$. Consider $\mu$, $\nu$ and $\kappa$ defined by formulas above for $(p,q)$ in an appropriate range, and apply \eqref{notjacobi} to $B$ with  $\alpha= \mu$, $\beta=\nu$, $\gamma = t+1$, $\delta= \kappa$. Then we  arrive at \eqref{magicformula_12}-\eqref{magicformula_14} after canceling common factors.
\end{proof}
 
The quiver obtained upon the completion of $\T^{(n-2)}_{\romtw}$ 
and shifting all vertices counterclockwise via $(p,q)\mapsto (p,q+1)$ will be denoted $Q_{\romth}(n)$. 
In $Q_{\romth}(n)$ the lacing and the knob are 
again incorporated into the web.

 \begin{proposition}
\label{endcycles} 
{\rm(i)} $Q_{\romth}(n)$ is decomposed into the web and the weave. 
The web is subdivided into consistently oriented triangles,
forming trapezoids. All trapezoids $(p,q)$ are SW for $q\in [2,n]$, except for $(2,n)$, and SE
for $q=1,n+1$, except for $(n-1,n+1)$.  
The trapezoids $(2,n)$ 
and $(n-1,n+1)$ are special; the first one contains edges
 $(2,n-1)\to (1,n-1)$, $(2,n)\to (1,n)$ and $(2,n-1)\to (2,n)$,
while the second one contains edges $(n-1,n+1)\to(n-2,n+1)$ and $(n-2,n+1)\to (n-2,n)$. 
Finally, $Q_{\romth}(n)$ contains the edge $(n,n-1)\to(n-1,n-1)$.
 
 The weave consists of a directed path
$(n-1,n-1) \to (1,n-1)\to (n-1,n-2)\to (1,n-2)\to\cdots (n-1,2) \to (1,2)$.  

{\rm(ii)} The cluster variables  $f_{\romth}(p,q)$ attached to vertices of $Q_{\romth}(n)$ are
\begin{equation*}
 f_{\romth}(p,q)= \begin{cases}
  \bangle{[n]}{1 \cup[n-p+2,n]},   & \quad \text{if  $q=n+1$}, \\
   \bangle{[n-1]}{1\cup[n-p+2,n]},   &\quad  \text{if $q=n$}, \\
 \bangle{[n+1]\cup[2n-q+2,2n]}{[n]\cup[3n-p-q+1,2n-1]},   &\quad  \text{if $q<n$, $p+q > n$}, \\
 \bangle{[n] \cup[2n-q+2,2n]}{[n+3-p-q,n]},   &\quad  \text{if  $q<n$, $p+q \leq n$}.
    \end{cases}
 \end{equation*}
\end{proposition}

\begin{proof}
Follows immediately from Lemma~\ref{cycles}. Note that for $l\in[n-2]$, $f^{(l)}_{\romtw}(1,n-l-1)= \bangle {[n]\cup[n+l+2,2n]}{[l+2,n]} = X_{[l+1]}^{[l+1]}=\thetta^{w_0}_{l+1}=f_{n-l,n-l}^{w_0}$, and
hence the corresponding vertices are frozen in $Q_{\romth}(n)$.
\end{proof}

The quiver $Q_{\romth}(5)$ is shown on Fig.~\ref{fig:ann5}.
 
 \begin{figure}[ht]
\begin{center}
\includegraphics[height=8cm]{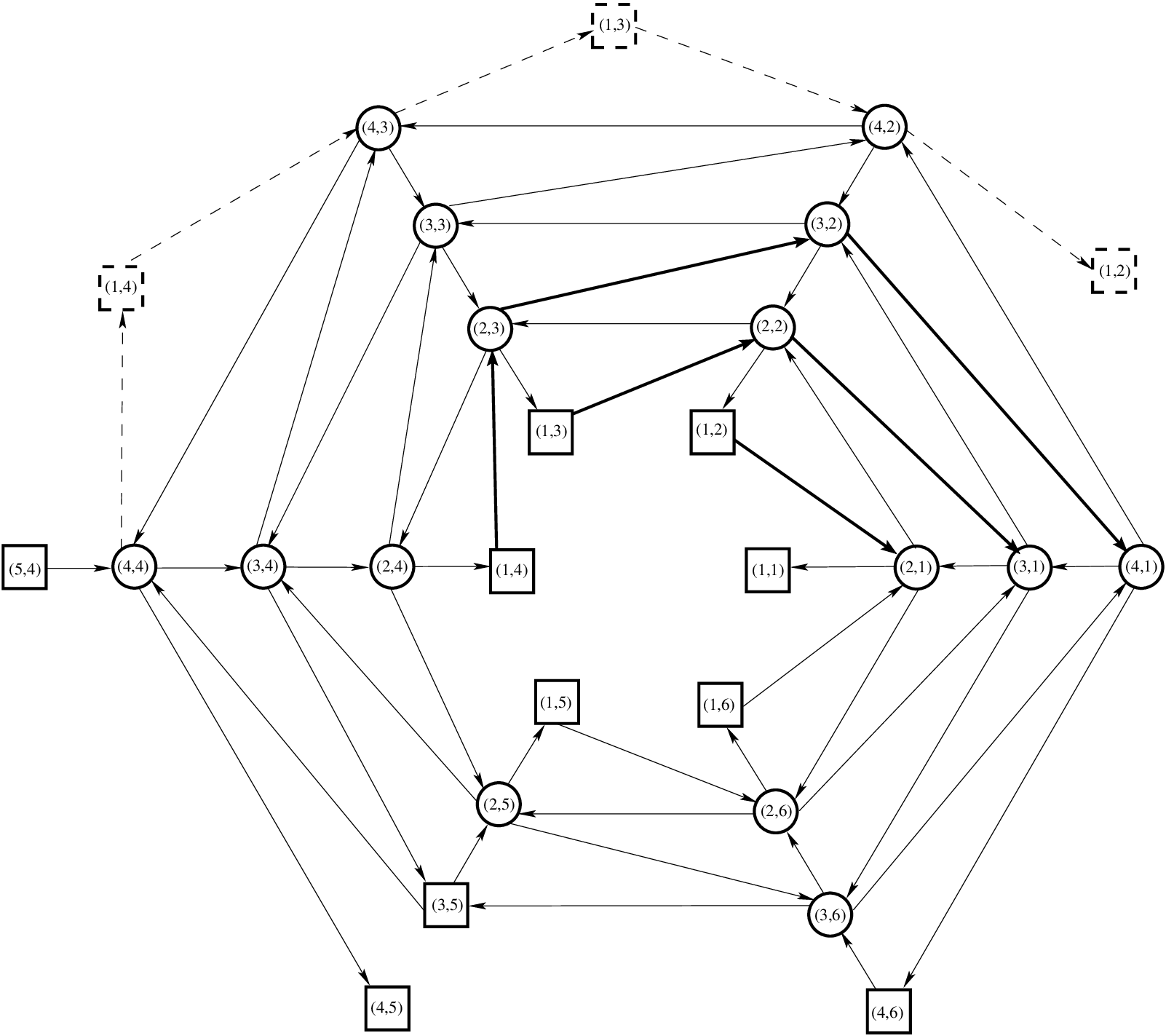}
\caption{Quiver $Q_{\romth}(5)$ }
\label{fig:ann5}
\end{center}
\end{figure}
 
 The third stage involves the sequence of mutations  along the path $(1,n-1)\to (2,n-2)\to \cdots
 \to (n-2,2)\to (n-1,1)$ in the opposite direction (starting at $(n-1,1)$ and ending at $(2,n-2)$). 
 After the mutation at
 $(i,n-i)$, $i\in[2,n-1]$, the mutated vertex is shifted  to the position $(i-1,n-i+1)$. So, at this 
 moment there are two vertices occupying the same position. 
 After the mutation at $(2,n-2)$, the ``new'' vertex $(1,n-1)$ is frozen,
and  the ``old'' $(1,n-1)$ it is shifted
 to the position $(n,n-2)$. We denote the resulting transformation $\T^{1}_{\romth}$. 
 Next, we define sequences of transformations
 $\T^{l}_{\romth}$ and $ \T^{(l)}_{\romth} = \T^{l}_{\romth}\circ\cdots\circ \T^{1}_{\romth}$, $l\in[n-2]$, inductively: $\T^{l}_{\romth}$ consists 
 in applying mutations along the path $(1,n-l)\to (2,n-l-1)\to \cdots
 \to (n-l-1,2)\to (n-l,1)$ in the opposite direction (starting at $(n-l,1)$ and ending at 
 $(2,n-l-1)$). The paths are shown in Fig.~\ref{fig:ann5} with thick solid lines. 
 After the mutation
 at  $(i,n-i-l+1)$, $i\in [2, n-l]$, the mutated vertex is shifted to the position $(i-1,n-i-l+2)$, so that there are 
 two vertices occupying the same position; after the mutation at $(2,n-l-1)$, the ``new'' $(1,n-l)$ is frozen,  
and the ``old'' $(1,n-l)$ is shifted to the position 
 $(n,n-l-1)$.   As before, denote 
 $Q^{(l)}_{\romth}(n)=\T^{(l)}_{\romth}(Q_{\romtw}(n))$ for $l\in [n-2]$, $Q^{(0)}_{\romth}(n)=Q_{\romth}(n)$, $f^{(0)}_{\romth}(p,q)=f_{\romth}(p,q)$.
 
 \begin{lemma}
\label{regularsteps} 
For any $l\in [n-2]$,
\begin{equation*}
f^{(l)}_{\romth}(p,q))= \begin{cases}
    \bangle{[n] \cup[2n-q+2,2n]}{1\cup [l+3,n]},   &\quad  \text{if $p+q = n - l +1$}, \\
    X^{[l]}_{[l]}, &  \quad\text{if $p = n$, $q= n-l$}, \\
    f^{(l-1)}_{\romth}(p,q)) & \quad \text{otherwise}.
 \end{cases}
 \end{equation*}
\end{lemma}

\begin{proof}
To describe the evolution of the quiver $Q_{\romth}(n)$ under $\T^{(l)}_{\romth}$, it will be convenient once again to decompose the web into 
the lacing and the truncated web: the lacing consists of the edges connecting vertices $(p,2)$ with vertices $(p,n+1)$, 
and the rest is the truncated web. 

One can check by induction that prior to applying a mutation in the sequence $\T^l_{\romth}$ to a vertex $(p,q)$, the current state of the quiver $Q^{(l-1)}_{\romth}(n)$ can be described as follows. 

(i) Vertex $(p,q)$ is five-valent if $l=1$ and four-valent otherwise.
There are two edges pointing from $(p,q)$: $(p,q) \to (p,q+1)$ and $(p,q)\to (p-1,q)$.
Besides, if $l=1$ there is an additional edge $(p,q) \to (n-1,n+1)$.
The edges pointing towards $(p,q)$ are

i.1) $(p-1,q+1)\to (p,q)$;

i.2) the {\it special\/} edge from the ``new'' copy of $(p,q)$ if $q\ne 1$; 

i.3) $(p-1,n+1)\to (p,1)$.

(ii) The truncated web is subdivided into consistently oriented triangles forming 
trapezoids. Note that triangles with vertices at $(p,q)$, $(p,q+1)$, $(p+1,q)$
and at $(p,q)$, $(p-1,q)$, $(p,q-1)$ are special: they correspond to oriented quadruples involving
the special edge.
Since no mutations are performed at the vertices $(p,n-1)$, $(p,n)$, and $(p,n+1)$ for $p\in [n-1]$,
the corresponding trapezoids $(p,n)$ and $(p,n+1)$ do not change. In what follows we only
describe the remaining trapezoids.
All trapezoids $(p,q)$, $q\in [2,n-1]$ are SW if $p\in [2,n]$, $q\in [n-l,n-1]$, or $p\in [2,n-1]$, 
$q\in [3,n-l-1]$, or $p\in [2,n-l-1]$, $q=2$, trapezoids $(p,1)$ are SE
for $p\in [2,n-l-1]$.
The trapezoid 
$(n-l-1,1)$  is incomplete: 
it lacks the lower base. 
  
 (iii) The lacing is only affected by the mutation at $(n-l,1)$. 
After the mutation and the subsequent shift it consists of
 the directed path
$(n-l-1,2)\to (n-l-1,n+1)\to (n-l,2) \to (n-l,n+1)\to\ldots \to (n-1,2)$, where $l\in [n-3]$.   
  
(iv) The weave is only affected by the shift of the vertex $(1,n-l)$ to the position $(n,n-l-1)$. 
Before the shift it consists of the directed path 
$(n-1,n-l)\to(1,n-l)\to(n-1,n-l-1)\to (1,n-l-1)\to \ldots \to(n-1,2) \to (1,2)$.

(v) After the first mutation in $\T^1_{\romth}$ and the subsequent shift there arises an {\it auxiliary\/} edge between $(p,q)$ and $(n-1,n+1)$. It consists, in fact, of two edges joining $(n-1,n+1)$
with both vertices occupying the same position; these two edges together with the special edge form
an oriented triangle. The auxiliary edge disappears at the end of  $\T^1_{\romth}$.

The quiver $Q^2_{\romth}(5)$ prior to the mutation at $(2,2)$ in
$\T^2_{\romth}$ is shown in Fig.~\ref{fig:ann6}. The lacing is shown with thick solid
lines, the weave with dashed lines, the special edge with an arc. There is no auxiliary edge, since $l>1$.
 
 \begin{figure}[ht]
\begin{center}
\includegraphics[height=8cm]{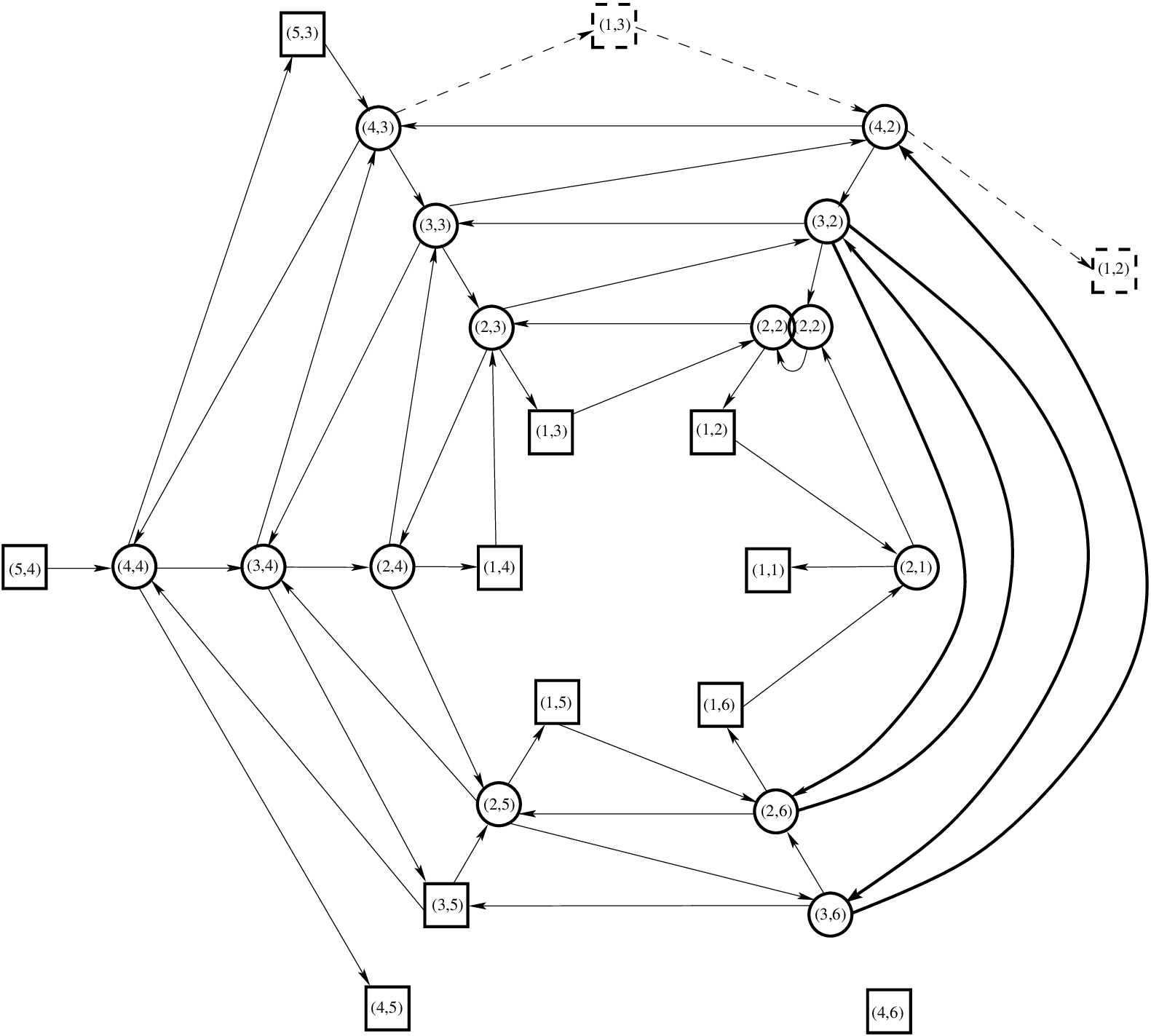}
\caption{Quiver $Q^2_{\romth}(5)$ prior to the mutation at $(2,2)$ in
$\T^2_{\romth}$ }
\label{fig:ann6}
\end{center}
\end{figure}

It follows from the description above that for $l=1$, 
the identities that need to be established are
\begin{align*}
&f^{(0)}_{\romth} (p,1) f^{(1)}_{\romth} (p-1,2) \\
&\qquad= f^{(0)}_{\romth} (p-1,1) f^{(0)}_{\romth}(p,n+1) f^{(0)}_{\romth} (p,2) + f^{(0)}_{\romth} (p-1,2) f_{\romth}^{(0)}(p-1, n+1)\ 
\end{align*}
with $p=n-1$ and
\begin{align*}
&f^{(0)}_{\romth} (p,q) f^{(1)}_{\romth} (p-1,q+1)\\
&\qquad = f^{(0)}_{\romth} (p-1,q) f^{(0)}_{\romth}(n-1,n+1) f^{(0)}_{\romth} (p,q+1) + f^{(0)}_{\romth} (p-1,q+1) f_{\romth}^{(1)}(p, q)\ 
\end{align*}
with $p=n-q$ for $q\in [2,n-2]$.

We note that 
 \begin{align*}
 & f^{(0)}_{\romth}(n-1,n+1) f^{(0)}_{\romth}(n-q, q+1) = \bangle {[n]} {1\cup  [3,n]}  
 \bangle {[n+1] \cup [2n-q+1, 2n]} {[n]}\\
 &\qquad = 
 x_{21} \bangle {[n+1] \cup [2n-q+1, 2n]} {[n]}={\bangle {[n]\cup  [2n - q + 1, 2n]} {1\cup [3, n]} }^{(2)}
 \end{align*}
and conclude that both identities above are of the form
{\small
\begin{align}
\nonumber
&  \bangle  {[n]\cup [\mu +1 , 2 n]}{[\nu,n] }\ \bangle  {[n]\cup [\mu , 2n]}{1\cup [\nu+1,n] }\ 
= \\ 
\label{magicformula_2}
&  \qquad
\bangle  {[n]\cup [\mu + 1, 2 n]}{[\nu + 1,n] }\ {\bangle {[n]\cup  [\mu, 2n]} {1\cup [\nu, n]} }^{(2)}
+ \bangle  {[n]\cup [\mu, 2 n]}{[\nu ,n] }\ \bangle  {[n]\cup [\mu +1 , 2n]}{1\cup [\nu + 1,n] },
\end{align}
}
where $\mu = 2n - q + 1$, $\nu = 3$.

For $l\in [2,n-2]$, 
the identities that need to be established are
\begin{align*}
&f^{(l-1)}_{\romth} (p,1) f^{(l)}_{\romth} (p-1,2) \\
&\qquad= f^{(l-1)}_{\romth} (p-1,1)  f^{(l-1)}_{\romth} (p,2) + f^{(l-1)}_{\romth} (p-1,2) f_{\romth}^{(l-1)}(p-1, n+1)\ 
\end{align*}
with $p=n-l$ and
\begin{align*}
&f^{(l-1)}_{\romth} (p+1,q) f^{(l)}_{\romth} (p,q+1)\\
&\qquad = f^{(l-1)}_{\romth} (p,q)  f^{(l-1)}_{\romth} (p+1,q+1) + f^{(l-1)}_{\romth} (p,q+1) f_{\romth}^{(l)}(p+1, q)\ 
\end{align*}
with $p=n-l-q$ for $q\in [2,n-2]$.
Both result in
{\small
\begin{align}
\nonumber
&  \bangle  {[n]\cup [\mu +1 , 2 n]}{[\nu,n] }\ \bangle  {[n]\cup [\mu , 2n]}{1\cup [\nu+1,n] }\ = \\ 
\label{magicformula_21}
&  \qquad
\bangle  {[n]\cup [\mu + 1, 2 n]}{[\nu + 1,n] }\ \bangle  {[n]\cup [\mu , 2n]}{1\cup [\nu,n] }
+ \bangle  {[n]\cup [\mu, 2 n]}{[\nu ,n] }\ \bangle  {[n]\cup [\mu +1 , 2n]}{1\cup [\nu + 1,n] }
\end{align}
}
 with $\mu=2n -q +1$, $\nu = l+2$. 
 
To prove~\eqref{magicformula_2}, \eqref{magicformula_21}, we consider the matrix $C$ obtained from $\bar U$  by deleting columns and rows indicated in the first factor of the first term on the right hand side of the corresponding relation. For rows and columns of $C$ we retain the same indices they had as rows and columns of $\bar U$.
Let $t$ be the smallest index such that  entries of the first sub-diagonal of $C$ are zero in rows with the index larger than $t$.  Then we define $A$ to be a leading principal $t\times t$ submatrix of $C$. For $\mu, \nu$ defined as above, apply \eqref{jacobi} to $A$ with  $\alpha= 1$, $\beta=\nu$, $\gamma = \mu$, $\delta= t$. Then we  obtain \eqref{magicformula_2}, \eqref{magicformula_21} after canceling common factors.

Note that  $f_{\romth}^{(n-q)}(1,q)= \bangle{[n]\cup [2n-q+2,2n]}{1\cup [n-q+3,n]}=X_{[2,n-q+2]}^{[n-q+1]} 
=f_{q-1,q}^{w_0}= \phhi_{n-q+1}^{w_0}$ for $q\in [2, n-1]$, which justifies freezing the 
last mutated vertex in each path.
\end{proof}

After $\T^{(n-2)}_{\romth}$ is complete, there are two isolated frozen vertices: $(1,1)$ and $(n-1,n+1)$.
 We delete them and note that the only remaining vertex in the first ray is $(1,n)$. Shift
 the vertices of the $(n+1)$st ray via $(p,n+1)\mapsto (p+1,1)$ 
 and denote the resulting quiver $Q_{\romfo}(n)$. Observe that $Q_{\romfo}(n)$ has only $n$ rays: 
 all the vertices
 from the $(n+1)$st ray in $Q_{\romth}(n)$ either moved or has been deleted. The frozen vertices 
 in $Q_{\romfo}(n)$ are $(1,q)$ for $q\in [2,n]$, $(n,q)$ for $q\in [n-1]$, $(2,1)$, $(n-1,n)$ and $(n-2,n)$. 
Lemma \ref{regularsteps} implies
 
 \begin{proposition}
 \label{endregularsteps}
{\rm (i)} The quiver $Q_{\romfo}(n)$ is planar. It is subdivided into consistently oriented triangles forming trapezoids. 
All trapezoids $(p,q)$ are SW if $p\in [2,n]$, $q\in [3,n-1]$ or $p\in [3,n]$, $q=2$, 
 The trapezoid 
$(n,2)$ is incomplete:  
it lacks 
the left side. Trapezoids $(p,1)$, $p\in [3,n-1]$, are slanted: their vertices are $(p,1)$, $(p-1,1)$,
$(p-2,n)$, $(p-1,n)$. They all are SE. 
The trapezoid $(2,n)$ is special; it remains unchanged
under $\T^{(n-2)}_{\romth}$.

{\rm (ii)} The cluster variables  $f_{\romfo}(p,q)$ attached to vertices of $Q_{\romfo}(n)$ are
\begin{equation*}
f_{\romfo}(p,q)= \begin{cases}
   \bangle{[n-1]}{1 \cup[n-p+2,n]},   &\quad  \text{if  $q=n$}, \\
    \bangle{[n] \cup[2n-q+2,2n]}{1\cup [n-p-q+4,n]},   &\quad  \text{if $p+q\in[3,n]$}, \\
   \bangle{[n+1]\cup [2n-q+2,2n]}{[n]\cup[3n-p-q+1,2n-1]},   &\quad  \text{if $p, q<n$, $p+q > n$}, \\
    X_{[n-q]}^{[n-q]},   & \quad \text{if  $p=n$, $q < n$}. 
    \end{cases}
 \end{equation*}
 \end{proposition}

 The quiver $Q_{\romfo}(5))$ is shown in Fig.~\ref{fig:ann7}.
 
 \begin{figure}[ht]
\begin{center}
\includegraphics[height=8cm]{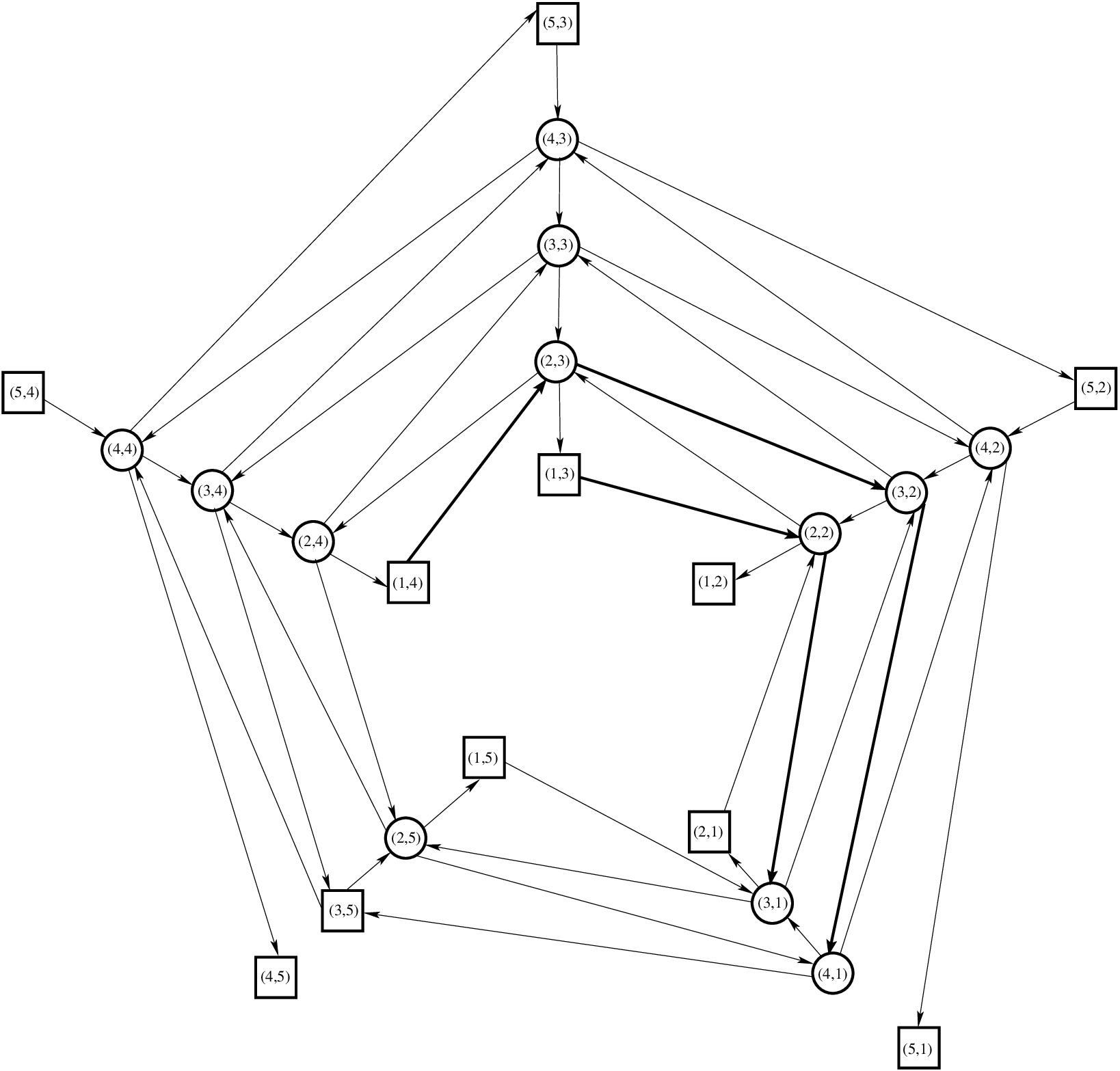}
\caption{Quiver $Q_{\romfo}(5)$ }
\label{fig:ann7}
\end{center}
\end{figure} 

The fourth stage is similar to the previous one both in its definition and analysis. 
It involves the sequence of mutations  along the path $(1,n-1)\to (2,n-2)\to \cdots \to
 (n-1,1)$ in the opposite direction (starting from $(n-1,1)$ and ending at $(2,n-2)$) 
 After the mutation at
 $(i,n-i)$, $i\in[2,n-1]$, the mutated vertex is shifted  to the position $(i-1,n-i+1)$. So, at this 
 moment there are two vertices occupying the same position. 
 After the mutation at $(2,n-2)$, the ``new'' vertex $(1,n-1)$ is frozen,
and  the ``old'' $(1,n-1)$ becomes isolated and is deleted
 to the position $(n,n-2)$.
We denote the resulting transformation $\T^{1}_{\romfo}$. 
 Next, we define sequences of transformations
 $\T^{l}_{\romfo}$ and $ \T^{(l)}_{\romfo} = \T^{l}_{\romfo}\circ\cdots\circ \T^{1}_{\romfo}$, $l\in[2,n-3]$, inductively: $\T^{l}_{\romfo}$ consists 
 in applying mutations along the path $(1,n-l)\to (2,n-l-1)\to \cdots
 \to (n-l-1,2)\to (n-l,1)$ in the opposite direction (starting at $(n-l,1)$ and ending at 
 $(2,n-l-1)$). The paths are shown in Fig.~\ref{fig:ann7} with thick solid lines. 
 After the mutation
 at  $(i,n-i-l+1)$, $i\in [2, n-l]$, the mutated vertex is shifted to the position $(i-1,n-i-l+2)$, so that there are 
 two vertices occupying the same position; after the mutation at $(2,n-l-1)$, the ``new'' $(1,n-l)$ is frozen,  
and the ``old'' $(1,n-l)$ becomes isolated and is deleted. As before, denote 
 $Q^{(l)}_{\romfo}(n)=\T^{(l)}_{\romfo}(Q_{\romth}(n))$ for $l\in [n-3]$, $Q^{(0)}_{\romfo}(n)=Q_{\romfo}(n)$, $f^{(0)}_{\romfo}(p,q)=f_{\romfo}(p,q)$.

The quiver obtained upon the completion of $\T^{(n-3)}_{\romfo}$ is denoted $Q_{\romfi}(n)$. 
Note that all the vertices of the first ray except for $(2,1)$ and $(n,1)$ has been moved to the 
second ray. It will be convenient
to put the vertices of $Q_{\romfi}(n)$ on the $n\times n$ grid in the following way: 
vertex $(p,q)$ remains at
$(p, q)$ for $p\in [n]$ and $q\in [2,n-1]$, or $p=2,n$, $q=1$, or $p=n-2,n-1$ and $q=n$, and
is placed at $(p+1,q+1)$ for $p\in[n-3]$ and $q=n$.   
In view of the similarity with the previous stage, we skip verification of the intermediate steps and 
simply present the description of $Q_{\romfi}(n)$ and cluster variables attached to its vertices.
 
 \begin{proposition}
 \label{sechalfregular}
 {\rm (i)}
 The quiver $Q_{\romfi}(n)$ is planar. 
 It can be decomposed into the {\rm front panel\/}, the {\rm fringe\/} and
the {\rm lacing}. The front panel consists of the vertices $(p,q)$ for 
$p\in [2,n]$ and $q\in [n-1]$ and of the
edges $(p,q)\to (p,q+1)$ for $p\in [2,n-1]$, $q\in [n-2]$, $(p,q)\to (p-1,q)$ for $p\in [2,n]$, $q\in [n-1]$,  except for $p=3,n$ and $q=1$, $(p,q)\to (p+1,q-1)$ for $p,q\in [2,n-1]$.

The fringe consists of the horizontal part containing vertices $(1,q)$ for $q\in [2,n-1]$ and the 
vertical part  
containing vertices $(n-2,n)$ and $(n-1,n)$; all vertices of the fringe are frozen. 
The edges connecting these vertices to the front panel are included in the fringe. They form two paths:
$(2,n-1)\to(1,n-1)\to (2,n-2)\to (1,n-2)\to\dots\to (2,2)\to (1,2)$ and $(n-2,n-1)\to(n-2,n)\to(n-1,n-1)
\to(n-1,n)$.

The lacing consists of the edges between the vertices of the front panel that do not belong to the front panel.
These edges form the path $(2,n-1)\to (4,1)\to (3,n-1)\to (5,1)\to\cdots\to (n-1,1)\to
(n-2,n-1)$.
 
{\rm (ii)} The cluster variables  $f_{\romfi}(p,q)$ attached to vertices of $Q_{\romfi}(n)$ 
are $f_{\romfi}(2,1)= \bangle{[n]}{1}=\phhi_N=\phhi_N^{w_0}$, 
$f_{\romfi}(n-1,n) = x_{1n}=\pssi_1^{w_0}$, $f_{\romfi}(n-2,n) = \pssi_2^{w_0}$ and
\begin{equation*}
f_{\romfi}(p,q)= \begin{cases}
    \bangle{[n-1] \cup[2n-q+2,2n]}{1\cup [n-p-q+5,n]},   & \quad \text{if $p+q\in [4,n]$}, \\
 \bangle{[n+1]\cup [2n-q+2,2n]}{[n]\cup[3n-p-q+1,2n-1]}, &\quad  \text{if $p, q<n$, $p+q > n$}, \\
    X_{[n-q]}^{[n-q]},   & \quad \text{if $p=n$, $q < n$}. 
    \end{cases}
 \end{equation*}
 \end{proposition}

Note that $f_{\romfi}(1,q)= \bangle{[n-1]\cup [2n-q+2,2n]}{1\cup [n-q+4,n]}=X_{[2,n-q+2]}^{[n-q+1]} 
=f_{q-2,q}^{w_0}= \pssi_{n-q+2}^{w_0}$ for $q\in [3,n-1]$. Therefore, the frozen vertices of 
$Q_{\romfi}(n)$ are $(1,q)$ for $q\in [2,n-1]$, $(n,q)$ for $q\in [n-1]$, $(2,1)$, $(3,1)$,
$(n-2,n)$ and $(n-1,n)$.

Our running example, the case of $n=5$, becomes too small to illustrate the next stages of the process, and
we switch to the case $n=7$.
 The quiver $Q_{\romfi}(7)$ is shown in Fig.~\ref{fig:grid2}. The lacing is shown with dashed lines.
 
 \begin{figure}[ht]
\begin{center}
\includegraphics[height=8cm]{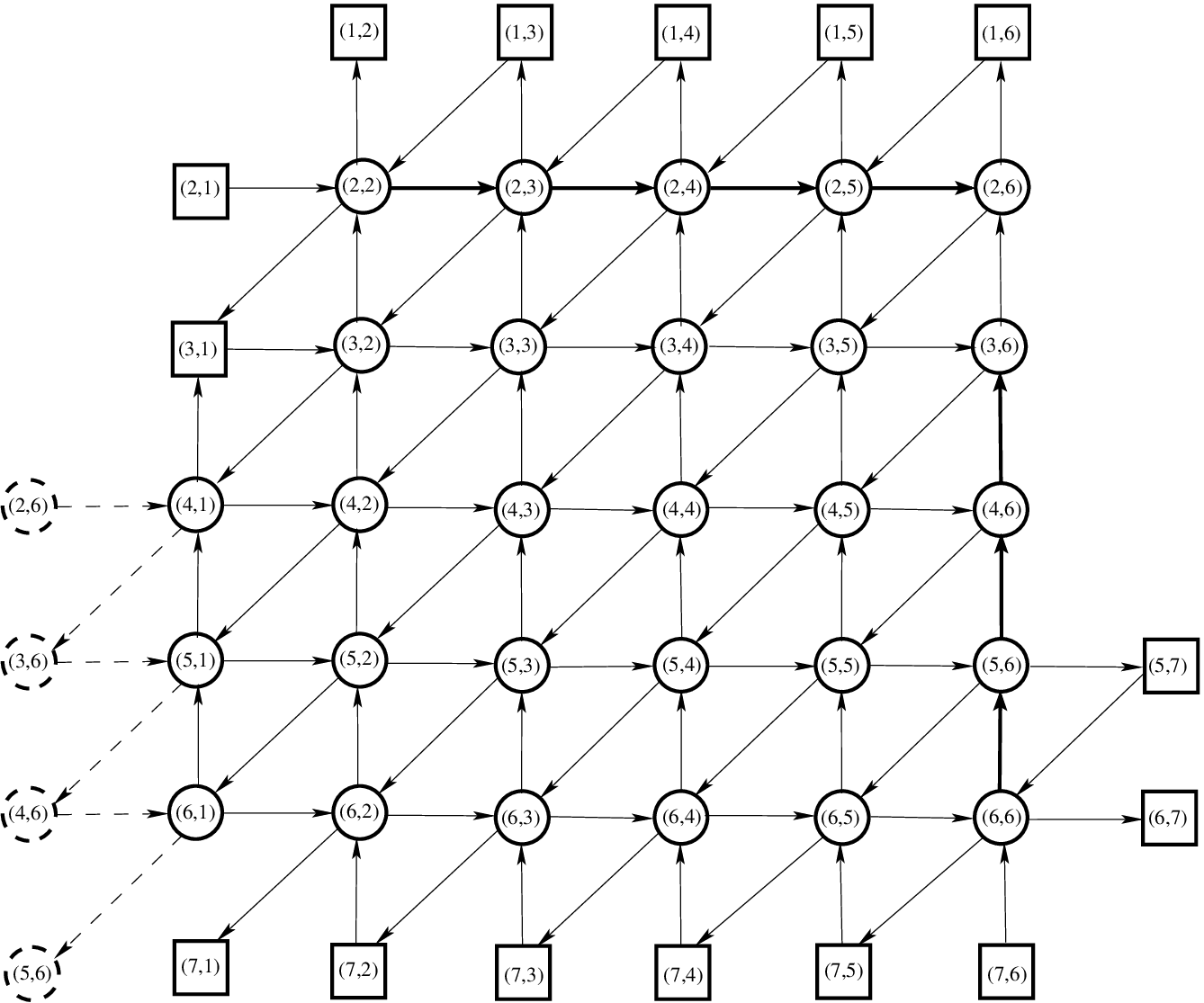}
\caption{Quiver $Q_{\romfi}(7)$ }
\label{fig:grid2}
\end{center}
\end{figure}  

Observe that the front panel is isomorphic, up to freezing the vertices $(p,1)$, $p\in [4,n-1]$, to the quiver
studied in \cite{GSV1}, \cite[Ch.~4.2]{GSVb} and related to the open cell of the Grassmannian $G_{n-1}(2(n-1))$;
we denote this quiver $Q_{Gr}(n-1)$.
On the fifths stage we apply to $Q_{\romfi}(n)$ the same sequence of transformations that was used in the proof of Lemma~4.13 in~\cite{GSVb}. 
Namely, define a sequence $\T_j^1$, $j\in [n-2]$, as mutations along the path $(j+1,2)\to(j+1,3)\to\dots\to
(j+1,n-j)$ in the opposite direction (starting from $(j+1,n-j)$) followed by mutations along the path 
$(n-1,n-j)\to(n-2,n-j)\to\dots\to(j+2,n-j)$ in the opposite direction (starting from $(j+2,n-j)$).
Similarly to the previous stages, we denote
 $\T_{(j)}^1=\T_{j}^1\circ\T_{j-1}^1\circ\dots\circ\T_{1}^1$, 
$j\in [n-2]$, and $\T^1_{\romfi}=\T_{(n-2)}^1$. After the end of $\T^1_{\romfi}$, the vertices of the
fringe are shifted in the following way: $(1,n-1)$ is moved to the position $(2,n)$ in the vertical part of the fringe; $(1,q)$, $q\in [2,n-2]$, is moved rightwards to the position $(1,q+1)$; 
$(p-1,n)$ is moved downwards to the position $(p,n)$;
for $p=n-3,n-2$, the vertex $(p,n)$ is  moved upwards to the position $(p-1,n)$.
The resulting quiver is denoted $Q^{(1)}_{\romfi}(n)$.

 \begin{lemma}
 \label{grassmanind}
{\rm (i)} The induced subquiver of $Q^{(1)}_{\romfi}(n)$
on the vertices $(p,q)$, $p\in [n-1]$, $q\in [2,n]$, can be decomposed into the
front panel and the fringe that are isomorphic to those for $Q_{\romfi}(n-1)$, 
up to freezing the vertices of the last 
row and the first two vertices in the first column.

{\rm (ii)} The vertices of the first and the second column are connected with the path
$(n,1)\to(n-1,2)\to(n-1,1)\to\dots\to (2,2)\to (2,1)$. The vertices of the $(n-1)$th and the
$n$th row are connected with the path
$(n-1,1)\to (n,1)\to(n-1,2)\to\dots\to (n-1,n-1)\to (n,n-1)$.

{\rm (iii)} The lacing 
consists of the paths $(2,1)\to(4,1)\to(2,2)$, 
$(2,n-1)\to(5,1)\to(3,n-1)\to(6,1)\to\dots\to(n-1,1)\to(n-3,n-1)$.
\end{lemma}
 
 \begin{proof}
 One can prove by induction that the evolution of $Q^{(1)}_{\romfi}(n)$ under $\T^1_{\romfi}$
 can be described as follows.
 
 {\rm (i)} The front panel is changed exactly as described in~{\rm \cite[Lemma 4.13]{GSVb}}, with the following amendment: after the 
mutation at $(p,2)$, $p\in[3,n-2]$, the edge $(p+1,1)\to(p,1)$ is deleted; after the mutation at $(n-1,2)$
the edge $(n-1,1)\to(n,1)$ is added. Therefore, the induced subquiver of
$Q^{(1)}_{\romfi}(n)$ on the vertices $(p,q)$, $p,q\in [2,n-1]$,
is isomorphic to $Q_{Gr}(n-2)$ up to freezing the vertices $(p,2)$, $p\in [2,n-1]$,
and $(n-1,q)$, $q\in [2,n-1]$. 
The edges of the front panel outside this induced subquiver form two paths:
$(n-1,1)\to(n,1)\to(n-1,2)\to(n,2)\to(n-1,3)\to\dots\to(n-1,n-1)\to(n,n-1)$ and
$(n-1,2)\to(n-1,1)\to(n-2,2)\to(n-2,1)\to\dots\to(2,2)\to(2,1)$.

{\rm (ii)} The fringe is only affected during $\T^1_1$. 
After the mutation at $(2,n-1)$ the vertex
$(1,n-1)$ is connected to the front panel via the path $(3,n-1)\to(1,n-1)\to(2,n-1)$. 
After the mutation at 
$(2,q)$, $q\in [2,n-2]$, the vertex $(1,q)$ is connected with the front panel via the path 
$(2,q+1)\to(1,q)\to(2,q)$.

After the mutation at $(p,n-1)$,
$p\in [3,n-2]$, the vertex $(p-1,n)$ is connected with the front panel via the path
$(p+1,n-1)\to(p-1,n)\to(p,n-1)$ {\rm(}the first of the two edges does not exist if $p=n-1${\rm).} 
In addition, for $p=n-3,n-2$ the vertex $(p,n)$ is connected 
with the front panel via the path $(p-1,n-1)\to(p,n)\to(p,n-1)$. 

After the end of $\T^1_{\romfi}$ the vertices of the fringe are moved in their new positions.

{\rm (iii)} The lacing is only affected during $\T_1^1$. 
Prior to the mutation at $(2,q)$, $q\in [2,n-1]$, 
the lacing consists of the paths $(2,q)\to(4,1)\to (2,q+1)$ and
$(3,n-1)\to(5,1)\to(4,n-1)\to(6,1)\dots\to(n-1,1)\to(n-2,n-1)$.
Prior to the mutation at $(p,n-1)$, $p\in [3,n-2]$, the lacing 
consists of the paths $(2,1)\to(4,1)\to(2,2)$, 
$(2,n-1)\to(5,1)\to(3,n-1)\to(6,1)\to\dots\to(p+1,1)\to(p-1,n-1)$ and
$(p,n-1)\to(p+2,1)\to(p+1,n-1)\to\dots\to(n-1,1)\to(n-2,n-1)$; the second of these paths is empty for $p=3$,
the third one is empty for $p=n-2$. Mutations at $(n-2,n-1)$ and $(n-1,n-1)$ do not change the lacing.
\end{proof}

The quiver $Q^{(1)}_{\romfi}(7)$ is shown in Fig.~\ref{fig:grid4}.

  \begin{figure}[ht]
\begin{center}
\includegraphics[height=8cm]{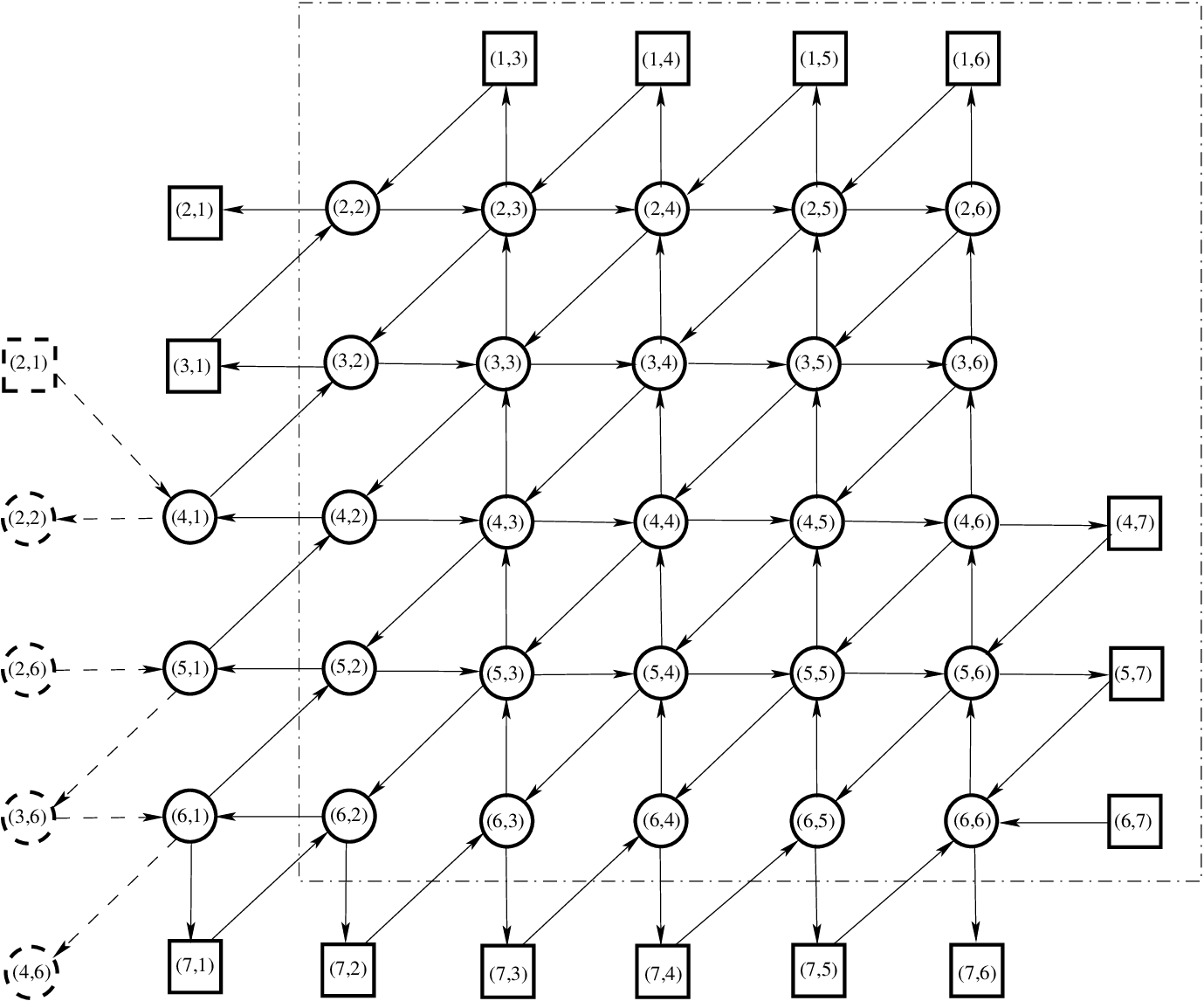}
\caption{Quiver $Q^{(1)}_{\romfi}(7)$ and its subquiver isomorphic to $Q_{\romfi}(6)$ 
up to freezing the first column and the last row }
\label{fig:grid4}
\end{center}
\end{figure} 

It follows from Lemma~\ref{grassmanind} that we can define the sequence $\T^2_{\romfi}$ as the
application of $\T^1_{\romfi}$ to the induced subquiver of
$Q^{(1)}_{\romfi}(n)$ on the vertices $(p,q)$, $p\in [n-1]$, $q\in [2,n]$, the sequence
$\T^3_{\romfi}$ as the
application of $\T^1_{\romfi}$ to the corresponding induced subquiver of the
resulting quiver, and so on up to $\T^{n-2}_{\romfi}$, which consists of the unique mutation at $(2,n-1)$. As before, we define $\T^{(l)}_{\romfi}=\T^l_{\romfi}\circ\cdots\circ\T^1_{\romfi}$,
$Q^{(l)}_{\romfi}(n)=\T^{(l)}_{\romfi}(Q_{\romfi}(n))$, $f^{(l)}_{\romfi}(p,q)=\T^{(l)}_{\romfi}
(f_{\romfi}(p,q))$, $l\in [n-2]$, $Q^{(0)}_{\romfi}(n)=Q_{\romfi}(n)$, $f^{(0)}_{\romfi}(p,q)=f_{\romfi}(p,q)$.

\begin{lemma}\label{grassmannian}
The functions $f_{\romfi}^{(l)}(p,q)$ attached to the vertices of
$Q^{(l)}_{\romfi}(n)$, $l\in [n-2]$, are described as follows.
 In the first row, $f_{\romfi}^{(l)}(1,l+1)= \phhi^{w_0}_{n-1}$, 
$f_{\romfi}^{(l)}(1,q)= \pssi^{w_0}_{n-q+l+2}$ for  $ l+2 < q < n $. 
In the $n$th column, 
$f_{\romfi}^{(l)}(n-2-l,n)= \pssi^{w_0}_{2}$, $f_{\romfi}^{(l)}(n-1-l,n)= \pssi^{w_0}_{1}$,
and $f_{\romfi}^{(l)}(p,n)= \pssi^{w_0}_{n-p+2}$ for $p\in [n-l, n-1]$. 
At other vertices,
 \begin{equation*}
f_{\romfi}^{(l)}(p,q)= \begin{cases}
   \bangle {[n-1] \cup [2n-q+l+2,2n]]} {1 \cup [l+4, n] \cup  [3n-p -q - l +1 ,2n-1]},   &  
  \text{if  $p\in [n+1- q, n- l]$}, \\
    \bangle {[n-1] \cup [2n-q+l+2,2n]]} {1 \cup [n-p -q + l +5, n] \cup  [2n - l ,2n-1]},   &  \text{if $q\in [l+1, n - p]$, $(p,q)\ne (2,l+1)$}, \\
    \bangle{[n]}{1\cup [2n-l,2n-1]},  &  \text{if $p=2$, $q=l+1$},\\
    f_{\romfi}^{(l-1)} (p,q) &  \text{otherwise}.    
    \end{cases}
 \end{equation*}
 \end{lemma}
 
 \begin{proof} The evolution of the front panel and the fringe under $\T^{l}_{\romfi}$, $l\in 
[2,n-2]$, is similar to the evolution described in Lemma~\ref{grassmanind}. 
Namely, only the restriction of the quiver to the vertices $(p,q)$, $p\in [n-l]$, $q\in [l+1,n]$, 
is changed. The restriction of the front panel to the vertices $(p,q)$, $p\in [2,n-l]$, 
$q\in [l+1,n-1]$, is isomorphic to $Q_{Gr}(n-l-1)$.
The edges $(p,l)\to(p+1,l)$, $p\in [2,n-l]$, and $(n-l+1,q)\to(n-l+1,q-1)$, $q\in [l+1,n-1]$ 
are added, the edges $(p,q)\to(p+1,q-1)$ for $p\in[2,n-l]$, $q=l+1$, and $p=n-l$, $q\in[l+1,n-1]$,  $(p,l)\to(p,l+1)$ for $p\in[2,n-l]$, $(n-l+1,q)\to(n-l,q)$ for $q\in[l+1,n-1]$ are reversed.

 The horizontal part of the fringe consists of vertices $(1,q)$, $q\in [l^*,n-1]$, the vertical part
 consists of $(p,n)$, $p\in [n-l^*,n-1]$, where $l^*=\min\{l+2,n-2\}$. The fringe is connected to the front panel
 with the paths $(2,n-1)\to(1,n-1)\to(2,n-2)\to(1,n-2)\to\dots\to(1,l+2)\to(2,l+1)$, $(n-1,n)\to(n-1,n-1)
 \to(n-2,n)\to(n-2,n-1)\to\dots\to(n-l,n)\to(n-l,n-1)$ and $(n-l-2,n-1)\to(n-l-2,n)\to(n-l-1,n-1)\to(n-l,n)\to(n-l,n-1)$.
For $l=n-3$ the first two edges of the third path are $(2,n-2)\to(1,n-1)\to(2,n-1)$. For $l=n-2$ the first path is  
empty, and the third path becomes $(2,n-3)\to(1,n-2)\to(2,n-2)\to(1,n-1)\to(2,n-1)$.
 
For $l<n-3$ the lacing consists of the paths 
$(2,1)\to(4,1)\to(2,2)\to (5,1)\to\dots\to(l+3,1) \to(l+1,2)$ and
$(2,n-1)\to(l+4,1)\to(3,n-1)\to(l+5,1)\to\dots\to(n-1,1)\to(n-l-2,n-1)$; 
 the second path is empty for $l=n-4$. 
 Transformations $\T^{n-3}_{\romfi}$ and $\T^{n-2}_{\romfi}$ do not
 change the lacing.

As on previous steps, we observe that all identities dictated by cluster transformations can be derived in a uniform
 way: if a transformation is applied at a vertex $(p,q)$, 
 we denote $\kappa=2n-q+l+1$ and 
\begin{align*} 
&\mu = l+3,\ \nu = 3n - p - q - l + 1\quad \text{for $p+q > n$}, \\
&\mu = n-p-q + l + 4,\ \nu = 2n - l\quad \text{for $p+q \leq n$}
\end{align*}
Then the needed identities are of the form
 {\small
\begin{align}
\nonumber
&  \bangle  {[n-1]\cup [\kappa, 2 n]}{1\cup [\mu,n] \cup [\nu+1,2n-1] }\ 
 \bangle  {[n-1]\cup [\kappa+1, 2 n]}{1\cup [\mu+1,n] \cup [\nu,2n-1] }\ = \\ 
\label{magicformula_51}
&  \qquad
\bangle  {[n-1]\cup [\kappa+1, 2 n]}{1\cup [\mu,n] \cup [\nu+1,2n-1] }\ 
{ \bangle  {[n-1]\cup [\kappa, 2 n]}{1\cup [\mu+1,n] \cup [\nu,2n-1] }}^{(s)}
\ \\ \nonumber
&\qquad\qquad\ +\ 
\bangle  {[n-1]\cup [\kappa, 2 n]}{1\cup [\mu,n] \cup [\nu,2n-1] }\ 
 \bangle  {[n-1]\cup [\kappa+1, 2 n]}{1\cup [\mu+1,n] \cup [\nu+1,2n-1] } 
\end{align}
}
if $(p,q)\ne (2,l+1)$ and 
{\small
\begin{align}
\nonumber
&  \bangle  {[n-1]\cup 2 n}{1\cup [\nu+1,2n-1] }\ 
 \bangle  {[n]}{1\cup [\nu,2n-1] }\ = \\ 
\label{magicformula_52}
&  \quad
\bangle  {[n]}{1\cup [\nu+1,2n-1] }\ 
 \bangle  {[n-1]\cup  2 n}{1 \cup [\nu,2n-1] }
\ +\ 
{\bangle  {[n]\cup 2 n}{1\cup [\nu,2n-1] }}^{(2)}\ 
 \bangle  {[n-1]}{1\cup [\nu+1,2n-1] }
\end{align}
}
if $(p,q)=(2, l+1)$.
 
 For $p > 2$, $ q < n-1$, the value of $s$ in \eqref{magicformula_51} equals one, and the
relation follows directly from the definition of $f^{(l)}_{\romfi}$.
 For $p=2$, $q > l+1$, we need to consider
 \begin{align*}
&f^{(l-1)}_{\romfi} (2,q) f^{(l)}_{\romfi} (2,q) \\
&\qquad= f^{(l-1)}_{\romfi} (2,q-1) f^{(l)}_{\romfi} (2,q+1)+ f^{(l-1)}_{\romfi} (3,q-1) \psi^{w_0}_{n-q+l+1} f^{(0)}_{\romfi} (l+3,1)\ 
\end{align*}
 and  notice that
 \begin{align*}
\psi^{w_0}_{n-q+l+1} f^{(0)}_{\romfi} (l+3,1) &= \psi^{w_0}_{n-q+l+1} \bangle {[n-1]} {1\cup [n - l + 1, n]}\\
&= {\bangle {[n-1]\cup [2n-q+l+1,2n] } {1\cup [n-q+l+3,n] \cup [2 n - l, 2n - 1]}}^{(2)}.
\end{align*}

For $p>2$, $q=n-1$, we look at
 \begin{align*}
&f^{(l-1)}_{\romfi} (p,n-1) f^{(l)}_{\romfi} (p,n-1) \\
&\qquad= f^{(l-1)}_{\romfi} (p+1,n-1) f^{(l)}_{\romfi} (p-1,n-1)+ f^{(l-1)}_{\romfi} (p+1,n-2) \psi^{w_0}_{l+3} \Delta(p),
\end{align*}
where
$$
\Delta (p)=\begin{cases}  
  f^{(0)}_{\romfi} (p+l+1,1), &\quad  \text{if $p\in [2, n-l-2]$},\\
  \psi^{w_0}_{2},   &\quad  \text{if $p=n-l-1$},\\
    \psi^{w_0}_{1},   &\quad  \text{if $p=n-l$}\\
  \end{cases}\ = \bangle {[n-1]} {1\cup [n - p - l + 3, n]},
$$
 and observe that
 \begin{equation*}
 \psi^{w_0}_{l+3} \bangle {[n-1]} {1\cup [n - p - l + 3, n]}
 = {\bangle {[n-1]\cup [n+l+2,2n] } {1\cup [l+4,n] \cup [2 n - p - l + 2, 2n - 1]}}^{(2)}.
\end{equation*}
 
 Finally, relation \eqref{magicformula_52} follows from
 \begin{align*}
f^{(l-1)}_{\romfi} (2,l+1) &f^{(l)}_{\romfi} (2,l+1) \\
&= f^{(l-1)}_{\romfi} (2,l) f^{(l)}_{\romfi} (2,l+2)+ f^{(l-1)}_{\romfi} (3,l) \phhi^{w_0}_{n-1} f^{(0)}_{\romfi} (l+3,1)
\end{align*}
after one notices that
\begin{align*}
& \phhi^{w_0}_{n-1} f^{(0)}_{\romfi} (l+3,1) = X_{[2,n]}^{[1,n-1]} \bangle {[n-1]} {1\cup [n - l + 1, n]} = {\bangle {[n]\cup 2n } {1\cup [2 n - l, 2n - 1]}}^{(2)}.
\end{align*}
 
 To prove \eqref{magicformula_51}, we  consider the matrix $C$ obtained from $\bar U$  by deleting columns 
indexed by $[n-1]\cup [\kappa+1, 2 n]$ and rows indexed by $1\cup [\mu+1,n] \cup [\nu+1,2n-1]$. 
Let $t$ be the smallest index such that  entries of the first sub-diagonal of $C$ are zero in rows with the index larger than $t$, then we define $A$ to be the leading principal $t\times t$ submatrix of $C$. If $(p,q)\ne (2,l+1)$,  we apply \eqref{jacobi} to $A$ with  $\alpha=\mu$, $\beta=\nu$, $\gamma = \kappa$, $\delta= t$, where  $\mu, \nu$ are defined as above. The needed identities are obtained  after canceling common factors.
If $(p,q) = (2,l+1)$, $A$ is defined in the same way (note, that the interval $[\mu+1, n]$ is empty in this case), and  one needs to use
\eqref{jacobi}  with  $\alpha=\nu$, $\beta=t$, $\gamma = \kappa$, $\delta= t$ to arrive at  \eqref{magicformula_52}.
 \end{proof}
 
 Among functions $f^{(n-2)}_{\romfi} $ we now find seven new representatives of the family $\{ f^{w_0}_{ij}\}$. 
 Namely, 
 $f^{(n-2)}_{\romfi} (2, n-i) = f^{(n-i-1)}_{\romfi} (2, n-i)\bangle{[n]}{1\cup [n+i+1,2n-1]} = \phhi_{n+i-1}$ for $i\in[1,3]$,  and
  $f^{(n-2)}_{\romfi} (3, n-i) = f^{(n-i-1)}_{\romfi} (3, n-i)\bangle{[n-1]}{1\cup [n+i+1,2n-1]} = \pssi_{n+i-1}$ for $i\in [1,4]$.
  After they are frozen, vertices $(1,n-2)$, $(1,n-1)$,
$(2,n-2)$, $(2,n-1)$, $(2,n)$ become isolated.

The quiver obtained upon the completion of $\T^{n-2}_{\romfi}$ is denoted $Q_{\romsi}(n)$. 
It will be convenient to put the vertices of $Q_{\romsi}(n)$ on the square grid in the following way: 
vertex $(p,q)$ remains at $(p, q)$ for $p\in [3,n]$ and $q\in [n]$, vertex $(2,q)$ is placed at 
$(q+3,0)$ for $q\in [n-3]$. Therefore, the frozen vertices of $Q_{\romsi}(n)$ are $(3,q)$ for
$q\in[n]$, $q\ne 3$, $(n,q)$ for $q\in [0,n-1]$, $(p,n)$ for $p\in [4,n-1]$, and $(4,0)$.

It will be convenient to subdivide $Q_{\romsi}(n)$ into parallelograms with vertices $(p,q)$, $(p-1,q+1)$, $(p-1,q)$, $(p,q-1)$. Such a parallelogram is denoted $(p,q)$; it
may be south (S), north (N), northeast (NE)
and southwest (SW) depending on the direction of its diagonal. Lemma~\ref{grassmannian} implies

\begin{proposition}
\label{indgrass}
{\rm (i)} $Q_{\romsi}(n)$ is subdivided into consistently oriented triangles,
forming parallelograms. 
All  parallelograms $(p,q)$ are S for $p=4$ and $q\in [2,n-1]$, or $p\in [5,n]$ and $q\in [n-1]$, and oriented 
for $p\in [4,n-1]$ and $q=1$. Additionally, there are edges $(p,0)\to (p-1,0)$ for $p\in[5,n]$.

The lacing consists of the path $(3,1)\to(5,0)\to(3,2)\to (6,0)\to\dots\to(n-1,0)\to(3,n-4)$.

{\rm (ii)} The cluster variables $f_{\romsi}(p,q)$ attached to the vertices of $Q_{\romsi}(n)$
are
\begin{equation*}
f_{\romsi}(p,q)= 
   \bangle {[n-1] \cup [3n - p - q +2,2n]} 
   {1 \cup [n - p + 4,n ] \cup  [2n - q  + 1,2n-1]}   
 \end{equation*}
for  $p\in [3,n- 1]$ and $q \in [n]$,
\begin{equation*}
f_{\romsi}(p,0)= 
   \bangle {[n]} 
   {1 \cup  [2n - p + 4,2n-1]}   
 \end{equation*}
 for $p\in[4, n]$ and
 \begin{equation*}
f_{\romsi}(n,q)= \thetta^{w_0}_{n-q}
 \end{equation*}
 for $q\in [n-1]$.
\end{proposition}
 
 The quiver $Q_{\romsi}(7)$ is shown in Fig.~\ref{fig:grid5}.
 
  \begin{figure}[ht]
\begin{center}
\includegraphics[width=12cm]{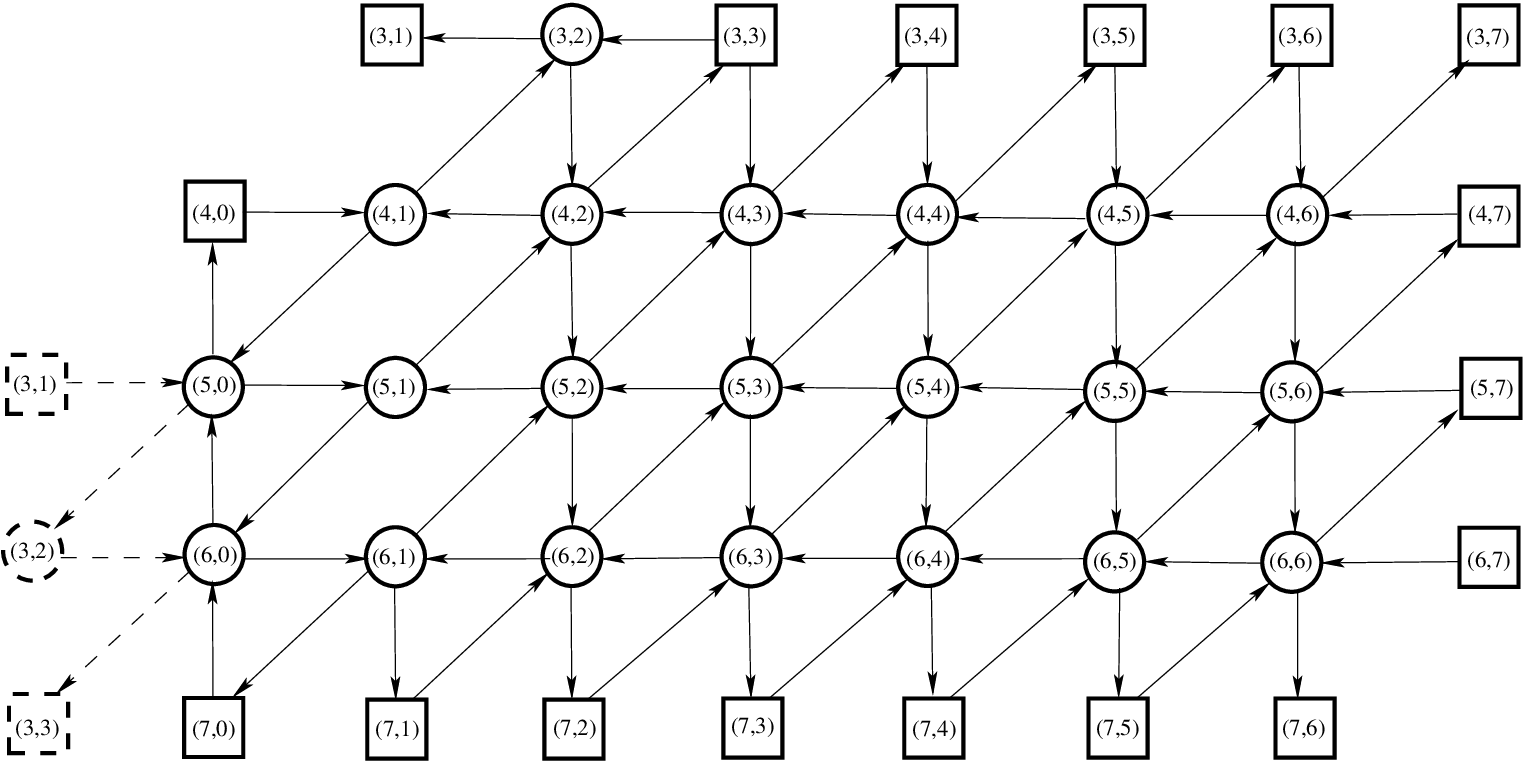}
\caption{Quiver $Q_{\romsi}(7)$}
\label{fig:grid5}
\end{center}
\end{figure} 

On the last, sixth stage we start with mutations at $(n-1,1), (n-2,1),\dots,(4,1)$, followed
by mutations at $(n-1,2), (n-2,2),\dots,(4,2)$, and so on until $(n-1,n-1), (n-2,n-1),\dots,(4,n-1)$.
 It will be convenient, after the mutation at $(n-1,q)$,
to switch vertices $(n,q-1)$ and $(n,q)$. This sequence is denoted $\T^1_{\romsi}$, and the resulting quiver $Q^{(1)}_{\romsi}(n)$. We relabel $Q^{(1)}_{\romsi}(n)$ by shifting
 $(3,q)\mapsto (q+4,-1)$, $q\in [n-4]$, similarly to what was done before the beginning of the sixth stage.

\begin{lemma}
\label{indgrass1}
{\rm (i)} $Q^{(1)}_{\romsi}(n)$ is subdivided into consistently oriented triangles,
forming parallelograms. 
All   parallelograms $(p,q)$ are S for $p=5$ and $q\in [n-2]$, or $p\in [6,n]$ and $q\in [0,n-2]$, and oriented 
for $p\in [5,n-1]$ and $q=0$. Additionally, there are edges $(p,-1)\to (p-1,-1)$ for $p\in[6,n]$.

The lacing consists of the path $(4,0)\to(6,-1)\to(4,1)\to (7,-1)\to\dots\to(n-1,-1)\to(4,n-4)$.

{\rm (ii)} The cluster variables $f^{(1)}_{\romsi}(p,q)$ attached to the vertices of $Q^{(1)}_{\romsi}(n)$
are
\begin{equation*}
f^{(1)}_{\romsi}(p,q)= 
   \bangle {[n] \cup [4n - p - q +2,3n+1]} 
   {1 \cup [2n - p + 4,2n-1] \cup  [3n - q - 1,3n-2]}   
 \end{equation*}
for  $p\in [4,n- 1]$ and $q\in[-1, n- 1] $,
\begin{equation*}
f^{(1)}_{\romsi}(p,-1)= 
   \bangle {[n-1]} 
   {1 \cup  [2n - p + 5,2n-1]}   
 \end{equation*}
 for $p\in[5, n]$, and    
 \begin{equation*}
f^{(1)}_{\romsi}(n,q)= \thetta^{w_0}_{n-q-1}
 \end{equation*}
 for $q\in [0,n-2]$.
\end{lemma}

\begin{proof}
One can check by induction that prior to applying a mutation in the sequence $\T^1_{\romsi}$ at
a vertex $(p,q)$, the current state of the quiver $Q_{\romsi}(n)$ can be described as follows.

(i) Vertex $(p,q)$ is five-valent if $p=n-1$ and four-valent otherwise. The edges pointing towards
$(p,q)$ are $(p,q-1)\to (p,q)$ and $(p,q+1)\to(p,q)$. The edges pointing from $(p,q)$ are

i.1) $(p,q)\to (p+1,q-1)$ and $((p,q)\to (p-1,q+1)$;

i.2) $(n-1,q)\to (n,q)$ (recall that $(n,q-2)$ and $(n,q-1)$ have been switched after the mutation
at $(n-1,q-1)$).

(ii) $Q_{\romsi}(n)$ is subdivided into consistently oriented triangles,
forming parallelograms. Parallelograms $(4,q')$ are NE for $q'\in[q-1]$, the parallelogram $(p+1,q)$
is NE for $p\ne n-1$. Parallelograms $(p',q')$ are oriented for $p'\in [5,p^*+1]$ and $q'=q-1$,
where $p^*=\min\{p,n-2\}$, for $p'\in [4,p]$ and $q'=q$, and for $p'\in[p+1,n-1]$ and $q'=q+1$.
All remaining parallelograms are S.

(iii) The lacing does not change.

The quiver $Q_{\romsi}(7)$ before the mutation at $(5,3)$ in $\T^1_{\romsi}$ is shown in Fig.~\ref{fig:grid6}. Vertices $(7,2)$ and $(7,3)$ should be switched now.

 \begin{figure}[ht]
\begin{center}
\includegraphics[width=12cm]{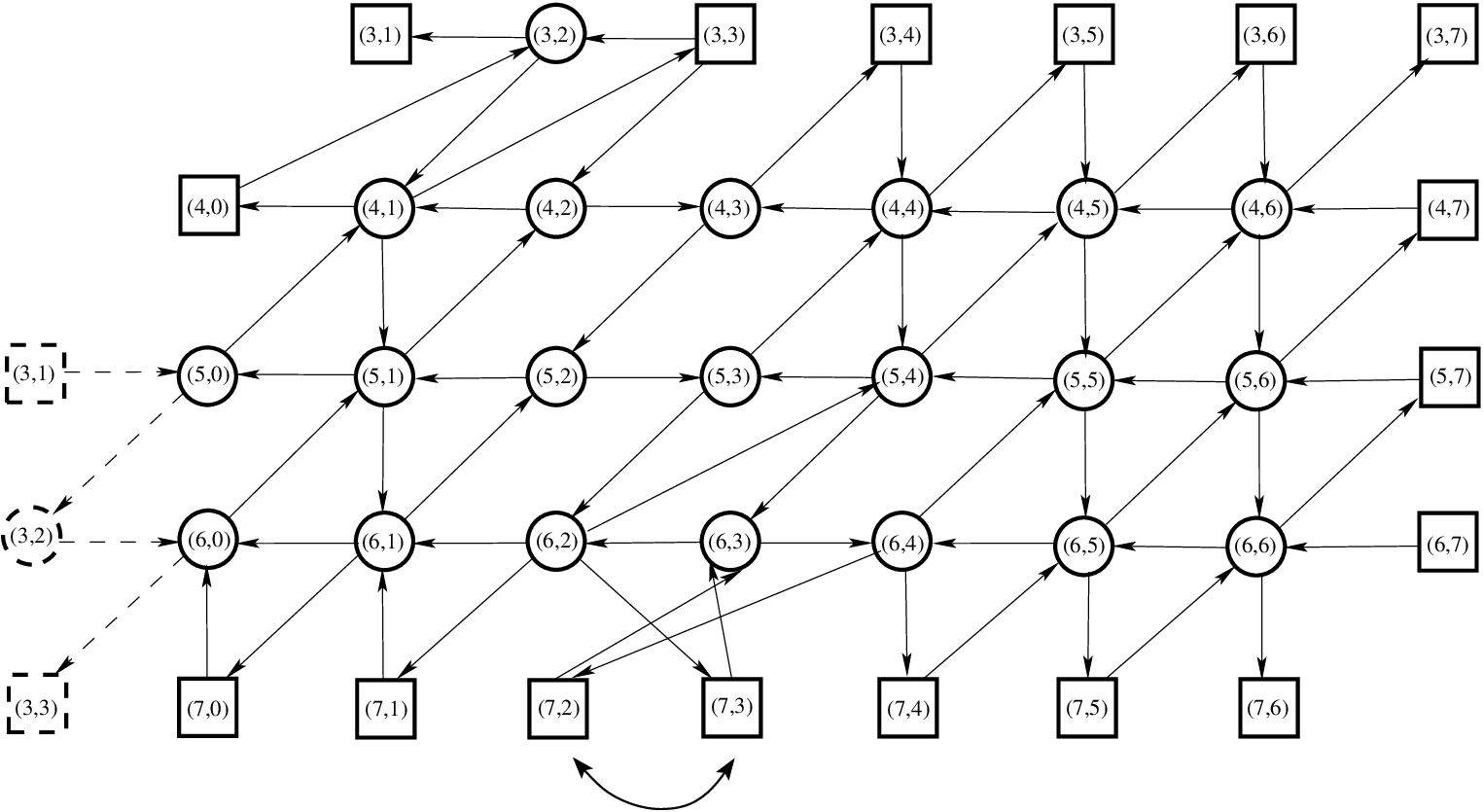}
\caption{Quiver $Q_{\romsi}(7)$ before the mutation at $(5,3)$ in $\T^1_{\romsi}$}
\label{fig:grid6}
\end{center}
\end{figure} 

The proof of the second part of the lemma is deferred until Lemma~\ref{indgrass2} below.
\end{proof}

It follows from Lemma~\ref{indgrass1} that 
$f_{\romsi}^{(1)}(n,-1) = \pssi^{w_0}_{n+3}$, 
$f_{\romsi}^{(1)}(4,0) = \phhi_N=\phhi_N^{w_0}$, 
$f_{\romsi}^{(1)}(5,-1) = \pssi_M=\pssi_M^{w_0}$, 
$f_{\romsi}^{(1)}(p,n-1) = \phhi^{w_0}_{2n+2 - p}$ for $p\in [4, n-1]$, and 
$f_{\romsi}^{(1)}(4, q) = \phhi^{w_0}_{3n - q-3}$ for $q\in [n-6, n-1]$. Therefore,
the corresponding vertices are frozen in $Q^{(1)}_{\romsi}(n)$, and hence vertices
$(3,q)$ for $q\in [n-3,n]$, $(p,n)$ for $p\in [4,n-1]$ and $(n,n-1)$ become isolated and
are deleted. We thus see that $Q^{(1)}_{\romsi}(n)$ looks very much like $Q_{\romsi}(n)$: it has
$n-3$ rows $4, 5,\dots,n$ and $n+1$ columns $-1,0, 1, \dots, n-1$, its lacing comes from the edges
between the upper two rows of $Q_{\romsi}(n)$, and the edges between its two leftmost columns
come from the lacing of $Q_{\romsi}(n)$. The quiver $Q^{(1)}_{\romsi}(7)$ is shown in Fig.~\ref{fig:grid7}. 

  \begin{figure}[ht]
\begin{center}
\includegraphics[width=12cm]{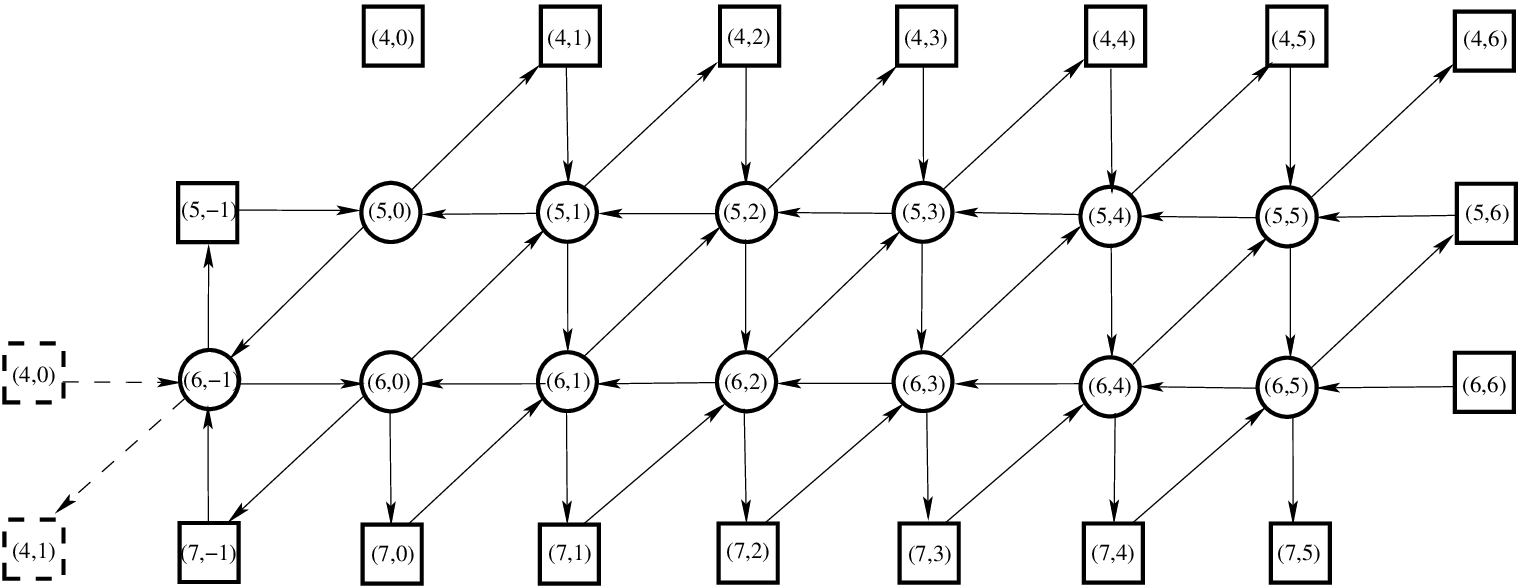}
\caption{Quiver $Q_{\romsi}^{(1)}(7)$}
\label{fig:grid7}
\end{center}
\end{figure} 

We now define the sequence $\T^2_{\romsi}$ as mutations at $(n-1,0), (n-2,0),\dots, (5,0)$
followed by mutations at $(n-1,1), (n-2,1),\dots,(5,1)$, and so on until $(n-1,n-2), (n-2,n-2),\dots,(5,n-2)$. As before, after the mutation at $(n-1,q)$ we 
switch vertices $(n,q-1)$ and $(n,q)$. The resulting quiver is denoted $Q^{(2)}_{\romsi}(n)$. We relabel it by shifting
 $(4,q)\mapsto (q+6,-2)$, $q\in [0,n-6]$. The quiver $Q^{(2)}_{\romsi}(n)$ is very similar to
 $Q^{(1)}_{\romsi}(n)$: it has one row less and the same number of columns. Thus, we can define
 $\T^{(3)}_{\romsi}$ similar to $\T^{(2)}_{\romsi}$ and so on. After $l$ steps we get to the quiver
 $Q^{(l)}_{\romsi}(n)$.

Denote
$$
r_s = s (n-1) +1,\qquad c_s = s(n+1) -2, \qquad s\in [k],
$$
where $k=\lfloor\frac{n+1}2\rfloor$.
Recall that the staircase shape of $\bar U$ is such that $\bar u_{ij}=0$ when $i> r_s$, $j < c_s + 3$ or $i < r_s + 1$, $j > c_{s+1}$.

\begin{lemma}
\label{indgrass2}
The description of functions $f_{\romsi}^{(l)}(p,q)$ attached to vertices of $Q_{\romsi}^{(l)}(n)$,  $l\in [n-4]$, depends on the parity of $l$. If $l=2s$, we have
\begin{equation}
 \label{romsi_even_1}
f_{\romsi}^{(l)}(p,q)= 
   \bangle {[n-1] \cup [\alpha^l_{pq},
   c_{s+2}]} 
   {1 \cup [\beta^l_{p},
   r_{s+1}] \cup  [\gamma^l_{q},
   r_{s+2}]}   
 \end{equation}
for  $p\in [l+3, n- 1]$ and $q\in [-l+1, n- l]$, where 
\begin{equation}\label{romsiabc}
\begin{aligned}
\alpha^l_{pq}&=(s+3)(n+1) - p - q - l-1,\qquad
\beta^l_{p}=(s+1)n - p + s + 4,\\ 
&\qquad\qquad\qquad\gamma^l_{q}=(s+2)n - q - 3 s + 1, 
\end{aligned}
\end{equation}
and 
\begin{equation}
 \label{romsi_even_2}
f_{\romsi}^{(l)}(p,-l)= 
   \bangle {[n]} 
   {1 \cup  [\beta_{p}^l+n,r_{s+2}]}   
 \end{equation}
 for $p\in[ l+4, n]$.
 
If $l=2s-1$, we have
\begin{equation}
 \label{romsi_odd_1}
f_{\romsi}^{(l)}(p,q)= 
   \bangle {[n] \cup [\alpha^{l+1}_{pq}+1,
   c_{s+2}]} 
   {1 \cup [\beta^{l+1}_{p}-1,
   r_{s+1}] \cup  [\gamma^{l+1}_{q}+1,
   r_{s+2}]}   
 \end{equation}
for  $p\in [l+3, n- 1]$ and $q\in [-l, n- l]$, 
and 
\begin{equation}
 \label{romsi_odd_2}
f_{\romsi}^{(l)}(p,-l)= 
   \bangle {[n-1]} 
   {1 \cup  [\beta^{l+1}_{p},r_{s+1}]}   
 \end{equation}
 for $p\in [l+4, n]$.
 
 In both cases 
 \begin{equation}
 \label{romsi_3}
f_{\romsi}^{(l)}(n,q)= \theta^{w_0}_{n-q-l}
 \end{equation}
 for $q\in [-l+1, n-l-1]$.
\end{lemma}
 
 \begin{proof} 
 We prove the Lemma by induction on $l$, taking into account that the case $l=0$ is already established in Proposition~\ref{indgrass}.

First, we observe that  one can encode $\phhi_{i}^{w_0}$ and $\pssi_j^{w_0}$ using $\bangle \ \ $ notation,
based on their description as minors of $\bar U$ given in Section~\ref{sec:Strans}: 
 \begin{equation}\label{bangle_w0}
 \begin{aligned}
 &\phhi_i^{w_0}=  \left\{  
 \begin{aligned}
 \bangle {[n]\cup [n+i+1, c_{s+1}]}{1\cup [i+2, r_{s+1}]},& \quad \text{if $r_s < i \leq c_{s+1} + 2- n$, 
 $s > 0$},\\
 \bangle {[n]\cup [n+i+1, c_{s}]}{1\cup [i+2, r_{s+1}]}, & \quad \text{if $c_{s} + 2- n < i \leq r_{s}$, 
 $s > 0$};
 \end{aligned}\right.
 \\ 
 &\pssi_j^{w_0}  =  \left\{
 \begin{aligned}
 \bangle {[n-1]\cup [n+j, c_{s+1}]}{1\cup [j+2, r_{s+1}]}, & \quad \text{if $r_s \leq j < c_{s+1} + 3- n$,
 $s > 0$},\\
 \bangle {[n-1]\cup [n+j, c_{s}]}{1\cup [j+2, r_{s+1}]}, & \quad \text{if $c_{s} + 3- n \leq j < r_{s}$, 
 $s > 0$}.
 \end{aligned}\right.
 \end{aligned}
\end{equation}

Together with the induction hypothesis, this implies
\begin{equation}\label{frozen_even}
\begin{gathered}
 f_{\romsi}^{(l)}(l+3,-l+1) =\pssi_M=\pssi^{w_0}_M, \qquad
f_{\romsi}^{(l)}(l+4,-l) = \phhi_N=\phhi^{w_0}_N,\\
 f_{\romsi}^{(l)}(n,-l) = \phhi^{w_0}_{(s+1)(n+1) + 1}, \\
  f_{\romsi}^{(l)}(p,n-l) = \pssi^{w_0}_{(s+1)(n+1) - p+1},\quad p\in [l+3, n-1],\\
 f_{\romsi}^{(l)}(l+3, q) = \pssi^{w_0}_{(s+2)n - q-3s-1},\quad q\in [n-2l-4, n-l],
\end{gathered}
\end{equation}
 if $l=2s$, and
 \begin{equation}\label{frozen_odd}
\begin{gathered} 
f_{\romsi}^{(l)}(l+3,1-l) = \phhi_N=\phhi^{w_0}_N,\qquad f_{\romsi}^{(l)}(l+4,-l) = \pssi_M=\pssi^{w_0}_M,\\
f_{\romsi}^{(l)}(n,-l) = \pssi^{w_0}_{s(n+1) + 2},\\ 
f_{\romsi}^{(l)}(p,n-l) = \phhi^{w_0}_{(s+1)(n+1) - p},\quad p\in [l+3, n-1],\\
 f_{\romsi}^{(l)}(l+3, q) = \phhi^{w_0}_{(s+2)n - q-3s},\quad q\in [n-2l-4, n-l],
\end{gathered}
\end{equation}
 if $l=2s-1$.
 These relations follow from  \eqref{romsi_even_1}, \eqref{romsi_odd_1} combined with \eqref{bangle_w0}. 
 As a representative example, we will verify the last two formulas in \eqref{frozen_even}. If $q=n-l$, put
  $j= (s+1)(n+1) - p+1$; then $\beta_p^l = j +2$, $\gamma_q^l = r_{s+1} + 2$, 
$\alpha_{pq}^l= n+j$, and \eqref{romsi_even_1} becomes $f_{\romsi}^{(l)}(p,n-l) =  \bangle {[n]\cup [n+j, c_{s+1}]}{1\cup [j+2, r_{s+1}]}$.
If $p=l+3$, put $j= (s+2)n - q-3s-1$; then $\beta_p^l = r_{s+1} +1$, $\gamma_q^l = j + 2$, 
$\alpha_{pq}^l= n+j$, and \eqref{romsi_even_1} becomes $f_{\romsi}^{(l)}(p,n-l) =  \bangle {[n]\cup [n+j, c_{s+1}]}{1\cup [j+2, r_{s+2}]}$. This confirms the formulas in question.
 
 The corresponding vertices are frozen in $Q^{(l)}_{\romsi}(n)$; consequently the vertices 
 $(l+2,q)$ for $q\in [n-2l-1,n-l]$ and $(p,n-l+1)$ for $p\in [l+2,n]$ that were present in $Q_{\romsi}^{(l-1)}(n)$
 become isolated and are now deleted. Therefore, vertices of
$Q_{\romsi}^{(l)}(n)$,  $l\in [n-4]$, 
are located at  positions $(p,q)$, $p\in [l+3, n]$, $q\in[-l, n-l]$, except for $(l+3,-l)$ and $(n, n-l)$. The vertices $(l+3, q)$ for $q\in[n-2l-4, n-l]$, $(n,q)$ for $q\in [-l, n-l-1]$, $(p,n-l)$
for $p\in [l+3,n-1]$, $(l+4,-l)$ and $(l+3, -l+1)$ are frozen (the last one does not exist if $l=n-4$).

The following statements can be proved
by induction using the description of the current structure
of $Q^{(l-1)}_{\romsi}(n)$ during the execution of $\T^l_{\romsi}$ similar to the one given in
the proof of Lemma~\ref{indgrass1} for the case $l=1$. 

(i) $Q^{(l)}_{\romsi}(n)$ is subdivided into consistently oriented triangles,
forming parallelograms. 
All parallelograms  $(p,q)$ are S for $p=l+4$ and $q\in [-l+2,n-l-1]$, or $p\in [l+5,n]$ and $q\in [-l+1,n-l-1]$, 
and oriented for $p\in [l+4,n-1]$ and $q=-l+1$ (oriented parallelograms do not exist if $l=n-4$).
Additionally, there are edges $(p, -l ) \to (p-1, - l)$ for $p\in [l+5,n]$.

(ii) The lacing 
consists of the path $(l+3, -l+1) \to (l+5, -l) \to 
(l+3, -l+2) \to (l+6, -l) \to \cdots (l+3, n -2l - 5) \to (n-1, -l) \to (l+3, n - 2l -4)$. 
 
Let us turn to the description of the functions.  
First of all, consider equation~\eqref{romsi_3}. Recall that
the vertices of the $n$th row are frozen, and hence the change in the formulas is caused by 
their reshuffling after mutations at the $(n-1)$th row. This reshuffling results in the shift of the 
vertices to the left by one, hence the index $n-q-l$ does not change, and~\eqref{romsi_3} follows.

Next, it follows from the above mentioned description of the current structure that the 
identities we need to establish are
 \begin{equation*}
f^{(l-1)}_{\romsi} (p,q) f^{(l)}_{\romsi} (p,q) = f^{(l-1)}_{\romsi} (p,q+1) f^{(l)}_{\romsi} (p,q-1) + f^{(l-1)}_{\romsi} (p-1,q+1) f^{(l)}_{\romsi} (p+1,q-1)\ 
\end{equation*}
 for $p\in [l+3, n-1]$, $q\in [-l+2, n-l-1]$, and
 \begin{align*}
f^{(l-1)}_{\romsi} (n-1,q) f^{(l)}_{\romsi} (n-1,q)&=\ f^{(l-1)}_{\romsi} (n-1,q+1) f^{(l)}_{\romsi} (n-1,q-1)\\ &\qquad+ f^{(l-1)}_{\romsi} (n-2,q+1)\Delta_l  \thetta^{w_0}_{n-q-l+1}\ 
\end{align*}
for $p=n-1$,  $q\in [-l+2, n-l-1]$, where
$$
\Delta_l=\begin{cases}
\pssi^{w_0}_{s(n+1) + 2}, \qquad \text{if $l=2s$},\\
\phhi^{w_0}_{s(n+1) + 1}, \qquad \text{if $l=2s-1$}.
\end{cases}
$$ 
For $q=-l+1$ all expressions of the form $f^{(l)}_{\romsi}(*,q-1)$
in the both formulas above should be replaced by $f^{(l-1)}_{\romsi}(*,q-1)$.

Let us start with the even case $l=2s$. First, we notice that
\begin{align*}
& \pssi^{w_0}_{s(n+1) + 2} \thetta^{w_0}_{n-q-l+1}
=  \bangle {[n-1]} {1\cup [s(n+1) + 4, r_{s+1}]} \thetta^{w_0}_{n-q-l+1}\\
&\quad = \bangle {[n-1] \cup [(s+2) n - q - s +3 , 
c_{s+2}] } 
{1\cup [s(n+1) + 4, r_{s+1}] \cup [\gamma^l_{q}+1, r_{s+2}]}^{(2)}.
\end{align*}
It is easy to check that rows and columns involved in the latter expression coincide with those 
in the right-hand side of
\eqref{romsi_even_1} if  in \eqref{romsi_even_1} one uses $p=n$ and replaces $q$ with $q-1$. This allows us
to treat both relations above simultaneously. To this end, denote 
by $I=1\cup [\beta_{p}^l,r_{s+1}]\cup [\gamma_{q}^l+1,r_{s+2}]$ and $J=[n-1]\cup [\alpha_{pq}^l+1,c_{s+2}]$
the intersection of row (respectively, column) sets used in \eqref{romsi_odd_1}, \eqref{romsi_even_1}
to define $f^{(l-1)}_{\romsi} (p,q)$  and $f^{(l)}_{\romsi} (p,q)$. Then relations we need to establish have a form
 {\small
\begin{equation}
\label{last_1}
\begin{aligned}
\bangle  {[n]  }{(\beta^{l}_{p}-1) \cup I }\ &\bangle  { [n-1] }{\gamma^l_{q} \cup  I } \\
&=  
\bangle  {[n]   }{(\beta^{l}_{p}-1)\cup\gamma^l_{q}\cup  I }\ \bangle  { [n-1] }{ I }
+  \bangle  {[n] }{\gamma^l_{q}\cup I }\ {\bangle  { [n-1] }{(\beta^{l}_{p}-1) \cup I }}^{(t)}
\end{aligned}
\end{equation}
}
if $\alpha^l_{pq} > c_{s+2}$, which is equivalent to $p+q+l \leq n+1$ (note that in this case $J=[n-1]$), or
 {\small
\begin{align}
\label{last_2}
&  \bangle  {n \cup  J }{(\beta^{l}_{p}-1) \cup I }\ \bangle  { \alpha^l_{pq} \cup J }
{\gamma^l_{q} \cup  I }\ =  
\bangle  {n \cup \alpha^l_{pq} \cup  J }{(\beta^{l}_{p}-1)\cup\gamma^l_{q}\cup  I }\ 
\bangle  { J }{ I }
+  \bangle  {n \cup   J }{\gamma^l_{q}\cup I }\ \bangle  {\alpha^l_{pq} \cup  J }
{(\beta^{l}_{p}-1) \cup I }^{(t)}\ 
\end{align}
}
otherwise. Here $t=2$ if $p=n-1$ and $t=1$ otherwise.

To prove \eqref{last_1},\eqref{last_2}, we consider the  matrix $C=\bar U^{\hat J}_{\hat I}$.
If $\alpha_{pq}^l > c_{s+2}$, we let $\delta$ be the smallest index such that  entries of the first sub-diagonal of $C$ are zero in rows with the index larger than $\delta$. Next, we define $A$ to be the leading principal $\delta\times \delta$ submatrix of $C$ and  apply \eqref{jacobi} to $A$ with  $\alpha=\beta_{p}^l -1$, $\beta=\gamma_{q}^l$, $\gamma = \alpha_{pq}^l$ to obtain \eqref{last_1} after cancellations.

Next, we will show that  every factor of every term  featured in \eqref{last_2} is, in fact, equal to the minor of $A$ with the indicated index sets. Thus
the relation is simply an instance of the Desnanot--Jacobi identity \eqref{jacobi}. We will establish this property for $ \bangle  {n \cup  J }{(\beta^{l}_{pq} -1)\cup I }= f^{(l-1)}_{\romsi} (p,q)$. Other cores  featured in \eqref{last_2} can be treated in the same way.

Consider the submatrix $\bar C = C^{\hat n}_{\widehat {\beta^{l}_{p}-1}}$ of $C$. 
Using  the inequality $\alpha^{l}_{pq}+1-n\geq \beta^{l}_{p}-2$, which amounts to $q \leq n - l+1$, 
and the equality $n+\beta^{l}_{p}-r_{s+1}-3=\alpha^{l}_{pq}-\gamma^{l}_{q}$, 
one can check that the diagonal entries of $\bar C$ break into three contiguous segments 
\begin{align*}
&\bar u_{2,n+1}, \bar u_{3,n+2}, \ldots, \bar u_{\beta^{l}_{p}-2,n+\beta^{l}_{p}-3},\\ 
&\bar u_{r_{s+1}+1, n+\beta^{l}_{p} -2}, \bar u_{r_{s+1}+2,n+\beta^{l}_{p}-1}, \ldots, \bar u_{\gamma^{l}_{p}, \alpha^{l}_{pq}},\\ 
&\bar u_{r_{s+2} + 1, c_{s+2}+1}, \bar u_{r_{s+2} + 2, c_{s+2}+2}\ldots\ .
\end{align*} 

Since $\bar u_{r_{s+2} + 1, c_{s+2}+1}=0$, the submatrix of $\bar U$ whose determinant is equal to  $\bangle  {n \cup  J }{(\beta^{l}_{p}-1) \cup I }$ is also a submatrix of $\bar C$. In fact, it is a submatrix of the contiguous submatrix $C'$ of $\bar C$ whose lower right corner is 
$\bar u_{\gamma^{l}_{q}, \alpha^{l}_{pq}}$. 
To establish the needed property of $ \bangle  {n \cup  J }{(\beta^{l}_{p}-1) \cup I }$, it remains to show that $C'$ is irreducible. Let us prove 
that the entries of its first sub- and super-diagonal are nonzero.
From the construction of $\bar U$ we know that $\bar u_{i+2,n+i}$ and $\bar u_{i+2,n+i+2}$ are nonzero for all $i$, thus we only need to check the rows containing the diagonal elements from the second segment above.
They do have the required property, since they  are within the $(s+2)$nd strip of $U(X,X)$ and  
$c_{s+1} + 1 < n+ \beta^{l}_{p} -2 < c_{s+2} + 1$, which amounts to the inequality $l+2 < p < n$.

In the odd case, $l=2s-1$, we have to use the translation $\tau_3$ to re-write all expressions  $\bangle {[n-1]\cup J} {1 \cup I}$ as
$\left(\phhi^{w_0}_{n-1}\right)^{-1} {\bangle {[n]\cup 2n\cup  (J + n+1)} {1 \cup (I + n -1)}}^{(2)}$, where  
$I +i$ means that $i$ is added to every index in the set $I$. We also use the identity
\begin{align*}
&\phhi^{w_0}_{s(n+1) + 1} \thetta^{w_0}_{n-q-l+1}
 =  \bangle {[n]} {1\cup [s(n+1) + 3, r_{s+1}]} \thetta^{w_0}_{n-q-l+1}\\
&\qquad = {\bangle {[n] \cup [(s+2) n - q - s +4, c_{s+2}] } 
{1\cup [s(n+1) + 3, r_{s+1}] \cup [\gamma^{l+1}_{q}+2, r_{s+2}]}}^{(2)}.
\end{align*}

After this, relations that we need to establish again reduce to identities similar to  \eqref{last_1}, \eqref{last_2}, and
the same kind of  argument can be applied to verify them.
 \end{proof}
 
 To finish the proof of Theorem \ref{transform2}, we observe that the last two equations in \eqref{frozen_even} show
 that among the variables attached to frozen vertices in $Q_{\romsi}^{(2s)}$ there are functions $\pssi_i^{w_0}$ with
 $i\in [s(n+1)+3, (s+1) (n+1) + 2]$. Since $s$ ranges from $0$ to $\frac {n-5} 2 = k-3$, and since $\pssi_{1}^{w_0}$ and  $\pssi_{2}^{w_0}$ have been already obtained on stage I, 
 this means that we have recovered
 $\pssi_i^{w_0}$ for $i \in [1, (k-2)(n+1) + 2 = M]$. 
 Similarly, the last two equations in \eqref{frozen_odd} imply
 that among the variables attached to frozen vertices in $Q_{\romsi}^{(2s-1)}$ there are functions $\phhi_i^{w_0}$ with
 $i\in [s(n+1)+2, (s+1) (n+1) + 1]$ for $s$ ranging from $1$ to $\frac {n-6} 2=k-3$ and with $i\in [(k-2)(n+1)+2, (k-1) (n+1)=N]$ for $s= k-2$. Since $\phhi^{w_0}_i$ has been obtained on stage I for $i=1$,
 on stage III for $i\in [2,n-1]$ and on stage V for $i\in [n,n+2]$,  
 this means that  we have recovered $\phhi_i^{w_0}$ for all $i\in [N]$. Finally, $\thetta^{w_0}_i$
has been obtained on stage I for $i=1$ and on stage II for $i\in [2,n-1]$, whereas $\thetta^{w_0}_n$
is a stable variable. 

It remains to check that $Q^*_{CG}(n)$ defined at the beginning of Section~\ref{Ttrans} is isomorphic to $Q^{opp}_{CG}(n)=-Q_{CG}(n)$. By Proposition~\ref{end4path}, the vertices $(n-1,n-1)$ and $(n-2,n-1)$ of the quiver $Q_{\romtw}(n)$ are frozen, and the corresponding variables are $\pssi_1^{w_0}$ and $\pssi_2^{w_0}$, respectively. Therefore, in $Q^*_{CG}(n)$ they are moved to positions $(n,1)$ and $(n-2,n)$, respectively. The edge between these vertices in $Q^*_{CG}(n)$ is inherited from $Q_{\romon}^{(n-1)}$, hence, by Remark~\ref{first_last}(ii), it points from $(n,1)$ to $(n-2,n)$. It remains to observe that in $Q_{CG}(n)$ the edge points in the opposite direction.
Therefore,  the proof of Theorem \ref{transform2} is complete.
 
\section{Technical results on cluster algebras}

In this  Section we prove several technical lemmas that are used in the proofs of the main results.

\begin{lemma}
\label{Q-restore}
Let $(V, \Poi)$ be a Poisson variety possessing a coordinate system $\wx=(x_1,\dots,x_{n+m})$
that satisfies~\eqref{cpt} for some
$\omega_{ij}\in \mathbb{Z}$. Assume also that $V$ admits an action of  
$\left ( \mathbb{C}^*\right )^m$ that induces a local toric action of rank $m$ on $\wx$. 
Suppose there exists $\wB$ with a skew-symmetric irreducible principal part 
such that the cluster structure
$\CC(\wB)$ with  the initial extended cluster $\wx$ and stable variables $x_{n+1},\dots,x_{n+m}$  
is compatible with $\Poi$ and the local toric action above extends to a global toric action. 
Then $\wB$ is unique up to multiplication by a scalar.
\end{lemma}

\begin{proof} Let $\Omega^\wx=(\omega_{ij})_{i,j=1}^{n+m}$ and let $W$ be the weight matrix defined by the local toric action. By Proposition~\ref{Bomega} above and~\cite[Lemma~2.3]{GSV1}, 
compatibility conditions mean that $\wB$ satisfies equations $\wB \Omega^\wx = \left [ D\; 0\right ]$ 
and  $\wB W = 0$, where $D$ is an invertible diagonal matrix. 
The columns of $W$ span the kernel of $\wB$, since the rank of the global toric action is maximal. 
On the other hand,
the submatrix $\left (\Omega^\wx\right )^{[n]}$ is of full rank, and the span of its columns has a
 trivial intersection with the kernel of $\wB$. Define an $(n+m)\times (n+m)$ matrix $A = \left [ \left (\Omega^\wx\right )^{[n]}\; W \right ]$. Then $A$ is invertible and
$\wB A =  \left [ D\; 0\right ]$. Thus $\wB = D (A^{-1})_{[n]}$, and its principal part is irreducible 
if only if the principal part of  $(A^{-1})_{[n]}$ is. Therefore, 
$D=\lambda \one_{n}$ for some $\lambda \in \mathbb{Z}$,
since the principal part of $\wB$ is assumed to be skew-symmetric. This
completes the proof.
\end{proof}

\begin{remark}
\label{afterQ-restore}
{\rm (i) Let $Q$ be the quiver corresponding to $(A^{-1})_{[n]}$, where $A$ is defined in the proof above. Then Lemma~\ref{Q-restore} states that if  $(\wx, Q')$ is an initial seed of another cluster structure compatible with both the same Poisson bracket and the same toric action, then $Q'$ is obtained from $Q$ by replacing every edge with $\lambda$ copies of that edge if $\lambda$ is positive, or with $\lambda$ copies of the reversed edge if $\lambda$ is negative. In this case, we will write
$Q'=\lambda Q$.

(ii) It is easy to see that the formula for $\wB$ in the proof of Lemma \ref{Q-restore} depends on the column span of $W$ rather than on $W$ itself.
}
\end{remark}

 \begin{lemma}
 \label{twoloc}
Let $\CC$ be a cluster structure of geometric type and $\UU$ be the corresponding
upper cluster algebra.
Suppose $\frac{M_1}{f_1^{m_1}}=\frac{M_2}{f_2^{m_2}}=M$
for  $M_1, M_2\in \UU$, $m_1, m_2 \in \mathbb{N}$  and coprime cluster
variables $f_1\ne f_2$. Then $M\in \UU$.
 \end{lemma}
 
 \begin{proof}
 Let $x_1,...,x_n$ be cluster variables in some cluster $t$, and $L_t$ be the algebra of Laurent polynomials in $x_1,...,x_n$.  Then $M_1, M_2,f_1,f_2\in L_t$.

Consider $\tilde M=M_1 f_2^{m_2}=M_2 f_1^{m_1}\in \UU$.
Recall that  the algebra $L_t$  is a unique factorization domain and  $f_1, f_2$ are coprime.
Therefore $\tilde M=f_1^{m_1} f_2^{m_2} M$ with $M\in L_t$. Hence, $M$ is a Laurent polynomial in any cluster $t$, which accomplishes the proof.
\end{proof}

\begin{lemma} \label{MishaSha}
Let $\wB$ and 
$\hB$ be integer $n\times (n+m)$ matrices that differ in the last column only.
Assume that there exist $\tilde w, \hat w\in\C^{n+m}$ such that $\wB\tilde w=\hB\hat w=0$ and
$\tilde w_{n+m}=\hat w_{n+m}=1$. Then for any cluster $(x_1',\dots,x_{n+m}')$ in $\CC(\wB)$ there exists a collection of numbers
$\gamma_i'$, $i\in [n+m]$,  such that $x_i' x_{n+m}^{\gamma_i'}$ satisfy exchange relations of 
the cluster structure $\CC(\hB)$. In particular, for the initial cluster $\gamma_i=\hat w_i-\tilde w_i$,
$i\in [n+m]$.
\end{lemma}

\begin{proof} Apply the same sequence of cluster transformations  
to $(x_1,\dots,x_{n+m})$ and $(\hat x_1,\dots,\hat x_{n+m})$
to obtain cluster variables 
$\tilde y=\tilde g(x_1,\dots,x_{n+m})$ in $\CC(\wB)$ and $\hat y=\hat g(\hat x_1,\dots,\hat x_{n+m})$ in $\CC(\hB)$. Note that the restrictions of $\tilde g(z_1,\dots,z_{m+n})$ and $\hat g(z_1,\dots, \allowbreak z_{m+n})$
to $z_{n+m}=1$ coincide, since $\wB$ and $\hB$ differ only in the last column corresponding to a stable
variable.

By \cite[Lemma 2.3]{GSV1}, local toric actions with exponents $\tilde w$ and $\hat w$, respectively, can be extended to global toric actions, hence there exist $\tilde w_y'$ and $\hat w_y'$ such that 
$\tilde y d^{\tilde w_y'}=\tilde g(x_1d^{\tilde w_1},\dots,x_{n+m}d^{\tilde w_{n+m}})$ and 
$\hat yd^{\hat w_y'} =
\hat g(\hat x_1d^{\hat w_1},\dots,\hat x_{n+m}d^{\hat w_{n+m}})$ for any $d\in\C^*$. 
Assume $\hat x_i=x_ix_{n+m}^{\gamma_i}$, $i\in[n+m]$, for some $\gamma_i$ to be chosen later. For any 
fixed values of $x_1,\dots,x_{n+m}$ with $x_{n+m}\ne0$ we can take $d=x_{n+m}^{-1}$ to get
$$
\tilde y x_{n+m}^{-\tilde w_y'}=\tilde g(x_1x_{n+m}^{-\tilde w_1},\dots,x_{n+m}x_{n+m}^{-\tilde w_{n+m}})
$$ 
and 
$$
\hat y x_{n+m}^{-\hat w_y'} =
\hat g(x_1x_{n+m}^{\gamma_1-\hat w_1},\dots,x_{n+m}x_{n+m}^{\gamma_{n+m}-\hat w_{n+m}}).
$$
Let us choose $\gamma_i=\hat w_i-\tilde w_i$, $i\in [n+m]$. Then the arguments in the right hand sides of the above expressions coincide; moreover, the last argument equals $1$. Therefore, $\tilde y x_{n+m}^{-\tilde w_y'}=
 \hat y x_{n+m}^{-\hat w_y'}$, and hence $\hat y=\tilde y x_{n+m}^{\hat w_y'-\tilde w_y'}$ as required.
\end{proof}

  \begin{lemma}
 \label{antipoiss}
Let  $W_0$ be the matrix corresponding to the longest permutation $w_0$. The map $X
\mapsto W_0 X W_0$ is anti-Poisson with respect to
the Cremmer--Gervais Poisson bracket on $\Mat_n$.
 \end{lemma}
 
 \begin{proof} The map $-w_0$ permutes simple positive roots. For any
Belavin--Drinfeld triple
 $T=(\Gamma_1,\Gamma_2, \gamma)$, denote, $\Gamma_{1}^{w_0} = - w_0(\Gamma_{2})$ and
 $\Gamma_{2}^{w_0} = - w_0(\Gamma_{1})$.
 Define a map $\gamma^{w_0} \: \Gamma_{1}^{w_0} \to \Gamma_{2}^{w_0}$ by
$\gamma^{w_0} ( -w_0 (\gamma(\alpha)))=
 -w_0 (\alpha)$, $\alpha \in \Gamma_1$. Clearly,
$T^{w_0}=(\Gamma^{w_0}_1,\Gamma^{w_0}_2, \gamma^{w_0})$ is also a valid
Belavin--Drinfeld triple. If $r$ and $r^{w_0}$ are r-matrices corresponding to $T$
and $T^{w_0}$, then Proposition \ref{bdclass} implies that $\Ad_{w_0}\otimes
\Ad_{w_0} (r) = r^{w_0}_{21}$. This means that the conjugation by $w_0$ is an
anti-Poisson
 map from $(G,\{,\}_r)$ to $(G,\{,\}_{r^{w_0}})$. It remains to observe that for
$G=SL_n$ equipped with the Cremmer--Gervais Poisson bracket,
 $T$ and $T^{w_0}$ coincide.
 \end{proof}
 
\section*{Acknowledgments}
 M.~G.~was supported in part by NSF Grants DMS \#0801204 and  DMS \# 1101462. 
M.~S.~was supported in part by NSF Grants DMS \#0800671 and DMS \# 1101369.  
A.~V.~was supported in part by ISF Grant \#162/12. This paper was partially written during
our joint stays at MFO Oberwolfach (Research in Pairs program, August 2010), at the Hausdorff Institute
(Research in Groups program, June-August, 2011) and at the MSRI (Cluster Algebras program, August-December, 2012). We are grateful to these institutions for
warm hospitality and excellent working conditions. We are also grateful to our home institution for support in arranging visits by collaborators.
Special thanks are due to Bernhard Keller, whose quiver mutation applet proved to be an indispensable tool in our work on this project and to Sergei Fomin for illuminating discussions.

 \end{document}